\documentclass[11pt,a4paper, reqno]{amsart}
\usepackage{mathrsfs}

\usepackage{mathtools}

\usepackage{amsmath,amsfonts,verbatim}
\usepackage{latexsym}
\usepackage{amssymb,leftidx}
\usepackage{extarrows}
\usepackage{overpic}
\usepackage{color}
\usepackage{epsfig}
\usepackage{subfigure}
\usepackage{tikz}
\usetikzlibrary{graphs}
\usetikzlibrary{positioning}
\usepackage{tikz-cd}
\usepackage{caption}
\usepackage{yhmath}
\usepackage{geometry}

\newcommand{\mk}{\mathfrak}
\newcommand{\msf}{\mathsf}

\newcommand{\Econe}{X_0}
\newcommand{\WF}{\mathrm{WF}}

\newcommand{\supp}{\ensuremath{\mathrm{supp}}}
\newcommand{\sct}{\ensuremath{\mathrm{sc}}}

\newcommand{\conic}{\ensuremath{\mathrm{conic}}}
\newcommand{\Ymetric}{\mk{h}_0}
\newcommand{\XYmetric}{\mk{h}}

\newcommand{\Rprojlow}{\tilde{\Pi}_{\calclow}}

\newcommand{\Lprojlow}{\Pi_{\calclow}}
\newcommand{\Lprojhigh}{\Pi_{\calchigh}}
\newcommand{\hatLprojlow}{\hat{\Pi}_{\calclow}}
\newcommand{\hatLprojhigh}{\hat{\Pi}_{\calchigh}}
\newcommand{\RSlow}{\mathrm{RS}^{\calclow}}

\newcommand{\projhighlow}{\Pi_{\smf}}

\newcommand{\RYstar}{\mathscr{R}_Y^*}

\newcommand{\RY}{\mathscr{R}_Y}
\newcommand{\Ryz}{\mathscr{R}_{y_0}}

\newcommand{\RXstarb}{\mathscr{R}_{X,\rmb}^{*}}

\newcommand{\Diagb}{\mathrm{Diag_b}}

\newcommand{\Lsharplow}{L^\#}
\newcommand{\Lsharphigh}{G^{\#}}
 
\newcommand{\srange}{\mathsf{E}}
\newcommand{\srangehigh}{\mathsf{E}^{\calchigh}}

\newcommand{\ULCPlow}{\mathcal{U}_{\mathrm{LCP}}^{\calclow}}
\newcommand{\ULCPhigh}{\mathcal{U}_{\mathrm{LCP}}^{\calchigh}}
\newcommand{\Jlow}{\mathfrak{J}_{\calclow}}
\newcommand{\Jhigh}{\mathfrak{J}_{\calchigh}}

\newcommand{\calc}{\flat,\#} 
\newcommand{\calclow}{\flat} 
\newcommand{\calchigh}{\#} 
\newcommand{\rmb}{\mathrm{b}}

\newcommand{\NDb}{N^*\Delta_{\rmb}}
\newcommand{\NDhigh}{N^*\Delta_{\smf}}
\newcommand{\bundlehigh}{\Phi}

\newcommand{\IF}{\mathfrak{F}_X}
\newcommand{\IFint}{\mathfrak{F}_X^{\circ}}
\newcommand{\IFz}{\mathfrak{F}_{X_0}}
\newcommand{\IFintz}{\mathfrak{F}_{X_0}^{\circ}}

\newcommand{\la}{\ensuremath{\langle}}
\newcommand{\ra}{\ensuremath{\rangle}}

\newcommand{\LB}{\ensuremath{\mathrm{LB}}}
\newcommand{\RB}{\ensuremath{\mathrm{RB}}}
\newcommand{\BFS}{\mathrm{BF}}

\newcommand{\lb}{\ensuremath{\mathrm{lb}}}
\newcommand{\rb}{\ensuremath{\mathrm{rb}}}

\newcommand{\smf}{\mathrm{sf}}

\newcommand{\bfs}{\mathrm{bf}}

\newcommand{\zf}{\mathrm{zf}}

\newcommand{\paramconic}{\mathsf{P}}

\usepackage[colorlinks=true,linkcolor=blue,backref=page]{hyperref}
\hypersetup{urlcolor=blue, citecolor=red}

\usepackage{amssymb}
\usepackage{mathrsfs}
\usepackage{amscd}
\usepackage{bbm}
\usepackage{cite}

\newcommand\specm{dE_{\sqrt{P}}(\lambda)}

\newcommand\R{\mathbb{R}}

\newcommand\Z{\mathbb{Z}}
\newcommand\N{\mathbb{N}}
\newcommand\Id{\mathrm{Id}}

\newcommand\A{\textbf{A}}

\usepackage{graphicx}
\usepackage{float}

\usepackage{xcolor}
\usepackage{booktabs}

\numberwithin{equation}{section}
\newtheorem{proposition}{Proposition}[section]
\newtheorem{definition}{Definition}[section]
\newtheorem{lemma}{Lemma}[section]
\newtheorem{theorem}{Theorem}[section]
\newtheorem{corollary}{Corollary}[section]
\newtheorem{remark}{Remark}[section]

\begin{document}
\title[Geometric focusing and dispersive estimates]{The effect of geometric focusing on dispersive estimates for Schr\"odinger and wave equations}

\author{Qiuye Jia}
\address{Mathematical Sciences Institute, the Australian National University; }
\email{Qiuye.Jia@anu.edu.au; }

\author[Junyong Zhang]{Junyong Zhang}
\address{Junyong Zhang\newline
  Department of Mathematics, Beijing Institute of Technology, Beijing 100081}
\email{zhang\_junyong@bit.edu.cn}

\begin{abstract}
We classify the long-time decay rate in dispersive estimates for the Schr\"odinger and wave equations on non-trapping asymptotically conic manifolds and exact metric cones in terms of the intensity of geometric focusing. 
Letting $X_0$ be a metric cone, one of our main results demonstrates that each multiplicity of conjugate points within distance $\pi$ on $Y=\partial X_0$ leads to a $|t|^{1/2}$-loss in the long-time decay order and a half-order shift in the regularity index in the dispersive estimate for the Schr\"odinger equation.
Unexpectedly, conjugate point pairs on $Y$ at distance $\pi$ do not cause loss when the Legendre submanifold carrying the wave propagation satisfies a natural admissible condition that we propose. In sum, we give a robust framework for proving dispersive estimates that is stable under geometric perturbations and also accommodates perturbations by potentials.
\end{abstract}

\maketitle

\begin{center}
 \begin{minipage}{120mm}
 \end{minipage}
 \end{center}


\tableofcontents

\section{Introduction}

\subsection{The setup}
\label{subsec:set-up}

Dispersive estimates have been a central topic in the study of dispersive equations for decades (see, e.g. Bony-H\"afner\cite{Bony-Hafner}, Bouclet-Burq \cite{Bouclet-Burq}, Royer\cite{Royer}, Koch-Tataru\cite{Koch-Tataru}, Tataru \cite{Tataru}, Schlag \cite{Schlag} and the references therein). However, deriving $L^1-L^\infty$ decay estimates on curved spaces poses new significant challenges and the literature remains sparse. 
This is due to the fact that the refocusing effect of the geodesic flow can change the dispersion mechanism of the PDE compared with the flat setting.
More precisely, the energy of solutions to Schr\"odinger and wave equations propagates along the geodesic flow of the background space, so when geodesics rejoin, caustic-like concentration of energy can appear and lead to a slower decay rate compared with the flat case.

 In this paper, we investigate the question of how the focusing effect of the geodesic flow affects pointwise decay estimates on manifolds. 
We present in this paper a classification of decay rates for solutions to the Schrödinger and wave equations on asymptotically conic manifolds, which to some extent quantifies the propagation of singularities for the Schrödinger equation \cite{Hassell-Tao-Wunsch,hassell-wunsch2005schrodinger}.
This classification is governed by the strength of the focusing effect of the geodesic flow, which we define precisely in the sequel \eqref{eq:IF-full-def}. 
In fact, as we will discuss in Remark~\ref{remark:microlocalized-est-Schrodinger}, we can have even more refined microlocal information: one can partition the spectral measure with each piece associated with a part of the geodesic flow (see Section~\ref{subsec:microlocal-partition-low} and Section~\ref{subsec:microlocal-partition-high-combined}), and the decay rate in the dispersive estimate will depend only on the focusing effect associated with this part of the geodesic flow.

Let $Y$ be a compact manifold without boundary. The cone over $Y$ is $X_0=C(Y)=(0,\infty)_r\times Y$. Here $Y$ is called the cross section or link of $X_0$.
A conic metric on $X_0$ is one of the form
\begin{equation}\label{eq:def-exact-conic-metric}
g_0 = dr^2+r^2 \Ymetric(y),
\end{equation}
where $\Ymetric$ is a Riemannian metric on $Y$. An asymptotically conic metric is one on a manifold $X^\circ$ which, outside a compact region in the interior, is identified with $(r_0,\infty)_r\times Y$, with a metric that on this conic end tends to $g_\infty$ as $r\to\infty$ in a specified way. An example is the Euclidean metric, for which the cross section is the standard
sphere, and indeed metrics asymptotic to the Euclidean one at
infinity, or more generally to perturbations of the Euclidean metric by changing the metric on the link, namely the sphere `at infinity'.

For our purposes, it is useful to compactify our space. Let $x=r^{-1}$, so $r\to\infty$ corresponds to $x\to 0$. We add a boundary $\{0\}_x\times Y$ to $X^\circ$, thus compactifying it to $X$. An asymptotically conic metric, as introduced by Melrose \cite{Melrose1994}, is a Riemannian metric on $X$ which is of the form
\begin{equation}\label{conic_metric_2}
g=\frac{dx^2}{x^4}+\frac{\XYmetric(x,y)}{x^2}
\end{equation}
near $\partial X$, where $\XYmetric$ is a smooth symmetric 2-cotensor on $X$; $g$ is thus asymptotic to the metric of the exact cone with $\Ymetric = \XYmetric|_{x=0}$ on the cross section $Y$.
In this paper, our metric $g$ will always be assumed to be \emph{non-trapping}: all unit speed geodesics on $X$ tend to $\partial X$ eventually.
This requirement enters in the high-energy regime, where the geodesic flow over the interior of $X$ matters as well.

Let $\Delta_g \geq 0$ be the (positive) Laplacian on $X$ associated with $g$. 
Though our main interest lies in the interplay between the geometry of $X$ and the dispersion effect, we still allow quite general potential perturbations that include potentials that decay at the scaling-critical level and are not necessarily small. Concretely, the operator we will consider is
\begin{equation} \label{eq:P-def}
P = \Delta_g + V,
\end{equation}
where $V \in x^2C^\infty(X)$ and the restriction on it arises from the b-operator living at zero (scattering) energy at infinity: 
\begin{equation}
P_{\rmb} = x^{-1}Px^{-1} =  -(x\partial_x)^2 + \Delta_{\Ymetric} + \frac{(n-2)^2}{4} + V_0+xW,
\end{equation}
where $V_0 = x^{-2}V|_{x=0}$, $W \in \mathrm{Diff}^2_{\rmb}(X)$.    
The order of the spectral measure, as a Legendre distribution, will be affected by the elliptic parametrix of $P_{\rmb}$, which in turn will be affected by the eigenvalues of the operator 
\begin{equation} \label{eq:def-PY}
P_Y = \Delta_{\Ymetric} + \frac{(n-2)^2}{4} + V_0
\end{equation}
on the cross section $Y$.
In this paper, we make the assumption that $V_0$ is positive in the sense that the smallest eigenvalue $\nu_0^2$ of $P_Y$ satisfies
\begin{equation} \label{eq:condition-nu-0} 
\nu_0 \geq (n-2)/2.
\end{equation}
In particular, this condition is satisfied for any $V \in x^{2+\epsilon}C^\infty(X)$ with $\epsilon>0$ or any $V_0 \geq 0$.
We also make the assumption that 
\begin{equation} \label{assumption:P-positive}
P \text{ has no non-positive eigenvalues or zero-resonance. }    
\end{equation}
By a zero-resonance, we mean a solution to $Pu=0$ with $u \notin L^2(X)$.

Our analysis could also handle operators on exact metric cones with inverse square potentials. Recall that an exact cone is $X_0 = (0,\infty)_r \times Y$ mentioned above equipped with a metric of the form \eqref{eq:def-exact-conic-metric}, and the operator in this case takes the form
\begin{equation}  \label{eq:P-def-exact-cone}
    P = \Delta_{g_0} + r^{-2}V_0(y).
\end{equation}
We still impose conditions \eqref{eq:condition-nu-0}  and \eqref{assumption:P-positive}.
Notice that this is not covered by the asymptotically conic manifolds directly since both the potential and the manifold have singularities as $r \to 0$.
The singularity of the potential is straightforward to see. 
To see the singularity of the manifold itself, a direct computation shows that the sectional curvature of such $X_0$ has an $r^{-2}$ level singularity as $r \to 0$ in general. Such a singularity can be removed only if the cross section $Y$ has constant sectional curvature $1$, which means that $Y=\mathbb{S}^{n-1}$ and $X_0$ is the Euclidean space with the origin removed. 
In this case over exact metric cones, we still impose the requirement \eqref{eq:condition-nu-0}.

\begin{remark}
When \eqref{eq:condition-nu-0} is not satisfied, one can still obtain the same type of oscillatory integral representation of the spectral measure but with orders shifted. See \cite[Corollary~1.5, \, Section~3,7]{GHS2} about how $\nu_0$ enters the order of the spectral measure and decay of propagators of PDEs.
With this expression, one can obtain the corresponding dispersive estimates with loss. 
Because this phenomenon has been addressed in \cite{JZ1,JZ2,JZ3} and the classification according to the geometric scenario is already complicated enough, the details about the effect of $V_0$ violating \eqref{eq:condition-nu-0} will be pursued elsewhere.
\end{remark}

In our setting, since we allow fairly general geometric behaviour over any fixed compact region in the interior of $X$, bicharacteristic lines can 
be confined to a bounded spatial region. 
These bicharacteristic lines will lead to the phenomenon that the energy will remain concentrated and violate the dispersive estimates. 
To preclude this, we assume that our metric $g$ is non-trapping in the following sense.

\begin{definition} \label{def:non-trapping}
Let $(X,g)$ be an asymptotically conic manifold. We say that $(X,g)$ is non-trapping if any unit speed geodesic with respect to $g$ escapes any bounded region $K \subset X^{\circ}$ in finite time.
\end{definition}

We study the decay property of solutions to the Cauchy problems of the Schr\"odinger equation:
\begin{equation} \label{eq:IVP-Schrodinger}
\begin{cases}
(i\partial_t+P) u=0,\\
u|_{t=0}=u_0;
\end{cases}
\end{equation}
and the wave equation:
\begin{equation} \label{eq:IVP-wave}
\begin{cases}
(\partial_t^2 + P) u=0,\\
u|_{t=0}=u_0,\quad \partial u|_{t=0}=u_1.
\end{cases}
\end{equation}

Concretely, we will investigate the dispersive estimates that their solutions satisfy.

\subsection{The focusing effect of the geodesic flow and main results}
\label{subsec:the_focusing_effect_of_the_geodesic_flow_and_main_results}

Now we discuss how we quantify the focusing effect of the geodesic flow on $X$, which gives a criterion for the dispersion effect on asymptotically conic manifolds in terms of a simple geometric quantity. 
Roughly speaking, the quantity that we use to quantify the focusing effect of the geodesic flow is the maximum of the multiplicity of conjugate point pairs on $X$.

For simplicity, we only consider the part restricted to $ \partial X = Y$ and postpone the full definition to Section~\ref{subsec:focusing-effect}.

We will show that only conjugate points at distance less than or equal to $\pi$ on $Y$ will affect the dispersive estimate.
The reason for this range arises from the structure of the geodesic flow of $X$, which we briefly recall next. 
First, the scattering cotangent bundle ${}^{\sct}T^*X$ on $X$ is characterized as follows. Let $x$ be the defining function of $\partial X$ at infinity as in \eqref{conic_metric_2} and $y$ be a coordinate system on $Y = \partial X$. Then over the interior, ${}^{\sct}T^*X$ is the same as the usual cotangent bundle $T^*X^{\circ}$, while near infinity it takes
\begin{equation} \label{eq:sc-frame}
\frac{dx}{x^2}, \; \frac{dy}{x}
\end{equation}
as a local frame and we write the canonical one form as
\begin{equation} \label{eq:sc-1-form}
 \tau \frac{dx}{x^2} + \mu \cdot \frac{dy}{x}.
\end{equation}
It has a natural symplectic structure `extending' the one on $T^*X^{\circ}$. When we consider Hamilton vector fields of a function $\bullet$ on ${}^{\sct}T^*X$, this will be the symplectic structure with respect to which it is defined and we denote it by $H_{\bullet}$.
Then the threshold $\pi$ arises from the model case when $X$ is an exact metric cone near infinity.
Let $g_0$ be as in \eqref{eq:def-exact-conic-metric}, which is the exact conic metric that \eqref{conic_metric_2} is asymptotic to. Here $\Ymetric$ is the 2-tensor obtained by extending $\XYmetric|_{x=0}$ so that it is independent of $x$ near $\partial X$.

The computation in \cite[Section~3]{melrose1996scattering} shows that the rescaled Hamilton vector of $G_0=g_0^{-1}(\zeta,\zeta)$ (with $g_0^{-1}$ being the inverse of $g_0$,
and $\zeta=(\tau,\mu)$) takes the form
\begin{align} \label{eq: rescaled HG}
\msf{H}_{G_0}:=\frac{1}{2}x^{-1}|\mu|^{-1}H_{G_0} = \tau \frac{x}{|\mu|}\partial_x - |\mu|^{-1}\Ymetric(y,\mu)\partial_\tau + \tau \frac{\mu}{|\mu|} \cdot \partial_{\mu} + \frac{1}{2|\mu|}H_{\Ymetric}.
\end{align}
Here the rescaling is chosen so that it has unit speed on the cross section $Y$. Correspondingly, the rescaled bicharacteristic lines of $g_0$, i.e., integral curves of $\msf{H}_{G_0}$ on conic manifolds take the form:
\begin{equation}\begin{aligned}\label{eq:conic-bichar}
    & x=\frac{x_0}{\sin s_0}\sin (s+s_0),\ \tau=\cos (s+s_0),\ |\mu|=\sin (s+s_0),\\
    & (y,\hat\mu)=\exp(sH_{\frac{1}{2} \Ymetric })(y_0,\hat\mu_0),\ s\in(-s_0,-s_0+\pi),
\end{aligned}\end{equation}
where $\tau,\mu$ are scattering frequencies as in \eqref{eq:sc-1-form} and $\exp(sH_{\frac{1}{2} \Ymetric })$ means $s$-time flow along this vector field.
We refer the reader to \cite[Section~3]{melrose1996scattering} for more discussions about the dynamics. It is then straightforward to see from \eqref{eq:conic-bichar} that only geodesic flow within time $\pi$ on $Y$ enters the geodesic flow on $X$, which in turn captures the main part of the propagation of waves on $X$. Indeed, as shown in \cite{JZ1}, the geometric focusing on $Y$ at time strictly larger than $\pi$ does not affect dispersive estimates on conic manifolds, while the focusing exactly at time $\pi$ is more delicate to analyze and we will discuss it in Section~\ref{subsec:Strategy-of-proof}.

Due to the extra complication when we take the interior of $X$ into consideration, we postpone the precise definition of the \emph{indices of geometric focusing} $\IF$ and $\IFint$ to Section~\ref{subsec:focusing-effect}. Roughly speaking, $\IF$ and $\IFint$ are both the maximal multiplicity of conjugate points on $X$ except that $\IF$ counts conjugate point pairs in which both points may lie on the boundary, whereas $\IFint$ only counts conjugate point pairs with at most one point on the boundary. 
Then our main result is that $\IF$ (or $\IFint$) is precisely the change of decay rate or the power in weights in dispersive estimates compared with the Euclidean case. We state two types of results concerning the Schr\"odinger equation \eqref{eq:IVP-Schrodinger} first. The effect of geometric focusing is reflected in different ways in them. We start with the pointwise estimate on the propagator. 
\begin{theorem}[Estimate for the Schr\"odinger propagator]
\label{thm:Schrodinger-propagator-est-1}
Let $(X,g)$ be an asymptotically conic manifold of dimension $n \geq 3$ that is non-trapping in the sense of Definition~\ref{def:non-trapping}, let
$P$ be as in \eqref{eq:P-def}, and let $\IF$ be as in \eqref{eq:IF-full-def} below, then
\begin{equation} \label{eq:Schrodinger-kernel-1}
|e^{itP}(z,z')| \lesssim |t|^{-\frac{n}{2}}\Big(1+\big(\frac{\la z \ra \la z' \ra}{|t|}\big)^{\IF/2}\Big).
\end{equation}
If in addition $X$ is admissible in the sense of Definition~\ref{definition:Lbf-boundary-admissible}, then $\IF$ above can be replaced by $\IFint$ in \eqref{eq:IFint-def}:
\begin{equation} \label{eq:Schrodinger-kernel-2}
|e^{itP}(z,z')| \lesssim |t|^{-\frac{n}{2}}\Big(1+\big(\frac{\la z \ra \la z' \ra}{|t|}\big)^{\IFint/2}\Big).
\end{equation}
For $P$ on an exact cone $(X_0,g_0)$ as in \eqref{eq:P-def-exact-cone}, \eqref{eq:Schrodinger-kernel-1} (resp. \eqref{eq:Schrodinger-kernel-2}) holds with $\IF$ (resp. $\IFint$) replaced by $\IFz$ (resp. $\IFintz$) defined in \eqref{eq:IFz-def} (resp. \eqref{eq:IFintz-def}).
\end{theorem}

See Corollary~\ref{coro:Schrodinger-propagator-est-3} for the proof of Theorem~\ref{thm:Schrodinger-propagator-est-1}.
In each setting listed above, note that
\begin{equation}
\Big(1+\big(\frac{\la z \ra \la z' \ra}{|t|}\big)^{\IF/2}\Big) \sim \big(1+\frac{\la z \ra}{|t|^{1/2}}\big)^{\IF/2}\big(1+\frac{\la z' \ra}{|t|^{1/2}}\big)^{\IF/2},
\end{equation}
the pointwise estimate on the kernel implies the following dispersive estimates by pairing with the initial data in $u_0$, multiplying corresponding weights, and integrating in $z'$ directly.

\begin{theorem}  \label{thm:dispersive-Schrodinger-1}
With the notation as above, suppose that $u$ is the solution to \eqref{eq:IVP-Schrodinger}, then
\begin{equation} \label{eq:est-dispersive-1-1}
\| \big(1+\la z \ra/|t|^{1/2}\big)^{-\IF/2} u(t,\cdot) \|_{L^\infty} \leq C 
|t|^{-\frac{n}{2}} \| (1+\la z \ra/|t|^{1/2})^{\IF/2} u_0 \|_{ L^1 } .
\end{equation}
If in addition $X$ is admissible in the sense of Definition~\ref{definition:Lbf-boundary-admissible}, then $\IF$ above can be replaced by $\IFint$ in \eqref{eq:IFint-def}:
\begin{equation} \label{eq:est-dispersive-1-2}
\|\big(1+\la z \ra/|t|^{1/2}\big)^{-\IFint/2} u(t,\cdot) \|_{L^\infty} \leq C 
|t|^{-\frac{n}{2}} \| \big(1+\la z \ra/|t|^{1/2}\big)^{\IFint/2} u_0 \|_{ L^1 } .
\end{equation}
For an exact cone $(X_0,g_0)$, the same estimate as above holds, except that $\IF$ and $\IFint$ are replaced by $\IFz$ in \eqref{eq:IFz-def} and $\IFintz$ in \eqref{eq:IFintz-def}, respectively.
\end{theorem}

Some typical examples that this main theorem treats are as follows.
\begin{corollary} \label{coro:after-main-thm}
Let the cross section $Y$ be the unit sphere $\mathbb{S}^{n-1}$. Then $X_0$ is $\R^n\backslash\{0\}$ equipped with the flat metric and it is admissible with $\IFintz=0$. So \eqref{eq:est-dispersive-1-2} with $\IFint$ replaced by $\IFintz$ gives the dispersive estimate on $\R^n$, allowing potentials like $r^{-2}V_0(y)$ satisfying \eqref{eq:condition-nu-0}. 
Let $Y$ be a closed manifold with sectional curvature less than $1$. Then $X_0=C(Y)$ satisfies $\IFz=\IFintz=0$. Hence \eqref{eq:est-dispersive-1-1} shows that the solutions of the Schr\"odinger equations on such manifolds satisfy the same dispersive estimate as on $\R^n$.
\end{corollary}
\begin{remark}
  We postpone the proof of Corollary~\ref{coro:after-main-thm} to the end of Section~\ref{subsec:dispersive-conic-pair-contribution}.
  The first part of Corollary~\ref{coro:after-main-thm} answers the question in \cite[Remark~1.12]{Fenelli-etal-timedecay} on the dispersive estimates for inverse square potentials. Concretely, Corollary~\ref{coro:after-main-thm} generalizes \cite[Theorem~1.11.(i)]{Fenelli-etal-timedecay}, which assumes $n=3$ and that $V_0(y)$ is a constant.
One can see from the proof of Corollary~\ref{coro:after-main-thm} that its second part holds under the weaker condition that conjugate points on $Y$ have distance larger than $\pi$. In this way, we partially recover the main result of \cite{JZ1} when \eqref{eq:condition-nu-0} is satisfied.
\end{remark}

\begin{remark}
    As one can see from the main technical result, Theorem~\ref{thm:Schrodinger-propagator-est-1}, and the proof that deduces Theorem~\ref{thm:dispersive-Schrodinger-1} from that, \eqref{eq:est-dispersive-1-1} holds with the following more general family of weights:
\begin{equation}
\| w_1(z,t)^{-1} u(t,\cdot) \|_{L^\infty} \leq C 
|t|^{-\frac{n}{2}} \| w_2(z,t) u_0 \|_{ L^1 } ,   
\end{equation}
where $w_1(z,t), w_2(z',t) \gtrsim 1$ and
\begin{equation}
 w_1(z,t) w_2(z',t) \gtrsim   1+(\la z \ra \la z' \ra/|t|\big)^{\IF/2}.
\end{equation}
The same generalization applies to \eqref{eq:est-dispersive-1-2} with $\IF$ replaced by $\IFint$, and to the case of exact cones as well. 
We are stating the particular choice in Theorem~\ref{thm:dispersive-Schrodinger-1} because it is the most symmetric version.
\end{remark}

The second way to encode the loss caused by geometric focusing uses the homogeneous Sobolev space associated with $P$. That is, for $m \in \N$, we define $\dot{W}^{1,m}$ to be the completion of the Schwartz function class with respect to the norm
\begin{equation}
\| u \|_{\dot{W}^{1,m}} = \| P^{m/2} u \|_{L^1}.
\end{equation}
This is indeed a norm under the assumption \eqref{assumption:P-positive}. Then our result is as follows.

\begin{theorem}[Estimate for the Schr\"odinger propagator, the second version]
\label{thm:Schrodinger-propagator-est-2}
For $P$ in \eqref{eq:P-def} on $(X,g)$ and $\IF$ defined by \eqref{eq:IF-full-def}, we have
\begin{equation} \label{eq:Schrodinger-propagator-est-2-1}
\big|e^{itP}(z,z') P^{-\IF/2}\big| \lesssim |t|^{-\frac{n-\IF}{2}}.
\end{equation}
If in addition $X$ is admissible in the sense of Definition~\ref{definition:Lbf-boundary-admissible}, then $\IF$ above can be replaced by $\IFint$ in \eqref{eq:IFint-def}:
\begin{equation} \label{eq:Schrodinger-propagator-est-2-2}
\big|e^{itP}(z,z') P^{-\IFint/2}\big| \lesssim |t|^{-\frac{n-\IFint}{2}}.
\end{equation}
If $P$ is instead defined in \eqref{eq:P-def-exact-cone} on an exact cone $(X_0,g_0)$, then the same estimate holds with $\IF$ (resp. $\IFint$) replaced by $\IFz$ (resp. $\IFintz$) defined in \eqref{eq:IFz-def}.
\end{theorem}
See Corollary~\ref{coro:Schrodinger-propagator-est-4} for the proof of Theorem~\ref{thm:Schrodinger-propagator-est-2}.
In the same way as the discussion after Theorem~\ref{thm:Schrodinger-propagator-est-1}, this implies the first part of Theorem~\ref{thm:dispersive-Schrodinger-2}, i.e. \eqref{eq:est-dispersive-2-1}, by pairing with the initial data in $u_0$ and integrating in $z'$ directly.

\begin{theorem}  \label{thm:dispersive-Schrodinger-2}
Suppose $(X,g)$ is an asymptotically conic manifold that is non-trapping in the sense of Definition~\ref{def:non-trapping}, let $\IF$ be as in \eqref{eq:IF-full-def} and let $u$ be the solution to \eqref{eq:IVP-Schrodinger}; then we have
\begin{equation} \label{eq:est-dispersive-2-1}
\| u(t,\cdot) \|_{L^\infty} \leq C |t|^{ -\frac{n-\IF}{2} } \| u_0 \|_{ \dot{W}^{1,\IF}  } .
\end{equation}
If in addition $X$ is admissible in the sense of Definition~\ref{definition:Lbf-boundary-admissible}, then $\IF$ above can be replaced by $\IFint$ in \eqref{eq:IFint-def}:
\begin{equation} \label{eq:est-dispersive-2-2}
\|  u(t,\cdot) \|_{L^\infty} \leq C |t|^{ -\frac{n-\IFint}{2} } \| u_0 \|_{ \dot{W}^{1,\IFint}  } .
\end{equation}
For an exact cone $(X_0,g_0)$, the same estimate holds as above, except that $\IF$ and $\IFint$ are replaced by $\IFz$ in \eqref{eq:IFz-def} and $\IFintz$ in \eqref{eq:IFintz-def}, respectively.
\end{theorem}

\begin{remark} 
We illustrate the intuition behind the numerology by giving a few examples of conic manifolds.
Let $X_0$ be an exact cone whose cross section $Y$ is a sphere $\mathbb{S}^{n-1}_\varsigma$ of radius $\varsigma$.
When $\varsigma > 1$, as shown in Proposition~\ref{prop:example-admissible} below, $X_0$ is admissible with $\IFz = 0$, $\IFintz = 0$. 
In the case $\varsigma=1$, $X_0$ recovers $\R^n$ as a special case with $\IFz = n-2$, $\IFintz = 0$. 
Notice that the singularity of the cone at $r=0$ is removable in this case, and we can interpret it as an asymptotically conic manifold with $\IF = n-2,\IFint=0$ as well.
When $\varsigma<1$, such $X_0$ with cross section $\mathbb{S}^{n-1}_\varsigma$ is also admissible, but with $\IFintz = n-2$. In this case, $\IFz$ does not enter the dispersive estimate directly since it is admissible and the conclusion on $\IFz$ is more involved,  
see Proposition~\ref{prop:example-admissible} for details.
The main result of \cite{Taira-dispersive-cone} shows that in this case, the estimates in Theorem~\ref{thm:dispersive-Schrodinger-1} and Theorem~\ref{thm:dispersive-Schrodinger-2} are saturated and the loss is optimal.
The method of \cite{Taira-dispersive-cone} relies on explicit eigenfunctions on $Y=\mathbb{S}_{\varsigma}^{n-1}$ and the expression of the propagator $e^{itP}$ near antipodal points, and therefore cannot be applied to general cases.
\end{remark}
\begin{remark} In particular, the result also shows that all cones over a cross section $Y$ with sectional curvature less than $1$ have the same decay rate in the dispersive estimate as in the Euclidean case.
\end{remark}

\begin{remark} 
Local energy decay estimates (i.e. $L^2\to L^2$ estimates with spatial weights) in a similar geometric setting are proved by Bony–Häfner \cite{Bony-Hafner}, Bouclet–Burq \cite{Bouclet-Burq}, Grasselli \cite{Grasselli-conic-decay}, and Royer \cite{Royer}. The results here provide pointwise estimates, which differ from theirs.
\end{remark}

The focusing of the geodesic flow plays a similar role in the dispersive estimate of the wave equation, which is our next main result. We will state results concerning $e^{\pm it\sqrt{P}}$ for convenience, but they imply dispersive estimates for Cauchy problems \eqref{eq:IVP-wave} just by some simple combinations like $\cos(t\sqrt{P}) = \frac{1}{2}(e^{it\sqrt{P}}+ e^{-it\sqrt{P}})$.

\begin{theorem}  \label{thm:dispersive-wave}
Let $\IF$ be as in \eqref{eq:IF-full-def}, $u_0 \in L^1(X)$, and suppose that $\varphi_K(\lambda) \in C_c^\infty(\R)$ is supported in the region $\lambda \sim 2^K$ with $K\in\Z$, then
\begin{equation} \label{eq:dispersive-wave-main-1}
    \| \varphi_K(\sqrt{P}) e^{it\sqrt{P}} u_0 \|_{L^\infty}   \lesssim 
    2^{Kn } (1+2^K|t|)^{-\frac{n-1}{2}}
    (1+2^K|t|)^{\frac{\IF}{2}} \| \varphi_K(\sqrt{P})  u_0 \|_{L^1}.
\end{equation}
If in addition $X$ is admissible in the sense of Definition~\ref{definition:Lbf-boundary-admissible}, then $\IF$ above can be replaced by $\IFint$ in \eqref{eq:IFint-def}:
\begin{equation} \label{eq:dispersive-wave-main-2}
    \| \varphi_K(\sqrt{P}) e^{it\sqrt{P}} u_0 \|_{L^\infty}   \lesssim 
    2^{Kn} (1+2^K|t|)^{-\frac{n-1}{2}}
    (1+2^K|t|)^{\frac{\IFint}{2}} \| \varphi_K(\sqrt{P})  u_0 \|_{L^1}.
\end{equation}
For an exact cone $(X_0,g_0)$, we have the same estimate as above except that $\IF$ and $\IFint$ are replaced by $\IFz$ in \eqref{eq:IFz-def} and $\IFintz$ in \eqref{eq:IFintz-def} respectively.
\end{theorem}

We conclude this part by discussing the magnetic potential. Our analysis can handle certain magnetic potentials on $\R^n$ as well, which in particular includes the famous Aharonov-Bohm type potential\footnote{The physically well-known Aharonov-Bohm potential originates in the two dimensional setting. While our analysis is carried out in three or higher dimensions.}. The corresponding Schr\"odinger operator takes the form
\begin{align} \label{eq:magnetric-Schrodinger}
\Big(i\nabla+\frac{{\A}(\hat{z})}{|z|}\Big)^2+\frac{a(\hat{z})}{|z|^2}, \quad \A \cdot \hat{z} = 0.
\end{align}
We are not including this in the main results because the main theme of this article is about the focusing effect of the geodesic flow, and we will remove the spectral condition \eqref{eq:condition-nu-0} to investigate the effect of attractive potentials on the dispersive estimate in a forthcoming work \cite{JZ-magnetic}.
The extension to operators taking the form \eqref{eq:magnetric-Schrodinger} can be carried out as follows.
As pointed out in \cite[Remark~4.2]{GHS2}, the construction of the resolvent and in turn that of the spectral measure, only relies on the condition that the sub-principal symbol of the operator vanishes at $\Lsharplow$, which is the Legendre submanifold carrying the the purely outgoing or incoming oscillation.
In terms of the construction in \cite{hassell2001resolvent}, this corresponds to the fact that the conclusion of \cite[Lemma~2.3]{hassell2001resolvent} still holds in this setting even though the short-range condition there is violated. Then the construction of the resolvent goes through without extra complications.


\subsection{Strategy of the proof}
\label{subsec:Strategy-of-proof}

We discuss the strategy of the proof in this subsection. 
We will use the characterization of the spectral measure of $P$ given in \cite{GHS2,Hassell-Wunsch-semiclassical-resolvent}, which we recall here informally in the case \eqref{eq:condition-nu-0}. 
\begin{theorem}\label{thm:spectral-measure-concise}
The spectral measure associated with $P$ in \eqref{eq:P-def} is a (conormal) Legendre distribution associated with the intersecting pair of Legendre submanifolds with conic points $(L^{\bfs}, \Lsharplow)$ in the low-energy regime, and is a Legendre distribution associated with the intersecting pair of Legendre submanifolds with conic points $(L^{\bfs,\calchigh},\Lsharphigh)$ in the high-energy regime.
For $P$ in \eqref{eq:P-def-exact-cone}, its spectral measure is the same type of Legendre distribution associated with $(L^{\bfs}, \Lsharplow)$ in both low and high-energy regimes.
\end{theorem}

We refer the reader to Theorem~\ref{thm:spectral-measure-complete} for a complete version.
This is a concise way to say that the spectral measure is a sum of oscillatory integrals that we will describe in detail in Section~\ref{sec:microlocalized-spectral-measure}. 
The key observation that allows us to quantify the effect of geometric focusing (which is quantified by the degeneracy of the exponential map) is that the number of extra parameters in a non-degenerate parametrization of a Legendre distribution\footnote{One can think of this as a Fourier integral operator, but with the `boundary face' being the spatial infinity instead of fiber infinity.}  is equal to the rank drop of the projection from the `propagating Legendrian' (see the discussion about \eqref{eq:Lbf-definition-gamma^2}), which in turn equals the rank drop of the exponential map.

The order of the spectral measure and also the propagators of Schr\"odinger or wave equations as Legendre distributions or Fourier integral operators, or their mapping property on $L^2$-based Sobolev spaces, does not depend on the concrete behaviour of the geodesic flow.
On the other hand, as one can see from the definitions in Section~\ref{sec:microlocalized-spectral-measure}, the number of parameters in the parametrization affects the $L^\infty$-type bound on spectral measures in a very straightforward way: it becomes worse by $\frac{1}{2}$ order in terms of the vanishing order at all boundary surfaces whenever there is one extra parameter\footnote{The numerology of Fourier integral operators in \cite{hormander2009analysis} says that the symbol order of the amplitude becomes lower by $\frac{1}{2}$-order whenever there is one more parameter, which seems to say that the $L^\infty$-behaviour of the kernel becomes better. But this is an illusion: the volume form carries one order growth. In sum, the same phenomenon occurs: the $L^\infty$-bound of the kernel gets worse by $\frac{1}{2}$-order.}. One of the major goals of this article is to translate this into the effect on the dispersion of solutions to the Schr\"odinger and wave equations.

On the other hand, for conjugate points at distance greater than $\pi$ on $Y$, we see from \eqref{eq:conic-bichar} that they are not connected by geodesics on $X$. Indeed, our results confirm that those conjugate points do not affect the decay rate in dispersive estimates.
As aforementioned, this has been confirmed in \cite{JZ1} from a different approach in the exact conic case in the presence of singular potentials.

 
Now we turn to discuss the effect near `endpoints' of geodesics, which turns out to be more delicate than the interior part.
From \eqref{eq:conic-bichar}, it seems that the geometric focusing at time $\pi$ (in terms of the geodesic flow on $Y$) will enter the geodesic flow on $X$ and potentially will affect dispersive estimates.
However, the dispersive estimate on the flat Euclidean space is well-known and this corresponds to the case $Y = \mathbb{S}^{n-1}$, which corresponds to the case where $Y$ has geometric focusing at time $\pi$ but no such focusing before time $\pi$. 
In this case, not only the multiplicity of conjugate points itself, but the precise characterization of the geodesic flow also plays a role.
So we will seek a general condition that enables one to `ignore' the focusing effect of the geodesic flow at time $\pi$ on $Y$, which corresponds to the geodesic flow on $X$ from one end at infinity to the other end, in dispersive estimates.
Such a condition is given in Definition~\ref{definition:Lbf-boundary-admissible} mentioned above in our main results. Roughly speaking, it says that the focusing of the geodesic flow on $Y$ at time $\pi$ is uniform in all directions.

Technically, estimating the contribution of this part  comes down to the pointwise estimate of the piece associated with the intersecting Legendre pair with conic points, see Section~\ref{subsec:conic-pair-geometry-and-phase-function} and  Section~\ref{subsec:Legendrian-distributions-high} for its precise local expression as oscillatory integrals. 
This was introduced by Melrose-Zworski \cite{melrose1996scattering} and further extended by Hassell-Vasy \cite{hassell1999spectral,hassell2001resolvent} to the case with codimension two corners and by Guillarmou-Hassell-Sikora \cite{GHS1,GHS2} to the low-energy regime. This distribution class gives a calculus that includes many objects of fundamental importance (e.g. resolvents, Poisson operator, spectral measure of Laplace type operators). 
The estimate in Proposition~\ref{prop:conic-pair-pointwise-bound-low} shows that the pointwise estimate of such Legendre distributions obeys the same numerology as if there is no conjugate point, or equivalently no extra parameters in the oscillatory integral.
Conceptually, this is because the amplitude of our oscillatory integral has a factor that vanishes at the intersection of the pair of Legendre submanifolds, and this is exactly where the Hessian of the phase function in our oscillatory integral degenerates. This extra vanishing compensates the potential loss in the pointwise estimate of the oscillatory integral. This is similar to the `damping oscillatory integrals' studied by Sogge-Stein \cite{Sogge-Stein-average}, Cowling-Disney-Mauceri-M\"uller \cite{Cowling-etal-damping}, Gressman \cite{damping-osc-Gressman}, Lee-Oh \cite{lee2025damping} and many others.
In order to show that the vanishing factor is enough to compensate the loss in the pointwise estimate, in Section~\ref{sec:projection-and-phase}, we prove that though the projection from general Legendre submanifolds can degenerate in an uncontrolled manner, the projection from one that arises as a `propagating Legendre submanifold', which is the flow out of the geodesic flow, can only degenerate simply. That is, whenever the projection degenerates, its derivative vanishes to the first order. In addition, those singular points will not accumulate.
This allows us to obtain a lower bound for the Hessian matrix of the phase around its degenerate critical points and trade off this degeneracy against the vanishing of the amplitude.

\subsection{Literature review and related problems}
\label{subsec:literature_review}

Dispersive estimates, encompassing both time-decay estimates and Strichartz estimates, have been extensively studied over several decades. In the constant-coefficient setting, these estimates are intimately connected to the restriction theorem in harmonic analysis. A wide range of dispersive estimates are known to be valid for fundamental models such as the wave equation and the Schr\"odinger equation; we refer to Ginibre-Velo \cite{Ginibre-Velo} and Keel-Tao \cite{Keel-Tao} for results. These estimates have been proved to be valuable tools in the analysis of nonlinear problems.

These estimates are considerably more delicate in the variable coefficients case, which introduces new features tied to the nontrivial behaviour (e.g. focusing and trapping) of the Hamilton flow in the phase space. Early progress on Strichartz estimates was made for the wave equation with variable coefficients, starting with the smooth case in Kapitanski\u{\i} \cite{Kapitanski} and Mockenhaupt-Seeger-Sogge \cite{Mockenhaupt-Seeger-Sogge}, and later extended to $C^2$-coefficients by Smith \cite{Smith} and Tataru \cite{Tataru1,Tataru2}. 
Parallel results for the Schrödinger equation were established by Staffilani-Tataru \cite{Staffilani-Tataru} for $C^2$ coefficients and by Burq-Gérard-Tzvetkov \cite{Burq-Gerard-Tzvetkov} for smooth coefficients. In the broader context of elliptic variable-coefficient operators, Sogge's $L^q$-eigenfunction bounds on compact manifolds are also fundamental. Long-range perturbations were considered in Burq-Tzvetkov\cite{Burq-Tzvetkov}. 
In the same setting as the present paper, Hassell–Tao–Wunsch \cite{Hassell-Tao-Wunsch} established local-in-time Strichartz estimates and conjectured their global-in-time counterparts (see \cite[Section 7]{Hassell-Tao-Wunsch}). This conjecture was subsequently resolved by Hassell and the last author \cite{Hassell-Zhang2016Strichartz,Zhang}.

In contrast to Strichartz estimates, decay estimates—in particular pointwise dispersive estimates—are more delicate, as they require more precise information about the underlying characteristic dynamics. For example, the decay estimates in weighted 
$L^2$ space are studied in Bouclet-Burq \cite{Bouclet-Burq} and Royer\cite{Royer}. We refer to Iandoli-Ivanovici \cite{Iandoli-Ivanovici} for the pointwise dispersive estimates outside a cylinder.

To study the dispersive estimates in the asymptotically conical setting, we need the theory of Legendre distributions on manifolds with boundaries or corners developed by Melrose-Zworski \cite{melrose1996scattering}, 
Hassell-Vasy \cite{hassell1999spectral,hassell2001resolvent}, Hassell-Wunsch \cite{hassell-wunsch2005schrodinger,Hassell-Wunsch-semiclassical-resolvent} and Guillarmou-Hassell-Sikora \cite{GHS1,GHS2}. 
To obtain the behaviour of the propagator using the spectral measure, one needs to investigate the Schwartz kernel of the resolvent and spectral measure in different frequencies, that is,  in the cases for a fixed $\lambda$ (e.g. \cite{hassell1999spectral,hassell2001resolvent}), for the high-energy i.e. $\lambda\to\infty$ (e.g. see \cite{Hassell-Wunsch-semiclassical-resolvent}) and for the low-energy i.e. $\lambda\to 0$ (e.g. see \cite{GH-2008, GHS1, GHS2}) respectively. These constructions are also inspired by the general program for geometric scattering theory initiated by Melrose \cite{Melrose-APS,Melrose1994,Melrose-geometric-sc}, the classical theory of Fourier integral operators by H\"ormander \cite{FIO1} and Duistermaat-H\"ormander \cite{FIOII} and the theory of paired Lagrangian distributions treated by Melrose-Uhlmann \cite{Melrose-Uhlmann-intersecting}, Guillemin-Uhlmann \cite{Guillemin-Uhlmann-intersecting} and Joshi \cite{Joshi:Thesis}. 
We also remark that the characterizations of the Schr\"odinger propagator as a Legendre distribution in various settings are given in Hassell-Wunsch \cite{hassell-wunsch2005schrodinger,Hassell-Wunsch-semiclassical-resolvent}, where wavefront bounds are derived.

It is worth mentioning that the worst-case scenario for $L^p \to L^p$ estimates of Fourier integral operators is precisely the opposite of that for $L^1 \to L^\infty$ estimates, or more generally for $L^p \to L^{p'}$ estimates. Concretely, for the $L^p \to L^p$ estimate of $e^{it\sqrt{P}}$ in the exact conic case, a rescaling argument shows that the regularity difference in the estimate at $t=1$ determines the growth rate in $t$. See also the classical results in \cite{Peral-Lp} and \cite{Myachi-Hp-Lp}. The same phenomenon is studied in \cite{Seeger-Sogge-Stein-FIO} for general Fourier integral operators associated with canonical graphs. As pointed out in \cite[Theorem~2.1]{Seeger-Sogge-Stein-FIO}, the estimate is saturated precisely when the projection from the associated Lagrangian submanifold to the base manifold has full rank at certain points.
This means that the $L^p \to L^p$ estimates behave in the opposite way to $L^1 \to L^\infty$ dispersive estimates since, as we show here, the non-degeneracy of such projections is in favour of the best possible decay rate in dispersive estimates.



\subsection{Organization of the Paper}

The article is organized as follows. In Section~\ref{sec:microlocal-tools}, we introduce the (parametrized) double space and the pseudodifferential algebra that we will use.
Then we recall the definition of Legendre submanifolds that we will use and the theory of Legendre distributions in Section~\ref{sec:Legendrian-distributions}. 
Afterwards, oscillatory integral expressions of the spectral measure of $P$ are given in Section~\ref{sec:microlocalized-spectral-measure}. 
The delicate part of the dispersive estimates is the contribution from Legendre distributions associated with the conic intersecting pair of Legendre submanifolds. In Section~\ref{sec:projection-and-phase}, we analyze the projection from those Legendre submanifolds in the phase space to the base manifold and phase functions used in Legendre distributions. 
Then we give the pointwise bound of such Legendre distributions in Section~\ref{sec:conic-points-pointwise-bound}.
Finally, our main results, dispersive estimates for Schr\"odinger equations and wave equations, are proved in Section~\ref{sec:dispersive-Schrodinger} and Section~\ref{sec:dispersive-wave}.

Readers willing to accept the theory of Legendre distributions and the oscillatory integral expression of the microlocalized spectral measure as a black box can start with Section~\ref{sec:microlocalized-spectral-measure} and read Sections~\ref{sec:projection-and-phase}-\ref{sec:dispersive-wave} for the proof of dispersive estimates.

\section{Microlocal Tools}
\label{sec:microlocal-tools}

In this section, we review the theory of Legendre distributions on manifolds with boundaries and corners developed by Melrose-Zworski \cite{melrose1996scattering}, 
Hassell-Vasy \cite{hassell1999spectral,hassell2001resolvent}, and Guillarmou-Hassell-Sikora \cite{GHS1,GHS2}.

\subsection{The low-energy, high-energy, and combined spaces}
\label{subsec:low-high-combined-spaces}

In this subsection we introduce a space that combines the construction of the low-energy space in \cite{GH-2008, GHS1,GHS2} and the high-energy (semiclassical) space in \cite{Hassell-Wunsch-semiclassical-resolvent}. 
The basic building block is the b-double space obtained by blowing up the corner $\partial X \times \partial X \subset X \times X$:
\begin{equation}
\label{eq:b-double-space}
X_\rmb^2 = [X \times X; \partial X \times \partial X],
\end{equation}
and we denote the blow down map by 
\begin{equation} \label{eq:beta-b}
\beta_{\mathrm{b}}: X_{\rmb}^2 \to X^2.
\end{equation}
We denote the lift of $\partial X \times X$ by $\lb$, the lift of $X \times \partial X$ by $\rb$ and the lift of $\partial X \times \partial X$ by $\bfs$. We refer the reader to \cite{Melrose-APS} for its detailed definition and the definition of blow ups.

Now we turn to the construction of the low-energy space in \cite{GH-2008, GHS1,GHS2}. 
Let $X_I^2: = [0,1] \times X \times X$ be the family of double spaces $X \times X$ parametrized by the energy level $\lambda$. To analyze the behaviour near the corner $\partial X \times \partial X$ and the behaviour as $\lambda \to 0$, we consider the b-low-energy space $X_{\rmb,\flat}^2$, which is obtained by blowing up $X_I^2$ along three corners of codimension two:
\begin{align*}
C_{ZL} = \{0\} \times \partial X \times X,
\;
C_{ZR} = \{0\} \times X \times \partial X,
\;
C_{LR} = [0,1] \times \partial X \times \partial X,
\end{align*}
and the corner of codimension three:
\begin{align*}
C_{ZLR} = \{0\} \times \partial X \times \partial X.
\end{align*}
That is, we set
\begin{align} \label{eq:def-low-space}
X_{\rmb,\flat}^2: = [ X_I^2 ; 
C_{ZLR},C_{ZL}, C_{ZR}, C_{LR}],
\end{align}
with blow-down map $\beta_{\rmb,\flat}: X_{\rmb,\flat}^2 \to X_I^2$. See Figure~\ref{figure:b-low-energy-space} for a graphic illustration.

\begin{figure}[h!] 
\includegraphics[scale = 0.7]{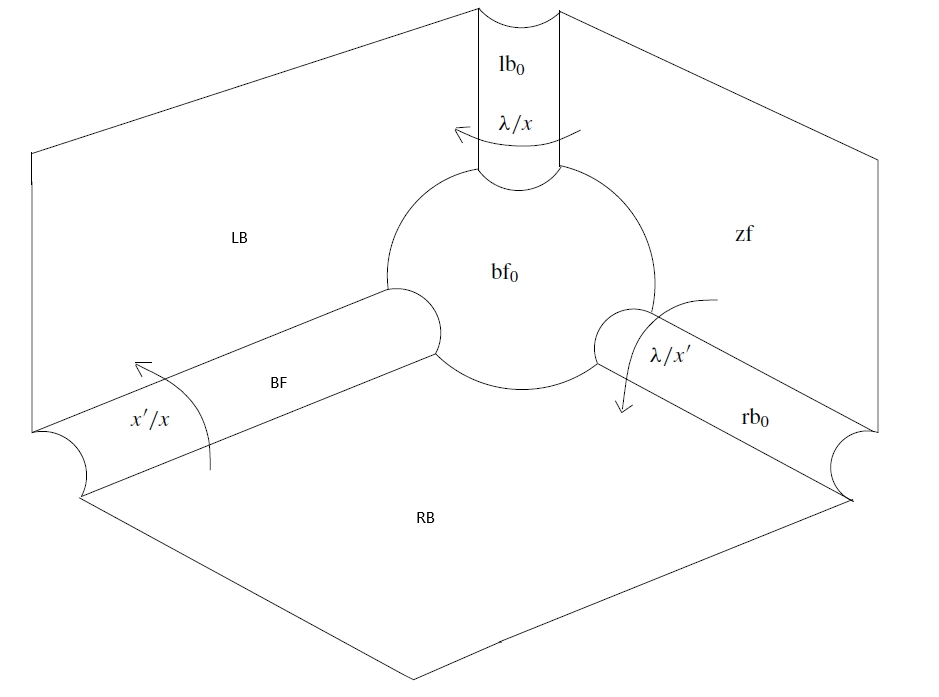}
\caption{The b-low-energy space $X_{\rmb,\flat}^2$.}
\label{figure:b-low-energy-space}
\end{figure}



We should mention that $\lambda$, the parameter in the $[0,1]$-component 
should be interpreted as the $\sct$-energy. The frequencies (or energy) measured in terms of $\rmb$-cotangent bundle are parameters like $\rho=\frac{\lambda}{x}, \rho'=\frac{\lambda}{x'}$ on those surfaces introduced by the blow up, where $x,x'$ are defining functions of $\partial X$ of the left and right $X$-factors respectively.
We also mention here that as pointed out in \cite{GH-2008}, the $\zf$ (the zero face), which is the lift of $\{0\} \times X \times X$ in $X_{\rmb,\flat}^2$ is canonically diffeomorphic to $X_{\rmb}^2$. This reflects the connection between this low-energy space and the recent work on the second microlocalization in \cite{Vasy2021resolvent, Vasy-resLag} that encodes the b-analysis and the sc-analysis simultaneously. See \cite[Section~5]{Vasy2021resolvent} and \cite[Section~2]{Vasy-resLag} for more details.\footnote{One can see this connection more clearly if we do not perform the resolution of Vasy for the resolved sc-b-cotangent bundle at $\lambda>0$ (which is called $\sigma$ there) but instead only at $x=\lambda=0$.
See the construction of the `compressed cotangent bundle' in \cite[Section~2.3]{GHS2}. }

Compared with the low-energy case, the high-energy space $X_{\rmb, \#}^2$ is simpler:
\begin{align} \label{eq:def-high-energy-space}
X_{\rmb, \calchigh}^2: = [0,h_0) \times X_{\rmb}^2, \quad h_0>1,
\end{align}
where $X_{\rmb}^2 = [X \times X; \partial X \times \partial X]$ is the b-double space as before.
Here the semiclassical parameter $h \in [0,h_0)$ is related to the $\lambda$ above by viewing $h^{-1}$ as the energy level, which means the part with $h$ near $0$ is indeed the `high-energy' part.

Finally, we combine the low-energy and high-energy spaces above to define the combined space $X_{\calc}^2$ to be the manifold with corners obtained by gluing the low-energy space and the high-energy (semiclassical) space via the identification $\lambda \sim h^{-1}$:
\begin{equation} \label{eq: definition, combined double space}
X_{\calc}^2: = X_{\rmb,\flat}^2 \sqcup X_{\rmb, \#}^2 / \sim,
\end{equation}
where $\sim$ identifies $(p,\lambda)$ with $(p,h^{-1})$ for $\lambda \in [h_0^{-1},1]$ and $h \in [1,h_0)$ and $p \in X_{\rmb}^2$.


In summary, $X_{\calc}^2$ has eight boundary faces 
\begin{equation} \label{eq:faces}
 \zf,\, \lb_0, \, \bfs_0, \, \rb_0,\, \LB, \, \BFS, \, \RB,\, \smf,
\end{equation}
and we summarize their boundary defining functions below. Let $\bullet$ be one of the faces listed in \eqref{eq:faces}. We use $\rho_{\bullet}$ to denote its boundary defining function. 
Then for `low-energy' faces at $\lambda = 0$, we have:
\begin{align} \label{eq:defining-functions-low-1}
\begin{split}
& \zf = \big\{ \frac{\lambda}{x} = 0, \frac{\lambda}{x'} = 0  \big\}, \quad \rho_{\zf} = \frac{\lambda}{\lambda+x}+\frac{\lambda}{\lambda+x'}; 
\\ & \lb_0 = \big\{ \frac{x}{x'} = 0,\lambda = 0  \big\}, \quad \rho_{\lb_0} = \frac{x}{x+x'}+\lambda;
\\ &  \bfs_0 = \big\{ x=x'=\lambda =0 \big\}, \quad \rho_{\bfs_0} = \lambda+x+x';
\\& \rb_0 = \big\{ \frac{x'}{x} = 0, \lambda = 0 \big\} , \quad \rho_{\rb_0}  = \frac{x'}{x+x'} + \lambda.
\end{split}
\end{align}

The faces with $\lambda \in [0,\infty]$ \footnote{$[0,\infty]$ means compactified $[0,+\infty)$ using the identification $\lambda \sim h^{-1}$ for large $\lambda$ and $h \in [0,h_0)$.} which are away from $\zf$ are:
\begin{align}
\label{eq:defining-functions-low-3}
\begin{split}
& \LB = \big\{ \frac{x}{x'} = 0, \frac{x}{\lambda}=0  \big\},
\quad \rho_{\LB} = \frac{x}{\lambda+x} + \frac{x}{x+x'};
\\ & \BFS = \big\{ x = x' = 0, \frac{x+x'}{\lambda} = 0  \big\}, \quad \rho_{\BFS}= \frac{x+x'}{\lambda};
\\ & \RB = \big\{ \frac{x'}{x} = 0, \frac{x'}{\lambda}=0  \big\},
\quad \rho_{\RB} = \frac{x'}{\lambda+x'}+\frac{x'}{x+x'};
\\ & \smf = \big\{ \lambda = +\infty \big\}, \qquad
\rho_{\smf} = h = \lambda^{-1}.
\end{split}
\end{align}




Then the conormal bundle of the lifted diagonal (but only down to the boundary of $\bfs_0$, away from $\zf$) in $X_{\calc}^2$ gives the phase space through which we will construct our microlocal partition of unity, and the important feature of it is the scaling of the frequency variables.
This is effectively a vector bundle over $[0,\infty]_\lambda \times X$ after identifying the lifted diagonal with $X$. 
One can think of this as a family of scattering cotangent bundles in which the characteristic variety of $\Delta_g-\lambda^2$ is rescaled to the unit cosphere bundle.
Concretely, we write a point in the cotangent bundle as:
\begin{equation}\label{eq:contact-form}
\lambda \nu  d(\frac{1}{x}) + \sum_{i=1}^{n-1} \lambda \mu_i \frac{dy_i}{x}
\end{equation}
near $\partial X$, or
\begin{equation}\label{eq:contact-form-interior}
 \sum_{i=1}^{n} \lambda \zeta_i dz_i
\end{equation}
away from $\partial X$.
This construction is valid up to $\lambda \to \infty$. In the region with large $\lambda$, one can rewrite the canonical one form above using the semiclassical parameter $h = \lambda^{-1}$:
\begin{equation} \label{eq:1-form-high-single}
\nu  h^{-1}d(\frac{1}{x}) + \sum_{i=1}^{n-1}  \mu_i \frac{dy_i}{hx},
\end{equation}
and it gives a vector bundle down to $h=0$ and we denote this part with $h$ close to $0$ by ${}^{\calchigh}T^*X$, which effectively is just the semiclassical scattering cotangent bundle. 
Finally we glue these two parts together via the identification $h \sim \lambda^{-1}$ again and denote this combined phase space by ${}^{\calc}T^*X$.
We further radially compactify its fibers to obtain the compactified combined cotangent bundle, which is an $n$-ball bundle over $[0,\infty]_\lambda \times X$ that we denote by
\begin{equation} \label{eq:compactified-combined-cotangent-bundle}
{}^{\calc}\overline{T}^*X.
\end{equation}

\subsection{The low-energy, high-energy, and combined pseudodifferential algebra}
\label{subsec:low-high-combined-PsiDO}

In this subsection, we introduce a pseudodifferential algebra that combines the ones introduced in \cite{GH-2008,Hassell-Wunsch-semiclassical-resolvent,Hassell-Zhang2016Strichartz} that are adapted to the low- and high-energy spaces.


The space of $m$-th order low-energy pseudodifferential operators\footnote{See \cite[Section~2.3]{GH-2008} for a more general version of this pseudodifferential algebra. }, denoted by 
\begin{equation} \label{eq:PsiDO-low-def}
\Psi_{\flat}^{m}(X;\Omega_{\flat}^{1/2}).
\end{equation}
consists of those operators that have Schwartz kernels that are sections of $\Omega_{\flat}^{1/2}$ of the form characterized below.
Here $\Omega_{\flat}^{1/2}$ is the half-density bundle whose sections, near $\BFS \cap \bfs_0$ and the lifted diagonal, are non-zero smooth multiples of
\begin{equation} \label{eq:low-energy-density-1}
\Big| \frac{d\rho d\rho' dydy' d\lambda}{\rho^{n+1} (\rho')^{n+1} \lambda }\Big|^{1/2} 
\sim \lambda^n \Big|\frac{dgdg'd\lambda}{\lambda}\Big|^{1/2},
\end{equation}
where $dg,dg'$ are Riemannian densities with respect to $g$ on $X$ lifted from the left and right factors to $X_{\rmb,\flat}^2$.
Near $\zf$, they are smooth multiples of
\begin{equation} \label{eq:low-energy-density-2}
\Big|\frac{dxdx'dydy'd\lambda}{xx'\lambda}\Big|^{1/2}.
\end{equation}
Finally, near $\LB \cap \lb_0$, they are smooth non-zero multiples of
\begin{equation} \label{eq:low-energy-density-3}
\lambda^{n/2} \Big| \frac{dgdg'_{\rmb}d\lambda}{\lambda} \Big|^{1/2},  
\end{equation}
and similarly near $\RB \cap \rb_0$ with primed and unprimed objects switched.
The advantage of this density bundle is that it incorporates the transition from the sc-region to the b-region as we approach $\zf$.
See \cite[Section~2,4]{GHS1} for detailed discussions on this density bundle, which is denoted by $\Omega_{k,b}^{1/2}$ there.

We now characterize the kernels of operators in $\Psi_{\flat}^{m}(X;\Omega_{\flat}^{1/2})$.
They can be written as a sum of two operators of the following types.
The first type consists of pseudodifferential operators with kernels that are supported in the region $\rho = \frac{x}{\lambda},\rho' = \frac{x'}{\lambda} \leq C<\infty$.
The other type consists of multiplication operators by smooth functions of $\rho$ that are supported on $\rho \geq \frac{C}{2}$, which has Schwartz kernel that is a multiple of the delta function on the diagonal.
For the first type, near the diagonal the operators can be represented by oscillatory integrals of the form
\begin{align} \label{eq:quant-PsiDO}
\lambda^n \Big( \int_{\R^n} e^{ \frac{i\lambda}{x} ( (1-\sigma)\nu + (y-y') \cdot \mu ) }
a(\lambda,\rho,y,\mu,\nu)d\mu d \nu \Big) \big|dgdg'\frac{d\lambda}{\lambda}\big|^{1/2},
\end{align}
where $\sigma= \frac{\rho}{\rho'} = \frac{x}{x'}$ and $a(\lambda,\rho,y,\mu,\nu)$ is a classical symbol of order $m$ in $(\mu,\nu)$, smooth in other variables and supported where $\rho \leq C$. 
Here $dg,dg'$ are volume forms associated with $g$ lifted from the left and right factors.
When we compose such pseudodifferential operators with Legendre distributions below, the composition is taken only in the $X$-variables, with the spectral parameter $\lambda$ fixed. Consequently, we suppress the $|\frac{d\lambda}{\lambda}|^{1/2}$-factor in the pseudodifferential operator and retain this factor only in the Legendre distribution.
Also, since we are only concerned with the low-energy part now, we will only consider those operators with $a(\lambda,\rho,y,\mu,\nu)$ supported only on $\lambda \leq \lambda_0$ for a constant $\lambda_0>0$.

One can define a more general class of pseudodifferential operators like $x^{-l}\lambda^{-k_\flat}\Psi_{\flat}^{m}$, but we will not pursue it here, since the class without extra weights is already sufficient to include our microlocal partition of unity.

Recall that ${}^{\calc}\overline{T}^*X$ is the compactified phase space constructed in Section~\ref{subsec:low-high-combined-spaces}. 
Then its partial boundary consists of three parts:
\begin{itemize}
\item The semiclassical part: ${}^{\calc}\overline{T}_{ \{h=0\} }^*X$.

\item The fiber infinity part: ${}^{\calc} S^*X$, which is the $(n-1)$-sphere bundle formed by the part with $|(\nu,\mu)| = \infty$ in ${}^{\calc}\overline{T}^*X$

\item The spacetime boundary part ${}^{\calc}T^*_{ [0,\infty] \times \partial X}X$, which is the restriction of ${}^{\calc}\overline{T}^*X$ to $[0,\infty]_\lambda \times \partial X$.
\end{itemize}

\begin{remark}
Here `partial boundary' is emphasizing that the part at $\lambda=0$ is not completely included, but only the part over $\partial X$, hence this is smaller than the boundary of ${}^{\calc}\overline{T}^*X$.
In contrast, in the high-energy regime $h=0$, we will include the entire bundle, even over the interior at finite frequencies. 
\end{remark}

Then the wavefront set in the low-energy regime is defined as follows.
\begin{definition} \label{def:low-energy-WF}
For $A \in \Psi_{\flat}^{m}(X;\Omega_{\flat}^{1/2})$ with left symbol $a(\lambda,\rho,y,\mu,\nu)$
as in \eqref{eq:quant-PsiDO}, we define $\WF'_{\flat}(A)$ as a subset of ${}^{\calc} S^*X \cup {}^{\calc}T^*_{ [0,\infty] \times \partial X}X$
(in fact only the part with $\lambda \leq \lambda_0$) as follows.
For $q \in \big({}^{\calc} S^*X \cup {}^{\calc}T^*_{ [0,\infty] \times \partial X}X\big)$, we say $q \notin \WF'_{\flat}(A)$ if there is a $\chi \in C_c^\infty({}^{\calc}\overline{T}^*X)$ with $\chi(q)=1$ such that $\chi a \in \mathcal{S}({}^{\calc}\overline{T}^*X)$.\footnote{Here the Schwartz function space is defined to be the space of smooth functions that are rapidly decaying as $x \to 0$ and $\mu,\nu \to \infty$, while the latter is automatic by the assumption on $a,\chi$.}
\end{definition}

%

For the high-energy part, as in \cite[Section~3]{Hassell-Zhang2016Strichartz}, we use the semiclassical scattering pseudodifferential algebra $\Psi_{\sct,h}(X;\Omega_{\calchigh}^{1/2})$ introduced by Wunsch and Zworski \cite[Appendix~A]{wunsch-zworski-resonance-distribution}.
Concretely, for $h>0$, this is just the scattering pseudodifferential algebra. 
We do not actually perform the resolution at the diagonal at $h=0$ and give the conormality criterion, but rather give an oscillatory integral representation of $\Psi_{\sct,h}(X;\Omega_{\calchigh}^{1/2})$ directly.

The half-density bundle $\Omega_{\calchigh}^{1/2}$ we are using for the high-energy regime is simpler than the low-energy one, since only scattering type behaviour is needed, whereas the low-energy case concerns both the scattering and the b-regimes.
Its sections are non-zero smooth multiples of
\begin{equation} \label{eq:high-energy-half-density}
\Big|\frac{dgdg'dh}{ h^{2n+2} }\Big|^{1/2} = h^{-n} |dgdg'd\lambda|^{1/2},
\end{equation}
which is essentially the same as \eqref{eq:low-energy-density-1}.

Now we define the pseudodifferential algebra $\Psi_{\sct,h}(X;{}^{\Phi}\Omega^{1/2})$. For $A \in \Psi_{\sct,h}^{m,0,0}(X;{}^{\Phi}\Omega^{1/2})$, away from $[0,h_0) \times \mathrm{bf}$ but near $[0,h_0) \times \Delta_{\rmb}$, it is just a semiclassical pseudodifferential operator and can be written as 
\begin{align} \label{eq:PsiDO-high-def-1}
h^{-n} \Big( \int_{\R^n} e^{i(z-z') \cdot \zeta/h}a(h,z,\zeta)d\zeta \Big) |dg dg'|^{1/2},
\end{align}
where $a$ is a symbol of order $m$ with respect to $\zeta$. 
Near $[0,h_0) \times (\mathrm{bf} \cap \Delta_{\rmb})$ , it can be represented by
\begin{align} \label{eq:PsiDO-high-def-2}
h^{-n} \Big( \int_{\R^n} e^{i \frac{(1-\sigma)\nu + (y-y') \cdot \mu}{hx} } a(h,x,y,\nu,\mu)d\mu d\nu \Big) |dg dg'|^{1/2}.
\end{align}
Here $h_0 \geq \lambda_0^{-1}$.
More generally, we set
\begin{equation}
  \Psi_{\calchigh}^{m,l,k}(X;{}^{\Phi}\Omega^{1/2})=\Psi_{\sct,h}^{m,l,k}(X;{}^{\Phi}\Omega^{1/2})
  = x^{-l}h^{-k}\Psi_{\sct,h}^{m,0,0}(X;{}^{\Phi}\Omega^{1/2}).
\end{equation}

Then similar to Definition~\ref{def:low-energy-WF}, we define the wavefront set $\WF'_{\sct,h}(A)$ as follows.
\begin{definition} \label{def:high-WF}
For $A \in \Psi_{\sct,h}^{m,0,0}(X;{}^{\Phi}\Omega^{1/2})$,
its wavefront set $\WF'_{\sct,h}(A)$ is defined, as a subset of ${}^{\calc}\overline{T}_{ \{h=0\} }^*X \cup {}^{\calc} S^*X \cup {}^{\calc}T^*_{ [0,\infty] \times \partial X}X$ (in fact, only the part with $h \in [0,h_0)$, or $\lambda \geq h_0^{-1}$) as follows.
For $q \in \big( {}^{\calc}\overline{T}_{ \{h=0\} }^*X \cup {}^{\calc} S^*X \cup {}^{\calc}T^*_{ [0,\infty] \times \partial X}X \big)$, we say $q \notin \WF'_{\sct,h}(A)$ if there is a $\chi \in C_c^\infty({}^{\calc}\overline{T}^*X)$ with $\chi(q)=1$ such that $\chi a \in \mathcal{S}({}^{\calc}\overline{T}^*X)$. 
\end{definition}
By definition, $\WF'_{\sct,h}(A)$ coincides with $\WF'_{\flat}(A)$ over the region where $\lambda \in [h_0^{-1},\lambda_0]$.


Now we combine the above definitions to define the combined pseudodifferential algebra. We will use $[0,\infty]$ to denote $[0,\lambda_0]_\lambda \sqcup [0,h_0)_h/\sim$ via the identification $h^{-1} \sim \lambda$ and it is equipped with the smooth structure inherited from $[0,h_0)_h$ near $\infty$.

\begin{definition} \label{def:combined-PsiDO}
The operator class $\Psi_{\calc}^{m,l,k}$ consists of operators $A$ such that
\begin{itemize}
\item For $\phi,\psi \in C^\infty([0,\infty] \times X)$ supported in $\lambda \leq 1$, we have
\begin{align*}
\phi A \psi \in \Psi_{\flat}^{m}.
\end{align*}

\item For $\phi,\psi \in C^\infty([0,\infty] \times X)$ supported in $h \leq h_0$ (or $\lambda \geq h_0^{-1}$), we have
\begin{align*}
\phi A \psi \in \Psi_{\calchigh}^{m,l,k}=\Psi_{\sct,h}^{m,l,k}.
\end{align*}

\item For $\phi,\psi \in C^\infty([0,\infty] \times X)$ with disjoint supports, $\phi A \psi$ has kernel that is Schwartz on $X^2_{\calc}$. 

\end{itemize}
\end{definition}
Correspondingly, the combined wavefront set of $A \in \Psi_{\calc}^{m,0,0,}$, which we denote by $\WF'_{\calc}(A)$ is defined by
\begin{equation} \label{eq:combined-WF-def}
\WF'_{\calc}(A) = \WF'_{\flat}(A) \cup \WF'_{\sct,h}(A).
\end{equation}

\section{Legendre submanifolds and Legendre distributions}
\label{sec:Legendrian-distributions}

We briefly recall the theory of Legendre distributions in this section. 
The aim of this section is just to save readers from unravelling the theory of Legendre distributions and the scattering fibered structure in the series of work \cite{melrose1996scattering,hassell1999spectral,hassell2001resolvent,Hassell-Wunsch-semiclassical-resolvent}, so we only state results at the level of generality that is enough for our applications, but refer to \cite{melrose1996scattering,hassell1999spectral,hassell2001resolvent,GHS2,Hassell-Wunsch-semiclassical-resolvent} for more details and more general situations.


\subsection{Legendre submanifolds and Legendre distributions at low energies}
\label{subsec:Legendrian-geometry-low}

As mentioned above, \cite[Theorem~3.10]{GHS2} states that the spectral measure for the Laplacian on an asymptotically conic manifold is, for low energies, a Legendre distribution associated with a pair of Legendrian submanifolds with conic points: the `propagating Legendrian' $L^{\bfs}$ and the `incoming/outgoing Legendrian' $L^\sharp$.
Pairs of such Legendre submanifolds are called \emph{Legendrian conic pairs} for short and denoted by $\mathrm{LCP}$ in subscripts below.

The goal of this subsection is to briefly introduce those geometric objects carrying the `important part' of the spectral measure. 
We first have to introduce the contact manifold in which these Legendre submanifolds live.


To begin with, let 
\begin{equation} \label{def:pulled-back-bundle}
{}^{\Phi}T^* X_{\mathrm{b}}^2 = \beta_{\rmb}^*({}^{\sct}T^*X \times {}^{\sct}T^*X )
\end{equation}
be the bundle formed by pulling back ${}^{\sct}T^*X \times {}^{\sct}T^*X$ (viewed as a bundle over $X^2$) to $X_{\mathrm{b}}^2$ via the natural blow down map $\beta_{\mathrm{b}}$ in \eqref{eq:beta-b}. 
In particular, the coordinates away from $\partial (X_{\rmb}^2)$ can be determined by writing the canonical 1-form as
\begin{equation} \label{eq:1-form-Phi-bundle-interior}
\zeta \cdot dz + \zeta' \cdot dz'.
\end{equation}
When we approach the boundary in terms of the left or right variables, we use coordinates as in \eqref{eq:sc-1-form} lifted from the left or right factor instead.

Then ${}^{\Phi}T^* X_{\mathrm{b}}^2$ inherits a symplectic structure\footnote{This symplectic structure degenerates over the boundary. Nevertheless, the contact structure induced below is non-degenerate.}
from the natural one of ${}^{\sct}T^*X \times {}^{\sct}T^*X$, which in turn induces a contact structure on ${}^{\Phi}T^*_{\bfs} X_{\mathrm{b}}^2$, which is its restriction to the boundary hypersurface $\bfs$ created by the blowup \eqref{eq:b-double-space}.
Our Legendre submanifolds will be those of ${}^{\Phi}T^*_{\bfs} X_{\mathrm{b}}^2$. In terms of $X_{\rmb,\flat}^2$, one can think of this $\bfs$ as $\bfs \times \{ 1 \} \subset \BFS$, but it encodes information of its dilated copies at all fixed $\lambda$. See \cite[Section~2]{GHS2} for more details.

More concretely, in local coordinates 
\begin{equation} \label{eq:coordinates-bf-fibered-bundle}
(\sigma, y, y', \nu, \nu',\mu, \mu') 
\end{equation}
for ${}^\Phi T^*_{\bfs} X^2_b$, where $\sigma = x/x'$, $(\mu, \nu)$ are as in \eqref{eq:contact-form}, and the un-primed/primed coordinates are lifted from the left/right copies of ${}^{\sct} T^* X$\footnote{This convention will be kept throughout this paper, unless otherwise stated.}, the contact form has an expression
\begin{equation} \label{eq:contact-form-main}
d\nu - \mu \cdot dy + \sigma (d\nu' - \mu' \cdot dy').
\end{equation}
A Legendre submanifold should be defined to be a $2n-1$-dimensional submanifold of this $4n-1$-dimensional space on which the contact form vanishes. 
However, we need to address the degeneracy of \eqref{eq:contact-form-main} as $\sigma \to 0$, which corresponds to the situation when we approach the corner $\bfs \cap \lb$.
To deal with this, we introduce fibrations of $\lb,\rb$ as follows.
Let $Z_{\lb},Z_{\rb}$ be $\partial X$ parametrized by the left and right factors of $\lb$ and $\rb$ respectively.
Then we consider
\begin{equation} \label{eq:def-phi-lb}
    \phi_{\lb}: \quad \lb \to Z_{\lb} 
\end{equation}
defined by
\begin{equation}
    (q,q') \to q \in \partial X,
\end{equation}
where we identified $\lb$ with $\partial X \times X$. We define $\phi_{\rb}$ in the same way, except for switching the left and right factors.
Then we define
\begin{equation}  \label{eq:NZ-lb-def}
    {}^{\Phi}N^*Z_{\lb} =  {}^{\sct}T^*_{Z_{\lb}}\big([0,x_0)_x \times Z_{\lb} \big) ={}^{\sct}T^*_{\partial X}X,
\end{equation}
and similarly for ${}^{\Phi}N^*Z_{\rb}$. The identification between ${}^{\Phi}N^*Z_{\lb}$ and ${}^{\sct}T_{\partial X}^*X$ gives it a contact structure and this precisely corresponds to the first part $d\nu - \mu \cdot dy$ in \eqref{eq:contact-form-main}.
The fibration in \eqref{eq:def-phi-lb} induces a fibration 
\begin{equation} \label{eq:def-tilde-phi-lb}
    \tilde{\phi}_{\lb}:
    {}^{\Phi}T_{\lb \cap \bfs}^*X_{\rmb}^2 \to {}^{\Phi}N^*Z_{\lb}
\end{equation}
defined by
\begin{equation}
    (y,y',\nu,\nu',\mu,\mu') \to (y,\nu,\mu).
\end{equation}
Similarly, we define 
\begin{equation}
    \tilde{\phi}_{\rb}:
    {}^{\Phi}T_{\lb \cap \bfs}^*X_{\rmb}^2 \to {}^{\Phi}N^*Z_{\rb}
\end{equation}
by
\begin{equation}
    (y,y',\nu,\nu',\mu,\mu') \to (y',\nu',\mu').
\end{equation}

Letting $\bullet = \lb$ or $\rb$, we use $F_{\bullet}$ to denote the fiber of \eqref{eq:def-phi-lb}, which is just $X$ parametrized by the other variable (that is, if $\bullet = \lb$, then this $X$ is parametrized by the right variable, while if $\bullet = \rb$ then it is parametrized by the left variable). In addition, $\tilde{\phi}_{\bullet}$ is connected to $\phi_{\bullet}$ in the way that the fiber of $\tilde{\phi}_{\bullet}$ is ${}^{\sct}T_{\partial F_{\bullet} }^*F_{\bullet}$, which can be identified with ${}^{\sct}T_{\partial X}^*X$ parametrized by variables from the other side as above. 


\begin{definition} \label{def:Legendre-submanifold-codim2}
    A Legendre submanifold $L$ of ${}^{\Phi}T_{\bfs}^*X_{\rmb}^2$ is a submanifold (with boundary) such that
    \begin{itemize}
        \item the contact form \eqref{eq:contact-form-main} vanishes on $L$;
        \item $L$ is transversal to ${}^{\Phi}T_{\lb \cap \bfs}^*X_{\rmb}^2$ and ${}^{\Phi}T_{\rb \cap \bfs}^*X_{\rmb}^2$.
    \end{itemize}
\end{definition}

Let $L$ be as above. Its boundary has two parts: the part over the corner $\lb \cap \bfs$ and the part over the corner $\rb \cap \bfs$.
We denote them by
\begin{equation}
    \partial_{\bullet}L = L \cap {}^{\Phi}T_{\bullet \cap \bfs}^*X_{\rmb}^2,
\end{equation}
where $\bullet=\lb$ or $\rb$.

The following property of $L$ was taken as an assumption in \cite{hassell1999spectral,hassell2001resolvent}, but as shown in \cite{Hassell-Wunsch-semiclassical-resolvent}, in fact it can be proved from the assumptions in Definition~\ref{def:Legendre-submanifold-codim2} and we list it as a proposition.

\begin{proposition} {\cite[Proposition~4.3]{Hassell-Wunsch-semiclassical-resolvent}}
\label{prop:Legendre-fibration-low}
    Let $\tilde{\phi}_{\bullet}^{L}$ be the restriction of $\tilde{\phi}_{\bullet}$ to $\partial_{\bullet}L$, with $\bullet = \lb$ or $\rb$. Its image $L_{\bullet}$ is an immersed Legendre submanifold of ${}^{\Phi}N^*Z_{\bullet}$ (in the usual sense, instead of Definition~\ref{def:Legendre-submanifold-codim2}) and
    \begin{equation} 
\tilde{\phi}_{\bullet}^{L}: \; \partial_{\bullet}L  \to L_{\bullet}
    \end{equation}
    is locally a fibration with fibers being Legendre submanifolds of ${}^{\sct}T_{\partial F_{\bullet}}^*F_{\bullet}$.
\end{proposition}


\begin{remark}
In terms of the topological or smooth structure, $\partial_{\bullet}L$ in fact almost has a product structure.
But we call it a fibration because the contact structures on the two factors are not of equal importance over the interior of $\lb,\rb$ or corners $\lb \cap \bfs$ and $\rb \cap \bfs$, which is reflected by the fact that the second part on the right hand side of \eqref{eq:contact-form-main} becomes degenerate as $\sigma \to 0$.
In terms of analysis, near the corner $\bullet \cap \bfs$ with $\bullet$ being $\lb$ or $\rb$, this corresponds to the fact that the oscillation carried by the Legendre submanifold of ${}^{\sct}T_{\partial F_{\bullet}}^*F_{\bullet}$ is much slower (with $O(\sigma)$ being the ratio of frequencies) than the oscillation carried by $L_{\bullet}$. So such a fibration is a way of conducting multi-scale analysis.
\end{remark}

We now describe concrete Legendre submanifolds of ${}^{\Phi}T_{\bfs}^*X_{\rmb}^2$ that we will use. The first Legendre submanifold $L^\sharp$ is easy to define: it is the submanifold
\begin{equation} \label{eq:Lsharplow-def}
\Lsharplow = \Lsharplow_+ \cup \Lsharplow_- ,
\end{equation}
where 
\begin{equation} \label{eq:Lsharplow-pm-def}
  \Lsharplow_\pm = \{ (\sigma, y, y', \mu, \mu', \nu, \nu') \mid \mu = \mu' = 0, \ \nu = \nu' = \pm 1 \}.
\end{equation}


The other Legendre submanifold, $L^{\bfs}$, is more intricate and interesting.
It encodes the geodesic flow on the cone over $(Y, \Ymetric)$ where $\Ymetric$ is as in \eqref{eq:def-exact-conic-metric}. Let $(y, \eta)$ be an element of the cosphere bundle $S^* \partial X$ and $\gamma(s) = (y(s), \eta(s))$ be the geodesic with $(y(0), \eta(0)) = (y, \eta)$. Then  $L^{\bfs}$ is given by the union of the leaves $\gamma^2 = \gamma^2(y, \eta)$, with $\textrm{clos}$ standing for taking closure:
\begin{multline}\label{eq:Lbf-definition-gamma^2}
\gamma^2 = \textrm{clos} \big\{ (\sigma = x/x', y, y', \mu, \mu', \nu, \nu') \mid y = y(s_l), y' = y(s_r), \mu = \eta(s_l) \sin s_l, \\ \mu' = -\eta(s_r) \sin s_r, \nu = -\cos s_l, \nu' = \cos s_r, \sigma = \sin s_l/\sin s_r, (s_l, s_r) \in (0, \pi)^2 \big\},
\end{multline}
as $(y, \eta)$ ranges over $S^* \partial X$. We note that this closure includes the sets
\begin{equation}\label{eq:Tpm}
T_\pm = \big\{ (\sigma, y, y', \mu, \mu', \nu, \nu') \mid y=y', \ \sigma \in \R, \  \mu = \mu' = 0, \ \nu = -\nu' = \pm 1 \},
\end{equation}
corresponding to the limit $s_l, s_r \to 0$ and $s_l, s_r \to \pi$. Notice that this definition of $L^{\bfs}$ applies to $(X_0,g_0)$ directly and we will still denote it by $L^{\bfs}$ in that case.

In terms of \eqref{eq:Lbf-definition-gamma^2}, $\Lsharplow_+$ (resp. $\Lsharplow_-$) in \eqref{eq:Lsharplow-pm-def} will intersect $L^{\bfs}$ at points corresponding to $s_l \to \pi$ and $s_r \to 0$ (resp. $s_l \to 0$ and $s_r \to \pi$).

Conceptually, one can think of Legendre submanifolds above as families of Legendre submanifolds that are linear in $\lambda$ (or $h=\lambda^{-1}$, in the high-energy regime) in the fiber component, and rescaled to the one with $\lambda = 1$ to capture its geometric or dynamical structure. 

The statement that the spectral measure is a Legendre distribution with respect to the pair of Legendre submanifolds $(L^{\bfs}, L^\sharp)$ means that the Schwartz kernel of the spectral measure can be expressed as an oscillatory function or an oscillatory integral, with a phase function that `parametrizes' the Legendre submanifold. 
We now discuss the definition of parametrization of Legendre submanifolds. We state the definition of $L_{\lb}$ and $L_{\rb}$ first.

\begin{definition} \label{def:parametrization-lb-rb-low}
Let $L_{\lb}$ be a Legendre submanifold as in Proposition~\ref{prop:Legendre-fibration-low}. We say that $\Phi_{1}(y,v)$ with extra parameter $v \in \R^k$ (locally) parametrizes it, if $L_{\lb}$ locally can be written as 
\begin{equation}
    L_{\lb} = \{ \nu = \Phi_1, \mu = d_y\Phi_1 | \; d_v\Phi_1 = 0 \}.
\end{equation}
Parametrizations of $L_{\rb}$ are defined in the same way, except that $(y,\nu,\mu)$ is replaced by $(y',\nu', \mu')$.
\end{definition}

\begin{definition} \label{def:Legendre-parametrization-low}
Let $L$ be a Legendre submanifold of ${}^{\Phi}T^*_{\bfs} X_{\mathrm{b}}^2$ in the sense of Definition~\ref{def:Legendre-submanifold-codim2}. 
Away from corners $\lb \cap \bfs$, we say that $\Phi(\sigma, y, y', v)$ (depending on extra variables $v=(v_1, \dots, v_k) \in \R^k$), or more accurately $\Phi/x$,  (locally) parametrizes $L$, if $L$ locally can be written as
\begin{multline}
L = \big\{ \mu = d_y \Phi(\sigma, y, y',v), \ \mu' = \sigma^{-1} d_{y'} \Phi(\sigma, y, y',v), \\ \nu = \Phi(\sigma, y, y',v) - \sigma d_\sigma \Phi(\sigma, y, y',v), \ \nu' =  d_\sigma \Phi(\sigma, y, y',v) \mid d_v \Phi = 0 \big\}.
\label{LPhiparam}\end{multline}
Near the corner $\lb \cap \bfs$, we say 
\begin{equation} \label{eq:Phi-low-codim2}
    \Phi(\sigma,y,y',v,w) = \Phi_1(y,v) + \sigma \Phi_2(\sigma,y,y',v,w),
\end{equation}
parametrizes $L$, if $L$ can be locally written as 
\begin{align}    
\begin{split}
L = \big\{ \mu = d_y \Phi, &\ \mu' = \sigma^{-1} d_{y'} \Phi= d_{y'} \Phi_2, \nu = \Phi - \sigma d_\sigma \Phi, \\ 
&\nu' =   d_\sigma \Phi \mid d_v \Phi = 0, d_w\Phi_2 =0 \big\}.
\end{split}\label{LPhiparam-codim2}
\end{align}

\end{definition}

If we define the critical set of $\Phi$ by
\begin{equation} \label{eq:C-Phi-defn-1}
    C_\Phi = \{ (\sigma,y,y',v): d_v\Phi=0 \}, 
\end{equation}
in the case away from corners, while by 
\begin{equation} \label{eq:C-Phi-defn-2}
    C_\Phi = \big\{ (\sigma,y,y',v,w): d_v\Phi=0, d_w\Phi_2 = 0 \big\}, 
\end{equation}
when we are near $\lb \cap \bfs$, then the map sending $(\sigma,y,y',v)$ (resp. $(\sigma,y,y',v,w)$, in the second case) to $(\sigma,y,y',\mu,\mu',\nu,\nu')$ as in \eqref{LPhiparam} (resp. \eqref{LPhiparam-codim2}) is a diffeomorphism from $C_{\Phi}$ to $L^\bfs$ locally.
Next, we recall the basic property of parametrizations.


\begin{proposition} \label{prop:minimal-parametrization-low}
Let $L, \Phi$ be as in Definition~\ref{def:Legendre-parametrization-low}. Then $\Phi(\sigma, y, y',v)$ always exists locally. 
In addition, let $k$ be the drop of rank of the projection $L \to \bfs$ at $\msf{q}_0 \in L$, 
then the smallest number of extra parameters in a parametrization of $L$ near $\msf{q}_0$ is $k$.
In particular, $k \leq n-2$ when $L=L^{\bfs}$.
We will call such a parametrization \emph{minimal}.
\end{proposition}

\begin{proof}
 The existence of such parametrizations and in particular with number of parameters equal to the rank drop follows from the construction in \cite[Proposition~5]{melrose1996scattering} and \cite[Proposition~3.3]{hassell1999spectral}. 
 The minimality of such $k$ was shown in \cite[Equation~(6.6)]{melrose1996scattering} since $k_{\min}$ there is precisely the rank drop of the projection $L \to \bfs$ and we briefly explain this.
This is because, with a parametrization, one can select some components of $v$ and $(\sigma,y,y')$ to form a coordinate system of $C_{\Phi}$, which effectively gives a coordinate system on $L$ when we identify them via the parametrization map. Then the number of selected $v_i$ has to be at least the number of $dy_j$ or $dy'_j$ such that $dy_j|_{T_{\mathsf{q}_0}L}=0$ or $dy'_j|_{T_{\mathsf{q}_0}L}=0$, which is the rank drop of the projection, as desired. In addition, $k \leq n-2$ is because the rank drop of the projection $L^{\bfs}\to\bfs$ can be at most $n-2$, which follows from the fact that the exponential map (on $Y$) is always non-degenerate on the radial direction and can at most have rank drop $n-2$.
\end{proof}

As an example, consider the case of a Legendre submanifold $L$ that projects diffeomorphically to the base $\bfs$, in the sense that the projection from ${}^\Phi T^* _{\bfs}X^2_b$ to $\bfs$ restricts to a (local) diffeomorphism from $L$ to $\bfs$. In this case, we can take $v$ to be empty and there exists a function $\Phi : \bfs \to \R$ such that (locally) $L$ is the graph of the differential of the function $\Phi/x$, or in coordinates,
\begin{multline} \label{eq:Lparam-example-projectable}
L = \Big\{ \mu = d_y \Phi(\sigma, y, y'), \ \mu' = \sigma^{-1} d_{y'} \Phi(\sigma, y, y'), \\ \nu = \Phi(\sigma, y, y') - \sigma d_\sigma \Phi(\sigma, y, y'), \ \nu' = d_\sigma \Phi(\sigma, y, y') \Big\}.
\end{multline}

Observe that if we take the union of the points of \eqref{eq:Lbf-definition-gamma^2} with $s_l=s_r$, over all $(y, \eta)\in S^* \partial X$, then we get a codimension one submanifold of $L^{\bfs}$, which is also a codimension one submanifold of the conormal bundle of the diagonal $N^* \Diagb$, given by
$$
N^* \Diagb = \big\{ (\sigma, y, y', \mu, \mu', \nu, \nu') \mid y=y', \ \sigma = 1, \ \mu = -\mu', \ \nu = -\nu' \big\}.
$$
It turns out that in a deleted neighbourhood of $N^* \Diagb$, $L^{\bfs}$ projects in a 2:1 fashion to the base $\bfs$, i.e. $L^{\bfs} \setminus N^* \Diagb$ consists of 2 sheets, each of which projects diffeomorphically to the base $\bfs$, and is parametrized by the function $\pm d_{\conic}$, where $d_{\conic}$ is  the distance function  on the cone over $\partial X$.
 The conic distance $d_{\conic}$ has an explicit expression when
$d_{\partial X}(y,y') < \pi$.  Writing $r = 1/x$, $r' = 1/x' =
\sigma /x$, it takes the form
\begin{equation} \label{dconic}
d_{\conic}(y,y', r, r') = \sqrt{ r^2 + {r'}^2 - 2r r' \cos
d_{\partial X}(y, y')} = r \sqrt{1 + \sigma^2 -2\sigma \cos
d_{\partial X}(y, y')}. \end{equation}
Notice that $d_{\conic}(y,y', r, r')/r$ indeed has the form $\Phi(\sigma, y, y')/x$, and is smooth provided that $\cos d_{h}(y, y')$ is smooth, i.e., $d_{h}(y, y')$ is less than the injective radius on $(Y, h)$.


We now introduce the class of Legendre distributions. We first consider the low-energy regime. 
Near $\BFS \cup \LB \cup \RB$, it is a family of Legendre distributions parametrized by $\lambda$ and conormal in $\lambda$ at $\lambda=0$. Away from $\BFS \cup \LB \cup \RB$, it is polyhomogeneous with index family 
\begin{equation} \label{eq:cal-B-index}
\mathcal{B} = (\mathcal{B}_{\bfs_0}, \mathcal{B}_{\lb_0}, \mathcal{B}_{\rb_0} , \mathcal{B}_{\zf}) 
\end{equation}
consisting of index sets for each of the boundary hypersurfaces $\{\bfs_0,\lb_0,\rb_0, \zf \}$ of $X_{\rmb,\flat}^2$.
And we let $m,r_{\LB},r_{\RB} \in \R$ be orders associated with the `main face' $\BFS$ and $\LB,\RB$.

We first recall the definition of Legendre distributions at a fixed energy level introduced in \cite[Section~4]{hassell1999spectral} and \cite[Section~2]{hassell2001resolvent}.
Recall $X_{\rmb}^2$ in \eqref{eq:b-double-space} and $\lb,\rb,\bfs$ defined after it, we use ${}^{\Phi}\Omega^{1/2}$
to denote the half-density bundle obtained from lifting the scattering density of left and right factors to $X_{\rmb}^2$.
\footnote{This was introduced in \cite[Equation~(4.5)]{hassell1999spectral}, where it was called ${}^{\Phi}\Omega^{1/2}$. We changed the notation to match our notation for ${}^{\Phi}T^*X_{\rmb}^2$. }
So a typical non-vanishing section of it is $|dgdg'|^{1/2}$, where $|dg|$ and $|dg'|$ are Riemannian densities lifted to $X_{\rmb}^2$ from the left and right factors.

\begin{definition} \label{def:Legendre-fixed-energy}
Let $L$ be a Legendre submanifold of ${}^{\Phi}T_{\bfs}^* X_{\rmb}^2$. It in fact gives a family of Legendre submanifolds of ${}^{\Phi}T_{\bfs}^* X_{\rmb}^2$ via dilating by $\lambda$ in the fiber. We denote this Legendre submanifold by $\lambda L$. 
Then $I^{m,r_\lb,r_\rb}(X_{\rmb}^2,\lambda L;{}^{\Phi}\Omega^{1/2})$, 
which we call the class of Legendre distributions (for fixed energy level) associated with $\lambda L$, consists of distributions that can be written as 
\begin{equation*}
    u = \sum_{j=1}^6 u_j,
\end{equation*}
where
\begin{itemize}
  \item 
$u_1$ is a finite sum of terms supported near a point in the interior of $\lb$, and of the form
\begin{equation} \label{eq:Legendrian-dis-fixed-lambda-lb}
 x^{r_\lb-\frac{k}{2}} \int_{\R^k}  e^{i\lambda \Phi(y,v)/x} a(x,y,z',v)dv |dgdg'|^{1/2}
\end{equation}
where $z'$ is the coordinate system on $X$ lifted to $X_{\rmb}^2$ from the right factor, and $a \in C_c^\infty$, and $\Phi(y,v)$ parametrizes $L_{\lb}$ in Proposition~\ref{prop:Legendre-fibration-low} in the sense of Definition~\ref{def:parametrization-lb-rb-low} with $v \in \R^k$ being the parameter;
 

\item 
$u_2$ is a finite sum of terms supported near a point in the corner $\lb \cap \bfs$ taking the form
\begin{equation} \label{eq:Legendrian-dis-fixed-lambda-lb-bf}
(x')^{m - \frac{k+k'}{2} + \frac{n}{2} }\sigma^{r_\lb-\frac{k}{2}}
\int_{\R^{k+k'}} e^{i\lambda \Phi(\sigma,y,y',v,w)/x}
a(x',\sigma,y,y',v,w)  dvdw |dgdg'|^{1/2},
\end{equation}
where $\Phi$ is as in \eqref{eq:Phi-low-codim2}, with $v \in \R^{k}$, $w \in \R^{k'}$ as extra parameters and parametrizes $L$ locally, and $a \in C_c^\infty$;

\item $u_3$ is a finite sum of terms supported near a point in the interior of $\bfs$ taking the form
\begin{equation} \label{eq:Legendrian-dis-fixed-lambda-bf}
x^{m-\frac{k}{2}+\frac{n}{2}}\int_{\R^k} e^{i\lambda \Phi(\sigma,y,y',v) / x} a(x,\sigma,y,y',v) dv  |dgdg'|^{1/2},
\end{equation}
where $\Phi(\sigma,y,y',v)$ with $v \in \R^k$ being the extra parameter parametrizes $L$ as in \eqref{LPhiparam} and $a \in C_c^\infty$;

\item $u_4$ is a finite sum of terms supported near a point in the interior of $\rb$ of the same form as \eqref{eq:Legendrian-dis-fixed-lambda-lb}, except with primed and un-primed variables switched and with $k$ replaced by $k'$ as well.

\item $u_5$ is a finite sum of terms supported near a point in $\rb \cap \bfs$, and it is of the  same form as \eqref{eq:Legendrian-dis-fixed-lambda-lb-bf}, except with primed and un-primed variables switched and $\sigma$ replaced by $\theta = \sigma^{-1}$ and $k,k'$ interchanged as well;

\item $u_6 \in \mathcal{S}(X_{\rmb}^2)$.
\end{itemize}

\end{definition}

Using Definition~\ref{def:Legendre-fixed-energy}, the definition of low-energy Legendre distributions is as follows.
\begin{definition} \label{def:Legendre-distribution-low-energy}
Let $L$ be a Legendre submanifold of ${}^{\Phi}T_{\bfs}^* X_{\rmb}^2$ and denote the corresponding dilated Legendre submanifold by $\lambda L$ as above.
Then the class $I_{\calclow}^{m,r_{\LB},r_{\RB};\mathcal{B}}(X_{\rmb,\flat}^2,L ; \Omega_{\flat}^{1/2})$ of Legendre distributions consists of half-densities $u$ on $X_{\rmb,\flat}^2$ that can be decomposed into
\begin{equation*}
    u = \sum_{j=0}^6 u_j,
\end{equation*}
where 
\begin{enumerate}
\item $u_0$ is supported in $\{ \lambda \geq \epsilon \}$ for some $\epsilon > 0$, and $u_0 \otimes |d\lambda/\lambda|^{-1/2}$ is a family of Legendre distributions in $I^{m, r_{\LB}, r_{\RB}}(X_{\rmb}^2,  \lambda L; {}^{\Phi}\Omega^{1/2})$ with symbol depending smoothly on $\lambda$;

\item $u_1$ is supported close to $\bfs_0 \cap \BFS$ and away from $\LB \cup \RB$, and is given by a finite sum of expressions
\begin{equation}
\rho^{m-k/2+n/2} \lambda^n  \int_{\R^k} e^{i \Phi(\sigma, y, y' , v)/\rho} a(\lambda, \rho, \sigma, y, y' ,v) \, dv \Big| \frac{dg dg' d\lambda}{\lambda}\Big|^{1/2}
\label{u1}\end{equation}
where $\sigma = x/x'$, $\rho=x/\lambda$, $ \Phi$ locally parametrizes $L$ in the sense of Definition~\ref{def:Legendre-parametrization-low} 
and $a$ is polyhomogeneous conormal in $\lambda$, with respect to the index set $\mathcal{B}_{\bfs_0}$, at $\lambda = 0$ and is smooth in all other variables;

\item $u_2$ is supported close to $\bfs_0 \cap \BFS \cap \LB$, and is given by a finite sum of expressions
\begin{multline}
{\rho'}^{m-(k+k')/2+n/2} \sigma^{r_{\LB} - k/2} \lambda^n  \\ \times \int_{\R^{k+k'}} e^{i\big(  \Phi_1(y,v) + \sigma  \Phi_2(\sigma, y, y' , v, v') \big)/\rho}  a(\lambda, \rho', \sigma, y, y' , v, v') \, dv \, dv' \,  \Big| \frac{dg dg' d\lambda}{\lambda}\Big|^{1/2};
\label{u2}\end{multline}
where $\rho'=x'/\lambda$, $ \Phi_1 + \sigma  \Phi_2$ locally parametrizes $L$ in the sense of Definition~\ref{def:Legendre-parametrization-low} and $a$ is polyhomogeneous conormal in $\lambda$, with respect to the index set $\mathcal{B}_{\bfs_0}$, at $\lambda = 0$ and is smooth in all other variables;

\item $u_3$ is supported close to $\bfs_0 \cap \BFS \cap \RB$, and is given by an expression similar to $u_2$, with $(\sigma,x,y),k$ and $(\theta=x'/x,x',y'),k'$ interchanged, and $r_{\LB}$ replaced by $r_{\RB}$;

\item $u_4$ is supported close to $\LB \cap \bfs_0$ and away from $\BFS$, and is given by a finite sum of expressions of the form
\begin{equation}
\rho^{r_{\LB} - k/2} \lambda^n \int_{\R^k}
e^{i \Phi_1(y,v) /\rho} a(\rho, x', 1/\rho', y, y',  v) \, dv \Big| \frac{dg dg' d\lambda}{\lambda}\Big|^{1/2}
\label{u4}\end{equation}
where $ \Phi_1$ locally parametrizes $L_{\lb}$ and $a$ is polyhomogeneous conormal  in $(x', 1/\rho')$, with respect to the index sets $(\mathcal{B}_{\bfs_0}, \mathcal{B}_{\lb_0})$ and is smooth in all other variables;

\item $u_5$ is supported close to $\RB \cap \bfs_0$ and away from $\BFS$, and is given by a similar expression to $u_4$ with $(x,y,\rho)$ and $(x',y',\rho')$ interchanged, $L_{\lb}$ and $r_{\LB}$ replaced by $L_{\rb}$, $r_{\RB}$, $k$ replaced by $k'$, and $\mathcal{B}_{\lb_0}$ replaced by $\mathcal{B}_{\rb_0}$;

\item $u_6$ is supported away from $\BFS \cup \LB \cup \RB$ and is of the form $a \tau$ where $\tau$ is a smooth nonvanishing section of $\Omega_{\flat}^{1/2}$ and  $a$ is polyhomogeneous with index family $\mathcal{B}$ at $\bfs_0, \lb_0, \rb_0, \zf$.

\end{enumerate}
\end{definition}


Now we discuss some further simplifications due to the special structure of $L^{\bfs}$ and $L^{\#}$.

\begin{proposition} \label{prop:Lbf-Legendrian-dis-simplification-low}
    Let $L = L^{\bfs}$. Oscillatory integrals in the local representation of $u \in I_{\calclow}^{m,r_{\LB},r_{\RB};\mathcal{B}}(X_{\rmb,\flat}^2,L ; \Omega_{\flat}^{1/2})$ can be simplified so that no parameter is needed to parametrize $L^{\bfs}_{\lb}$ and $L^{\bfs}_{\rb}$. Concretely, we have:
    \begin{enumerate}
        \item  \label{Lbf-one-side-parametrization-fixed-energy}
        For $u \in I^{m,r_{LB},r_{\RB}}(X_{\rmb}^2,\lambda L;{}^{\Phi}\Omega^{1/2})$ as in Definition~\ref{def:Legendre-fixed-energy},
        the number of extra parameters $k$ in $u_1,u_2$ and $k'$ in $u_4,u_5$ can be taken to be $0$.

    \item \label{Lbf-one-side-parametrization-low}
    For $u \in I_{\calclow}^{m,r_{\LB},r_{\RB};\mathcal{B}}(X_{\rmb,\flat}^2,L ; \Omega_{\flat}^{1/2})$ in  Definition~\ref{def:Legendre-distribution-low-energy}, 
    we have a simplification for the $u_0$ term
    as in \eqref{Lbf-one-side-parametrization-fixed-energy} above, and in addition the number of extra parameters $k$ in $u_2,u_4$ and $k'$ in $u_3,u_5$ can be taken to be $0$.
        \end{enumerate}
When $L = L^{\#}$ instead, no extra parameter is needed in all parts.
\end{proposition}

\begin{proof}
The conclusion for $L = L^{\#}$ follows from the fact that it projects to the base diffeomorphically and the discussion in the example \eqref{eq:Lparam-example-projectable}.

Now we consider the case $L^{\bfs}$. The claimed simplification is because
\begin{equation}
    L^{\bfs}_{\lb} = \{ \nu = \pm 1, \mu = 0  \}  \subset {}^{\Phi}N^*Z_{\lb}, 
    \quad L^{\bfs}_{\rb} = \{ \nu' = \pm 1, \mu' = 0  \}  \subset {}^{\Phi}N^*Z_{\rb},
\end{equation}
which leads to the fact that the projections 
\begin{equation}
   L^{\bfs}_{\lb} \to Z_{\lb}, \quad L^{\bfs}_{\rb} \to Z_{\rb}
\end{equation}
are local diffeomorphisms (in fact, a $2$-sheet cover). So no extra parameter is needed for parametrizing $L^{\bfs}_{\lb}$ and $L^{\bfs}_{\rb}$ and the conclusion follows.
    
\end{proof}


\subsection{Legendre distributions of Legendrian conic pairs at low energies} 
\label{subsec:conic-pair-geometry-and-phase-function}

In this subsection, we briefly review the geometric structure of the intersecting pair of Legendre submanifolds with conic points $(L^{\bfs}, L^\#)$ and recall the class of Legendre distributions associated with it.
We will first recall the definition in a slightly more general setting and then point out a further simplification of the oscillatory integral expression of the spectral measure due to the special property of $L^{\bfs}$.


Consider the case away from codimension two corners first. 
We introduce coordinates on the bundle ${}^{\Phi}T^* X_{\mathrm{b}}^2$ introduced in Section~\ref{subsec:Legendrian-geometry-low} that we will use, which is convenient to define parametrization of Legendre submanifolds. 
Concretely, over a point $\msf{p} \in {}^{\Phi}T^* X_{\mathrm{b}}^2$ that is away from $\bfs \cap \rb$, we write the canonical 1-form as
\begin{equation} \label{eq:tautological-1-form-bf}
\alpha = \overline{\nu} d(\frac{1}{x}) + \nu_1 \frac{d\sigma}{x} + \mu \cdot \frac{dy}{x} + \mu'\cdot \frac{dy'}{x'},
\end{equation}
and use the coordinate system $(\sigma,y,y',\overline{\nu},\nu_1,\mu,\mu')$ on ${}^{\Phi}T^*_{\bfs} X_{\mathrm{b}}^2$.

If we consider coordinates $(x,y,x',y',\nu,\mu,\nu',\mu')$ on ${}^{\sct}T^*X \times {}^{\sct}T^*X$ by writing the contact form as in \eqref{eq:sc-1-form} for both left and right factors, then we know coordinates in \eqref{eq:tautological-1-form-bf} are related to them by \footnote{Rigorously speaking, only when $x'>0$.}
\begin{equation}
  \sigma = x/x', \, \overline{\nu} = \nu + \sigma \nu',  \, \nu_1 = \nu'. 
\end{equation}

Next we discuss the structure of $(L^{\bfs},L^\#)$ and the precise meaning of the statement that it is an intersecting conic pair.
We consider the resolved ${}^{\Phi}T^*_{\bfs} X_{\rmb}^2$:
\begin{equation} \label{eq:resolved-bundle}
[{}^{\Phi}T^*_{\bfs} X_{\rmb}^2; J^{\calclow}],
\end{equation}
where 
\begin{equation} \label{eq:def-J-low}
  J^{\calclow} = \{ x'=0,\overline{\nu}-(1+\sigma)\nu_1=0,\mu= \mu' =0 \}
\end{equation}
in terms of coordinates as in \eqref{eq:contact-form} and $x$ is a boundary defining function lifted from the left factor of $X_{\rmb}^2$. Here the condition $\overline{\nu}-(1+\sigma)\nu_1=0$ is equivalent to $\nu=\nu'$, when both of them are valid coordinates. 
We then denote the blow down map by
\begin{equation} \label{eq:beta-LCP-low-1}
\beta_{\mathrm{LCP},\calclow}:  [{}^{\Phi}T^*_{\bfs} X_{\rmb}^2; J^{\calclow}] \to {}^{\Phi}T^*_{\bfs} X_{\rmb}^2.
\end{equation}
Then the lift of $L^{\bfs}$ is defined to be
\begin{equation} \label{eq:hat-Lbf-def}
  \hat{L}^{\bfs} = \beta_{\mathrm{LCP},\calclow}^*(L^{\bfs}) := \mathrm{clos}(\beta_{\mathrm{LCP},\calclow}^{-1}(L^{\bfs} \backslash J^{\calclow})).
\end{equation}
The statement that $(L^{\bfs},L^\#)$ is a conic intersecting Legendre pair means that $\hat{L}^{\bfs}$ is a compact manifold with corners, which is proved in \cite[Section~5]{melrose1996scattering} and \cite[Proposition~4.1]{hassell2001resolvent}.

The local coordinates of $[{}^{\Phi}T^*_{\bfs} X_{\rmb}^2; J^{\calclow}]$, for example over the region 
$\frac{ \overline{\nu}-(1+\sigma)\nu_1 }{|\mu'|} \lesssim 1, \frac{|\mu|}{|\mu'|} \lesssim 1$ are given by
\begin{equation} \label{eq:coordinates-resolved-low-1}
  \Big(\sigma,y,y', \overline{\nu}, \hat{\nu}_1 = \frac{ \overline{\nu}-(1+\sigma)\nu_1 }{|\mu'|}, \frac{\mu}{|\mu'|},\hat{\mu'},|\mu'|\Big).
\end{equation}
As pointed out in the discussion after \cite[Equation~(7.22)]{hassell1999spectral}, this is the only region that we need to consider when we are considering the part of $\hat{L}^{\bfs}$ that is away from the corner $\rb \cap \bfs$.

Similarly, when we consider the part near $\rb \cap \bfs$, we only need to consider a similar region on which $|\mu|$ dominates other defining functions of the resolved part.

Now we discuss phase functions parametrizing $(L^{\bfs},\Lsharplow)$. We only discuss the case near $L^{\bfs} \cap \Lsharplow_+$ since simply switching the sign of phase functions handles the other case near $L^{\bfs} \cap \Lsharplow_-$.

We say that\footnote{We have used the simplification as in Proposition~\ref{prop:Lbf-Legendrian-dis-simplification-low} so that the first part of the phase function is just $1+\sigma$. In more general settings, this should be a function of $y$, and even a function depending on extra parameters, if one does not assume either member of the pair projects locally diffeomorphically to the base. }
\begin{align} \label{eq:phase-conic-bf}
\Phi(\sigma,y,y',v,s) = 1 + \sigma +  s \sigma \psi(\sigma,y,y', v,s).
\end{align}
parametrizes the conic pair $(L^{\bfs},L^{\#})$ near $\hat{q}_0  \in \partial\hat{L}^{\bfs} = \hat{L}^{\bfs} \cap \beta_{\mathrm{LCP},\calclow}^*(\Lsharplow)$ 
if locally $\hat{L}^{\bfs}$ can be represented as
\begin{equation} \label{eq:hat-Lbf-parametrization-bf}
\hat{L}^{\bfs} = \beta_{\mathrm{LCP},\calclow}^*
 \Big\{ d\big(\frac{\Phi}{\sigma x'}\big)  \big| \;  (\sigma,y,y',v,s) \in C_{\Phi} \Big\}
\end{equation}
where the last two components are the polar coordinates introduced by the resolution \eqref{eq:resolved-bundle} and $C_{\Phi}$ is the critical set of $\Phi$:
\begin{equation} \label{eq:C-Phi-conic-pair-defn}
C_{\Phi} = \{ (\sigma,y,y',v,s) : d_s\Phi = 0, \, d_v\psi = 0, \, s \geq 0 , v \in \R^k \}.
\end{equation}
More concretely, in terms of \eqref{eq:coordinates-resolved-low-1}, the right hand side of \eqref{eq:hat-Lbf-parametrization-bf} is given by
\begin{equation} \label{eq:hat-Lbf-parametrization-bf-2}
  \beta_{\mathrm{LCP},\calclow}^*\Big\{ \Big(\sigma,y,y',\Phi,\frac{-(1+\sigma)\sigma d_{\sigma}\psi - \psi}{|d_{y'}\psi|},\frac{d_y(\sigma\psi)}{|d_{y'}\psi|},\frac{d_{y'}\psi}{|d_{y'}\psi|},s|d_{y'}\psi|\Big) \Big| \; (\sigma,y,y',v,s) \in C_{\Phi} \Big\}.  
\end{equation}

Let $q_0=(\sigma_0,y_0,y_0',v_0,s=0) \in C_{\Phi}$ so that it is sent to $q_0$ as in \eqref{eq:hat-Lbf-parametrization-bf}. We say this parametrization is \emph{non-degenerate} if
\begin{equation} \label{eq:non-degenerate-parametrization-conic-low}
 ds, \; d_{(\sigma,y,y',v,s)}\psi, \; d_{(\sigma,y,y',v,s)}(\frac{\partial \psi}{\partial v^j}), \; j=1,2,...,k  \text{ \; are linearly independent at \; } q_0.
\end{equation}

Now we turn to the part near the corner $\lb \cap \bfs$.
Because we are in the region that $\nu,\nu'$ are close to $1$ and $\sigma \lesssim 1$, this part over $\lb \cap \bfs$ corresponds to the part, in terms of \eqref{eq:Lbf-definition-gamma^2}, $s_l \to \pi, \, s_r \to 0$ and with $\pi-s_l \ll s_r$. So correspondingly $|\mu| \ll |\mu'| \ll 1$ and $|\mu'|$ is chosen to be the `large parameter' for the polar coordinates in the tangential frequency.

All the definitions above hold uniformly down to $\sigma \to 0$, except that we make a stronger requirement for the non-degeneracy of the parametrization. 
Let $q_0=(0,y_0,y_0',v_0,s=0) \in C_{\Phi}$ so that it is sent to $q_0$ as in \eqref{eq:hat-Lbf-parametrization-bf}. We say this parametrization is non-degenerate near $q_0$ if
\begin{equation} \label{eq:non-degenerate-parametrization-conic-low-lb-corner}
 d_{y'}\psi, \;  d_{(y',v)}(\frac{\partial \psi}{\partial v^j}), \; j=1,2,...,k  \text{ \; are linearly independent at \; } q_0.
\end{equation}
Near the corner $\rb \cap \bfs$, we instead use phase functions of the form
\begin{align} \label{eq:phase-conic-low-bf-rb-corner}
\Phi(\theta,y,y',v,s) = 1 + \theta +  s \theta \psi(\theta ,y,y', v,s),
\end{align}
where $\theta = x'/x$. All the definitions of parametrization and its non-degeneracy are the same, except with $\sigma$ replaced by $\theta$.

Now we give the characterization of the minimal number of extra parameters in this setting, which is an analogue of Proposition~\ref{prop:minimal-parametrization-low}. 
\begin{proposition} \label{prop:minimal-parametrization-low-conic}
Let $\hat{L}^{\bfs}, \Phi$ be as above. Then $\Phi(\sigma, y, y',v,s)$ always exists locally. 
In addition, let $k$ be the drop of rank of the projection $\hat{L}^{\bfs} \to \bfs$ at $\msf{q}_0 \in L$, 
then the minimal number of extra parameters, including $s$, in a parametrization of $\hat{L}^{\bfs}$ near $\msf{q}_0$ is $k+1$. In addition $k \leq n-2$. 
We will call such a parametrization \emph{minimal}.
\end{proposition}

The existence of the parametrization will follow from the concrete construction that we will give in Section~\ref{subsec:Leg-concrete-setup}. 
The property of $k$ follows in the same way as in the proof of Proposition~\ref{prop:minimal-parametrization-low}.

Now we discuss the class of Legendre distributions associated with $(L^{\bfs},L^{\#})$ using Definition~\ref{def:Legendre-distribution-low-energy} with $L=L^{\bfs},L^{\#}$ as building blocks.
\begin{definition} \label{def:Legendrian-dis-conic-intersecting-low}
Let $(L^{\bfs}, L^{\#})$ be a pair of intersecting Legendre submanifolds with conic points in ${}^{\Phi}T^*_{\bfs}X_{\rmb}^2$. 
For $m,p, r_{\LB},r_{\RB} \in \R$ and $\mathcal{B}$ as in \eqref{eq:cal-B-index}, the space of Legendre distributions associated with $(L^{\bfs}, L^{\#})$, which we denote by  $I_{\calclow}^{m, p; r_{\LB}, r_{\RB}; \mathcal{B}}(X_{\rmb,\flat}^2, (L^{\bfs}, L^{\#}); \Omega_{\flat}^{1/2})$ consists of half-densities $u$ on $X_{\rmb,\flat}^2$ that can be written as a finite sum of terms $u = \sum_{j=0}^5 u_j$, where
\begin{itemize}

\item $u_0$ is supported in $\{ \lambda \geq \epsilon \}$ for some $\epsilon > 0$, and $u \otimes |d\lambda/\lambda|^{-1/2}$ is a family of 
Legendre distributions in $I_{\calclow}^{m, p; r_{\LB}, r_{\RB}}(X_{\rmb}^2, ( \lambda L^{\bfs},  \lambda L^{\#}); {}^{\Phi}\Omega^{1/2})$ with symbol depending smoothly on $\lambda$;

\item $u_1$ is an element of $I_{\calclow}^{m, r_{\LB}, r_{\RB}; \mathcal{B}}(X_{\rmb,\flat}^2, L^{\bfs}; \Omega_{\flat}^{1/2})$, microsupported away from $L^{\#}$;

\item $u_2$ is an element of $I_{\calclow}^{p, r_{\LB}, r_{\RB}; \mathcal{B}}(X_{\rmb,\flat}^2, L^{\#}; \Omega_{\flat}^{1/2})$, microsupported away from $L^{\bfs}$;

\item $u_3$ is supported close to $\bfs_0 \cap \bfs$, and away from $\lb \cup \rb$,  and is given by a finite sum of expressions 
\begin{multline}
 \lambda^n \int_0^\infty ds \int_{\R^k} e^{i \Phi(\sigma, y, y' , v,s)/\rho} \big( \frac{\rho}{s} \big)^{m-(k+1)/2+n/2} s^{p+n/2-1} \\ 
 \times
 a(\lambda, \frac{\rho}{s}, \sigma, y, y' ,v,s) \, dv \Big| \frac{dg dg' d\lambda}{\lambda}\Big|^{1/2}
\label{u3con}\end{multline}
where $ \Phi$ locally parametrizes $(L^{\bfs}, L^{\#})$ in the sense of \cite[Eq. (3.31)]{hassell1999spectral} 
and $a$ is polyhomogeneous conormal in $\lambda$, with respect to the index set $\mathcal{B}_{\bfs_0}$, at $\lambda = 0$ and is smooth in all other variables;

\item $u_4$ is supported close to $\bfs_0 \cap \bfs \cap \lb$ and is given by a finite sum of expressions
\begin{multline}
 \lambda^n \int_0^\infty ds \int_{\R^k} e^{i \Phi(\sigma, y, y' , v,s)/\rho} \big( \frac{\rho'}{s} \big)^{m-(k+1)/2+n/2} s^{p+n/2-1} \sigma^{r_{\RB}} \\ \times
 a(\lambda, \frac{\rho'}{s}, \sigma, y, y' ,v,s) \, dv \Big| \frac{dg dg' d\lambda}{\lambda}\Big|^{1/2}
\label{u4con}\end{multline}
where $ \Phi$ locally parametrizes $(L^{\bfs}, L^{\#})$ in the sense of \cite[Eq. (3.31)]{hassell1999spectral} and $a$ is polyhomogeneous conormal  in $\lambda$, with respect to the index set $\mathcal{B}_{\bfs_0}$, at $\lambda = 0$ and is smooth in all other variables;

\item $u_5$ is supported close to $\bfs_0 \cap \bfs \cap \rb$ and is given by a finite sum of expressions analogous to \eqref{u4con}, with $(\sigma,x,y,\rho')$ and $(\theta=x'/x,x',y',\rho)$ interchanged, and $r_{\LB}$ replaced by $r_{\RB}$.

\end{itemize}

\end{definition}

\subsection{Legendre submanifolds and distributions at high energies}
\label{subsec:Legendrian-geometry-high}

We discuss Legendre submanifolds and distributions in the high-energy regime in this subsection. The major difference with the low-energy regime is that we need to take the dynamics over the interior of $X$ into consideration now. 
We only discuss definitions at the level of generality that is sufficient for our applications instead of the more general scattering fibered setting in \cite{Hassell-Wunsch-semiclassical-resolvent}. 
In particular, in the notation used there, we always have $y_3,\mu_3,v_3$ being empty.

We start by recalling the bundle ${}^{\bundlehigh}T^*X_{\rmb,\calchigh}^2$ over $X_{\rmb,\calchigh}^2 = [0,h_0) \times X_{\rmb}^2$ that captures this high-energy propagation, which is also the contact manifold that our Legendre submanifold lives in. 
We give expressions for writing points in this cotangent bundle (which are covectors over the base point in $X_{\rmb,\calchigh}^2$) in different regions directly, but refer the reader to \cite[Section~3, 11]{Hassell-Wunsch-semiclassical-resolvent} for a detailed construction of this bundle, which needs extra fibrations characterized there, in particular \cite[Example~3.5]{Hassell-Wunsch-semiclassical-resolvent}.

Over $[0,h_0) \times (X_{\rmb}^2)^{\circ}$ \footnote{Here $(X_{\rmb}^2)^{\circ}$ denotes the interior of $X_{\rmb}^2$, which is just $X^\circ \times X^\circ$.},
let $z,z'$ be coordinates on $X^{\circ}$ lifted to $X^\circ \times X^\circ$ from the left and right factors respectively. We write the contact form as
\begin{equation} \label{eq:semiclassical-form-interior}
   \zeta \cdot \frac{dz}{h} + \zeta' \cdot \frac{dz'}{h} + \tau d(\frac{1}{h}).
\end{equation}
Here and below, letters appearing in front of $d(\bullet)$ are frequency variables dual to $\bullet$. When we are near the boundary but away from $\bfs$, say $x' \to 0$ but away from $x=0$, we instead write
\begin{equation} \label{eq:semiclassical-form-boundary-1}
   \zeta \cdot \frac{dz}{h} + \xi' \cdot d(\frac{1}{x'h}) + \mu \cdot \frac{dy}{xh}
   + \mu' \cdot \frac{dy'}{x'h} +  \tau' d(\frac{1}{h}).
\end{equation}
Near $\bfs$, say in the region that $\theta = \frac{x'}{x} \lesssim 1$, then the coordinate system on $X_{\rmb,\calchigh}^2$ is $(h,\theta,x,y,y')$ and we write the contact form as
\begin{equation} \label{eq:semiclassical-form-bf}
\xi d(\frac{1}{x h}) + \xi' d(\frac{1}{x\theta h}) +  \mu \cdot \frac{dy}{xh} + \mu' \cdot \frac{dy'}{x\theta h} + \overline{\tau} d(\frac{1}{h}) .
\end{equation}
In the region $\sigma = x/x' \lesssim 1$, we can write it in a way that is more parallel to \eqref{eq:tautological-1-form-bf} (with a different $\overline{\tau}$):
\begin{equation} \label{eq:semiclassical-form-bf-lb}
 h^{-1} \Big( \overline{\nu}d(\frac{1}{x}) + \nu_1 \frac{d\sigma}{x}+  \mu \cdot \frac{dy}{x} + \mu' \cdot \frac{dy'}{x'} \Big) + \overline{\tau} d(\frac{1}{h}) .
\end{equation}
The boundary face $\{h=0\} \times X_{\rmb}^2 \subset X_{\rmb,\calchigh}^2$ is denoted by $\smf$ (standing for `semiclassical face'), and then the contact manifold we consider is ${}^{\Phi}T_{ \smf }^*X_{\rmb,\calchigh}^2$.



Now we start to characterize our high-energy propagating Legendre submanifold, which we denote by $L^{\bfs,\calchigh}$. 
Let ${}^{\calchigh}\overline{T}^*X$ be the high-energy phase space defined in Section~\ref{subsec:low-high-combined-spaces}. We denote the characteristic set of $P$ by
\begin{equation} \label{eq:def-Sigma}
\Sigma = \mathrm{clos} \big\{  (z,\zeta) \in {}^{\calchigh}T_{\{h=0\}}^*X:  g^{ij}(z)\zeta_i\zeta_j = 1  \big\} \subset {}^{\calchigh}\overline{T}_{\{h=0\}}^*X,
\end{equation}
where the coordinate system $(z,\zeta)$ is over some region in the interior of ${}^{\calchigh}\overline{T}_{\{h=0\}}^*X$.
Roughly speaking, this is just the unit (semiclassical scattering) sphere bundle.
Let $G_{\tau}$ be the geodesic flow on ${}^{\calchigh}T^*_{\{h=0\}}X$ (identified with ${}^{\sct}T^*X$, on which the definition of the geodesic flow becomes natural) with unit speed. 
Then $L^{\bfs,\calchigh}$ over $\{ 0 \}_h \times (X^{\circ} \times X^{\circ})$ is given by
\begin{equation} \label{eq:L-high-interior}
(L^{\bfs,\calchigh})^\circ =  \big\{ (q,(q')',h=0,\tau): \; q,q' \in \Sigma, \; q = G_{\tau}(q' ) \big\}
\end{equation}
in terms of coordinates in \eqref{eq:semiclassical-form-interior}, where $q = (z,\zeta), \, q' = (z',\zeta') \in ({}^{\calchigh}\overline{T}_{\{h=0\}}^*X)^{\circ}$ and $(q')'$ means switching the sign of the frequency component of $q'$.
Finally, the complete $L^{\bfs,\calchigh}$ is taken as the closure of \eqref{eq:L-high-interior} in ${}^{\Phi}T_{ \smf }^*X_{\rmb,\calchigh}^2$.
As discussed in \cite[Section~3]{Hassell-Zhang2016Strichartz}, this is connected to $L^{\bfs}$, which carries the propagation for finite energies in the following way: if we restrict it to $\{0\}_h \times \bfs$ and forget the $\tau$-component, then it equals $L^{\bfs}$.

\begin{remark}
Switching from coordinates in \eqref{eq:semiclassical-form-interior} to \eqref{eq:semiclassical-form-bf}, in particular from $\tau$ to $\overline{\tau}$ is effectively a renormalization process.  
As shown in \eqref{eq:L-high-interior}, $\tau$ in \eqref{eq:semiclassical-form-interior} is precisely the travelling time on $L^{\bfs,\calchigh}$. 
This $\tau$-coordinate does tend to infinity as $z,z' \to \partial X$, but if we switch to coordinates in \eqref{eq:semiclassical-form-bf}, then $\overline{\tau}$ tends to a finite limit as $z,z' \to \partial X$ along a geodesic.
This in fact gives the sojourn time introduced by Guillemin \cite{Guillemin-sojourn}, see \cite[Section~15]{Hassell-Wunsch-semiclassical-resolvent} for detailed discussions.
This is also used in \cite{Jia26-sc-inverse} to determine the metric from the scattering map in the time-dependent setting.
\end{remark}

\begin{remark}
It might seem strange that the $L^{\bfs}$ seems to live only at $\BFS \cap \bfs_0$, where $\lambda = 0$, while suddenly this $L^{\bfs,\calchigh}$ lives at $\lambda = \infty$ and includes the interior part. But in fact, one should think of $L^{\bfs}$ as living at all $\lambda$ finite via the dilation. 
This phenomenon is similar to what happened in \cite{gell2022propagation}: the flow corresponding to infinite frequencies enters the interior while the flow corresponding to the finite-frequency part remains at the boundary.
\end{remark}

Now we recall from \cite[Section~4.2]{Hassell-Wunsch-semiclassical-resolvent} the definition of parametrizations of Legendre submanifolds of ${}^{\bundlehigh}T^*_{ \smf }X_{\rmb, \calchigh}^2$.
Let $G \subset {}^{\bundlehigh}T^*_{ \smf }X_{\rmb, \calchigh}^2$ be a Legendre submanifold. The case near the boundary and away from corners is the same as in \cite{melrose1996scattering}. The case near codimension two corners is the same as in \cite[Section~3]{hassell1999spectral} and \cite{hassell2001resolvent}. So we only discuss the definition of parametrization near $q_0 \in {}^{\bundlehigh}T^*_{ \{ \smf \cap \LB \cap \BFS \} }X_{\rmb, \calchigh}^2$. 
In this region, a local coordinate system on $X_{\rmb, \calchigh}^2$ is $(h,\theta=\frac{x'}{x},x,y,y')$.
Let the coordinate of $q_0$ be $q_0=(0,0,0,y_0,y'_0,\xi_0,\xi_0',\mu_0,\mu_0',\overline{\tau}_0)$ in terms of coordinates in \eqref{eq:semiclassical-form-bf}. The phase function we use is
\begin{equation}
\Phi(x,y,y',v_1,v_2) = \psi_1(y,v_1)+ \theta \psi_2(\theta,y,y',v_1,v_2) + \theta x \psi_3(\theta,x,y,y',v_1,v_2),
\end{equation}
where $v_j \in \R^{k_j}$ are parameters.
Here $\psi_1,\psi_2,\psi_3$ are smooth functions supported near $(y_0,v_{10})$, $(0,y_0,y'_0,v_{10},v_{20})$, $q'_0 = (0,0,y_0,y'_0,v_{10},v_{20})$ respectively, where
 $v_{j0} \in \R^{k_j}$ are arbitrarily chosen (hence in particular can be chosen to be $0$) and
\begin{equation}
d(\frac{\Phi}{\theta x h})(q'_0) = q_0, \quad d_{v_1,v_2}\Phi(q'_0) = 0.
\end{equation}

Then we say that $\Phi$ parametrizes $G$ near $q_0$ if $G$ is locally given by
\begin{equation} \label{eq:def-parametrization-high}
G = \{ d(\frac{\Phi}{\theta x h})(q') | q' \in C_{\Phi}  \},
\end{equation}
where $C_{\Phi}$ is the critical set\footnote{ All points here are restricted to a neighborhood of $q'_0$. } of $\Phi$:
\begin{equation} \label{eq:C-Phi-high-1}
C_{\Phi} = \{ q'=(\theta,x,y,y',v_1,v_2) | d_{v_1,v_2}\Phi(q') = 0 \}.
\end{equation}

Such a parametrization always exists locally, see \cite[Section~4.3]{Hassell-Wunsch-semiclassical-resolvent}.
Let  $v_i^j$ be the $j$-th component of $v_i$, then we say this parametrization or $\Phi$ is non-degenerate if
\begin{equation}
d_{(y,v_1)} \partial_{v_1^j}\psi_1,
\end{equation}
are independent at\footnote{Hence this non-degeneracy holds near there as well. The same applies to \eqref{eq:dy'v2psi2}.} $(y_0,v_{10})$,  and
\begin{equation} \label{eq:dy'v2psi2}
d_{(y',v_2)} \partial_{v_2^j}\psi_2
\end{equation}
are independent at $(y_0,y'_0,v_{10},v_{20})$.
We have the following analogue of Proposition~\ref{prop:minimal-parametrization-low} in the high-energy regime characterizing minimal parametrizations.  
\begin{proposition} \label{prop:minimal-parametrization-high}
Let $G, \Phi$ be as above. Then $\Phi$ always exists locally. In addition, let $k$ be the drop of rank of the projection $G \to \smf$ at $\msf{q}_0 \in G$, 
then the smallest number of extra parameters in a parametrization of $G$ near $\msf{q}_0$ is $k$.
We will call such a parametrization \emph{minimal}.
\end{proposition}

Next, we discuss Legendre distributions in the high-energy regime. Recall that $\Omega_{\calchigh}^{1/2}$ is the half-density bundle with its typical section taking the form \eqref{eq:high-energy-half-density}. 
Here we adopt the modification of the order convention as in \cite[Remark~4.3]{GHS1}\cite{Hassell-Zhang2016Strichartz} as compared to \cite{Hassell-Wunsch-semiclassical-resolvent} so that the order matches that of the low-energy part.
In this definition and also other classes of Legendre distributions below, we get one further simplification compared with \cite{Hassell-Wunsch-semiclassical-resolvent}: powers like $-\frac{f_i}{2}+\frac{N}{4}$ there arising from the scattering fibered structure vanish. 


\begin{definition} \label{def:Legendre-distribution-high-energy}
Let $G$ be a Legendre submanifold in ${}^{\bundlehigh}T^*_{\smf}X_{\rmb,\calchigh}^2$. 
The class of Legendre distributions associated with $G$, which we denote by $I^{m_{\calchigh},m ,r_{\LB},r_{\RB}}_{\calchigh}(X_{\rmb,\calchigh}^2,G; \Omega_{\calchigh}^{1/2})$
with orders $m_{\calchigh},m,r_{\LB},r_{\RB}$ being orders associated with $\smf, \BFS, \LB,\RB$ respectively, consists of distributions that can be written as a sum
\begin{equation*}
    u =  \sum_{j=0}^5 u_j,
\end{equation*}
with $u_j$ characterized below.

\begin{itemize}

\item  $u_0$ is supported near a point in the interior of $\RB \cap \smf$ and it is a finite sum of oscillatory integrals of the following form, satisfying the support conditions below
\begin{equation} \label{eq:high-Legendre-RB}
  h^{m_{\calchigh}-\frac{k_2+k_1}{2}-\frac{n}{2}} \theta^{r_{\RB}- \frac{k_2}{2} } \int_{\R^{k_1+k_2}} e^{i \frac{\Phi}{x'h}} a(h,x,\theta,y,y',v_1,v_2) dv_1dv_2   \big|dgdg'\frac{dh}{h^2}\big|^{1/2},
\end{equation}
where $v_1 \in \R^{k_1}, \, v_2 \in \R^{k_2}$ and $\Phi = \psi_1(y',v_1)+ \theta \psi_2(\theta,y,y',v_1,v_2)$ parametrizes $G$ and $a \in C_c^\infty$.

\item $u_1$ is supported near a point in the interior of $\LB \cap \smf$ and takes the same form as $u_0$, except with primed and un-primed variables switched, $r_{\RB}$ replaced by $r_{\LB}$ and $\theta$ replaced by $\sigma = \theta^{-1}$.

\item $u_2$ is supported near a point in the interior of $\BFS \cap \smf$ and it is a finite sum of oscillatory integrals taking the form
\begin{equation} \label{eq:high-Legendre-BF}
 h^{m_{\calchigh}-\frac{k_2}{2}-\frac{n}{2}} (x')^{m-\frac{k_2}{2}+\frac{n}{2}} \int_{\R^{k_2}} e^{i \frac{\Phi}{x'h}} a(h,x,\theta,y,y',v_2) dv_2   \big|dgdg'\frac{dh}{h^2}\big|^{1/2},
\end{equation}
where $v_2 \in \R^{k_2}$ and $\Phi(\theta,y,y',v_2)$ parametrizes $G$ and $a \in C_c^\infty$.

\item $u_3$ is a finite sum of oscillatory integrals of the following form in the region $x' \lesssim x$:
\begin{equation}
h^{m_{\calchigh} - \frac{k_2+k_1}{2}  - \frac{n}{2} } x^{m- \frac{k_2+k_1}{2} + \frac{n}{2} } \theta^{r_{\RB} -\frac{k_2}{2} }  \int_{\R^{k_1+k_2}} e^{ i \frac{\Phi}{h x'} } a(h,x,\theta,y,y',v_1,v_2)dv_1dv_2  \big|dg dg'\frac{dh}{h^2}\big|^{1/2},
\end{equation}
where $v_1 \in \R^{k_1},\, v_2 \in \R^{k_2}$ and $\Phi = \psi_1(y',v_1) + \theta\psi_2(\theta,y,y',v_1,v_2)$ parametrizes $G$;
while on the region $x \lesssim x'$, we consider oscillatory integrals of the same form but switch primed and un-primed variables and replace $r_{\RB}$ by $r_{\LB}$, $\theta$ by $\sigma = \theta^{-1}$.

\item $u_4$ is given by a finite sum of oscillatory integrals of the following form in the region $x' \lesssim x$:
\begin{equation}
x^{m- \frac{k+k_1}{2} + \frac{n}{2} } \theta^{r_{\RB} -\frac{k_2}{2} }  \int_{\R^{k_1+k_2}} e^{ i \frac{\Phi}{h x'} } a(h,x,\theta,y,y',v_1,v_2)dv_1dv_2  \big|dg dg'\frac{dh}{h^2}\big|^{1/2},
\end{equation}
where $v_1 \in \R^{k_1}, \, v_2 \in \R^{k_2}$, $\Phi = \psi_1(y',v_1) + \theta\psi_2(\theta,y,y',v_1,v_2)$ parametrizes $G$, and $a = O(h^N)$ for any $N$ with any fixed $C^k$-norm in all other variables.

\item $u_5 \in \mathcal{S}(X_{\rmb,\calchigh}^2)$.
\end{itemize}
\end{definition}

\subsection{Legendre distributions associated with Legendrian conic pairs at high energies}
\label{subsec:Legendrian-distributions-high}

In this subsection, we briefly recall the theory of Legendre distributions associated with Legendrian conic pairs in the high-energy regime from \cite{Hassell-Wunsch-semiclassical-resolvent}.
We only discuss the resolution and parametrization in the setting of codimension three (that is, at a codimension two corner of $X_{\rmb}^2$ at $h=0$), which is the most complicated case.
For details about the resolution and the parametrization in cases with lower codimensions, we refer the reader to \cite[Section~6.1.1.,\, 6.2.1,\, 6.5.1]{Hassell-Wunsch-semiclassical-resolvent}.
But we still give the definition of the corresponding class of Legendre distributions in Definition~\ref{def:Legendre-dis-conic-high-RB-LB} and use it as a building block in Definition~\ref{def:Legendrian-dis-conic-high} below.

Similar to the low-energy case, apart from $L^{\bfs,\calchigh}$, the other Legendre submanifold carrying purely outgoing/incoming oscillations is simpler.
In terms of coordinates in \eqref{eq:semiclassical-form-bf}, it is
\begin{equation} \label{eq:def-Lsharphigh}
\Lsharphigh = \Lsharphigh_+ \cup \Lsharphigh_-,
\end{equation}
where, similar to \eqref{eq:Lsharplow-pm-def} and in terms of coordinates in \eqref{eq:semiclassical-form-bf}, $\Lsharphigh_\pm$ are defined by
\begin{equation}
  \Lsharphigh_\pm = \{ \xi = \xi' = \pm 1, \; \mu=\mu' =0 \} \cap {}^{\Phi}T_{ \smf \cap \BFS }^*X_{\rmb,\calchigh}^2.
\end{equation}
In terms of coordinates in \eqref{eq:tautological-1-form-bf}, this means that our boundary defining function and other coordinates are chosen so that $\Lsharphigh_\pm$ is given by $\overline{\nu} = \pm (1+\sigma),\mu=\mu'=0$.

Then as shown in\footnote{The referred proposition only showed half of this claim for the forward propagating part, but the other half follows in the same way.} \cite[Proposition~11.1]{Hassell-Wunsch-semiclassical-resolvent}, 
$(L^{\bfs,\calchigh},\Lsharphigh)$ forms a pair of intersecting Legendre submanifolds with conic points in the sense that is similar to Section~\ref{subsec:conic-pair-geometry-and-phase-function}.
We blow up the span of $\Lsharphigh$, which is
\begin{equation} \label{eq:def-J-high}
J^{\calchigh} = \{ \xi = \xi', \; \mu=\mu' =0 \} \cap {}^{\Phi}T_{ \smf \cap \BFS }^*X_{\rmb,\calchigh}^2,
\end{equation}
and denote the blow-down map by
\begin{equation} \label{eq:def-high-blow-down}
    \beta_{\mathrm{LCP}, \calchigh}:  \; [{}^{\Phi}T^*_{\smf}X_{\rmb,\calchigh}^2;J^{\calchigh}] \to {}^{\Phi}T^*_{\smf}X_{\rmb,\calchigh}^2.
\end{equation}
The lift of $L^{\bfs,\calchigh}$ is defined to be 
\begin{equation} \label{eq:hatLbf-high-def}
    \hat{L}^{\bfs,\calchigh} = 
    \beta_{\mathrm{LCP}, \calchigh}^*(L^{\bfs,\calchigh})
    = \mathrm{clos} \Big( \beta_{\mathrm{LCP}, \calchigh}^{-1}( (L^{\bfs,\calchigh})^{\circ}) \Big).
\end{equation}
Then the statement that $(L^{\bfs,\calchigh},\Lsharphigh)$ forms a pair of intersecting Legendre submanifolds with conic points means $\hat{L}^{\bfs,\calchigh}$ is a smooth Legendrian submanifold (with corner) in the resolved bundle $[{}^{\Phi}T^*_{\smf}X_{\rmb,\calchigh}^2;J^{\calchigh}]$. See \cite[Section~6.1, Section~11]{Hassell-Wunsch-semiclassical-resolvent} for more details.

If one uses coordinates in \eqref{eq:semiclassical-form-bf-lb}, then similar to \eqref{eq:def-J-low}, $J^{\calchigh}$ can be instead defined by
\begin{equation} \label{eq:def-J-high-2}
J^{\calchigh} = \{x'=0,\, \overline{\nu}-(1+\sigma)\nu_1 = 0 , \, \mu=\mu' =0 \} \cap {}^{\Phi}T_{ \smf \cap \BFS }^*X_{\rmb,\calchigh}^2,
\end{equation}
and $\Lsharphigh$ is defined by
\begin{equation} \label{eq:concrete-Lsharphigh}
    \Lsharphigh = \{\nu_1=1,\overline{\nu}=1+\sigma,x'=0,\mu=\mu'=0\}.
\end{equation}
In the region $x'\lesssim|\mu'|$and without loss of generality we assume $\mu'_{n-1}$ is the dominating component in $\mu'$, then a coordinate system near the front face is 
\begin{equation} \label{eq:coordinates-resolved-high-region1}
\big( \sigma,\varrho=x'/\mu'_{n-1}, \, \overline{\nu},\hat{\nu}_1=\frac{\overline{\nu}-(1+\sigma)\nu_1}{\mu'_{n-1}}, \mu/\mu'_{n-1},\mu'_1/\mu'_{n-1},...,\mu'_{n-2}/\mu'_{n-1},\mu'_{n-1} \big),
\end{equation}
while in the region $|\mu'|\lesssim x'$, we use coordinates
\begin{equation} \label{eq:coordinates-resolved-high-region2}
\big( \sigma,x', \overline{\nu},\hat{\nu}_1=\frac{\overline{\nu}-(1+\sigma)\nu_1}{x'}, \mu/x',\mu'/x' \big).
\end{equation}
As explained in Remark~\ref{remark:no-need-part}, $\hat{L}^{\bfs,\calchigh}$ does not meet the region where $|\overline{\nu}-(1+\sigma)\nu_1| \gg x',|\mu'|$ and we do not need to consider this region below.

Similar to the low-energy case, the phase function that we use to parametrize those Legendrian conic pairs is different from the one\footnote{These are also different from phase functions parametrizing an intersecting pair of Legendrian submanifolds (without conic points). The typical example of this type of Legendrian pair is the conormal bundle of the (lifted) diagonal and our propagating Legendrian. We did not include them as we are not using those directly, but they do arise in the construction of resolvents, hence the spectral measure. See \cite{hassell2001resolvent}\cite{Hassell-Wunsch-semiclassical-resolvent}. } 
we are using for a single Legendre submanifold.
Concretely, near $q$ that is in the lift of ${}^{\bundlehigh}T^*_{\smf \cap \BFS \cap \RB}X_{\rmb,\calchigh}^2$, we use
\begin{equation} \label{eq:Phi-high-codim3-conic}
\Phi =  1 + \theta +s \theta \psi_2(\theta,y,y',s,v) + x \theta \psi_3(s,\theta, \frac{x}{s},y,y',v) ,    \quad  s \geq 0, \; v \in \R^{k}
\end{equation}
to parametrize $L^{\bfs,\calchigh}$. The case with smaller codimension is simpler and we refer the reader to \cite[Section~6.2]{Hassell-Wunsch-semiclassical-resolvent}.
Parametrization in this setting means that near the lift of $\BFS \cap \smf$, we have
\begin{equation} \label{eq:def-parametrization-high-energy-conic-codim3}
\hat{L}^{\bfs,\calchigh} = \beta_{\mathrm{LCP},\calchigh}^{-1} \big(\{ d(\frac{\Phi}{\theta x h}) (q')  | \; q' \in C_{\Phi} \} \big), 
\end{equation}
where $q'$ and $C_{\Phi}$ are:
\begin{equation} \label{eq:def-C-Phi-high-energy-conic}
    C_{\Phi} = \{ q' = (\theta, \frac{x}{s},y,y',v,s), v \in \R^k | d_{(s,v)}\Phi(q') = 0  \}.
\end{equation}
In addition, we say such a parametrization is non-degenerate if $C_{\Phi}$ is (locally) diffeomorphic to $\hat{L}^{\bfs,\calchigh}$. Equivalently, in terms of $\Phi$, with $j$-th component of $v$ denoted by $v^j$, this can explicitly be written as
\begin{equation}
   d_{y'}\psi_2, \quad d_{y',v}( \frac{\partial \psi_2}{\partial v^j})
\end{equation}
are linearly independent at $q' \in C_{\Phi}$ for $1\leq j \leq k$.
On the other hand, away from the lift of ${}^{\bundlehigh}T^*_{\smf \cap \BFS \cap \RB}X_{\rmb,\calchigh}^2$, which corresponds to being away from the part $\frac{x}{s}=0$ in the parametrization above, we use the phase function
\begin{equation} \label{eq:Phi-high-energy-conic-codim2}
    \Phi = 1 + \theta + \theta x \psi(\theta, x, y, y', v), \quad v \in \R^k.
\end{equation}
Now parametrization means
\begin{equation} \label{eq:def-parametrization-high-energy-conic-codim2}
    \hat{L}^{\bfs,\calchigh} = \beta_{\mathrm{LCP},\calchigh}^{-1}\big(\{ d(\frac{\Phi}{x \theta h}(q') | \;  q' \in C_{\Phi}  \} \big) ,
\end{equation}
where 
\begin{equation}
    C_{\Phi} = \{  q' = (\theta,x,y,y',v) |  d_v\psi = 0 \}.
\end{equation}
We say that this parametrization is non-degenerate if
\begin{equation}
d_{y'}\psi, \quad d_{y',v}( \frac{\partial \psi}{\partial v^i} )
\end{equation}
are linearly independent for $1 \leq i \leq k$.


\begin{proposition} \label{prop:minimal-parametrization-high-conic}
Let $\hat{L}^{\bfs,\calchigh}, \Phi$ be as above, then $\Phi$ always exists locally. 
In addition, let $k+1$ be the drop of rank of the projection $\hat{L}^{\bfs,\calchigh} \to \smf$ at $\msf{q}_0 \in \hat{L}^{\bfs,\calchigh} \cap \{x/|\mu|\} = 0$, 
then after potentially shrinking the region of parametrization, the smallest number of extra parameters, including $s$, in a parametrization of $\hat{L}^{\bfs,\calchigh}$ near $\msf{q}_0$ is $k+1$.
We will call such a parametrization \emph{minimal}.
\end{proposition}
The existence still follows from the construction that we will recall in Section~\ref{subsec:Leg-concrete-setup}. The minor difference in the numerology compared with Proposition~\ref{prop:minimal-parametrization-low-conic} is that the projection to the component $x = \frac{x}{|\mu|} \cdot |\mu|$ is always degenerate at $\{x=0=|\mu|\}$ since $dx = |\mu|d(\frac{x}{|\mu|}) + \frac{x}{|\mu|}d|\mu|$, and hence this rank drop always exceeds that in Proposition~\ref{prop:minimal-parametrization-low-conic} by exactly one. In some sense, this shows that the numerology in Proposition~\ref{prop:minimal-parametrization-high-conic} is actually the one that is compatible with the general principle that `the number of extra parameters equals the rank drop of the projection from the Legendre or Lagrangian submanifold', while the numerology in Proposition~\ref{prop:minimal-parametrization-low-conic} is the consequence of ignoring the $x$-direction.

Now we recall the theory of Legendre distributions associated with $(L^{\bfs,\calchigh},\Lsharphigh)$.
Let $G_1^{\calchigh}$ be the submanifold of ${}^{\bundlehigh}T_{\smf}^*X_{\rmb,\calchigh}^2$ defined by \footnote{ $G_1^{\calchigh}$ was used in \cite[Section~6.1.1]{Hassell-Wunsch-semiclassical-resolvent} to denote the projection of the current $G_1^{\calchigh}$ to the so-called ${}^{\Phi}N^*Z_1$ there, in which this projection becomes Legendrian. However, what is resolved in ${}^{\bundlehigh}T_{\smf}^*X_{\rmb,\calchigh}^2$ is still our current $G_1^{\calchigh}$, hence we just omit this complication.  }
\begin{equation}
    G_1^{\calchigh} = \{ \xi = 1, \mu = 0 \}
\end{equation}
in terms of coordinates in \eqref{eq:semiclassical-form-boundary-1} and we will only consider the case away from the part over $\BFS \cap \smf$. We use $J_1^{\calchigh}$ to denote the span of $G_1^{\calchigh}$, which means $J_1^{\calchigh} = \{ \mu  = 0 \}$.
The class of Legendre distributions associated with the pair $(L^{\bfs,\calchigh},G_1^{\calchigh})$, which is a building block of our final class of Legendre distributions, is defined as follows.

\begin{definition} \label{def:Legendre-dis-conic-high-RB-LB}
    The class of Legendre distributions associated with $(L^{\bfs,\calchigh},G_1^{\calchigh})$ with order $(m,p;r_{\LB},r_{\RB})$, which we denote by $I_{\calchigh}^{m_{\calchigh},p;r_{\LB},r_{\RB}}(X_{\rmb,\calchigh}^2,(L^{\bfs,\calchigh},G_1^{\calchigh});\Omega_{\calchigh}^{1/2})$, consists of distributions that can be written as
\begin{equation*}
    u = \sum_{j=1}^5 u_j 
\end{equation*}
with each piece characterized below and supported away from $\BFS$. 
\begin{itemize}
    \item $u_1 \in I_{\calchigh}^{m_{\calchigh},\infty;r_{\LB},r_{\RB} }(X_{\rmb,\calchigh}^2,L^{\bfs,\calchigh}; \Omega_{\calchigh}^{1/2})$ as defined in Definition~\ref{def:Legendre-distribution-high-energy}.
    \item $u_2$ is a finite sum of oscillatory integrals taking the form below and satisfying the following support conditions:
    \begin{itemize}
        \item Terms supported near a point in the interior of $\RB \cap \smf$ and having the form
\begin{equation} \label{eq:conic-high-codim2-RB}
 h^{m_{\calchigh} - \frac{k+1}{2} - \frac{n}{2} }   \int_{\R^k} \int_0^\infty  
 e^{i\frac{\Phi}{x'h}}
 (\frac{x'}{s})^{r_{\RB} - \frac{k+1}{2} } s^{p-1} a(s,\frac{x'}{s},h,z,y',v) ds dv   |dgdg'\frac{dh}{h^2}|^{1/2} ,
\end{equation}
where $s \in [0,\infty), \; v \in \R^k$, $\Phi=1+s\psi_1(y',s,v) + x'\psi_2(s,\frac{x'}{s},z,y',v)$ parametrizes $(L^{\bfs,\calchigh},G_1^{\calchigh})$ locally in the sense of \cite[Equation~(6.8)]{Hassell-Wunsch-semiclassical-resolvent} (and discussions there) and $a \in C_c^\infty$.

\item Terms supported near a point in the interior of $\LB \cap \smf$ and taking the same form as \eqref{eq:conic-high-codim2-RB} but with primed and un-primed variables switched, $\RB$ replaced by $\LB$.
      \end{itemize}
 \item  $u_3$ is a finite sum of oscillatory integrals of the following form, satisfying the support conditions below:
\begin{itemize}
    \item Terms supported near a point in the interior of $\RB \cap \smf $ and of the form:
\begin{equation}
 h^{m_{\calchigh} - \frac{k}{2} - \frac{n}{2} } (x')^{p-1} \int_{\R^k} e^{i\frac{\Phi}{x'h}}  a(h,x',z,y',v) dv \;|dgdg'\frac{dh}{h^2}|^{1/2},
\end{equation}
where $\Phi = 1+x'\psi_1(x',y',z,v)$ parametrizes $(L^{\bfs,\calchigh},G_1^{\calchigh})$ locally in the sense of \cite[Equation~(6.11)]{Hassell-Wunsch-semiclassical-resolvent}  (and discussions there) and $a \in C_c^\infty$. 

\item Terms that are supported near a point in the interior of $\LB \cap \smf$ and take the same form as \eqref{eq:conic-high-codim2-RB} but with primed and un-primed variables switched, $\RB$ replaced by $\LB$.
\end{itemize}

\item $u_4$ corresponds to the part that is away from $\smf$ and is given by a finite sum of oscillatory integrals of the following form, satisfying the support conditions below:
 \begin{itemize}
    \item Terms that are supported near a point in the interior of $\RB$
\begin{equation}
   h^{m_{\calchigh} - \frac{k+1}{2} - \frac{n}{2}}   \int_{\R^k}\int_0^\infty  e^{ i \frac{\Phi}{x'h} } (\frac{x'}{s})^{ r_{\RB} - \frac{k+1}{2} } s^{p-1} a(h,x',s,\frac{x'}{s},y',z,v) ds dv \;|dgdg'\frac{dh}{h^2}|^{1/2},
\end{equation}
where $\Phi = 1+s\psi_1(s,z,y',v)$ parametrizes $(L^{\bfs,\calchigh},G_1^{\calchigh})$.
Here we allow the support to reach $\smf$ but require $a(h,x',s,\frac{x'}{s},y',z,v)$ (hence the entire kernel) to be $O(h^\infty)$.

\item Terms that are supported near a point in the interior of $\LB$, with the same properties as above, except with primed and un-primed variables switched and $\RB$ replaced by $\LB$.
 \end{itemize}

\item  $u_5 \in x^{p}(x')^p h^{\infty} e^{i\Phi}C^\infty(X_{\rmb,\calchigh}^2)$, where $\Phi = \frac{1}{x'h}$ away from $\LB \cup \BFS$, and $\Phi = \frac{1}{xh}$ away from $\RB \cup \BFS$.
\end{itemize}
    
\end{definition}

\begin{remark}
    We explain the numerology compared with \cite[Section~6.5.1]{Hassell-Wunsch-semiclassical-resolvent}, in addition to the $\frac{1}{4}$-shift mentioned in \cite[Remark~4.3]{GHS1}.
    In that part, one should take $x_2$ there as the semiclassical parameter $h$, while the $x_1$ there is the boundary defining function of either $\LB$ or $\RB$.
    Now the fibration has only a one-dimensional fiber parametrized by $h$. So $-\frac{f_1}{2}$ can no longer cancel $\frac{N}{2}$. However, the fibered scattering half density there with the interior of $\RB$ or $\LB$ being a boundary face has extra $-(\frac{n-f_1}{2}) = -\frac{n-1}{2}$ power of $x'$ (say over the interior of $\RB$) now compared with $|dgdg'|^{1/2}$, hence the overall power compared with $|dgdg'|^{1/2}$ is the same as in the case with higher codimensions.
    More explicitly, in terms of this numerology, we have $f_1 = n$ in the codimension 2 case while $f_2 = 0, f_1 = n$ in the codimension 3 case.
\end{remark}

Finally we turn to the class of Legendre distributions associated with $(L^{\bfs,\calchigh},\Lsharphigh)$.

\begin{definition} \label{def:Legendrian-dis-conic-high}
The class of Legendre distributions associated with the pair of Legendre submanifolds with conic points $(L^{\bfs,\calchigh},\Lsharphigh)$, which we denote by 
\begin{equation}
I_{\calchigh}^{m_{\calchigh}, p;m,r_{\LB},r_{\RB} }(X_{\rmb,\calchigh}^2, (L^{\bfs,\calchigh},\Lsharphigh) ; \Omega_{\calchigh}^{1/2}),
\end{equation}
 consists of distributions that can be written as 
\begin{equation*}
    u = \sum_{j=1}^6 u_j,
\end{equation*}
where $u_j$ are as described below.
\begin{itemize}
    \item $u_1 \in I^{m_{\calchigh},m ,r_{\LB},r_{\RB}}_{\calchigh}(X_{\rmb,\calchigh}^2, L^{\bfs,\calchigh} ; \Omega_{\calchigh}^{1/2})$ as defined in Definition~\ref{def:Legendre-distribution-high-energy} and microsupported away from $\Lsharphigh$.
    
    \item $u_2 \in I_{\calchigh}^{m_{\calchigh},p;r_{\LB},r_{\RB}}(X_{\rmb,\calchigh}^2,(L^{\bfs,\calchigh},G_1^{\calchigh}); \Omega_{\calchigh}^{1/2})$ as defined in Definition~\ref{def:Legendre-dis-conic-high-RB-LB}.

\item $u_3$ is a finite sum of local expressions satisfying support conditions and taking the form below.
\begin{itemize}
    \item Terms that are supported away from $\smf \cap \BFS \cap \LB$ and take the form
    \begin{equation}
    h^{ m_{\calchigh} -\frac{k+1}{2} -\frac{n}{2} }  \int_{\R^k} \int_0^\infty e^{i\frac{\Phi}{hx\theta}} a(h,\theta,\frac{x}{s},y,y',v)
(\frac{x}{s})^{m-\frac{k+1}{2}+\frac{n}{2}}  s^{p-1+\frac{n}{2}} \theta^{ r_{\RB} } dsdv \;  |dgdg'\frac{dh}{h^2}|^{1/2},
    \end{equation}
    where $v \in \R^k$, $\Phi$ is as in \eqref{eq:Phi-high-codim3-conic} and parametrizes $(L^{\bfs,\calchigh},\Lsharphigh)$.
    \item Terms that are supported away from $\smf \cap \BFS \cap \RB$ with the same type of expression as above except with primed and un-primed variables switched, $\theta$ replaced by $\sigma = \theta^{-1}$, $r_{\RB}$ replaced by $r_{\LB}$.
\end{itemize}

\item $u_4$ is a finite sum of local expressions satisfying support conditions and taking the form below.
\begin{itemize}
    \item Terms that are supported away from $\smf \cap \BFS \cap \LB$ and take the form
    \begin{equation} \label{eq:conic-pair-u4-high}
  h^{ m_{\calchigh} -\frac{k}{2} -\frac{n}{2} } x^{p-1+\frac{n}{2}} \theta^{ r_{\RB} } \int_{\R^k}   e^{i\frac{\Phi}{hx\theta}} a(h,x,\theta,y,y',v) dv \;|dgdg'\frac{dh}{h^2}|^{1/2},
    \end{equation}
    where $v \in \R^k$, $\Phi$ is as in \eqref{eq:Phi-high-energy-conic-codim2} and parametrizes $(L^{\bfs,\calchigh},\Lsharphigh)$.
    \item Terms that are supported away from $\smf \cap \BFS \cap \RB$ with the same type of expression as above except with primed and un-primed variables switched, $\theta$ replaced by $\sigma = \theta^{-1}$, $r_{\RB}$ replaced by $r_{\LB}$.
\end{itemize}

\item $u_5$ is given by a finite sum of the same type of oscillatory integrals as in $u_3$, but with $\psi_3=0$ in $\Phi$ in \eqref{eq:Phi-high-codim3-conic} and the amplitude $a(\cdot)$ is $O(h^{\infty})$. \footnote{This term (resp. the $u_4$-term in Definition~\ref{def:Legendre-dis-conic-high-RB-LB}) can be absorbed into the $u_3$-term (resp. $u_2$-term there). We write them out because it is conceptually clearer to characterize the finite-energy case and the high-energy case separately. 
Also, this makes it easier for readers to refer to \cite[Section~6]{Hassell-Wunsch-semiclassical-resolvent}, in which case they are indeed different in the general machinery. This is because there is a potentially nontrivial fibration over $\smf$ there. However, as aforementioned, this fibration is always trivial in our applications and the so-called $v_d$ parameter is always empty.  }  


\item $u_6 \in \sigma^{r_{\LB}} \theta^{r_{\RB}} (x+x')^{p} h^\infty e^{i\Phi}C^\infty(X_{\rmb,\calchigh}^2)$, where $\Phi =  \frac{1+h}{x\theta h}$ away from $\LB$ while 
$\Phi = \frac{1+h}{x'\sigma h}$ away from $\RB$.

\end{itemize}

\end{definition}




\subsection{The focusing index} 
\label{subsec:focusing-effect}
In this subsection we define $\IF$ and $\IFint$ of an asymptotically conic manifold $X$ and their analogues on the exact cone $X_0$, which quantify the focusing effect of the geodesic flow on $X$, with the difference being that $\IF$ takes the focusing happening at $L^{\bfs,\calchigh} \cap \Lsharphigh$ into consideration while $\IFint$ does not. Since the boundary of $\hat{L}^{\bfs,\calchigh}$ gives $\hat{L}^{\bfs}$, definitions below using $\hat{L}^{\bfs,\calchigh}$ take care of the low-energy regime as well.

Consider $\IF$ first. Let $\hatLprojhigh$ be the projection
\begin{equation} \label{eq:def-Lprojhigh}
\hatLprojhigh: \; \hat{L}^{\bfs,\calchigh} \to \smf \simeq X_{\rmb}^2,
\end{equation}
and let $\Delta^{\calchigh}$ be the intersection of $\hat{L}^{\bfs,\calchigh}$ with the lifted diagonal:
\begin{equation}
    \Delta^{\calchigh} = \mathrm{clos} \Big(\beta_{\mathrm{LCP},\calchigh}^{-1}(\{ (q,(q)',h=0,\tau=0): \; q \in \Sigma \}) \Big),
\end{equation}
where $(q)'$ means switching the sign of the frequency part of $q$.
Over the part of $X$ that is away from $\partial X$, which is a compact manifold, the injective radius of $X$ has a lower bound $\gtrsim 1$.
For the part near $\partial X$, the metric is a perturbation of the conic metric, and we know that it has injective radius $\gtrsim 1$ in terms of the distance on $Y$, which means that it is $\gtrsim x^{-1}$ in terms of distance on $X$. In particular, there is a neighborhood $U_{\Delta^{\calchigh}}$ of $\Delta^{\calchigh}$ in $\hat{L}^{\bfs,\calchigh}$ such that the projection $\hatLprojhigh$ is a diffeomorphism on $U_{\Delta^{\calchigh}} \backslash \Delta^{\calchigh}$.  

Now we consider $\IFint$. It excludes the degeneracy of the projection $\hatLprojhigh$ at $\hat{L}^{\bfs,\calchigh} \cap \beta_{\mathrm{LCP},\calchigh}^{-1}(\Lsharphigh)$, which corresponds to pairs of end-points of bicharacteristic lines.
Let $\Lsharphigh$ be as in \eqref{eq:def-Lsharphigh}, which corresponds to the part with $\xi,\xi' = \pm 1$ in terms of coordinates in \eqref{eq:semiclassical-form-bf} (and its analogue in the region $\sigma =x/x' \lesssim 1$). Then the intersection $L^{\bfs,\calchigh} \cap \Lsharphigh$ corresponds to those $(q,q')$ as in \eqref{eq:L-high-interior} that are tending to endpoints of bicharacteristic lines.
Then we define
\begin{equation} \label{eq:IFint-def}
    \IFint = \max_{ \hat{L}^{\bfs,\calchigh} \backslash
     (U_{\Delta^{\calchigh}} \cup \beta_{\mathrm{LCP},\calchigh}^{-1}(\Lsharphigh) ) } \big(2n - \mathrm{rank} \, d\hatLprojhigh \big).
\end{equation}
Since $\hat{L}^{\bfs,\calchigh}$ is diffeomorphic to $L^{\bfs,\calchigh}$ away from $\Lsharphigh$, we know that $\IFint$ also equals the maximum of the rank drop of the projection from $L^{\bfs,\calchigh} \backslash (\Lsharphigh \cup \beta_{\mathrm{LCP},\calchigh}(\Delta^{\calchigh}))$.

Next we turn to $\IF$, it is  defined to be
\begin{equation} \label{eq:IF-full-def}
   \IF =  \max \Big(  \max_{ \beta_{\mathrm{LCP},\calchigh}^{-1}(L^{\bfs,\calchigh} \cap \Lsharphigh) } \big(2n - \mathrm{rank} \, d\hatLprojhigh \big) -1 , \; \IFint \Big).
\end{equation}
Roughly speaking, this is the maximal rank drop of the projection away from the diagonal and the part formed by pairs of endpoints of bicharacteristic lines.

By definition, we know
\begin{equation}
    \IFint \leq \IF.
\end{equation}

Now we turn to exact cones. In this setting, the analogue of the entire $L^{\bfs,\calchigh}$ defined using \eqref{eq:L-high-interior} for $X_0$ is not a smooth manifold anymore due to the singularity of $X_0$ as $r \to 0$, but fortunately the geometric information of $X_0$ is completely encoded in $Y$, whose geodesic flow is in turn completely encoded in $\hat{L}^{\bfs}$.
So we let $\hat{L}^{\bfs}$ be the same as before, i.e. defined via \eqref{eq:hat-Lbf-def}, which in turn uses $L^{\bfs}$ defined via \eqref{eq:Lbf-definition-gamma^2}.
Let $\hatLprojlow$ be the projection $\hat{L}^{\bfs} \to \bfs$ defined in \eqref{eq:def-Lprojlow},   $U_{\Delta^{\calclow}} = U_{\Delta^{\calchigh}} \cap \hat{L}^{\bfs}$ be the part of $U_{\Delta^{\calchigh}}$ lying over $\bfs$, and we define
\begin{equation} \label{eq:IFintz-def}
    \IFintz = \max_{ \hat{L}^{\bfs} \backslash
    \big( U_{\Delta^{\calclow}} \cup \beta_{\mathrm{LCP},\calclow}^{-1}(\Lsharplow) \big) } \big(2n -1 - \mathrm{rank} \, d\hatLprojlow \big).
\end{equation}
Then we define $\IFz$, which takes care of the part $\beta_{\mathrm{LCP},\calclow}^{-1}(L^{\bfs} \cap \Lsharplow)$ as well to be
\begin{equation} \label{eq:IFz-def}
    \IFz = \max \Big(  \max_{ \beta_{\mathrm{LCP},\calclow}^{-1}(L^{\bfs} \cap \Lsharplow) } \big(2n -1 - \mathrm{rank} \, d\hatLprojlow \big) , \; \IFintz \Big).
\end{equation}

\begin{remark}
    In principle, one can define $\IFz$ (resp. $\IFintz$) in the same way as \eqref{eq:IF-full-def} (resp. \eqref{eq:IFint-def}), but one needs to either deal with a singular object or truncate the analogue of $L^{\bfs,\calchigh}$ defined for $(X_0,g_0)$ to the part away from the vertex of the cone, which introduces several quite unnatural extra boundary surfaces. So we choose to define it using $\hat{L}^{\bfs}$ here.
\end{remark}

The importance of $\IF$, $\IFint$, $\IFz$, and $\IFintz$ lies in Proposition~\ref{prop:IF-relation-parameter-number} below, which gives the relationship between these geometric quantities and the local expression of the spectral measure.

\section{The microlocalized spectral measure}
\label{sec:microlocalized-spectral-measure}

In this section, we discuss the structure of the spectral measure of $P$ as a Legendre distribution associated with intersecting pairs of Legendre submanifolds with conic points. The global characterization is proven in \cite[Theorem~3.10]{GHS2}\cite[Corollary~1.2]{Hassell-Wunsch-semiclassical-resolvent}.
We construct the microlocal partition of unity so that we have a clear control of the degeneracy of the projection from Legendre submanifolds that the microlocalized pieces of the spectral measure are associated with. This will give us a clear characterization of how each degenerate point of the projection causes the loss in the dispersive estimate. See also Remark~\ref{remark:microlocalized-est-Schrodinger} below.

\subsection{The low-energy microlocal partition of unity}
\label{subsec:microlocal-partition-low}

In this subsection, we introduce the microlocal partition of unity that is used to localize the spectral measure so that we have a precise oscillatory integral representation of it.
Similar partitions have already been used in \cite[Lemma~5.3, Lemma~5.4]{GH2014uniform-Sobolev}  and \cite[Section~3]{Hassell-Zhang2016Strichartz}, but the role of the partition of unity in our setting is different from that in \cite{Hassell-Zhang2016Strichartz}, where this is used to separate conjugate points so that pairs of conjugate points almost `disappear' if we only use the part near the diagonal after this microlocalization.
On the other hand, in the current setting, we will not only use the contribution near the diagonal, but all pieces of the microlocalized spectral measure, hence the effect of pairs of conjugate points can't be avoided and our decomposition is used to characterize them more precisely.
This decomposition also separates the b-problem living at $\zf$ from other parts, which is dealt with by the elliptic theory in \cite[Section~6.1.1]{GHS2}. 

As we will discuss in Section~\ref{sec:microlocalized-spectral-measure}, the spectral measure of $P$ is a Legendre distribution associated with the Legendrian conic pairs $(L^{\bfs},L^\#)$ that we will define.
The microlocalized spectral measures are $Q_j^{\calclow} \specm Q_{j'}^\calclow$ where $Q_j^{\calclow},Q_{j'}^\calclow$ are members of the microlocal partition of unity.
Then as shown in \cite[Section~5]{GHS2}, these terms are also Legendre distributions, but associated only to part of the Legendre submanifold, namely to the subset
\begin{equation*}
\big\{ (\sigma, y, y', \mu, \mu', \nu, \nu') \in L^{\bfs} \mid 
(y, \mu, \nu) \in \WF'_{\calclow}(Q_j), (y', \mu', \nu') \in \WF_{\calclow}'(Q_{j'})
\big\},
\end{equation*}
where $\WF'_{\calclow}(\bullet)$ is the low-energy wavefront set defined in Definition~\ref{def:low-energy-WF}. 
In this way, we can separate conjugate point pairs with different rank drops and investigate their effects on the dispersive estimate individually.

In \cite{Hassell-Zhang2016Strichartz}, this microlocalization procedure is an important ingredient in the proof of the Strichartz estimates, while here this is merely a technical convenience to look at the spectral measure (hence propagators of PDEs) microlocally on the piece of $L^{\bfs}$ with desired geometric properties.
Concretely, we do not 
require the microlocalizers on both sides to be the same partition of identity, which is different from the microlocal partition of unity in \cite[Section~2D, Section~3C]{Hassell-Zhang2016Strichartz}\cite[Section~5, Section~7]{GHS1}. Another minor difference is that we use our combined pseudodifferential algebra, instead of defining the low-energy and high-energy part separately. We believe that the microlocalization technique employed here is powerful enough to treat problems beyond the scope of the $TT^*$ argument, such as $L^p$-$L^q$ resolvent estimates and pointwise estimates.


Since microlocalization is performed on the left and right variables separately, it is more convenient to use the following analogue of $L^{\bfs}$, which is in the product form, to capture the degeneracy of the projection:
\begin{align} \label{eq:scr-RY-*-definition}
\RYstar = \{ (y,\hat{\mu},y',\hat{\mu'}) \in S^*Y \times S^*Y: \exists s\in (-\pi,\pi): \exp(sH_{\frac{1}{2}\Ymetric^*})(y,\hat{\mu})
= (y',\hat{\mu'}) \}.
\end{align}
The partition is adapted to the degeneracy of the projection 
\begin{equation} \label{eq:def-Rprojlow-star}
\Rprojlow^*: \RYstar \to Y \times Y,
\end{equation}
which determines the pointwise bound of our spectral measure, hence the bound of propagators of our PDEs.
We consider the following decomposition of $\RYstar$:
\begin{equation}
\RYstar = \bigcup_{i=0}^{n-2} S_i^*,
\end{equation} 
where 
\begin{equation} \label{eq:def-Si}
S_i^* = \{ p \in \RYstar:  \mathrm{rank} \, d\Rprojlow^*(p)  = 2n-2-i \}.
\end{equation}

As discussed in Appendix~\ref{sec:riemannian_geometry_interpretation_of_the_focusing_intensity}, the rank of $d\Rprojlow^*$ is a lower semi-continuous function on $\RYstar$ and the degeneracy is always simple. So for each point $p = (y,\hat{\mu},y',\hat{\mu'}) \in S^*_i$, there is a neighborhood $U  \times U'$ in $S^*Y \times S^*Y$ such that the rank drop of $d\Rprojlow$ at $p$ is $i$ and it is at most $i$ on $(U \times U') \cap \RYstar$.
By the compactness of $\RYstar$, we know that it has a cover of the form
\begin{equation} \label{eq:R*-decomposition}
\RYstar \subset \bigcup_{\ell} U_\ell \times U'_\ell,
\end{equation}
where $\ell$ runs over a finite index set and such that the projections of $U_{\ell},U_{\ell}'$ are contained in a single coordinate chart of $Y$ and the rank drop of $d\Rprojlow^*$ on $(U_\ell \times U'_\ell) \cap \RYstar$ is at most $k_\ell \in \N$, and this rank drop is achieved at certain points.

\begin{remark}
In this open cover, it can happen that the degeneracy of $d\Rprojlow^*$ at some points is `hidden' by the degeneracy at other points because that point has a larger rank drop.
However, as we will see in Section~\ref{sec:dispersive-Schrodinger} and Section~\ref{sec:dispersive-wave}, only the largest $k_\ell$ dictates the numerology of the global dispersive estimate and this is independent of the choice of the cover.

On the other hand, we will also consider the more refined microlocal dispersive estimate, in which one can see that the contributions from $\specm$ associated with different parts of $L^{\bfs}$ (and $L^{\bfs,\calchigh}$ in \eqref{subsec:Legendrian-distributions-high}, if one takes the high-energy regime into consideration as well) give contributions of different decay rates to the dispersive estimate.
\end{remark}

Now we discuss the relationship between $\RYstar$ and $L^{\bfs}$. There is a natural fibration 
\begin{equation} \label{eq:hatLbf-R*-fiberation}
\mk{P}: \; \hat{L}^{\bfs} \to \RYstar
\end{equation}
given by\footnote{Coordinates for $\hat{L}^{\bfs}$ in \eqref{eq:fibration-Lbf-scrR*} are only valid away from $\rb$.
Near $\rb$, the only change needed is just switching primed and un-primed variables.}
\begin{equation} \label{eq:fibration-Lbf-scrR*}
(\sigma,y,y',\hat{\mu},\hat{\mu'},|\mu'|,\nu,\nu') \to (y,\hat{\mu},y',\hat{\mu'}).
\end{equation}
The fiber over $(y,\hat{\mu},y',\hat{\mu'})$ is precisely those lifted (boundary) bicharacteristic lines in ${}^{\sct}T^*_{\partial X}X$ passing through $(y,\mu,\nu)$ and $(y',\mu',\nu')$ with $|\mu|,\nu,\nu'$ determined by
\begin{equation}
|\mu|= \sigma|\mu'|, \, \nu = \pm (1-|\mu|^2)^{1/2}, \nu' = \mp (1-|\mu'|^2)^{1/2}.
\end{equation}
Then a decomposition of $\RYstar$ as in \eqref{eq:R*-decomposition} gives a decomposition of $\hat{L}^{\bfs}$ as follows. Let $\beta_{\mathrm{LCP},\calclow}$ be the blow-down map in \eqref{eq:beta-LCP-low-1}, for each $U_\ell \times U'_\ell$ in \eqref{eq:R*-decomposition}, we set $\mk{G}_{\ell}$ to be a tubular neighborhood of
\begin{equation} \label{eq:G-ell-01}
\beta_{\mathrm{LCP}}(\mk{P}^{-1}(U_\ell \times U'_\ell)).
\end{equation}
That is, $\mk{G}_{\ell} \cap L^{\bfs} = \beta_{\mathrm{LCP}}(\mk{P}^{-1}(U_\ell \times U'_\ell))$, and $\mk{G}_{\ell}$ is diffeomorphic to 
\begin{equation} \label{eq:G-ell-model}
(1-\delta,1+\delta) \times \beta_{\mathrm{LCP}}(\mk{P}^{-1}(U_\ell \times U'_\ell))
\end{equation}
with the first component parametrized by $|\mu|_h^2+\nu^2$ and $\delta \ll 1$ being a fixed small constant.
Consequently, we obtain a finite cover of $L^{\bfs}$:
\begin{equation} \label{eq:Lbf-covered-G-ell}
L^{\bfs} = \bigcup_{\ell} (\mk{G}_{\ell} \cap L^{\bfs}),
\end{equation}
such that the projection $L^{\bfs} \to \bfs$ on each $\mk{G}_{\ell} \cap L^{\bfs}$ has rank drop at most $k_\ell$, and it is achieved at certain points. In addition, after potentially further decomposing $\mk{G}_\ell$, we may assume that each piece $\mk{G}_{\ell} \cap L^{\bfs}$ can be parametrized in the sense of Definition~\ref{def:Legendre-parametrization-low} or \eqref{eq:hat-Lbf-parametrization-bf} by a single phase function $\Phi$.

For each fixed $\sigma=\sigma_0$, $L^{\bfs}$ can be equipped with the metric induced from ${}^{\sct}T^*_{\partial X}X \times {}^{\sct}T^*_{\partial X}X$, which in turn is induced by the metric $g$ on fibers and by $h$ on the base $\partial X \times \partial X$.
The topology induced by this metric is the same as the original one inherited from ${}^{\bundlehigh}T^*_{\bfs} X_{\mathrm{b}}^2$, under which $L^{\bfs} \cap \{\sigma = \sigma_0 \}$ is compact. In particular, we have a Lebesgue number $\delta_0$ associated with this covering: that is, if some subset of $L^{\bfs}$ has diameter less than $\delta_0$, then it is contained in one of $\mk{G}_{\ell} \cap L^{\bfs}$.
In addition, this $\delta_0$ can be taken to be uniform in $\sigma$. This is because we have an explicit construction of such $\delta_0$ from the proof of the Lebesgue lemma: the minimum of the averaged distance to complements of members of the open cover. 
Since our cover is independent of $\sigma$ by construction and $L^{\bfs}$ with fixed $\sigma$ is changing continuously in $\sigma$, this Lebesgue number is a continuous function in $\sigma$ (in $\sigma^{-1}$ in the region $\sigma \gg 1$) and can be made uniform in $\sigma$.

Now we start to construct our microlocal partition of unity. 
We separate the part near $\zf$ and away from $\zf$ first. Let $\chi \in C_c^\infty(\R)$ be an even smooth function non-decreasing on $[0,\infty)$ such that $\chi \equiv 1$ on $[0,\epsilon]$ and $\chi \equiv 0$ on $[2\epsilon,\infty)$. Then we set
$Q_{\zf}^{\calclow},Q_{\zf}^{'\calclow}$ to be the multiplication by $(1-\chi(\rho))$ and $(1-\chi(\rho'))$ respectively, which are localizers supported near $\zf$ in terms of left and right variables respectively.

The next step is to separate the part near the characteristic variety and the part away from it. We choose $\chi_1 \in C^\infty(\R)$ such that $\chi_1 \equiv 1$ on $[0,1-\frac{3}{4}\delta] \cup [1+\frac{3}{4}\delta,\infty)$ and $\chi_1 \equiv 0$ on $[1-\frac{\delta}{2},1+\frac{\delta}{2}]$, where $\delta$ is as in \eqref{eq:G-ell-model}
and set 
\begin{equation}
Q_1^{\flat} = \chi(\rho) \operatorname{Op}(\chi_1(|\mu|^2+\nu^2)) \in \Psi_{\flat}^0,
\end{equation}
where $\operatorname{Op}(\cdot)$ is the quantization map as in \eqref{eq:quant-PsiDO}. This $Q_1^{\flat}$ takes care of the part that is away from the characteristic variety of $\Delta_g-\lambda^2$.
Similarly, we set 
\begin{equation}
Q_1^{'\flat}= \chi(\rho') \operatorname{Op}(\chi_1(|\mu'|^2+(\nu')^2)) \in \Psi_{\flat}^0.
\end{equation}

Now we turn to the part that is near $L^{\bfs}$, which is the part that is actually interesting. 
By the compactness of $Y$ and consequently the compactness of the unit (scattering) cosphere bundle in ${}^{\sct}T^*_{\partial X}X$, we can choose a finite family of smooth functions on ${}^{\sct}T^*_{\partial X}X$, which we denote by $\{\chi_{j}\}_{j \in \Jlow}$ with $\Jlow$ being the index set, such that 
\begin{equation} \label{eq:sum-chi-j-1}
\sum_{j \in \Jlow \cup \{\zf,1\}} \chi_{j} \equiv 1
\end{equation}
on the region $1- \frac{3}{4}\delta \leq |\mu|_h^2+\nu^2  \leq 1+ \frac{3}{4}\delta$ and the diameter of $\supp \chi_{\ell,1} \cap \{ |\mu|_h^2 + \nu^2 = 1\}$ is less than $\frac{\delta_0}{2}$.
By our construction, the (left full) symbol of $\Id - Q_{\zf} - Q_1^{\flat}$ is supported in the region $\rho \geq \epsilon$ and $1- \frac{3}{4}\delta \leq |\mu|_h^2+\nu^2  \leq 1+ \frac{3}{4}\delta$, on which \eqref{eq:sum-chi-j-1} holds.
So we can decompose it (by multiplying the left symbol by $\chi_j$ and then quantize by \eqref{eq:quant-PsiDO}) as a sum
\begin{equation}
 \Id - Q_{\zf} - Q_1^{\flat} = \sum_{j \in \Jlow} Q_{j}^{\flat},
\end{equation}
where the support of the (left full) symbol of $Q_{j}^{\flat}$ with $j \in \Jlow$ has diameter less than $\frac{\delta_0}{2}$. In particular, the $\nu$-component lies in an interval of length $\leq \frac{\delta_0}{2}$.

For $j,j' \in \Jlow$, we let $W_j,W_{j'}$ be an open neighborhood of $\WF'(Q_j^\flat), \WF'(Q_{j'}^\flat)$ obtained by slightly enlarging them and continue to satisfy the assumptions on diameters. We define
\footnote{Here we introduced $W_j,W_{j'}$ just to avoid extra corners formed by the intersection of the inverse image of $\partial(\WF'(Q_j^\flat))$ and $\partial(\WF'(Q_j^\flat))$ under the left and right projections and $L^{\bfs}$. }
\begin{equation} \label{eq:Lbf-low-jj'}
L_{j,j'}^{\bfs} = 
\{ (\sigma,y,y',\mu,\mu',\nu,\nu') \in L^{\bfs}| (y,\mu,\nu) \in  W_j, (y',\mu',\nu') \in W_{j'}  \}.
\end{equation}

This is a piece of $L^{\bfs}$ that is microlocalized by our microlocal partition. As shown in \cite[Section~5]{GHS1}, $Q_j^\flat \specm Q_{j'}^\flat$ is a Legendre distribution associated with $L_{j,j'}^{\bfs}$ if it is away from $L^{\bfs}\cap\Lsharplow$ and associated with
the Legendrian conic pair formed by $L_{j,j'}^{\bfs}$ and $L^{\#}$.
As $L^{\#}$ always projects to the base diffeomorphically, we do not need to decompose it to capture the effect of the geometry.

By our assumption that the diameters of $W_j,W_{j'}$ are less than $\frac{\delta_0}{2}$, we know that the diameter of $L_{j,j'}^{\bfs}$ for fixed $\sigma = \sigma_0$ is less than $\delta_0$, which is the Lebesgue number. Consequently, it is contained in one of $\mk{G}_{\ell} \cap L^{\bfs} \cap \{\sigma = \sigma_0\}$.
Now let $\hat{L}^{\bfs}_{j,j'}=\beta_{\mathrm{LCP}}^{-1}(L_{j,j'}^{\bfs})$\footnote{The reason that we further lift to $\hat{L}^{\bfs}$ is that its invariance in $\sigma$ is more transparent.} be the lift of $L_{j,j'}^{\bfs}$ to $\hat{L}^{\bfs}$. 
Then the discussion above shows
\begin{equation}
\hat{L}^{\bfs}_{j,j'} \cap \{ \sigma = \sigma_0 \}
\subset \beta_{\mathrm{LCP}}^{-1}(\mk{G}_{\ell} \cap L^{\bfs}) \cap \{ \sigma = \sigma_0 \},
\end{equation} 
which gives
\begin{equation}
\hat{L}^{\bfs}_{j,j'} \subset \beta_{\mathrm{LCP}}^{-1}(\mk{G}_{\ell} \cap L^{\bfs}),
\end{equation}
since both sides are independent of $\sigma$ because they either contain or are disjoint from each fiber of \eqref{eq:fibration-Lbf-scrR*}.
This in turn shows
\begin{equation} \label{eq:Lbf-jj'-contained-G-l}
L^{\bfs}_{j,j'} \subset \mk{G}_{\ell} \cap L^{\bfs}.
\end{equation}
So we have a partition of $\Jlow \times \Jlow$:
\begin{equation}
\Jlow \times \Jlow = \bigcup_\ell \mk{J}_\ell,
\end{equation}
where $\mk{J}_\ell$ consists of $(j,j')$ satisfying \eqref{eq:Lbf-jj'-contained-G-l} and the $\mk{J}_\ell$ are disjoint for different $\ell$ (though the same $(j,j')$ may be contained in different $\mk{G}_{\ell} \cap L^{\bfs}$, we only select one of them).

In terms of our spectral measure, suppose $(j,j')$ satisfies \eqref{eq:Lbf-jj'-contained-G-l}, by the construction of $\mk{G}_{\ell}$, the rank drop from $L_{j,j'}^{\bfs}$ to $\overline{\R}_\sigma \times Y \times Y$ is at most $k_\ell$, which will be the number of extra parameters when we write Legendre distributions as oscillatory integrals in Section~\ref{subsec:osc-integral-formula}.
So we have constructed the microlocal partition in the low-energy regime and we summarize its properties below.
\begin{proposition} \label{prop:microlocal-partition-low}
Let $\mk{G}_{\ell}$ be the cover of $L^{\bfs}$ as in \eqref{eq:Lbf-covered-G-ell}, then there exists a microlocal partition of unity $Q_j^{\calclow} \in \Psi_{\calclow}^{0}(X)$:
\begin{equation*}
\mathrm{Id}=\sum_{j \in  \overline{\mk{J}}_{\calclow} } Q_j^{\calclow}  \; \text{ when } \; \lambda \lesssim 1,
\end{equation*}
where the index set is $\overline{\mk{J}}_{\calclow} = \{\zf,1\} \cup \Jlow$, such that:
\begin{enumerate}
    \item $Q_{\zf}^{\calclow}$ is a multiplication by $(1-\chi(\rho))$, which is a function that is supported on the region $\rho \geq C>0$ and is identically $1$ on the region $\rho \geq 2C$.
    \item $Q_1^{\calclow} \in \Psi_{\calclow}^0$ has wavefront set contained in $\rho \leq 2C$ and away from the projection of $L^{\bfs}$ to the right or left factor.
    \item  For $j,j' \in \Jlow$, there is an index $\ell$ such that \eqref{eq:Lbf-jj'-contained-G-l} holds. In particular, the rank drop of the projection from $L_{j,j'}^{\bfs}$ defined in \eqref{eq:Lbf-low-jj'} to $\bfs$ is at most $k_\ell$, the maximal rank drop from $L^{\bfs} \cap \mk{G}_{\ell}$ to $\bfs$.
    In addition, $\nu$ and $\nu'$-components in each $L^{\bfs}_{j,j'}$ lie in an interval of length at most $\frac{\delta_0}{2}$, where $\delta_0$ is the Lebesgue number given above.
\end{enumerate}
\end{proposition}

\subsection{The high-energy and combined microlocal partition}
\label{subsec:microlocal-partition-high-combined}

In this subsection, we give the microlocal partition of unity in the high-energy regime and combine it with the low-energy microlocal partition. 
The difference with the low-energy case is that now we need to take the geometry of the interior of $X$ into consideration.

We still start with the geometric decomposition of ${}^{\bundlehigh}T_{\smf}^*X_{\rmb,\calchigh}^2$ defined in Section~\ref{subsec:Legendrian-distributions-high}.
By its definition, every point of it can be written as
\begin{equation} \label{eq:high-form-decomposition}
   q     + \tau d(\frac{1}{h}),
\end{equation}
where $q \in \beta_{\rmb}^*({}^{\calchigh}T^*X \times {}^{\calchigh}T^*X)$ with $\beta_{\rmb}$ being the blow-down map in \eqref{eq:beta-b}.

So there are two natural projections $\pi_{L,\calchigh},\; \pi_{R,\calchigh}$ from ${}^{\bundlehigh}T_{\smf}^*X_{\rmb,\calchigh}^2$ to ${}^{\calchigh}T^*X$ via projecting to $q \in  \beta_{\rmb}^*({}^{\calchigh}T^*X \times {}^{\calchigh}T^*X)$ in \eqref{eq:high-form-decomposition} first and then projecting to ${}^{\calchigh}T^*X$ using the left or right projection.
In local coordinates over the interior, this is given by
\begin{equation} \label{eq:def-pi-L-high}
\zeta \cdot \frac{dz}{h} + \zeta' \cdot \frac{dz'}{h} + \tau d(\frac{1}{h}) \to \zeta \cdot \frac{dz}{h}
\end{equation}
and 
\begin{equation} \label{eq:def-pi-R-high}
\zeta \cdot \frac{dz}{h} + \zeta' \cdot \frac{dz'}{h} + \tau d(\frac{1}{h}) \to \zeta' \cdot \frac{dz'}{h}
\end{equation}
respectively.

In the same way as in Section~\ref{subsec:microlocal-partition-low}, we consider the projection
\begin{equation} \label{def:Lprojhigh-without-resolve}
    \Lprojhigh: L^{\bfs,\calchigh} \to X_{\rmb}^2.
\end{equation}
Similar to the low-energy case, since the projection from $\Lsharphigh$ to $X_{\rmb}^2$ is always a diffeomorphism, we do not need to decompose it.
Though we take the part over the interior of $X$ into consideration now, $L^{\bfs,\calchigh}$ is still compact and the construction of the finite cover is still valid.
\footnote{
It might seem strange that $L^{\bfs,\#}$ remains compact since $X$ should not be thought of as a `compact' object in terms of analysis. However, it is asymptotic to the exact cone over $Y$ and its geometric property near infinity is captured by $Y$. Away from $\partial X$, we again have a truly compact region. All those are finally packaged into the fact that we have those compactifications in which $L^{\bfs,\calchigh}$ remains compact.
}
The only subtle difference is that there is no natural choice of the distance function that is smooth down to the boundary to use to construct our Lebesgue number as before.
This can be remedied as follows.

First we give a metric on ${}^{\bundlehigh}T^*X_{\rmb,\calchigh}^2$, which is highly non-canonical, but only serves as an auxiliary tool to deduce some bounds of wavefront sets.
We let $\{ U_i \}$ be a finite open cover of ${}^{\bundlehigh}T^*X_{\rmb,\calchigh}^2$ such that ${}^{\bundlehigh}T^*X_{\rmb,\calchigh}^2$ is trivialized as in coordinate systems described in \eqref{eq:semiclassical-form-interior}\eqref{eq:semiclassical-form-bf}\eqref{eq:semiclassical-form-boundary-1}. 
Let $\rho_{\BFS} = x+x'$ be the defining function of $\BFS$, and fix a small constant $\rho_0>0$ as in \eqref{eq:product-rhoBFS-Lbfs} below, we choose those coordinate systems so that over the region $\{ \rho_{\BFS} \geq \frac{\rho_0}{2} \}$, which is away from $\BFS$, they are just a union of coordinate systems of ${}^{\calchigh}T^*X$ lifted from the left and right factors.
Now let $\mk{g}_i$ be the metric that is Euclidean in coordinate systems above\footnote{In other coordinates, we change it covariantly, hence this is indeed a metric locally and the glued version $\mk{g}$ is indeed a metric. } and let $\{\chi_i\}$ be a partition of unity subordinate to $\{ U_i \}$ and set
\begin{equation} \label{eq:L-high-metric-artificial-glued}
    \mk{g} = \sum_i \chi_i \mk{g}_i.
\end{equation}
Then this gives a metric on ${}^{\bundlehigh}T^*X_{\rmb,\calchigh}^2$ which induces the same topology as before and induces a metric on $L^{\bfs,\calchigh}$ that is smooth down to boundaries. 


Now we describe our cover of $L^{\bfs,\calchigh}$ by considering the part near $\BFS \cap \smf$ and the part away from $\BFS \cap \smf$ separately.
We notice that the part of $L^{\bfs,\calchigh}$ near $\BFS \cap \smf$ can be identified with 
\begin{equation} \label{eq:product-rhoBFS-Lbfs}
    [0,\rho_0)_{ \rho_{\BFS} } \times L^{\bfs},
\end{equation}
where $L^{\bfs}$ is the propagating Legendrian for low and fixed finite energy levels defined using \eqref{eq:Lbf-definition-gamma^2}.
Notice that the projection from it never degenerates on the $\rho_{\BFS}$-direction if $\rho_0$ is sufficiently small and $d\rho_{\BFS}$ is never linearly dependent with differentials of other coordinates on $L^{\bfs,\calchigh}$.
So we may just take the cover $\{ \mk{G}_\ell \}$ constructed in \eqref{eq:G-ell-model}, identify them as subsets of ${}^{\bundlehigh}T_{\smf \cap \BFS }^*X_{\rmb,\calchigh}^2$ and take product with $[0,\rho_0)_{ \rho_{\BFS} }$ to form
\begin{equation} \label{eq:def-tilde-mkG-ell}
    \tilde{\mk{G}}_{\ell} = [0,\rho_0)_{ \rho_{\BFS} } \times \mk{G}_{\ell} \subset {}^{\bundlehigh}T_{\smf}^*X_{\rmb,\calchigh}^2,
\end{equation}
and we have
\begin{equation}
    L^{\bfs,\calchigh} \cap \{ \rho_{\BFS} < \rho_0 \}
    =   \bigcup_{\ell}  ( L^{\bfs,\calchigh} \cap  \tilde{\mk{G}}_{\ell}).
\end{equation}
As discussed above, the rank drop of the projection from $L^{\bfs,\calchigh} \cap \{ \rho_{\BFS} < \rho_0 \} \cap  \tilde{\mk{G}}_{\ell}$ to $\smf$, which can be identified with $X_{\rmb}^2$,
is $k_\ell$ and is achieved at a certain point in this region.

Now we discuss the part $\{ \rho_{\BFS} \geq \frac{\rho_0}{2} \}$, which is away from $\BFS$.
Over this region, by our construction, $\mk{g}$ can be written as
\begin{equation} \label{eq:mkg,away-BFS}
    \mk{g} = dh^2 + \mk{g}_L + \mk{g}_R,
\end{equation}
where $\mk{g}_L,\mk{g}_R$ are metrics on ${}^{\calchigh}T^*_{ \{ h=0 \} }X$ lifted from the left and right factors and are glued using metrics that are Euclidean in the coordinates we chose.
Here both $\mk{g}_L$ and $\mk{g}_R$ are defined down to $\partial X$, though \eqref{eq:mkg,away-BFS} is not valid down to $\BFS$.

Now we let $\{ U_{\ell}^{\calchigh} \}$ be a finite open cover of the region $\{ \rho_{\BFS} \geq \frac{3}{4}\rho_0 \}$, which is compact, such that each $U_{\ell}^{\calchigh}$ is contained in $\{ \rho_{\BFS} \geq \frac{\rho_0}{2} \}$.
Then we let $\delta_{\calchigh}$ be a Lebesgue number associated with this open cover and the metric in \eqref{eq:mkg,away-BFS}. We will use this to define a partition of ${}^{\calchigh}\overline{T}_{\{h=0\}}^*X$, which in turn is used to construct our microlocal partition of unity.
Now we decompose ${}^{\calchigh}\overline{T}_{\{h=0\}}^*X$ as
\begin{equation} \label{eq:high-phase-cover}
    {}^{\calchigh}\overline{T}_{\{h=0\}}^*X =  U_1 \cup ( \bigcup_{j \in \mk{J}_{\calchigh,1} } U_{j}) \cup   ( \bigcup_{j \in \mk{J}_{\calchigh,2} } U_{j}) ,
\end{equation}
where 
\begin{itemize}
    \item $U_1$ covers the region that is away from $\Sigma$;
    \item $\bigcup_{j \in \mk{J}_{\calchigh,1}} U_{j}$ covers a tubular neighborhood of $\Sigma$
    over $\{ x < \frac{\rho_0}{4} \}$ and are contained in $\{ x < \frac{\rho_0}{2} \}$, with $\mk{J}_{\calchigh,1}$ being a finite index set.
    In addition, we may require them to satisfy the following conditions:
    \begin{itemize}
        \item Recall that $\delta_0$ is the Lebesgue number defined after \eqref{eq:Lbf-covered-G-ell}. We require each $U_j,\; j\in\mk{J}_{\calchigh,1}$, when intersected with $\{x=x_0\}$ for each fixed $x_0<\frac{\rho_0}{2}$, has diameter less than $\delta_0/2$, in terms of the metric induced by $h$ on ${}^{\calchigh}T_{\{h=0,x=x_0\}}^*X$, which can be identified with ${}^{\sct}T_{\partial X}^*X$. This is similar to our construction in Section~\ref{subsec:microlocal-partition-low}.
   \item We also require their diameters in terms of $\mk{g}_{L}$ and $\mk{g}_{R}$ to be smaller than $\frac{\delta_{\calchigh}}{2}$.
     \end{itemize}
     
\item For the part $( \bigcup_{j \in \mk{J}_{\calchigh,2} } U_{j})$, we require each $U_j, j \in \mk{J}_{\calchigh,2}$ to have a diameter less than $\frac{\delta_{\calchigh}}{2}$ in terms of $\mk{g}_{L}$ and $\mk{g}_{R}$.
\end{itemize}

Finally we let $\{ \chi_{j,\calchigh} \}$, $j \in \{1\} \cup \mk{J}_{\calchigh,1}  \cup \mk{J}_{\calchigh,2} $ be a partition of unity on ${}^{\calchigh}\overline{T}_{\{h=0\}}^*X$ that is subordinate to the cover \eqref{eq:high-phase-cover} and set (with $\mathrm{Op}$ denoting the quantization in \eqref{eq:PsiDO-high-def-1} or \eqref{eq:PsiDO-high-def-2} depending on the region):
\begin{equation}
    Q_{j}^{\calchigh} = \mathrm{Op}(\chi_{j,\calchigh}).
\end{equation}
For $j,j' \in \mk{J}_{\calchigh,1}  \cup \mk{J}_{\calchigh,2}$, we let $W_j$ be a small open neighborhood of $\WF_{\calchigh}'(Q_j)$ that is still contained in $U_j$ and define
\begin{equation} \label{eq:Lbf-high-jj'}
    L^{\bfs,\calchigh}_{j,j'} =
\{  q \in L^{\bfs,\calchigh} | \;  \pi_{L,\calchigh}(q) \in W_j, \; \pi_{R,\calchigh}(q) \in W_{j'} \},
\end{equation}
where $\pi_{L,\calchigh}, \pi_{R,\calchigh}$ are projections defined in \eqref{eq:def-pi-L-high} and \eqref{eq:def-pi-R-high} respectively.
Then we have the following property of $L^{\bfs,\calchigh}_{j,j'}$, which is the motivation for our construction.


\begin{proposition} \label{prop:microlocal-partition-high}
There exists a microlocal partition of unity $Q_j^{\calchigh} \in \Psi_{\calchigh}^{0,0,0}$ such that
\begin{equation*}
\mathrm{Id}=\sum_{j \in  \{1\} \cup \mk{J}_{\calchigh,1}  \cup \mk{J}_{\calchigh,2} } Q_j^{\calchigh} \; \text{ when } \; \lambda \gtrsim 1
\end{equation*}
such that 
\begin{enumerate}
\item $\WF_{\calchigh}'(Q_{1}^{\calchigh})$ is away from $\Sigma$.
\item For each pair of $j,j' \in  \mk{J}_{\calchigh,1}  \cup \mk{J}_{\calchigh,2}$, with $L^{\bfs,\calchigh}_{j,j'}$ defined as in \eqref{eq:Lbf-high-jj'} above, there is one $\tilde{\mk{G}}_{\ell} \subset \{ \rho_{\BFS} < \rho_0 \}$ or $U_{\ell}^{\calchigh} \subset \{ \rho_{\BFS} < \frac{\rho_0}{2} \}$ such that
\begin{equation} \label{eq:Lbf-high-jj'-contained-1}
   L^{\bfs,\calchigh}_{j,j'}  \subset  \tilde{\mk{G}}_{\ell},
\end{equation}
or 
\begin{equation} \label{eq:Lbf-high-jj'-contained-2}
   L^{\bfs,\calchigh}_{j,j'}  \subset  U_{\ell}^{\calchigh}.
\end{equation} 
In particular, the rank drop of the projection from $L_{j,j'}^{\bfs,\calchigh}$ defined in \eqref{eq:Lbf-high-jj'} to $X_{\rmb}^2$ is at most $k_\ell$, which is the maximal rank drop from $L^{\bfs,\calchigh} \cap \tilde{\mk{G}}_{\ell}$ to $X_{\rmb}^2$.
In addition, the diameters of the left and right projection of $L^{\bfs,\calchigh}_{j,j'}$ are at most $\delta_{\calchigh}$ for a fixed small number $\delta_{\calchigh}>0$.
\end{enumerate}

\end{proposition}

\begin{proof}
There are three cases in terms of where indices $j,j'$ lie:
\begin{itemize}
    \item both $j,j'$ are in $\mk{J}_{\calchigh,1}$;
    \item $j,j'$ are in $\mk{J}_{\calchigh,1}$ and $\mk{J}_{\calchigh,2}$ respectively;
    \item both $j,j'$ are in $\mk{J}_{\calchigh,2}$.
\end{itemize}

For the first case, which will correspond to the part near $\BFS$, we know 
$x,x'<\frac{\rho_0}{2}$ on $L^{\bfs,\calchigh}_{j,j'} $, hence
\begin{equation*}
    L^{\bfs,\calchigh}_{j,j'} \subset \{ \rho_{\BFS} < \rho_0 \}.
\end{equation*}
Then we use the fact that each $W_j$ (since it is contained in $U_j$, its diameter is bounded by the diameter of $U_j$) has diameter less than $\frac{\delta_0}{2}$ for each fixed $x$ in terms of the metric on ${}^{\sct}T^*_{\partial X}X$.
So the discussion in \ref{subsec:microlocal-partition-low} applies to $L^{\bfs,\calchigh}_{j,j'}$ with fixed $x$. By our definition of $\tilde{\mk{G}}_{\ell}$ in \eqref{eq:def-tilde-mkG-ell}, we know that there is a $\tilde{\mk{G}_{\ell}}$ that contains $L^{\bfs,\calchigh}_{j,j'}$.

Now we consider the second and third cases, which correspond to being near $\LB$,$\RB$ and only away from all of $\LB,\RB,\BFS$ respectively. 
In this case, since at least one of $x,x'$ is larger than $\frac{\rho_0}{2}$ on $L^{\bfs,\calchigh}_{j,j'}$, we know \eqref{eq:mkg,away-BFS} is valid. Using the triangle inequality, we know that the diameter of each $L^{\bfs,\calchigh}_{j,j'}$ is upper bounded by $\frac{\delta_{\calchigh}}{2}+\frac{\delta_{\calchigh}}{2} = \delta_{\calchigh}$, hence it is contained in one of $U_{\ell}^{\calchigh}$ by the definition of $\delta_{\calchigh}$ and this completes the proof.
\end{proof}


\begin{remark}
We briefly explain the reason why we consider the part near and away from $\BFS \cap \smf$ separately. 
Since we microlocalize by multiplying pseudodifferential operators from the left and right, our microlocalization is localizing on both the left and right projections of points in $L^{\bfs,\calchigh}$.
However, we can have points with very close $x,x'$-components but very different value of $\sigma= x/x'$ (or $\theta = x'/x$), which means that they are already as close as possible in terms of our microlocalizations but still not close to each other when lifted to ${}^{\bundlehigh}T^*_{\smf}X_{\rmb,\calchigh}^2$.
However, this can only happen near $\BFS \cap \smf$ and being near this region means that we are close to the boundary in both the left and right factors, which further implies that the geometric focusing, or equivalently the rank drop of the projection, is from the geometry of $\partial X = Y$ and becomes independent of $\sigma$, as in the low-energy case.
\end{remark}


Finally, since our requirements for the microlocal partition over the overlapped part between the high-energy and low-energy regions coincide, we can simply glue them using the partition of unity in $\lambda$, potentially after taking union of index sets and open covers in both cases, to form the following partition.
To state the conclusion, we set $\overline{\mk{J}}_{\calc} =  \Jlow \cup \Jhigh \cup \{\zf,1\}$, where 
\begin{equation} \label{eq:def-Jhigh}
\Jhigh = \mk{J}_{\calchigh,1} \cup \mk{J}_{\calchigh,2}.    
\end{equation}

\begin{proposition} \label{prop: microlocal-partition-combined}
There exists a microlocal partition of unity $Q_j^{\calc} \in \Psi_{\calc}^{0,0,0}(X)$:
\begin{equation*}
\mathrm{Id}=\sum_{j \in  \overline{\mk{J}}_{\calc} } Q_j^{\calc}
\end{equation*}
that satisfies conditions in Proposition~\ref{prop:microlocal-partition-low} when $\lambda \lesssim 1$ and Proposition~\ref{prop:microlocal-partition-high} when $\lambda \gtrsim 1$. 
\end{proposition}


\subsection{Oscillatory integral representations of the microlocalized spectral measure}
\label{subsec:osc-integral-formula}

In this subsection, we give the oscillatory integral representation of the spectral measure in both the low- and high-energy regimes.
This is done by unravelling the information in \cite[Theorem~3.10]{GHS2} and \cite[Corollary~1.2]{Hassell-Wunsch-semiclassical-resolvent} (a concise version of this was recalled in Theorem~\ref{thm:spectral-measure-concise}), conducting microlocalization, and combining with our discussion in previous sections (in particular, Proposition~\ref{prop:minimal-parametrization-low}). 
Those expressions have been given in Section~\ref{sec:Legendrian-distributions} for general Legendre distributions and the goal of this subsection is to specialize to the case of the spectral measure and give a decomposition of our spectral measure into such local pieces that capture our geometric information encoded in the microlocal partition of unity we constructed before.
To this end, we recall the complete version of \cite[Theorem~3.10]{GHS2},  
which captures the decay order of $\specm$ with respect to boundary hypersurfaces of $X_{\rmb,\flat}^2$.

First of all, we recall basic facts about the microlocalization of Legendre distributions, which are used to describe our microlocalized spectral measure.

We consider the low-energy part first. 
Let $\pi_L,\pi_R$ be the projection from ${}^{\Phi}T_{\bfs}^*X$ to  ${}^{\flat}T_{\partial X}^*X \cap \{\lambda = \lambda_1\} \subset {}^{\calc}\overline{T}^*X$ at fixed $\lambda = \lambda_1 \leq \lambda_0$
(here and below, we omit the $\lambda$-component when we restrict to fixed $\lambda$)
 via projecting to the left and right variables respectively:
\begin{align}
& \pi_L: \; (\sigma, y, y', \mu, \mu', \nu, \nu')
\to (y,\mu,\nu),\\
& \pi_R: \; (\sigma,y,y',\mu,\mu',\nu,\nu') \to (y',\mu',\nu'),
\end{align}
in terms of coordinates in \eqref{eq:contact-form}, \eqref{eq:contact-form-interior} and \eqref{eq:coordinates-bf-fibered-bundle}.
  
The following lemma is a modified version of \cite[Lemma~5.4]{GHS1}, which says that if we apply a microlocalizer that is trivial near $L^{\bfs}$ on one side, then the output is residual (i.e. has infinite order decay on all boundary faces) on that side.

\begin{lemma}{\cite[Lemma~5.4]{GHS1}} \label{lemma:low-energy-trivial-composition}
Suppose $F \in I_{\calclow}^{m,p,r_{\LB},r_{\RB};\mathcal{B}}(X_{\rmb,\flat}^2;(L^{\bfs},L^\#);\Omega_{\flat}^{1/2})$ and $Q \in \Psi^{0}_{\flat}(X;\Omega_{\calclow}^{1/2})$ satisfy: $\big(\WF'_{\flat}(Q) \cap \{\lambda = \lambda_1\} \big) \cap \pi_L(L^{\bfs} \cup L^\#) = \emptyset$ for all fixed $\lambda_1$, then 
\footnote{Our notation is slightly different from  \cite[Lemma~5.4]{GHS1}, with $Q$ here corresponding to $\Id-Q$ there. }
\begin{equation}
QF \in I_{\calclow}^{\infty,\infty,\infty,r_{\RB};\mathcal{B}}(X,(L^{\bfs},L^\#);\Omega_{\flat}^{1/2}).
\end{equation}
Similarly, suppose $Q ' \in  \Psi^{0}_{\flat}(X;\Omega_{\calclow}^{1/2})$ satisfies: $\big( \WF'_{\flat}(Q') \cap \{\lambda =\lambda_1\}\big) \cap\pi_R(L^{\bfs} \cup L^\#)= \emptyset$ for all fixed $\lambda_1$, then 
\begin{equation}
FQ' \in I_{\calclow}^{\infty,\infty,r_{\LB},\infty;\mathcal{B}}(X,(L^{\bfs},L^\#);\Omega_{\flat}^{1/2}).
\end{equation}
\end{lemma}

The high-energy analogue of the lemma above is \cite[Lemma~7.3]{GHS1}, which we recall next.
\begin{lemma}{\cite[Lemma~7.3]{GHS1}} 
\label{lemma:high-energy-trivial-composition}
Suppose $F \in I_{\calchigh}^{m_{\calchigh}, p;m,r_{\LB},r_{\RB} }(X_{\rmb,\calchigh}^2, (L^{\bfs,\calchigh},\Lsharphigh) ; \Omega_{\calchigh}^{1/2})$ and $Q \in \Psi^{0}_{\calchigh}(X;\Omega_{\calclow}^{1/2})$ satisfy: $\big(\WF'_{\calchigh}(Q) \cap \{h = h_1\} \big) \cap \pi_{L,\calchigh}(L^{\bfs,\calchigh} \cup G_1^{\calchigh}) = \emptyset$ for all fixed $h_1$, with $\pi_{L,\calchigh}$ defined in \eqref{eq:def-pi-L-high}, then 
\begin{equation}
QF \in I_{\calchigh}^{\infty,\infty;\infty,\infty,r_{\RB}}(X_{\rmb,\calchigh}^2,(L^{\bfs,\calchigh},G_1^{\calchigh});\Omega_{\calchigh}^{1/2}).
\end{equation}
Similarly, suppose $Q ' \in  \Psi^{0}_{\calchigh}(X;\Omega_{\calclow}^{1/2})$ satisfies: $\big( \WF'_{\calchigh}(Q') \cap \{h =h_1\}\big) \cap\pi_{R,\calchigh}(L^{\bfs,\calchigh} \cup G_1^{\calchigh})= \emptyset$ for all fixed $h_1$, with $\pi_{R,\calchigh}$ defined in \eqref{eq:def-pi-R-high}, then 
\begin{equation}
FQ' \in I_{\calchigh}^{\infty,\infty;\infty,r_{\LB},\infty}(X_{\rmb,\calchigh}^2,(L^{\bfs,\calchigh},G_1^{\calchigh});\Omega_{\calchigh}^{1/2}).
\end{equation}
\end{lemma}

To combine statements concerning the low and high-energy regimes, we introduce the following combined version of Definition~\ref{def:Legendrian-dis-conic-intersecting-low} and Definition~\ref{def:Legendrian-dis-conic-high}.

\begin{definition}
The class of combined Legendre distributions associated with $(L^{\bfs},L^{\#})$ at low-energy and associated with $(L^{\bfs,\calchigh},\Lsharphigh)$, which we denote by
\begin{equation} \label{def:combined-Legendrian-distributions}
I_{\calc}^{m_{\calchigh},m, p,r_{\LB},r_{\RB} ; \mathcal{B}}(X_{\rmb,\calc}, (L^{\bfs},L^{\#}), (L^{\bfs,\calchigh},\Lsharphigh); \Omega_{\calc}^{1/2}),
\end{equation}
consists of ($\Omega_{\calc}^{1/2}$-valued) distributions on $X_{\calc}^2$ (defined in \eqref{eq: definition, combined double space}) such that for a function $\chi \in C_c^\infty([0,\infty))$ that is identically $1$ near $0$, and extended to $[0,\infty]$ via extending by $0$ near $\infty$, we have 
\begin{equation}
    \chi u \in I_{\calclow}^{m, p; r_{\LB}, r_{\RB}; \mathcal{B}}(X_{\rmb,\flat}^2, (L^{\bfs}, L^{\#}); \Omega_{\flat}^{1/2}),
\end{equation}
and 
\begin{equation}
    (1-\chi) u \in I_{\calchigh}^{m_{\calchigh}, p;m, r_{\LB}, r_{\RB}; \mathcal{B}}(X_{\rmb,\calchigh}^2, (L^{\bfs,\calchigh}, \Lsharphigh); \Omega_{\calchigh}^{1/2}).
\end{equation}

Here $m_{\calchigh},m,p,r_{\LB},r_{\RB}$ and the index family $\mathcal{B}$ are indices associated with $\smf$, $\BFS$, $L^\#$ (and $\Lsharphigh$), $\LB$, $\RB$ and the family of faces $(\bfs_0,\lb_0,\rb_0,\zf)$, respectively.

\end{definition}

Since the definitions of $I_{\calchigh}^{\bullet}$ and $I_{\calclow}^{\bullet}$ coincide when $\lambda$ is away from $0$ and $\infty$, the definition above does not depend on the choice of $\chi$. Using ~\eqref{def:combined-Legendrian-distributions}, the complete characterization of the spectral measure is as follows.

\begin{theorem} \label{thm:spectral-measure-complete}
For $P$ in \eqref{eq:P-def}, the spectral measure of $\sqrt{P}$, which we denote by $\specm$ is in the class
\begin{equation}
I_{\calc}^{m_{\calchigh},m,p,r_{\LB},r_{\RB} ; \mathcal{B}}(X_{\rmb,\flat}, (L^{\bfs},L^{\#}); \Omega_{\flat}^{1/2})  \otimes |\lambda d\lambda|^{1/2},
\end{equation}
where $m_{\calchigh} = \frac{1}{2}, \, m=-\frac{1}{2}, \, p = \frac{n-2}{2}, \, r_{\LB} = r_{\RB} = \frac{n-1}{2}$, and 
$\mathcal{B} = (\mathcal{B}_{\bfs_0}, \mathcal{B}_{\lb_0}, \mathcal{B}_{\rb_0} , \mathcal{B}_{\zf})$ with those four index sets satisfying:
\begin{align}
\begin{split}
\min \mathcal{B}_{\bfs_0} = -1, \;
\min \mathcal{B}_{\lb_0} = \frac{n}{2}-1 =  \min \mathcal{B}_{\rb_0},  \;
\min \mathcal{B}_{\zf} = n-1.
\end{split}    
\end{align}
For $P$ in \eqref{eq:P-def-exact-cone}, $\specm \in I_{\calclow}^{m, p; r_{\LB}, r_{\RB}; \mathcal{B}}(X_{\rmb,\flat}^2, (L^{\bfs}, L^{\#}); \Omega_{\flat}^{1/2})$ for $\lambda \lesssim 1$ and the same local expressions hold up to $\lambda \to \infty$.
\end{theorem}

\begin{proof}
The first part concerning the asymptotically conic case follows by combining \cite[Theorem~3.10]{GHS2} and \cite[Corollary~1.2]{Hassell-Wunsch-semiclassical-resolvent}.
We also note that the conclusion for $\lambda$ in a fixed interval $[a,b]$ with $0<a<b<\infty$ can be deduced from \cite[Theorem~7.2]{hassell1999spectral} as well. The point of using the two results above is to enable us to obtain the uniform behaviour as $\lambda \to 0$ and $\lambda \to \infty$.

For the second part concerning $P$ in \eqref{eq:P-def-exact-cone}, it follows from the same proof of \cite[Theorem~3.10]{GHS2}, which uses the resolvent and Stone's Formula to construct the spectral measure, except that now we use the resolvent characterized in \cite[Theorem~5.1]{GHS2} as a Legendre distribution  directly.


\end{proof}



Now we use this to give oscillatory integral representations of the microlocalized spectral measure in terms of the decomposition:
\begin{equation} \label{eq:spectral-measure-decomposition}
\specm  = \sum_{j,j' \in \overline{\mk{J}}_{\calclow} }  Q_j^{\calclow} \specm Q_{j'}^{\calclow},
\end{equation}
where $\overline{\mk{J}}_{\calclow}=\{\zf,1\} \cup \Jlow$.
To this end, we use wavefront bounds under compositions in \cite[Section~5]{GHS1} to specify where the singularity of those pieces on the right hand side of \eqref{eq:spectral-measure-decomposition} lies.


We consider the case when both $j$ and $j'$ are $\zf$ or $1$ first.

\begin{proposition} \label{prop:localized-specm-two-side-residual-low}
Let $Q_{\zf}^{\calclow},Q_{\zf}^{'\calclow},Q_1^{\calclow}$ be as in Section~\ref{subsec:microlocal-partition-low}, then each of 
\begin{equation}
  Q_{\zf}^{\calclow} \specm Q_{\zf}^{'\calclow}, \; Q_{\zf}^{\calclow} \specm Q_{1}^{\calclow}, 
  \; Q_{1}^{\calclow} \specm Q_{\zf}^{'\calclow}, \; Q_{1}^{\calclow} \specm Q_{1}^{\calclow}
\end{equation}
   is a multiple of $|dgdg'|^{1/2} \otimes |d\lambda|$ with coefficient that is polyhomogeneous on $X_{\rmb,\calclow}^2$. More precisely, they are of the form:
\begin{equation} \label{eq:local-specm-0}
         \lambda^{n-1}  a  |dgdg'|^{1/2} \otimes |d\lambda| ,
\end{equation}
  where $a$ satisfies that for any $N,K \in \N$ there is a constant $C_{N,K}$ such that
\begin{equation} \label{eq:QzfEQzf-symbol-est}
  |(\lambda \partial_{\lambda})^N a| \leq C_{N,K}  (1+\lambda d_{\conic})^{-K}.
\end{equation}
\end{proposition}

\begin{proof}

Consider $Q_{\zf}^{\calclow} \specm Q_{\zf}^{'\calclow}$ first.  
Since $Q_{\zf}^{\calclow},Q_{\zf}^{'\calclow}$ are multiplication operators by $(1-\chi(\rho)),(1-\chi(\rho'))$ respectively, we know that the kernel of $Q_{\zf}^{\calclow} \specm Q_{\zf}^{'\calclow}$ is supported in the region $\rho,\rho' \gtrsim 1$, which is away from $\LB,\RB,\BFS$.
Then the oscillatory integrals in Definition~\ref{def:Legendrian-dis-conic-intersecting-low} and the definitions referred to therein are not oscillating anymore and can be integrated to give a polyhomogeneous function of the same index set as the amplitude together with those decay factors.\footnote{One can also choose $\chi$ so that only the $u_6$-term in Definition~\ref{def:Legendre-distribution-low-energy} survives after this localization and the conclusion follows. 
And one can see from this that requirements on different terms in Definition~\ref{def:Legendre-distribution-low-energy} coincide on the overlapped regions. } 
So it has the form as in \eqref{eq:local-specm-0}, except that we need to explain the decay order.

Using the condition that the support is away from $\LB,\RB,\BFS$ again, this term has arbitrary order of decay as any of $\rho_{\LB},\rho_{\RB},\rho_{\BFS}$ tends to 0.
Then the estimate \eqref{eq:QzfEQzf-symbol-est} follows from the observation that 
a product of boundary defining functions of $\LB,\RB,\BFS$ is $O((1+\lambda d_{\conic})^{-1})$.

Now we explain the decay order $\lambda^{n-1}$.
As stated in Theorem~\ref{thm:spectral-measure-complete}, $\specm$ 
is polyhomogeneous with index set $\mathcal{B}$ in Theorem~\ref{thm:spectral-measure-complete} as a section of $\Omega_{\calclow}^{1/2} \otimes |\lambda d\lambda|^{1/2}$. So we should compare a section of $\Omega_{\calclow}^{1/2} \otimes |\lambda d\lambda|^{1/2}$ with $|dgdg'|^{1/2} \otimes |d\lambda|$.
By the definition of $\Omega_{\flat}^{1/2}$ in \eqref{eq:low-energy-density-1}\eqref{eq:low-energy-density-2}\eqref{eq:low-energy-density-3}, after tensoring with $|\lambda d\lambda|^{1/2}$, they vanish, compared with $|dgdg'|^{1/2} \otimes |d\lambda|$, to order $n$ at $\bfs_0$,
order $0$ in the interior of $\zf$, order $\frac{n}{2}$ at $\lb_0$ and $\rb_0$ respectively. Combining orders above with orders specified in Theorem~\ref{thm:spectral-measure-complete} gives the vanishing order $n-1$ at $\bfs_0, \lb_0, \rb_0, \zf$.
Finally, the $\lambda^{n-1}$-factor follows by observing that $\lambda$ is a product of boundary defining functions of $\lb_0,\rb_0,\bfs_0$ and $\zf$.


The conclusion for other terms follows in the same way. The only difference is that when the microlocalizer is replaced by $Q_{1}^{\calclow}$ on that side, the property being supported in $\rho$ or $\rho' \gtrsim 1$ is replaced by being Schwartz in $\rho$ or $\rho'$, using Lemma~\ref{lemma:low-energy-trivial-composition}. So oscillatory factors like $e^{i\Phi/\rho},e^{i\Phi/\rho'}$ won't affect regularity anymore since losses like $\rho^{-1}$ or $(\rho')^{-1}$ introduced by differentiating them can be absorbed now. Consequently, the entire kernel is a function with the same regularity as the amplitude $a$, being polyhomogeneous with assigned index set and the rest of the discussion about the precise decay order  as $\lambda \to 0$ is the same as the case of $Q_{\zf}^{\calclow} \specm Q_{\zf}^{'\calclow}$ above.

\end{proof}

Now consider the case that exactly one of $j$ and $j'$ is $\zf$ or $1$. 
\begin{proposition} \label{prop:localized-specm-one-side-residual-low}

\begin{enumerate}
  For $j \in \Jlow$, we have the following oscillatory integral representations of the microlocalized spectral measure.
\item \label{item:QjEQ1-osc}
$Q_j^{\calclow} \specm Q_1^{\calclow}$ can be expressed as a finite sum of terms taking one of the following forms:
\begin{itemize}
  \item Terms supported in the region $\theta = x'/x \lesssim 1, \lambda \lesssim 1$ and take the form
\begin{equation} \label{QjEQ1-low-rb}
\lambda^{n-1} a(\lambda,\theta,y,y',x'/\lambda) |dgdg'|^{1/2} \otimes |d\lambda|.
\end{equation}
\item Terms supported in the region $\sigma = x/x' \lesssim 1, x/\lambda \lesssim 1$ and take the form
  \begin{gather}
\quad  e^{\pm i\lambda/x} \lambda^{n-1}  a(\lambda,\sigma,y,y',x'/\lambda) |dgdg'|^{1/2} \otimes |d\lambda|, \label{QjEQ1-c-low}
\end{gather}
\end{itemize}
In each case, $a(\lambda,\sigma,y,y',x'/\lambda)$ is compactly supported in its variables and satisfies
\begin{equation} \label{eq:phg-conormal-2}
  |(\lambda \partial_{\lambda})^N a | \lesssim (1+\lambda d(z,z'))^{-\frac{n-1}{2}}, 
\end{equation}
for any $N \in \N$.

\item \label{item:Q1EQj-osc} 
$Q_1^{\calclow} \specm Q_j^{\calclow}$ is a finite sum of terms of the same form as in Part~\eqref{item:QjEQ1-osc}, except with all primed and un-primed variables switched in all expressions. 

\item  \label{item:QjEQzf-single-term}
$Q_j^{\calclow} \specm Q_{\zf}^{'\calclow}$ is a finite sum of terms of the following form:
\begin{itemize}
  \item Terms supported in the region  $\theta = x'/x \lesssim 1$, $\lambda/x' \lesssim 1$ and take the form
\begin{equation} \label{QjEQzf-low-rb}
\lambda^{n-1}  a(\lambda,\sigma,y,y',\lambda/x') |dgdg'|^{1/2} \otimes |d\lambda|.
\end{equation}
\item Terms supported in the region $\sigma = x/x' \lesssim 1$, $\lambda/x' \lesssim 1$ and take the form
  \begin{gather}
\quad  e^{i\lambda/x} \lambda^{n-1} \sigma^{ \frac{n-1}{2} } a(\lambda,\sigma,y,y',\lambda/x') |dgdg'|^{1/2} \otimes |d\lambda| . \label{QjEQzf-c-low}
\end{gather}
\end{itemize}
In each case, $a(\lambda,\sigma,y,y',\lambda/x')$ satisfies
\begin{equation}
  |(\lambda \partial_{\lambda})^N a | \lesssim (1+\lambda d(z,z'))^{-\frac{n-1}{2}}, 
\end{equation}
for any $N \in \N$.

\item \label{item:QzfEQj-single-term}
$Q_{\zf}^{\calclow} \specm Q_j^{\calclow}$ is a finite sum of terms of the same form as in Part~\eqref{item:QjEQzf-single-term} with all primed and un-primed variables switched.

\end{enumerate}

\end{proposition}

\begin{proof}
We consider \eqref{item:QjEQ1-osc} first and \eqref{item:Q1EQj-osc} follows by the same proof, or simply taking the adjoint. 
For \eqref{item:QjEQ1-osc}, consider terms supported in $\sigma \lesssim 1$ first.
We observe that $Q_j^{\calclow} \specm Q_1^{\calclow}$ can be expressed as a sum of a finite number of terms of one of the following forms
\begin{gather}
\lambda^{n-1}  \int_{\R^k} e^{i\lambda\Phi(\sigma,y,y',v)/x}  
\sigma^{ \frac{n-1}{2} } \tilde{a}(\lambda,\sigma,y,y',\frac{x'}{\lambda},v)dv \;  |dgdg'|^{1/2} \otimes |d\lambda| \quad \mathrm{ or }   \label{QjEQ1-or-zf-low} \\
\lambda^{n-1}  \int_{\R^{k}} \int_0^\infty e^{i\lambda\Phi(\sigma,y,y',v,s)/x} 
  s^{n-2} \sigma^{ \frac{n-1}{2} } \tilde{a}(\lambda,\sigma,y,y',\frac{x'}{\lambda},v,s) \, dv \, ds 
\; |dgdg'|^{1/2} \otimes |d\lambda|
   \label{QjEQ1-or-zf-conic-low} 
  \end{gather}
  and are supported in the region $\sigma = x/x' \leq 2$, $x'/\lambda \lesssim 1$. Here $\Phi(\sigma,y,y',v)$ in \eqref{QjEQ1-or-zf-low} parametrizes $L^{\#}$ (in this case $k=0$ and $v$ is absent) or $L^{\bfs}$ in the sense of Definition~\ref{def:Legendre-parametrization-low}, while $\Phi(\sigma,y,y',v,s)$ in \eqref{QjEQ1-or-zf-conic-low} parametrizes the pair $(L^{\bfs},L^\#)$ in the sense of Section~\ref{subsec:conic-pair-geometry-and-phase-function}.

In each case, $\tilde{a}(\cdot)$ is a smooth function compactly supported in the $v$ and $s$ variables (where present), such that 
\begin{equation} \label{eq:est-a-osc-residual-rb}
    |(\lambda\partial_\lambda)^N \tilde{a}|\leq C_{N,M} (x'/\lambda)^M = C_{N,M}(\rho')^M
\end{equation}
for all $N,M \in \mathbb{N}$.
Those expressions follow from the definition of Legendre distributions in Definition~\ref{def:Legendre-distribution-low-energy} and the definitions referred to therein, except for a power of $\frac{x'}{\lambda}$.
So we only need to justify \eqref{eq:est-a-osc-residual-rb} since this is able to absorb powers of $\frac{x'}{\lambda}$. This arbitrary order of decay in $(x'/\lambda)$ follows from the observation that in this low-energy regime $x'/\lambda$ is a product of boundary defining functions of $\RB$ and $\BFS$, on which we have arbitrary order of decay by Lemma~\ref{lemma:low-energy-trivial-composition} \footnote{In fact, the more essential reason traces back to how Lemma~\ref{lemma:low-energy-trivial-composition} is proved in \cite[Lemma~5.4]{GHS1}: the $\frac{\bullet}{\rho'}$ part of the phase is non-stationary and we can apply a non-stationary phase argument to have arbitrary order of decay in $\rho'$. }.

Next we justify the oscillatory factor. The phase functions appearing above are of the form 
\begin{equation}
 \pm \frac{1+\sigma+\sigma \psi}{\rho} = \pm \big( \frac{1}{\rho} + \frac{1+\psi}{\rho'} \big).
\end{equation}
Consider the case with $+$ sign below and the other case can be treated in the same way. Hence we can extract the $e^{i/\rho}=e^{i\lambda/x}$-factor and the remaining part of the oscillatory integral uses $\frac{1+\psi}{\rho'}$ as its phase. 
This gives a function that has the same regularity property as the amplitude $a$ since differentiating the oscillating factor at most introduces $(\rho')^{-1}$ factors, while our amplitude satisfies \eqref{eq:est-a-osc-residual-rb} and this can be absorbed. In particular, it is conormal in $\lambda$.

Next we justify \eqref{eq:phg-conormal-2}. Discussion above already shows that $a$ in \eqref{eq:phg-conormal-2} is conormal in $\lambda$ and has arbitrary order of decay at  $\RB$ and $\BFS$. Then together with the $\sigma^{\frac{n-1}{2}}$-factor, we know that it has at least $\frac{n-1}{2}$-order of decay at all of $\LB,\BFS$ and $\RB$. Applying Lemma~\ref{lemma:bdf-1} justifies \eqref{eq:phg-conormal-2}. 

For those terms supported in the region that $\sigma \gtrsim 1$, we use phase functions of the form
\begin{equation}
\frac{\Phi}{\rho'},
\end{equation}
where $\Phi$ is a smooth function and the oscillatory integral integrates to be a function with the same regularity and decay properties as before for the same reason as above.

Now we discuss \eqref{item:QjEQzf-single-term}, which follows by a similar (in fact simpler) argument as above.
Recall from Section~\ref{subsec:microlocal-partition-low} that $Q_{\zf}^{'\calclow}$ is just the multiplication operator $1-\chi(\rho')$, which localizes those oscillatory integrals to $\rho' \gtrsim 1$.
Recalling the form of our Legendre distributions in Definition~\ref{def:Legendrian-dis-conic-intersecting-low} and the definitions referred to therein,  
for those terms supported in the region that $\sigma \gtrsim 1$, we use a phase function of the form
\begin{equation}
\frac{\Phi}{\rho'},
\end{equation}
where $\Phi$ is a smooth function and this is not oscillating and contributes a term that has regularity and decay property that is at least as good as the amplitude $a$ together with prefactors indicating decay orders: polyhomogeneous in $\lambda$ and smooth in all other variables. In fact, $\rho \gtrsim 1$ on the support of those terms as well.

On the other hand, for those terms supported in the region $\sigma \lesssim 1$, we are using phase functions of the form 
\begin{equation}
  \frac{1+\sigma+\sigma \psi}{\rho} = \frac{1}{\rho} + \frac{1+\psi}{\rho'}.
\end{equation}
Hence we can extract the $e^{i/\rho}=e^{i\lambda/x}$-factor and again the remaining part of the phase $\rho' \gtrsim 1$ integrates to give a smooth function times those prefactors, which gives \eqref{QjEQzf-c-low}.

\end{proof}

\begin{remark}
There are two types of results that we have obtained by `switching primed and un-primed variables' in the proof, which are slightly different.
An example of the first type is between oscillatory integrals in two parts of the proof of Part~\eqref{item:QjEQ1-osc}, where switching those variables is because we gave local expressions in this manner in Definition~\ref{def:Legendrian-dis-conic-intersecting-low}.
On the other hand, the switching between Part~\eqref{item:QjEQ1-osc} and Part~\eqref{item:Q1EQj-osc} is because we have switched the microlocalizers on two sides. So we take both parts before and after switching in Part~\eqref{item:QjEQ1-osc}, and then switch primed and un-primed variables to reduce to the corresponding term in Part~\eqref{item:QjEQ1-osc}.
In particular, the major difference that appears in the statement is that in Part~\eqref{item:QjEQ1-osc}, for both parts before and after switching, the arbitrary order decay is in $(x'/\lambda)$.
While in Part \eqref{item:Q1EQj-osc}, for both parts before and after switching, the arbitrary order decay is in $(x/\lambda)$.
The same comparison applies to Part~\eqref{item:QjEQzf-single-term} and Part~\eqref{item:QzfEQj-single-term}.
\end{remark}





We summarize the key properties of the microlocalized spectral measure in the following theorem.

\begin{theorem}
\label{thm:microlocalized-spectral-measure-integral-form-low} 
Let $Q_j^{\calclow}, Q_{j'}^{\calclow}$ be two elements of the low-energy partition of unity as in Section~\ref{subsec:microlocal-partition-low}, 
then for $\lambda \leq 2$, and as a multiple of $|dg dg'|^{1/2} \otimes |d\lambda|$, the Schwartz kernel of $Q_j^{\calclow} \specm Q_{j'}^{\calclow}$ can be expressed as a sum of a finite number of terms of one of the following forms
\begin{gather}
\lambda^{n-1}  \int_{\R^k} e^{i\lambda\Phi(\sigma,y,y',v)/x}  \big(\frac{x'}{\lambda}\big)^{(n-1)/2 - k/2} \sigma^{ \frac{n-1}{2} } a(\lambda,\sigma,y,y',\frac{x'}{\lambda},v)dv  \; |dg dg'|^{1/2} \otimes |d(\frac{1}{h})| \label{QiEQj-lo}
 \end{gather}
or
 \begin{gather}
\begin{split}
\lambda^{n-1}  \int_{\R^{k}} \int_0^\infty e^{i\lambda\Phi(\sigma,y,y',v,s)/x} \big(\frac{x'}{\lambda s}\big)^{(n-1)/2 - (k+1)/2}
   s^{n-2} \sigma^{ \frac{n-1}{2} } a(\lambda,\sigma,y,y',\frac{x'}{\lambda s},v,s) \, dv \, ds \\ \times |dg dg'|^{1/2} \otimes |d(\frac{1}{h})|    
\end{split}  \label{QiEQj-s-lo}
  \end{gather}
  in the region $\sigma = x/x' \leq 2$, $x'/\lambda \leq 2$, or
  \begin{gather}
\quad  \lambda^{n-1} \sigma^{ \frac{n-1}{2} }  a(\lambda,\sigma,y,y',x'/\lambda) \; |dg dg'|^{1/2} \otimes |d(\frac{1}{h})| \label{QiEQj-c-low}
\end{gather}
in the region $\sigma = x/x' \leq 2$, $x'/\lambda \geq 1$. Here $\Phi(\sigma,y,y',v)$ parametrizes $L^{\#}$ (in this case $k=0$ and $v$ is absent) or $L^{\bfs}$ in the sense of Definition~\ref{def:Legendre-parametrization-low}, while $\Phi(\sigma,y,y',v,s)$ parametrizes the pair $(L^{\bfs},L^\#)$ in the sense of Section~\ref{subsec:conic-pair-geometry-and-phase-function}.

In each case, $a(\cdot)$ is a smooth function compactly supported in the $v$ and $s$ variables (where present), such that 
\begin{equation} \label{eq:est-a-osc-main}
    |(\lambda\partial_\lambda)^N a|\leq C_N
\end{equation}
for all $N \in \mathbb{N}$.
In each case, the number $k$ of extra parameters we are integrating over can be taken to be:
\begin{itemize}
\item When at least one of $j,j' \in \{ \zf, 1 \}$, then we can take $k=0$. 
\item For $j,j' \in \Jlow$, let $(j,j')$ be as in \eqref{eq:Lbf-jj'-contained-G-l}, we can take $k$ to be $k_\ell$, which is the maximal rank drop of the projection from $\mk{G}_{\ell} \cap L^{\bfs}$ to $\bfs$.
\end{itemize}

For $\sigma \geq 1/2$, the Schwartz kernel has a similar description by switching un-primed and primed variables and replacing $\sigma$ by $\sigma^{-1}$, as follows immediately from the symmetry of the kernel under interchanging the left and right variables.
\end{theorem}


\begin{proof}
 
The statement for the case where at least one of  $j,j'$ is $\zf$ or $1$ follows from Proposition~\ref{prop:localized-specm-two-side-residual-low} and Proposition~\ref{prop:localized-specm-one-side-residual-low}.

Now we consider the case $j,j' \in \Jlow$. 
As aforementioned, by the wavefront bound of compositions in \cite[Section~5]{GHS1},  $Q_j^\flat \specm Q_{j'}^\flat$ is a Legendre distribution associated with the Legendrian conic pair $(L_{j,j'}^{\bfs},L_{j,j'}^{\#})$ of the same order as in Theorem~\ref{thm:spectral-measure-complete}.
Here $L_{j,j'}^{\bfs}$ is defined in \eqref{eq:Lbf-low-jj'} and $L_{j,j'}^{\#}$ is defined in the same manner, which is the part of $L^{\#}$ that is localized by $Q_j^\flat$ and $Q_{j'}^\flat$ from left and right respectively:
\begin{equation} \label{eq:L-sharp-localized}
L_{j,j'}^{\#} = 
\{ (\sigma,y,y',\mu,\mu',\nu,\nu') \in L^{\#}| \; (y,\mu,\nu) \in  W_j, (y',\mu',\nu') \in  W_{j'}  \}.
\end{equation}
where $W_j,W_{j'}$ are the same open neighborhoods of $\WF'_{\flat}(Q_j^\flat), \WF'_{\flat}(Q_{j'}^\flat)$ respectively as in \eqref{eq:Lbf-low-jj'}.

Finally, the claim about $k$ follows from combining one of Proposition~\ref{prop:minimal-parametrization-low}, Proposition~\ref{prop:minimal-parametrization-low-conic}, Proposition~\ref{prop:minimal-parametrization-high}, Proposition~\ref{prop:minimal-parametrization-high-conic}, depending on which type of Legendre distribution it is, with the construction  of $Q_j^{\calclow}$ in Section~\ref{subsec:microlocal-partition-low}, in particular its property \eqref{eq:Lbf-jj'-contained-G-l}. 
For the part of $Q_j^{\calclow} \specm Q_{j'}^{\calclow}$ associated with the Legendrian conic pair $(L_{j,j'}^{\bfs}, \Lsharplow_{j,j'})$, the projection from $L_{j,j'}^{\#}$ to $\bfs$ is always a local diffeomorphism.
\end{proof}

Next we restate the conclusion above in a form that is more convenient to use in the proof of dispersive estimates, in which we absorb all decaying factors in $\BFS,\LB,\RB$ into the amplitude $a(\cdot)$. To this end, we first give a lemma bounding the product of boundary defining functions of $\BFS, \LB, \RB$.
\begin{lemma} \label{lemma:bdf-1}
Let $\rho_{\LB},\rho_{\BFS},\rho_{\RB}$ be boundary defining functions of $\LB,\BFS,\RB$ as in \eqref{eq:defining-functions-low-3}, then we have:
\begin{equation}
    \rho_{\BFS} \rho_{\LB} \rho_{\RB} \lesssim (1+\lambda \max\{|z|,|z'|\})^{-1}.
\end{equation}
\end{lemma}
\begin{proof}
    The inequality follows by considering different regions and observing that the right hand side vanishes to at most first order at $\BFS,\LB,\RB$.
\end{proof}

\begin{corollary} \label{coro:microlocalized-spectral-measure-osc-int-form-low}
With $Q_j^{\calclow},Q_{j'}^{\calclow},\Phi$ as above, for $\lambda \leq 2$, as a multiple of $|dg dg'|^{1/2} |d\lambda|$, the Schwartz kernel of $Q_j^{\calclow} \specm Q_{j'}^{\calclow}$ can be expressed as a sum of a finite number of terms of one of the following forms and a Schwartz term:
\begin{itemize}
    \item When $j,j' \in \Jlow$ are as in \eqref{eq:Lbf-jj'-contained-G-l} and $L^{\bfs}_{j,j'}$ defined as in \eqref{eq:Lbf-low-jj'} is away from $L^{\bfs} \cap \Lsharplow$, it takes the form
\begin{align} \label{QiEQj-lo-2}
\lambda^{n-1}  \int_{\R^k} e^{i\lambda\Phi(\sigma,y,y',v)/x} a(\lambda,\sigma,y,y',\frac{x'}{\lambda},v)dv \; |dg dg'|^{1/2} |d\lambda|
\end{align}
in the region $\sigma = x/x' \leq 2$, $x'/\lambda \leq 2$.
Here $k$ can be taken as any integer $k \geq k_\ell$ with $k_\ell$ being the maximal rank drop of the projection from $\mk{G}_{\ell} \cap L^{\bfs}$ to $\bfs$ and for all $\alpha \in \mathbb{N}$:
\begin{equation} \label{eq:spectral-measure-symbol-bound-low-1}
|\partial_{\lambda}^\alpha a(\lambda,z,z',v)| 
\leq C_{\alpha} \lambda^{-\alpha} (1+\lambda |z|)^{  - \frac{n-1-k}{2} } \sigma^{k/2}.
\end{equation}

\item  When $j,j' \in \Jlow$ and $L^{\bfs}_{j,j'}$ instead meets $L^{\bfs} \cap \Lsharplow$, it takes the form
\begin{align}\label{QiEQj-s-lo-2}
\lambda^{n-1}  \int_{\R^{k}} \int_0^\infty e^{i\lambda\Phi(\sigma,y,y',v,s)/x} s^{\frac{n+k}{2}-1 }a(\lambda,\sigma,y,y',\frac{x'}{\lambda s},v,s) \, dv \, ds  \;|dg dg'|^{1/2} |d\lambda| 
\end{align}
in the region $\sigma = x/x' \leq 2$, $x'/\lambda \leq 2$. Here $k$ can be taken as any integer $k \geq k_\ell$ and for all $\alpha \in \mathbb{N}$:
\begin{equation} \label{eq:spectral-measure-symbol-bound-low-2}
|\partial_{\lambda}^\alpha a(\lambda,z,z',v,s)| 
\leq C_{\alpha} \lambda^{-\alpha} (1+\lambda |z|)^{  - \frac{n-k-2}{2} } \sigma^{(k+1)/2}.
\end{equation}

\item When at least one of $j,j'$ is $1$ or $\zf$, it takes the form
  \begin{align} \label{QiEQj-c-low-absorbed-a}
\quad  \lambda^{n-1} a(\lambda,\sigma,y,y',x'/\lambda) \; |dg dg'|^{1/2} |d\lambda| 
  \end{align}
in the region $\sigma = x/x' \leq 2$, $x'/\lambda \geq 1$ and for all $\alpha \in \mathbb{N}$:
\begin{equation} \label{eq:spectral-measure-symbol-bound-low-3}
|\partial_{\lambda}^\alpha a(\lambda,z,z',v,s)| 
\leq C_{\alpha} \lambda^{-\alpha} (1+\lambda |z|)^{  - \frac{n-1}{2} }.
\end{equation}

\end{itemize}

As before, for $\sigma \geq 1/2$, the Schwartz kernel has a similar description by switching un-primed and primed variables and replacing $\sigma$ by $\theta=\sigma^{-1}$, as follows immediately from the symmetry of the kernel under interchanging the left and right variables.


\end{corollary}

\begin{proof}
We see from Theorem~\ref{thm:microlocalized-spectral-measure-integral-form-low} that, after absorbing all those boundary defining functions into the amplitude except for the $\lambda^{n-1}$-factor, we can write the microlocalized spectral measure as an oscillatory integral as above with (when $\sigma = x/x' \lesssim 1$):
\begin{equation} \label{eq:spectral-measure-symbol-bound-lowbdf}
|\partial_{\lambda}^\alpha a(\lambda,z,z';v)| 
\leq C_{\alpha} \lambda^{-\alpha} \big(\rho_{\BFS}\rho_{\LB} \rho_{\RB} \big)^{ \frac{n-1-k}{2} } \sigma^{k/2}
\end{equation}
Applying Lemma~\ref{lemma:bdf-1} and noticing that there is a `remaining' $\sigma^{k/2}$-factor, the desired estimate \eqref{eq:spectral-measure-symbol-bound-low-1} follows.
The estimate in \eqref{eq:spectral-measure-symbol-bound-low-2} and \eqref{eq:spectral-measure-symbol-bound-low-3} follows in the same way using the corresponding oscillatory integral expressions in Theorem~\ref{thm:microlocalized-spectral-measure-integral-form-low}.
\end{proof}



Now we turn to the high-energy case. We first consider the high-energy analogue of Proposition~\ref{prop:localized-specm-two-side-residual-low}.

\begin{proposition} \label{prop:localized-specm-two-side-residual-high}
Let $Q_{1}^{\calchigh}$ be as in Section~\ref{subsec:microlocal-partition-high-combined}, then
\begin{equation}
 Q_{1}^{\calchigh} \specm Q_{1}^{\calchigh}
\end{equation}
is a multiple of $|dgdg'|^{1/2} \otimes |\frac{dh}{h^2}|$ with coefficient that is rapidly decreasing at all of $\smf, \LB,\BFS,\RB \subset X_{\rmb,\calchigh}^2$.
More precisely, let $(h,\msf{q})$ be a coordinate system on $X_{\rmb,\calchigh}^2$ depending on the region, then it is of the form:
\begin{equation} \label{eq:local-specm-Q1EQ1-high}
         h^{-(n-1)}  a(h,\msf{q})  |dgdg'|^{1/2} \otimes |d(\frac{1}{h})| ,
\end{equation}
where $a(\msf{q},h)$ satisfies that for any $N,K \in \N$ there is a constant $C_{N,K}$ such that\footnote{Estimate \eqref{eq:Q1EQ1-symbol-est-high} is weaker than being smooth in $h$. We are stating in this way so that the estimate can be unified with the low-energy case. }
\begin{equation} \label{eq:Q1EQ1-symbol-est-high}
  |(h \partial_{h})^N a(h,\msf{q})| \leq C_{N,K}  h^K(1+ h^{-1} d_{\conic})^{-K}.
\end{equation}
\end{proposition}

The proof of Proposition~\ref{prop:localized-specm-two-side-residual-high} is almost the same as that of Proposition~\ref{prop:localized-specm-two-side-residual-low}, so we only briefly sketch it here.  
Using Lemma~\ref{lemma:high-energy-trivial-composition}, for any $K,N$ we have an oscillatory integral with amplitude that has sufficiently large decay order at all boundary surfaces. The oscillatory integral can just be integrated to give a function $a(h,\msf{q})$ satisfying \eqref{eq:Q1EQ1-symbol-est-high}.

Now we consider the case that only one of $j$ and $j'$ is $1$, which is the high-energy analogue of Proposition~\ref{prop:localized-specm-one-side-residual-low}.

\begin{proposition} \label{prop:localized-specm-one-side-residual-high}
Let $(h,\mathsf{q})$ be a coordinate system on $X_{\rmb,\calchigh}^2$, then:
\begin{enumerate}
\item \label{item:QjEQ1-osc-high}
$Q_j^{\calchigh} \specm Q_1^{\calchigh}$ can be expressed as a finite sum of terms of one of the following forms:
\begin{itemize}
  \item Terms supported in the region  $\theta = x'/x \lesssim 1$ and take the form
\begin{equation} \label{QjEQ1-high-rb}
h^{-(n-1)} \theta^{ \frac{n-1}{2} } a(h,\msf{q}) |dgdg'|^{1/2} \otimes |d(\frac{1}{h})|;
\end{equation}
\item Terms supported in the region $\sigma = x/x' \lesssim 1$ and take the form
  \begin{gather}
\quad  e^{i/hx}  h^{-(n-1)} \sigma^{ \frac{n-1}{2} } a(h,\mathsf{q}) |dgdg'|^{1/2} \otimes |d(\frac{1}{h})|. \label{QjEQ1-c-high}
\end{gather}
\end{itemize}
In each case, $a$ satisfies
\begin{equation} \label{eq:phg-conormal-2-high}
  |(h \partial_{h})^N a(h,\mathsf{q}) | \lesssim 1, 
\end{equation}
for any $N \in \N$.

\item \label{item:Q1EQj-osc-high} 
$Q_1^{\calchigh} \specm Q_j^{\calchigh}$ is a finite sum of terms of the same form as in Part~\eqref{item:QjEQ1-osc}, except with all primed and un-primed variables switched in all expressions. 
\end{enumerate}
\end{proposition}

The proof of Proposition~\ref{prop:localized-specm-one-side-residual-high} is essentially the same as that of Proposition~\ref{prop:localized-specm-one-side-residual-low}.
When we have $Q_1^{\calchigh}$ on the left (resp. right) side, then the phase is non-stationary in left (resp. right) variables and the oscillatory integral has arbitrary fixed order of decay (depending on how many times we have integrated by parts for the non-stationary phase argument) at $\BFS,\smf$ and $\LB$ (resp. $\RB$) and the oscillatory integral can be integrated to give a function satisfying \eqref{eq:phg-conormal-2-high}.


Now we consider the case in which both $j,j' \in \mk{J}_{\calchigh,1}  \cup \mk{J}_{\calchigh,2}$.
Then \cite[Lemma~7.1, Corollary~7.2]{GHS1} shows\footnote{In fact, a more refined decomposition of wavefront sets of Legendre distributions is given there, according to the face on which the Schwartz property fails. Here we are only using the full wavefront set, which is enough to characterize our parametrizations and microlocal partitions. }
that $Q_j^{\calchigh} \specm Q_{j'}^{\calchigh}$ is still a Legendre distribution as we defined in Definition~\ref{def:Legendrian-dis-conic-high},
but only associated with $(L^{\bfs,\calchigh}_{j,j'},\Lsharphigh)$ with $L^{\bfs,\calchigh}_{j,j'}$ defined in \eqref{eq:Lbf-high-jj'}, which is still a pair of Legendrian submanifolds with conic points.

Combining the discussion above with Proposition~\ref{prop:localized-specm-two-side-residual-high} and Proposition~\ref{prop:localized-specm-one-side-residual-high}, we have the following high-energy analogue of Theorem~\ref{thm:microlocalized-spectral-measure-integral-form-low}.

\begin{theorem}
\label{thm:microlocalized-spectral-measure-integral-form-high} 
Let $Q_j^{\calchigh}, Q_{j'}^{\calchigh}$ be two elements of the high-energy partition of unity as in Section~\ref{subsec:microlocal-partition-high-combined} and $(h,\msf{q})$ be a coordinate system of $X_{\rmb,\calchigh}^2$, then for $h \leq 2$, and as a multiple of $|dg dg'|^{1/2} \otimes |d(\frac{1}{h})|$, the Schwartz kernel of $Q_j^{\calchigh} \specm Q_{j'}^{\calchigh}$ can be expressed as a sum of finitely many terms of one of the following forms
\begin{align} \label{eq:thm4.3-1}
& h^{ - \frac{n-1}{2} - \frac{k}{2} } \int_{\R^k} e^{i\Phi(\msf{q},v)/hx}  (x')^{\frac{n-1-k}{2}} \sigma^{ \frac{n-1}{2} } a(h,\msf{q},v)dv \; |dg dg'|^{1/2} \otimes |d(\frac{1}{h})| 
  \end{align}
or
  \begin{align}
\begin{split} \label{eq:thm4.3-2}
h^{ - \frac{n-1}{2} - \frac{k+1}{2} }  \int_{\R^{k}} \int_0^\infty e^{i \Phi(\msf{q},v,s)/hx} \big(\frac{x'}{s}\big)^{(n-1)/2 - (k+1)/2}
\\ s^{n-2} \sigma^{ \frac{n-1}{2} } a(h,\msf{q},v,s) \, dv \, ds \; |dg dg'|^{1/2} \otimes |d(\frac{1}{h})|   
\end{split}
  \end{align}
in the region $\sigma = x/x' \leq 2$.
Here $\Phi(\msf{q},v)$ parametrizes $\Lsharphigh$ (in this case $k=0$ and $v$ is absent) or $L^{\bfs,\calchigh}$ in the sense of \eqref{eq:def-parametrization-high}, 
while $\Phi(\msf{q},v,s)$ parametrizes the pair $(L^{\bfs,\calchigh},\Lsharphigh)$ in the sense of \eqref{eq:def-C-Phi-high-energy-conic}.

In each case, $a(\cdot)$ is a smooth function compactly supported in all of its variables, such that 
\begin{equation} \label{eq:est-a-osc-main-high}
    |(h\partial_h)^N a|\leq C_N
\end{equation}
for all $N \in \mathbb{N}$.
In each case, the number $k$ of extra parameters we are integrating over can be taken to be:
\begin{itemize}
\item When at least one of $j,j'$ is $1$, we can take $k=0$. 
\item For $j,j' \in \Jhigh$, let $(j,j')$ be as in \eqref{eq:Lbf-jj'-contained-G-l}, we can take $k$ to be any integer $k \geq k_\ell$,  
where $k_\ell$ is the maximal rank drop of the projection from $\mk{G}_{\ell} \cap L^{\bfs,\calchigh}$ to $\smf$.
\end{itemize}

For $\sigma \geq 1/2$, the Schwartz kernel has a similar description by switching un-primed and primed variables and replacing $\sigma$ by $\theta=\sigma^{-1}$, as follows immediately from the symmetry of the kernel under interchanging the left and right variables.
\end{theorem}


\begin{corollary} \label{coro:microlocalized-spectral-measure-osc-int-form-high}
Let $Q_j^{\calchigh},Q_{j'}^{\calchigh},\Phi, (h=\lambda^{-1},\msf{q})$ as in Theorem~\ref{thm:microlocalized-spectral-measure-integral-form-high}. For $h \leq 2$, as a multiple of $|dg dg'|^{1/2} \otimes |d\lambda|$, the Schwartz kernel of $Q_j^{\calchigh} \specm Q_{j'}^{\calchigh}$ can be expressed as a sum of a finite number of terms of one of the following forms, together with a Schwartz term:
\begin{itemize}
    \item When $j,j' \in \Jhigh$ are as in \eqref{eq:Lbf-high-jj'-contained-1} and $L^{\bfs,\calchigh}_{j,j'}$ defined as in \eqref{eq:Lbf-high-jj'} is away from $L^{\bfs,\calchigh} \cap \Lsharphigh \cap \{x'/s = 0\}$, it takes the form
\begin{align} \label{QiEQj-high-2}
\lambda^{n-1}  \int_{\R^k} e^{i\lambda\Phi(\msf{q},v)/x} a(\lambda^{-1},\msf{q},v) dv \; |dg dg'|^{1/2} \otimes |d\lambda|
\end{align}
in the region $\sigma = x/x' \leq 2$.
Here $k$ can be taken as any integer $k \geq k_\ell$ with $k_\ell$ being the maximal rank drop of the projection from $\tilde{\mk{G}}_{\ell} \cap L^{\bfs,\calchigh}$ to $\bfs$ and for all $\alpha \in \mathbb{N}$:
\begin{equation} \label{eq:spectral-measure-symbol-bound-high-1}
|\partial_{\lambda}^\alpha a(\lambda^{-1},z,z',v)| 
\leq C_{\alpha} \lambda^{-\alpha} (1+\lambda |z| )^{  - \frac{n-1-k}{2} } \sigma^{k/2}.
\end{equation}

\item  When $j,j' \in \Jhigh$ and $L^{\bfs,\calchigh}_{j,j'}$ instead meet $L^{\bfs,\calchigh} \cap \Lsharphigh \cap \{x'/s = 0\}$, it takes the form
\begin{align}  \label{eq:coro4.2-2}
\lambda^{n-1}  \int_0^\infty  \int_{\R^{k}} e^{i\lambda\Phi(\msf{q},v,s)/x} s^{\frac{n+k}{2}-1}
 a(\lambda^{-1},\sigma,\frac{x'}{s},y,y',v,s) \, dv \, ds \; |dg dg'|^{1/2} \otimes |d\lambda|
\end{align}
in the region $\sigma = x/x' \leq 2$. Here $k$ can be taken as any integer $k+1 \geq k_\ell$ with $k_\ell$ being the maximal rank drop of the projection from $\tilde{\mk{G}}_{\ell} \cap L^{\bfs,\calchigh}$ to $\smf = X_{\rmb}^2$ and for all $\alpha \in \mathbb{N}$:
\begin{equation} \label{eq:spectral-measure-symbol-bound-high-2}
|\partial_{\lambda}^\alpha a(\lambda^{-1},z,z',v,s)| 
\leq C_{\alpha} \lambda^{-\alpha} (1+\lambda |z|)^{  - \frac{n-k-2}{2} } \sigma^{\frac{k+1}{2}}.
\end{equation}

\item When at least one of $j,j'$ is $1$, it takes the form
  \begin{align} \label{QiEQj-c-high-absorbed-a}
\quad  \lambda^{n-1} a(\lambda^{-1},\sigma,y,y',x'/\lambda) \; |dg dg'|^{1/2} |d\lambda| 
  \end{align}
in the region $\sigma = x/x' \leq 2$, $x'/\lambda \geq 1$ and for all $\alpha \in \mathbb{N}$:
\begin{equation} \label{eq:spectral-measure-symbol-bound-high-3}
|\partial_{\lambda}^\alpha a(\lambda^{-1},z,z',v,s)| 
\leq C_{\alpha} \lambda^{-\alpha} (1+\lambda |z|)^{  - \frac{n-1}{2} }.
\end{equation}
\end{itemize}


As before, for $\sigma \geq 1/2$, the Schwartz kernel has a similar description by switching un-primed and primed variables and replacing $\sigma$ by $\theta=\sigma^{-1}$, as follows immediately from the symmetry of the kernel under interchanging the left and right variables.


\end{corollary}

\begin{remark}
The smallest allowed $k$ here, compared with $k_\ell$ in \eqref{eq:coro4.2-2}, is shifted by $1$ compared with the low-energy case. This is because of the same reason as the discussion after Proposition~\ref{prop:minimal-parametrization-high-conic}: $k_\ell$ has increased by $1$ since we are taking the degeneracy in the $x'$-direction into account now, which will always be degenerate at $x'/s = 0$.
\end{remark}

\begin{proof}
Consider \eqref{QiEQj-high-2} first.
We see from Theorem~\ref{thm:microlocalized-spectral-measure-integral-form-high}, after absorbing all those boundary defining functions into the amplitude except for the $\lambda^{n-1}=h^{-(n-1)}$ factor, that we can write the microlocalized spectral measure as an oscillatory integral as above with
\begin{equation} \label{eq:spectral-measure-symbol-bound-lowbdf-high}
|\partial_{\lambda}^\alpha a(\lambda,z,z';v)| 
\leq C_{\alpha}  h^{\alpha+\frac{n-1-k}{2} }  \rho_{\BFS}^{ \frac{n-1-k}{2} } \rho_{\LB}^{\frac{n-1-k}{2}} \rho_{\RB}^{\frac{n-1-k}{2}},
\end{equation}
where $h^{\frac{n-1-k}{2}}$ comes from writing $h^{-\frac{n-1-k}{2}}$ as $h^{-(n-1)} \times h^{ \frac{n-1-k}{2} }$.
Noticing that when $\lambda \gtrsim 1$, a product of boundary defining functions for $\BFS$, $\LB$ and $\RB$ is $O((1 + d(z,z'))^{-1})$, we have
\begin{align}
\begin{split}
 h^{ \frac{n-1-k}{2} } \rho_{\BFS}^{ \frac{n-1-k}{2} } \rho_{\LB}^{\frac{n-1-k}{2}} \rho_{\RB}^{\frac{n-1-k}{2}} \lesssim & \lambda^{-\frac{n-1-k}{2}}  (1+d(z,z'))^{\frac{n-1-k}{2}} 
\\ \lesssim & \lambda^{n-1} (1+\lambda d(z,z'))^{-\frac{n-1-k}{2}},
\end{split}
\end{align}
which finishes the proof of this part. 
Conclusions concerning \eqref{eq:coro4.2-2} and \eqref{QiEQj-c-high-absorbed-a} can be shown in the same way, using the corresponding expressions in Theorem~\ref{thm:microlocalized-spectral-measure-integral-form-high}.
\end{proof}


We summarize below the relationship between the minimal number of extra parameters in oscillatory integrals above and $\IF,\IFz,\IFint,\IFintz$ defined in Section~\ref{subsec:focusing-effect}.
\begin{proposition}  \label{prop:IF-relation-parameter-number}
    The maximal number of extra parameters needed for minimal parametrizations for oscillatory integrals representing the microlocalized $\specm$ for $(X,g)$ not associated with the conic intersecting pairs is $\IFint$.
    The maximal number of extra parameters other than $s$ needed for minimal parametrizations including the part of the conic intersection pairs is $\IF$.
    If we were considering an exact cone $(X_0,g_0)$ instead, then the same conclusion holds with $\IF,\IFint$ replaced by $\IFz,\IFintz$ respectively.
\end{proposition}

\begin{proof}
    This follows from combining statements about $k$, the number of extra parameters other than $s$, in Proposition~\ref{prop:minimal-parametrization-low}, Proposition~\ref{prop:minimal-parametrization-low-conic}, Proposition~\ref{prop:minimal-parametrization-high}, and Proposition~\ref{prop:minimal-parametrization-high-conic}.
\end{proof}

\section{Analysis near the intersecting conic pair}
\label{sec:projection-and-phase}

In this section, we analyze the phase function parametrizing an intersecting pair of Legendre submanifolds with conic points.
The goal is to demonstrate that conjugate point pairs at distance $\pi$ on $Y$ won't affect dispersive estimates. In particular, this enables us to include classical results on the Euclidean space and perturbations of it as special cases of our result.

The content of this section has two parts. The first part is the relationship between the degeneracy of $d\hatLprojlow$ (or $d\hatLprojhigh$) and the degeneracy of the Hessian (in extra parameters) of our phase function at its critical points.
Roughly speaking, the degeneracy of the projection is `equivalent' to the degeneracy of this Hessian.
In the setting of classical Fourier integral operators, a qualitative version was given in \cite[Theorem~3.1.4]{FIO1}.
We will give a quantitative version that quantifies this degeneracy when we are approaching the place where $d\hatLprojlow$ (or $d\hatLprojhigh$) degenerates.
The second part is the function value of the phase function itself at its critical points.
As shown in \cite[Proposition~2.6]{Hassell-Zhang2016Strichartz}, near the diagonal, this function equals the distance on the cone over $Y$ and we will prove a generalization of this, removing the restriction to be near the diagonal, taking topological obstructions (i.e., geodesic loops) into consideration. 






\subsection{Preliminaries}
\label{subsec:Leg-concrete-setup}

To begin with, we introduce the local geometric set up near $L^{\bfs} \cap \Lsharplow$ in the low-energy setting and near $L^{\bfs,\calchigh} \cap \Lsharphigh$ in the high-energy setting.
The purpose is to separate the effect of geometric focusing happening away from $L^{\bfs} \cap \Lsharplow$ (resp. $L^{\bfs,\calchigh} \cap \Lsharphigh$) and investigate the effect of focusing happening at $L^{\bfs} \cap \Lsharplow$ (resp. $L^{\bfs,\calchigh} \cap \Lsharphigh$) solely.
In the low-energy setting, in terms of the geodesic flow on $Y$, this corresponds to separating the part strictly within time $\pi$ from the part at time $\pi$.


Consider the low-energy case first.
Let $\beta_{\mathrm{LCP},\calclow}$ be the blow down map as in \eqref{eq:beta-LCP-low-1}.  Let $\msf{q}_0 \in \beta_{\mathrm{LCP},\calclow}^{-1}(L^{\bfs} \cap \Lsharplow)$ and suppose that the rank of
\begin{equation} \label{eq:def-Lprojlow}
 \hatLprojlow:\; \hat{L}^{\bfs}  \to \bfs
\end{equation}
is $2n-1-k$ at $\msf{q}_0$, which means its rank drop is $k$ compared with the maximal rank.
Then we consider a neighborhood $\ULCPlow$ of $\msf{q}_0$ on which the rank drop of the projection $\hatLprojlow$ is at most $k$. In addition, we assume that on $\beta_{\mathrm{LCP},\calclow}^{-1}(\ULCPlow)$, the degeneracy of $\hatLprojlow$ only happens at the lift of $L^{\bfs} \cap \Lsharplow$.
We should emphasize that this is not an extra assumption since the situation in which the degeneracy of $\hatLprojlow$ outside the lift of $L^{\bfs} \cap \Lsharplow$ is present is covered by Section~\ref{sec:microlocalized-spectral-measure} and Section~\ref{subsec:est-Schrodinger-kernel}. 

Similarly, let $\beta_{\mathrm{LCP},\calchigh}$ be the blow-down map in \eqref{eq:def-high-blow-down} and suppose that the rank of the projection $\hatLprojhigh$ defined in \eqref{eq:def-Lprojhigh}
is $2n - k$ at $\msf{q}_0 \in \beta_{\mathrm{LCP},\calchigh}^{-1}(L^{\bfs,\calchigh} \cap \Lsharphigh)$.
Then we consider a neighborhood $\ULCPhigh$ of $\msf{q}_0$ on which the rank drop of the projection $\hatLprojhigh$ is at most $k$ and the degeneracy of $\hatLprojhigh$ only happens on $\beta_{\mathrm{LCP},\calchigh}^{-1}(L^{\bfs,\calchigh} \cap \Lsharphigh)$.

Next we recall the construction of parametrizations of $(L^{\bfs},L^{\#})$ defined for an asymptotically conic manifold $(X,g)$ or an exact cone $(X_0,g_0)$. See \cite[Proposition~3.5]{hassell1999spectral} for more details, which in turn used the construction in \cite[Proposition~6]{melrose1996scattering}.  
Let $\msf{q}_0 \in \beta_{\mathrm{LCP},\calclow}^{-1}(L^{\bfs} \cap \Lsharplow)$ and $\ULCPlow$ be as above. 
Without loss of generality, after a change of coordinates, we may assume that $s=\mu'_{n-1}$ is the dominating component and $\hat{\mu}'_j = \mu'_j/\mu'_{n-1} = 0$ for $j=1,...,n-2$.
There is a splitting of coordinates $y'=(y'_{I},y'_{II})$ and correspondingly $\hat{\mu}' = (\hat{\mu}'_{I},\hat{\mu}'_{II})$ such that 
\begin{equation} \label{eq:splitted-coordinate-1}
 \mathcal{Z}_{\calclow} =  (\sigma,y,y_{II}',\hat{\mu}'_I = \mu'_I/s,s)
\end{equation}
 is a coordinate system on $\hat{L}^{\bfs}$.
So in terms of coordinates in \eqref{eq:coordinates-resolved-low-1} with $|\mu'|$ replaced by $\mu'_{n-1}$, $\hat{L}^{\bfs}$ can locally be written as
\begin{align} \label{eq:hatLbf-low-components}
\begin{split}
\hat{L}^{\bfs} = & \{  \overline{\nu} = N(\mathcal{Z}_{\calclow}) = 1 + \sigma + \sigma \mu'_{n-1}N_2(\mathcal{Z}_{\calclow}), 
\hat{\nu}_1 = \hat{N}_1(\mathcal{Z}_{\calclow}),
\\  & \; y'_I = Y'_I(\mathcal{Z}_{\calclow}), 
 \;  y'_{n-1} = Y'_{n-1}(\mathcal{Z}_{\calclow}),\; \hat{\mu'}_{II} = \hat{M'}_{II}(\mathcal{Z}_{\calclow}) \},
\end{split}
\end{align}
where $N,N_2,\hat{N}_1,Y'_I,Y'_{n-1},\hat{M'}_{II}$ are smooth functions of $\mathcal{Z}_{\calclow}$. Here the particular form of $N(\mathcal{Z}_{\calclow})$ follows by observing that, on $\hat{L}^{\bfs}$, $\overline{\nu} = 1$ whenever $x'=0$ and $\overline{\nu}=1+\sigma$ whenever $s=0$.

So with $(v = \hat{\mu'}_I,s= \mu'_{n-1})$ being extra parameters, a parametrization of $\hat{L}^{\bfs}$ as in \eqref{eq:hat-Lbf-parametrization-bf-2} is given by
\footnote{
When we verify that this gives a parametrization, we need to use the fact that $\hat{L}^{\bfs}$ is Legendrian, which leads to
\begin{equation} \label{eq:1-form-vanish-low}
  dN + N_1 d\sigma + s \hat{M} \cdot dy + s\hat{\mu}'_I \cdot dY_I' + s\hat{M}'_2 \cdot dy'_{II} = 0,
\end{equation}
where $N_1 = \frac{N-s\hat{N}_1}{1+\sigma}$ is how $\nu_1$ is determined by $\hat{\nu}_1,\overline{\nu}$ and $s$. This is also the reason that $\hat{N}_1,\hat{M},\hat{M}'_2$ are not involved explicitly in $\Phi$, since they will be given by derivatives of other parts.}
\begin{equation} \label{eq:conic-pair-phase-construction-low}
\Phi = N + s \sigma \Big( (y_I'-Y_I') \cdot \hat{\mu'}_I + (y_{n-1}'-Y_{n-1}') \Big).
\end{equation}

We only need finitely many local expressions of the spectral measure as oscillatory integrals and local parametrizations, and for each such local expression, there is a positive distance (in terms of Euclidean distance in $\hat{\mu'}_I$) between the boundary of the amplitude $a(\bullet)$ and the boundary of the region where the parametrization is valid. 
We denote the minimum of such distances by  $\delta_0$. So without loss of generality, throughout this and the next section, whenever we write a local expression of a Legendre distribution (associated with Legendre conic pairs), we know that the region that $\hat{L}^{\bfs}$ (or $\hat{L}^{\bfs,\calchigh}$ discussed below) contains at least a $\delta_0$-neighborhood in $\hat{\mu'}_I$ of the support of the amplitude in oscillatory integrals.

Now we discuss the phase function in the high-energy case. Again, we only consider the case near $\LB \cap \BFS$ and the case near $\RB \cap \BFS$ can be dealt with in the same way.
We still consider the part where $s = \mu'_{n-1}$ dominates other components of $\mu'$.
Also, we choose the boundary defining function $x$ on $X$ so that $\Lsharphigh$ is defined by $\{\nu_1=1,\overline{\nu}=1+\sigma,x'=0,\mu=\mu'=0\}$ as in \eqref{eq:concrete-Lsharphigh}.

Consider the part near the lift of $\{ x' = 0 \}$ (hence $x'/s \lesssim 1$) first and we use coordinates in \eqref{eq:coordinates-resolved-high-region1}.
We only discuss the `codimension three' case below (i.e., the part near $h=0,\sigma=0,x'=0$) since the `codimension two' case is simpler and in fact can be dealt with by taking $\sigma \sim 1$ in the proof below.
Phase functions parametrizing this part are constructed as follows.
In the same way as in \eqref{eq:splitted-coordinate-1}, near a point $q \in \hat{L}^{\bfs,\calchigh}$ that is near the face $\{\frac{x'}{s} = 0\}$, we can select the coordinate system to be (with $s=\mu'_{n-1}$)
\begin{equation} \label{eq:splitted-coordinate-high}
  \mathcal{Z}_{\calchigh} = (\sigma,y,y_{II}',\hat{\mu}'_I,s,\varrho = \frac{x'}{s}).
\end{equation}
In the same way as \eqref{eq:hatLbf-low-components},  $\hat{L}^{\bfs,\calchigh}$ can be written as:
\begin{align} \label{eq:hatLbf-high-components-region1}
\begin{split}
\hat{L}^{\bfs,\calchigh} = 
& \{  \overline{\nu} = N(\mathcal{Z}_{\calchigh}) = 1 + \sigma + \sigma \mu'_{n-1}N_2(\mathcal{Z}_{\calchigh}), 
\hat{\nu}_1 = \hat{N}_1(\mathcal{Z}_{\calchigh}),
\\  & \; y'_I = Y'_I(\mathcal{Z}_{\calchigh}), 
 \;  y'_{n-1} = Y'_{n-1}(\mathcal{Z}_{\calchigh}),  \hat{\mu'}_{II} = \hat{M'}_{II}(\mathcal{Z}_{\calchigh}) \},
\end{split}
\end{align}
where $N,N_2,\hat{N}_1,Y'_I,Y'_{n-1},\hat{M},\hat{M'}_{II}$ are smooth functions. Again, the particular form of $N(\mathcal{Z}_{\calchigh})$ follows by observing that due to the form of $\Lsharphigh$ in \eqref{eq:concrete-Lsharphigh}, and since the boundary of $\hat{L}^{\bfs,\calchigh}$ is where it meets $\Lsharphigh$, we know $\overline{\nu} = 1$ whenever $x'=0$ and $\overline{\nu}=1+\sigma$ whenever $s=0$ on $\hat{L}^{\bfs,\calchigh}$.

Then, similar to \eqref{eq:conic-pair-phase-construction-low}, the phase function parametrizing it can be chosen to be 
\begin{align} \label{eq:conic-pair-phase-construction-high-region1}
\Phi = N + s \sigma \Big( (y_I'-Y_I') \cdot \hat{\mu'}_I + (y_{n-1}'-Y_{n-1}') \Big),
\end{align}
where the difference with \eqref{eq:conic-pair-phase-construction-low} is that $N,Y_I',Y_{n-1}'$ are functions of $\mathcal{Z}_{\calchigh}$ now.
See discussions after \cite[Equation~(6.20)]{Hassell-Wunsch-semiclassical-resolvent} for the detailed verification that this parametrizes $\hat{L}^{\bfs,\calchigh}$.

In the high-energy case, we need to take into account the region on $\hat{L}^{\bfs,\calchigh}$ on which $s/x' \lesssim 1$ corresponding to those geodesics that travel through the interior of $X$.
This part of $\hat{L}^{\bfs,\calchigh}$ is parametrized by phase functions as in \cite[Equation~(6.9)(6.15)]{Hassell-Wunsch-semiclassical-resolvent}, and contributes to the spectral measure terms like $u_4$ in \eqref{eq:conic-pair-u4-high}.
In this region, we instead use coordinates given by
\begin{equation}
 \tilde{\mathcal{Z}}_{\calchigh} = (\sigma,x',y,y_{II}',\hat{\mu}'_I,\varkappa = \frac{s}{x'}).
\end{equation}
Also, $\hat{L}^{\bfs,\calchigh}$ continues to have the form \eqref{eq:hatLbf-high-components-region1} except that we now use $\tilde{\mathcal{Z}}_{\calchigh}$ as coordinates:
\begin{align} \label{eq:hatLbf-high-components-region2}
\begin{split}
\hat{L}^{\bfs,\calchigh} = 
& \{  \overline{\nu} = N(\tilde{\mathcal{Z}}_{\calchigh}) = 1 + \sigma + \sigma \mu'_{n-1}N_2(\tilde{\mathcal{Z}}_{\calchigh}), 
\hat{\nu}_1 = \hat{N}_1(\tilde{\mathcal{Z}}_{\calchigh}),
\\  & \; y'_I = Y'_I(\tilde{\mathcal{Z}}_{\calchigh}), 
 \;  y'_{n-1} = Y'_{n-1}(\tilde{\mathcal{Z}}_{\calchigh}),  \hat{\mu'}_{II} = \hat{M'}_{II}(\tilde{\mathcal{Z}}_{\calchigh}) \}
\end{split}
\end{align}
In addition, \eqref{eq:conic-pair-phase-construction-high-region1} continues to give a parametrization after rewriting it as
\begin{align} \label{eq:conic-pair-phase-construction-high-region2}
\Phi = N + x' \sigma \Big( (y_I'-Y_I') \cdot \hat{\mu'}_I + \varkappa (y_{n-1}'-Y_{n-1}') \Big).
\end{align}
Notice that in this case, $\hat{\mu}'_I = \mu'_I/x'$, so this phase function coincides with \eqref{eq:conic-pair-phase-construction-high-region1} in the overlapped region and now we have $(\hat{\mu'}_I,\varkappa)$ as extra parameters.

The construction of phase functions above not only demonstrates the existence of parametrizations, but also has the advantage that their derivatives in extra parameters are explicit, which we summarize as follows.
\begin{proposition} \label{prop:LCP-phase-derivatives}
For $\Phi$ in \eqref{eq:conic-pair-phase-construction-low} or \eqref{eq:conic-pair-phase-construction-high-region1}, we have
\begin{equation}
 \partial_{\hat{\mu'}_I}\Phi = s\sigma(y_I' - Y_I'),
 \; \partial_s \Phi = \sigma (y'_{n-1} - Y'_{n-1}).
\end{equation}
Equivalently, if we write $\Phi = 1+\sigma + s\sigma \psi$, then
\begin{equation} \label{eq:v-critical-fixed-s}
  \partial_{\hat{\mu'}_I}\psi = y_I' - Y_I',
  \; \partial_s(s\psi) = y'_{n-1}-Y'_{n-1}. 
\end{equation}
For $\Phi = 1 + \sigma + x'\sigma \psi$ in \eqref{eq:conic-pair-phase-construction-high-region2}, we have
\begin{equation} \label{eq:5.1-2} 
\partial_{\hat{\mu'}_I}\Phi = \sigma x' (y'_I - Y'_I), \; \partial_{\varkappa}\Phi = \sigma x' (y'_{n-1}-Y'_{n-1}).
\end{equation}
Equivalently,
\begin{equation}
 \partial_{\hat{\mu'}_I}\psi = (y'_I - Y'_I), \; \partial_{\varkappa}\psi = (y'_{n-1}-Y'_{n-1}).
\end{equation}

\end{proposition}

\begin{proof}
     Just as in the verification that those phase functions are valid parametrizations, we exploit \eqref{eq:1-form-vanish-low} expanded in local coordinates in the computation of those partial derivatives. Then all terms with derivatives falling on $Y'_I$ or $Y'_{n-1}$ sum to be $0$ and the conclusion follows.
\end{proof}


\subsection{The Hessian of phase functions}
\label{subsec:phase-Hessian}

We analyze the Hessian of the phase function parametrizing the Legendrian conic pair in this subsection.
We first introduce the admissible condition on our Legendrian conic pairs that allows us to eliminate the effect of focusing at $L^{\bfs,\calchigh} \cap \Lsharphigh$. 
One can still have corresponding dispersive estimates without those assumptions with a loss that is analogous to the case with the presence of conjugate points before reaching this intersecting conic pair. The goal of introducing this concept is to give a conceptual explanation of why the degeneracy of $\Lprojlow$ or $\Lprojhigh$ at the intersecting conic pair does not cause loss in the Euclidean case. Roughly speaking, it is because the degeneracy happens in a uniform manner and this is a natural condition under which the lossless dispersive estimate continues to hold, generalizing the dispersive estimate on Euclidean spaces.


\begin{definition} \label{definition:Lbf-boundary-admissible}
We say $(\hat{L}^{\bfs,\calchigh},\Lsharphigh)$ is admissible (we will also say $X$ is admissible since $X$ determines $(\hat{L}^{\bfs,\calchigh},\Lsharphigh)$)
if $Y'_I,Y'_{n-1}$ in different regions satisfy conditions below.
\begin{itemize}
\item For $Y'_I,\, Y'_{n-1}$ as in \eqref{eq:hatLbf-high-components-region1}, we have
\begin{align} \label{eq:defn-condition-YI-Yn-high-region1}
\partial_{(\hat{\mu'}_I,\varrho)}(Y'_I,Y'_{n-1})|_{s=0} = 0.
\end{align}

\item  For $Y'_I,\, Y'_{n-1}$ as in \eqref{eq:hatLbf-high-components-region2}, we have
\begin{align} \label{eq:defn-condition-YI-Yn-high-region2}
\partial_{(\hat{\mu'}_I,\varkappa)}(Y'_I,Y'_{n-1})|_{x'=0} = 0.
\end{align}
\end{itemize}
Let $Y'_I,Y'_{n-1}$ be as in \eqref{eq:hatLbf-low-components} instead. Then we say that $(L^{\bfs},\Lsharplow)$ is admissible if 
\footnote{For the asymptotically conic case, this is not an extra assumption. It can be implied from the admissible condition on $(\hat{L}^{\bfs,\calchigh},\Lsharphigh)$ by restricting to $\varrho = 0$. However, for the exact cone case, since we haven't defined its $L^{\bfs,\calchigh}$, this is an independent definition. 
}
\begin{align} \label{eq:defn-condition-YI-Yn-1-s=0}
\partial_{\hat{\mu'}_I}(Y'_I,Y'_{n-1})|_{s=0} = 0.
\end{align}
One can verify that this definition is independent of the choice of coordinates.
\end{definition}

\begin{remark}
This condition can be relaxed to certain convexity of the boundary of $\hat{L}^{\bfs,\calchigh}$ and a uniform control on higher order derivatives of $Y'_{n-1},Y'_{I}$. See Appendix~\ref{appendix:relax-condition}.
In addition, we expect this condition to be sharp for the contribution from $L^{\bfs,\calchigh} \cap \Lsharphigh$ to be harmless to the dispersive estimate. See Remark~\ref{remark: admissible-convex-version} for more discussions. 

Since the restriction to $s=0$ corresponds to endpoints of bicharacteristic lines, this condition can also be viewed as a condition on the classical scattering map associated with $g$. 
How the long time behaviour of solutions is carried by such an object is addressed in \cite{HJ2026} in the setting of time-dependent Schr\"odinger equations.

In addition, by the definition of $Y_I,Y'_{n-1}$, we know \eqref{eq:defn-condition-YI-Yn-high-region1} holds at the point that reaches the local maximal rank drop. So our condition says that the degeneracy of the projection $\hat{L}^{\bfs,\calchigh} \to X_{\rmb}^2$ is uniform at the conic intersecting pair.
One can see from the treatment of the degenerate stationary phase in Section~\ref{sec:conic-points-pointwise-bound}, Section~\ref{subsec:dispersive-conic-pair-contribution} and Section~\ref{subsec:dispersive-conic-pair-contribution-wave}, the same lossless microlocalized dispersive estimate is expected to hold with the condition $\partial_{\hat{\mu'}_I}Y'_I|_{s=0} = 0$ relaxed to requiring $\partial_{\hat{\mu'}_I}Y'_I|_{s=0} \leq 0$, which is the sign that is in favour of the non-degeneracy of the Hessian (including $s$-derivatives) of the phase function.
We are making this choice $\partial_{\hat{\mu'}_I}Y'_I|_{s=0} = 0$ here to avoid the extra complications.
\end{remark}

Unless otherwise stated, for the rest of this section and in Section~\ref{sec:conic-points-pointwise-bound}, $(\hat{L}^{\bfs,\calchigh},\Lsharphigh)$ is assumed to be admissible in the sense above.


Fix $\msf{q}_0 = (\sigma_0,y_0,y'_{II,0},\hat{\mu'}_{I,0}, s_0) \in \hat{L}^{\bfs}$ (or $\msf{q}_0 \in \hat{L}^{\bfs,\calchigh}$) around which we are defining the parametrization. 
To facilitate discussion of properties of $Y'_I,Y'_{n-1}$, without loss of generality, we may assume that our coordinates are chosen so that $\partial_{y'_i}$ is an orthonormal basis in terms of the metric $\Ymetric$ at $y_0'=(Y'_{I,0}(\sigma_0,y_0,y'_{II,0}),y'_{II,0},Y'_{n-1,0}(\sigma_0,y_0,y'_{II,0}))$.
Equivalently:
\begin{equation} \label{eq:Ymetric-diagonalize}
  \Ymetric^{ij} = \delta^{ij} \; \text{ at } \; y_0'.
\end{equation}
In addition, we can further apply an orthogonal transformation, which does not affect \eqref{eq:Ymetric-diagonalize}, so that our momentum is $\mu'/|\mu'| = (0,...,1)$ at $\msf{q}_0$.
Similarly, we diagonalize $\Ymetric$ at a fixed point in the high-energy setting (i.e., when $Y'_I,Y'_{n-1}$ are as in \eqref{eq:hatLbf-high-components-region1} or \eqref{eq:hatLbf-high-components-region2}) as well.

Consider $Y'_I,\, Y'_{n-1}$ as in \eqref{eq:hatLbf-low-components} (i.e., in the low-energy setting) first.
When this admissible condition is satisfied, $Y_I'|_{s=0},Y'_{n-1}|_{s=0}$ are only functions of  $(\sigma,y,y_{II}')$ and we can write (with $s=\mu_{n-1}'$):
\begin{align} \label{eq:Y'-I-expansion}
  Y'_{I}(\sigma,y,y_{II}',\hat{\mu'}_I,s) = Y'_{I,0}(\sigma,y,y_{II}') + s\tilde{Y'}_I(\sigma,y,y_{II}',\hat{\mu'}_I,s),
\end{align}
and
\begin{align} \label{eq:Y'-n-1-expansion}
  Y'_{n-1}(\sigma,y,y_{II}',\hat{\mu'}_I,s) = Y'_{n-1,0}(\sigma,y,y_{II}')+s\tilde{Y'}_{n-1}(\sigma,y,y_{II}',\hat{\mu'}_I,s).
\end{align}




Now we consider the expansion of $Y'_I,\, Y'_{n-1}$ in \eqref{eq:hatLbf-high-components-region1}.
When the admissible condition in Definition~\ref{definition:Lbf-boundary-admissible} is satisfied, for $Y'_I,\, Y'_{n-1}$ in \eqref{eq:hatLbf-high-components-region1}, we have (again with $s = \mu_{n-1}'$),
\begin{align} \label{eq:Y'-I-expansion-high-region1}
  Y'_{I}(\sigma,x',y,y_{II}',\hat{\mu'}_I,x'/s) = Y'_{I,0}(\sigma,y,y_{II}') + s \tilde{Y'}_{I}(\sigma,x',y,y_{II}',\hat{\mu'}_I,x'/s),
\end{align}
and
\begin{align} \label{eq:Y'-n-1-expansion-high-region1}
  Y'_{n-1}(\sigma,x',y,y_{II}',\hat{\mu'}_I,x'/s) = Y'_{n-1,0}(\sigma,y,y_{II}') +s \tilde{Y'}_{n-1}(\sigma,x',y,y_{II}',\hat{\mu'}_I,x'/s).
\end{align} 
We also have this type of expansion for $Y_{I}'$ and $Y'_{n-1}$ in the region $\frac{s}{x'} \lesssim 1$, see \eqref{eq:Y'-I-expansion-high-region2}\eqref{eq:Y'-n-1-expansion-region2}.

\begin{proposition} \label{prop:Y'I-derivative;low-high-region1-region-def}
Let $Y'_I$ be as in \eqref{eq:hatLbf-low-components}, then potentially after shrinking the range of parametrization, 
we have
\begin{align} \label{eq:derivative-Y'I-low}
   \partial_{\hat{\mu'}_I} Y'_I = -s H_I,
\end{align}
with $H_I$ being uniformly non-degenerate (in fact close to the identity) for $s \in [0,s_0]$ for some $s_0>0$.
Similarly, let $Y'_I$ be as in \eqref{eq:hatLbf-high-components-region1}, then \eqref{eq:derivative-Y'I-low} continues to hold except that $Y'_I$ has $\varrho$-dependence now.
In addition, there is a neighborhood of $\beta_{\mathrm{LCP},\calchigh}^{-1}(L^{\bfs,\calchigh} \cap \Lsharphigh)$ (resp. $\beta_{\mathrm{LCP},\calclow}^{-1}(L^{\bfs} \cap \Lsharplow)$) such that the degeneracy of $\hatLprojhigh$ (resp. $\hatLprojlow$) only happens on $\beta_{\mathrm{LCP},\calchigh}^{-1}(L^{\bfs,\calchigh} \cap \Lsharphigh)$ (resp. $\beta_{\mathrm{LCP},\calclow}^{-1}(L^{\bfs} \cap \Lsharplow)$).
\end{proposition}

\begin{proof}
Consider the case that $Y'_I$ is as in \eqref{eq:hatLbf-low-components} and the case for $Y_I'$ in \eqref{eq:hatLbf-high-components-region1} follows in the same way, except for the extra $\varrho$-dependence, which is not involved in the proof.

By \eqref{eq:Y'-I-expansion}, \eqref{eq:derivative-Y'I-low} is equivalent to the statement that $\partial_{\hat{\mu'}_I}\tilde{Y'}_I$ is non-degenerate.
So we only need to prove this when $(\sigma,y,y'_{II}) = (\sigma_0,y_0,y'_{II,0})$ and the conclusion will follow by continuity, after potentially shrinking the range of parametrization (and we decompose each term using the original parametrization into a finite sum after doing so).

By \eqref{eq:Ymetric-diagonalize}\eqref{eq: rescaled HG}, the expansion of $Y'_{I}$ is
\begin{equation}  \label{eq:5.4-1} 
  Y'_I(\sigma_0,y_0,y'_{II,0},\hat{\mu'}_I,s) = Y'_{I,0} - s_r \frac{\mu'_I}{|\mu'|}  + O(s_r^2).
\end{equation}
Here we use the fact that $Y'_{I}$ is obtained from $Y'_{I,0}$ along the backward flow of $\msf{H}_{G_0}$ at a flow time $s_r$ in \eqref{eq:Lbf-definition-gamma^2}. Now we have 
\begin{equation}
\frac{s_r}{|\mu'|} = \frac{s_r}{\sin s_r} 
= 1 + O(s_r^3),
\end{equation}
and $s \sim s_r, \hat{\mu'}_I =\mu'/s$, we know 
\begin{equation}  \label{eq:Y'-I-expansion-concrete-low}
  Y'_I(\sigma_0,y_0,y'_{II,0},\hat{\mu'}_I,s) = Y'_{I,0}(\sigma_0,y_0,y'_{II,0}) - s \hat{\mu'}_I  + O(s^2).
\end{equation}
Then we know 
\begin{equation}
    \partial_{\hat{\mu'}_I}\tilde{Y'}_I = -\Id + O(s)
\end{equation}
at this point, hence non-degenerate nearby as well. In addition, for $Y'_{n-1}$ similarly we have
\begin{equation}  \label{eq:Y'-n-1-expansion-concrete-low}
  Y'_{n-1}(\sigma_0,y_0,y'_{II,0},\hat{\mu'}_I,s) = Y'_{n-1,0}(\sigma_0,y_0,y'_{II,0}) - s  + O(s^2).
\end{equation}

Now we turn to the high-energy setting in which $Y_I'$ is defined via \eqref{eq:hatLbf-high-components-region1}. 
Now we need to take the geodesic flow in the interior of $X$ into account. 
We first derive, in the single space setting, the expression for Hamilton vector fields like \eqref{eq: rescaled HG} when one takes the $x$-dependence of $\XYmetric$ into consideration. 
The symplectic form on ${}^{\sct}T^*X$ (considered as the right factor of the double space here and with $\tau'=-\nu'$) is given by
\begin{equation}
  \omega'_{\sct} = d\Big(\tau'\frac{dx'}{(x')^2}+\mu' \cdot \frac{dy'}{x'}\Big) = d\tau' \wedge \frac{dx'}{(x')^2} + d\mu' \wedge \frac{dy'}{x'} - \frac{dx'}{(x')^2} \wedge \frac{\mu' \cdot dy'}{x'}.
\end{equation}
Using the definition of $H_p$ that
\begin{equation}
  dp(H') = \omega_{\sct}(H',H_p)
\end{equation}
holds for any smooth vector field $H'$, we can solve for coefficients of $H_p$ and obtain 
\begin{align} \label{eq:perturbed-sc-Hamilton-vector}
\begin{split}
H_p = & \partial_{\tau'}p ((x')^2\partial_{x'}) + \partial_{\mu'}p(x'\partial_{y'})
-((x')^2\partial_{x'}p + x'\mu' \cdot \partial_{\mu'}p) \partial_{\tau'} 
\\ & -(x'\partial_{y'}p-x'(\partial_{\tau'}p)\mu') \cdot \partial_{\mu'}.
\end{split}
\end{align}
In particular, when we take $p = G' = (\tau')^2 + |\mu'|^2_{\XYmetric}$, we have

\begin{align} \label{eq: rescaled-perturbed-HG}
\begin{split}
\msf{H}_{G'}: & = \frac{1}{2}(x')^{-1}|\mu'|^{-1}H_{G'} 
\\ & = \tau' \frac{x'}{|\mu'|}\partial_{x'} - \big(|\mu'|^{-1}\XYmetric(x',y',\mu')+\frac{x'}{|\mu'|}\partial_{x'}\XYmetric(x',y',\mu')\big)\partial_{\tau'} 
+ \tau' \frac{\mu'}{|\mu'|} \cdot \partial_{\mu'} + \frac{1}{2|\mu'|}H_{\XYmetric}.
\end{split}
\end{align}
Since $\XYmetric-\Ymetric = O(x')$, we know the $\partial_{y'}$-component here is an $O(x')$ perturbation of $\msf{H}_{G_0}$, which in turn is $O(s)$ in the region $x'/s \lesssim 1$, after being lifted to the resolved bundle in which we blow-up $\{\mu'=0,x'=0\}$.
Consequently, after integrating this vector field from $s=0$, our current $Y_I'$ is an $O(s^2)$ (in any fixed $C^k$ or Sobolev norm) perturbation of that in the low-energy setting: 
\begin{equation}  \label{eq:Y'-I-expansion-concrete-high-region1}
Y'_I(\sigma_0,y_0,y'_{II,0},\hat{\mu'}_I,s,\varrho) = Y'_{I,0}(\sigma_0,y_0,y'_{II,0}) - s \hat{\mu'}_I  + O(s^2),
\end{equation}
and \eqref{eq:derivative-Y'I-low} continues to hold if we choose $s_0$ to be small.

Now we prove the last claim about the projection $\hatLprojhigh$. 
In the region $x'/s \lesssim 1$ of $\hat{L}^{\bfs,\calchigh}$, this also follows from \eqref{eq:Y'-I-expansion-concrete-high-region1} and the analogue of \eqref{eq:Y'-n-1-expansion-concrete-low} in the high-energy setting:
\begin{equation}  \label{eq:Y'-n-1-expansion-concrete-high-region1}
  Y'_{n-1}(\sigma_0,y_0,y'_{II,0},\hat{\mu'}_I,s,\varrho) = Y'_{n-1,0}(\sigma_0,y_0,y'_{II,0}) - s  + O(s^2).
\end{equation}
Then we see that $\partial_{(s,\hat{\mu'}_I)}(Y'_{n-1},Y'_I)$ is non-degenerate (though not uniformly as $s \to 0$) for all $s>0$.
In the region $x'/s \gtrsim 1$, this follows from the same proof except that the expansions of $Y'_I,Y'_{n-1}$ are replaced by \eqref{eq:Y'-I-expansion-high-region2-concrete}\eqref{eq:Y'-n-1-expansion-region2-concrete} instead.
The conclusion for $\hatLprojlow$ follows from the conclusion for $\hatLprojhigh$ since it is just the restriction of $\hatLprojhigh$ to $\{x'/s = 0\}$.
\end{proof}
 

A direct consequence of this is the estimate of the Hessian of phase functions parametrizing our conic pairs of Legendre submanifolds. 
\begin{corollary} \label{coro:Hessian-lowerbound-low;high-region1}
Let $\Phi$ be as in \eqref{eq:conic-pair-phase-construction-low} and written as $1+\sigma+s\sigma \psi$; then we have
\begin{equation} \label{eq:Hessian-psi-nondegenerate-low-highregion1}
  \partial^2_{\hat{\mu'}_I\hat{\mu'}_I}\psi =   s H_I,
  \end{equation}
with $H_I$ being uniformly non-degenerate (in fact close to the identity) for $s$ small.
In the high-energy case, let $\Phi$ be as in \eqref{eq:conic-pair-phase-construction-high-region1}, then \eqref{eq:Hessian-psi-nondegenerate-low-highregion1} continues to hold, except that the function now has extra $\varrho$-dependence.
\end{corollary}

\begin{proof}
This follows from Proposition~\ref{prop:LCP-phase-derivatives} and Proposition~\ref{prop:Y'I-derivative;low-high-region1-region-def}.
\end{proof}


Now we discuss the phase function parametrizing $\hat{L}^{\bfs,\calchigh}$ in the region $s/x' \lesssim 1$ as in \eqref{eq:conic-pair-phase-construction-high-region2}.
The major difference with the case that $x'/s \lesssim 1$ is that, in order to be uniform down to $|\mu'|/x' = 0$, we can't assume that one of $\mu_i$, say $\mu_{n-1}$, dominates others since any neighborhood of $\{|\mu'|/x' = 0\}$ will contain points with ratios like $\mu_i/\mu_j$ taking arbitrary values.

For the analogue of the expansion \eqref{eq:Y'-I-expansion} and \eqref{eq:Y'-I-expansion-high-region1} in this region, the admissible condition in Definition~\ref{definition:Lbf-boundary-admissible} shows $Y'_I,Y'_{n-1}$ are of the form
\begin{align} \label{eq:Y'-I-expansion-high-region2}
  Y'_{I}(\sigma,x',y,y_{II}',\hat{\mu'}_I,s/x') = Y'_{I,0}(\sigma,y,y_{II}') + x' \tilde{Y}'_I(\sigma,x',y,y_{II}',\hat{\mu'}_I,s/x'),
\end{align}
 and similarly
\begin{align} \label{eq:Y'-n-1-expansion-region2}
  Y'_{n-1}(\sigma,x',y,y_{II}',\hat{\mu'}_I,\varkappa) = Y'_{n-1,0}(\sigma,y,y_{II}') + x' \tilde{Y}'_{n-1}(\sigma,x',y,y_{II}',\hat{\mu'}_I,s/x').
\end{align}

In this region, to obtain the analogue of \eqref{eq:derivative-Y'I-low}, we instead consider the rescaled geodesic flow that has almost unit speed in $x'$, which means that  instead of \eqref{eq: rescaled-perturbed-HG}, the Hamilton vector field that gives non-degenerate smooth flow (on the right factor), hitting $\partial \hat{L}^{\bfs,\calchigh}$ transversally is given by:
\begin{align} \label{eq: rescaled-perturbed-HG-region2}
\begin{split}
\msf{H}_{G'}: & = \frac{1}{2}(x')^{-2}H_{G'} 
\\ & = \tau' \partial_{x'} - (x'^{-1}\XYmetric(x',y',\mu')+ \partial_{x'}\XYmetric(x',y',\mu'))\partial_{\tau'} 
+ \tau' \frac{\mu'}{x'} \cdot \partial_{\mu'} + \frac{1}{2x'}H_{\XYmetric},
\end{split}
\end{align}
which is just \eqref{eq: rescaled-perturbed-HG} rescaled by multiplying $|\mu'|/|x'| \lesssim 1$.
So in the region $x'/|\mu'| \lesssim 1$ we use the flow taking $s = \mu_{n-1}$ as the flow time and in the region $x'/|\mu'| \gtrsim 1$ we use the flow taking $x'$ as the flow time. They have the same endpoint (i.e., restriction to $s=0$ and $x'=0$ respectively). 
So, starting from such an endpoint and considering the backward flow, since the linear level behaviour (i.e., $\partial_{y'}$-components of the vector field) of $Y'_I,Y_{n-1}'$ comes purely from $\frac{1}{2x'}H_{\XYmetric}$, we can view $x'$ as the flow time and it coincides with $Y'_I,Y_{n-1}'$ in the setting given by the flow of \eqref{eq: rescaled-perturbed-HG} parametrized by $s$ at time $s=x' \varkappa$, modulo an overall $O((x')^2)$ (in any fixed $C^k$-norm) error.
Since now the rescaled frequency becomes $\hat{\mu'_I} = \mu'_I/x',\kappa = \mu'_{n-1}/x'$, this expansion, when we fix the end-point to be $y_0'$ as in \eqref{eq:Ymetric-diagonalize}, is
\begin{align} \label{eq:Y'-I-expansion-high-region2-concrete}
  Y'_{I}(\sigma_0,x',y_0,y_{II,0}',\hat{\mu'}_I,s/x') = Y'_{I,0}(\sigma_0,y_0,y_{II,0}') - x' \hat{\mu'}_I + O((x')^2),
\end{align}
 and similarly
\begin{align} \label{eq:Y'-n-1-expansion-region2-concrete}
  Y'_{n-1}(\sigma_0,x',y_0,y_{II,0}',\hat{\mu'}_I,\varkappa) = Y'_{n-1,0}(\sigma_0,y_0,y_{II,0}') - x' \varkappa + O((x')^2).
\end{align}
Then we have the following characterization of derivatives of $Y'_I,Y'_{n-1}$.
\begin{proposition} \label{prop:lower-bound-hessian-high-region2}
Let $Y'_I, \, Y'_{n-1}$ be as in \eqref{eq:hatLbf-high-components-region1} but with $\tilde{\mathcal{Z}}_{\calchigh}$ in place of $\mathcal{Z}_{\calchigh}$; then we have
\begin{align} \label{eq:5.4-6}
  \begin{pmatrix}
    \partial_{\varkappa} Y_{n-1}' & \partial_{\hat{\mu'}_I} Y_{n-1}' \\
    \partial_{\varkappa} Y'_I & \partial_{\hat{\mu'}_I} Y'_I
  \end{pmatrix} = - x' H_2,
\end{align}
with $H_2$ being uniformly non-degenerate (in fact close to the identity) in the region $\varkappa \gtrsim 1$, $x' \ll 1$.
\end{proposition}

\begin{proof}
By \eqref{eq:Y'-I-expansion-high-region2}\eqref{eq:Y'-n-1-expansion-region2}, we know that the left hand side of \eqref{eq:5.4-6} takes the form $x'H_2$ for some smooth $H_2$. By \eqref{eq:Y'-I-expansion-high-region2-concrete}\eqref{eq:Y'-n-1-expansion-region2-concrete}, we know $H_2 = \Id +O(s)$ when $(\sigma,y,y'_{II}) = (\sigma_0,y_0,y'_{II,0})$, hence it remains non-degenerate nearby and this concludes the proof.
\end{proof}

This gives the following characterization of our phase function $\Phi$ in \eqref{eq:conic-pair-phase-construction-high-region2}.
\begin{corollary} \label{coro:lower-bound-hessian-high-region2}
Let $\Phi$ be as in \eqref{eq:conic-pair-phase-construction-high-region2} and written as $1+\sigma+s\sigma \psi$ and denote $(\varkappa,\hat{\mu'}_I)$ by $\tilde{v}$; then we have
\begin{equation} \label{eq:Hessian-psi-nondegenerate-high-region2}
  \partial^2_{\tilde{v}\tilde{v}}\psi = x'H_2,
\end{equation}
with $H_2$ being uniformly non-degenerate (in fact close to the identity) for $x'$ small.
\end{corollary}

\begin{proof}
  In the same fashion as the proof of Corollary~\ref{coro:Hessian-lowerbound-low;high-region1},  using \eqref{eq:1-form-vanish-low}, we have
\begin{equation}
  \partial_{\varkappa} \psi = y'_{n-1}-Y'_{n-1}, \; \partial_{\hat{\mu'}_I}\psi = y'_I - Y'_I.
\end{equation}
Then applying Proposition~\ref{prop:lower-bound-hessian-high-region2} concludes the proof.
  
\end{proof}
 
Now we show that the most classical case $X_0 = \overline{\R^n}$, which is the radial compactification of $\R^n$, equipped with the flat metric is admissible in the sense of Definition~\ref{definition:Lbf-boundary-admissible}.

\begin{proposition} \label{prop:example-admissible}
Let $\mathbb{S}_{\varsigma}^{n-1}$ be the sphere with radius $\varsigma$ and let $X_0$ be the cone over $\mathbb{S}_{\varsigma}^{n-1}$, then $X_0$ is admissible in the sense of Definition~\ref{definition:Lbf-boundary-admissible} and the focusing indices of $X_0$ are classified as follows according to $\varsigma$:
\begin{enumerate}
\item When $\varsigma = 1$, the cone is the (compactified) Euclidean space $X_0 = \overline{\R^n}$ equipped with the flat metric, it has $\IFintz = 0$, $\IFz=n-2$;
\item When $\varsigma>1$, we have $\IFintz=\IFz=0$;
\item When $\varsigma<1$, we have $\IFintz = \IFz = n-2$.
\end{enumerate}
\end{proposition}

\begin{proof}
We first consider the case $\varsigma=1$, i.e. $X_0 = \overline{\R^n}$, and indicate changes needed to deal with other cases in the end.
Without loss of generality, we consider the region $\sigma = x/x' \lesssim 1$.
Let $\mathsf{p} \in \partial \hat{L}^{\bfs,\calchigh}$ be the point around which we are considering the parametrization. 
Consider the case that $\msf{p}$ is in the region $|x'|/|\mu'| \lesssim 1$ first.
Denote its projection to $\bfs$ by $(\sigma_0,x'=0,y_0,y_0')$. 
For coordinates on the right factor of $Y \times Y$, after a linear change of coordinates, we may assume that $\mu'/s = (0,...,1) \in \R^{n-1}$ at $\msf{p}$.
Then we use `the same' coordinate for the first factor, which means letting $-y'$ be the coordinate of the antipodal point of $y' \in \R^{n-1}$.
In this coordinate system, we know\footnote{It might be more straightforward to see this by pulling-back one of the factors by the antipodal map and considering the diagonal.}
\begin{align} \label{eq:dy-combination-vanish}
d(y_1+y_1'), \, ..., d(y_{n-2}+y'_{n-2}) = 0
\end{align}
at $\msf{p}$. By our definition of the splitting $(y_I',y'_{II},y'_{n-1})$ at the beginning of Section~\ref{subsec:phase-Hessian}, we can choose $I=\{1,2,...,n-2\}$, $II = \emptyset$. 
In terms of the notation in \cite[Proposition~6]{melrose1996scattering}, considering the coordinate system $(\sigma,y,y_1+y'_1,...,y_{n-2}+y'_{n-2},y'_{n-1})$ on $\bfs$, we know that the dual variable of $y_i+y'_{i}$ equals that of $y'_i$ and the coordinate system we can choose is 
\begin{align*}
(\sigma,y,y'_{n-1}, \hat{\mu'}_1,...,\hat{\mu'}_{n-2},s= \mu_{n-1}',x'/s).
\end{align*}
At $s=0$ and restricted to $L^{\bfs,\calchigh}$, we have $y'=-y$ since the geodesic flow on the sphere at time $\pi$ sends a point to its antipodal point, which means
\begin{align*}
Y'_{I}|_{s=0} = -(y_1,...,y_{n-2}), \; Y'_{n-1}|_{s=0} = -y_{n-1},
\end{align*}
which shows that \eqref{eq:defn-condition-YI-Yn-high-region1} and \eqref{eq:defn-condition-YI-Yn-high-region2} hold and $X_0$ is admissible.
To show $\IFintz=0$, we notice that the injective radius of $\mathbb{S}^{n-1}$ is $\pi$, hence the projection of $L^{\bfs}$ to $\bfs$ is a diffeomorphism away from $L^{\bfs}\cap\Lsharplow$. 
In addition, discussion above shows that the rank drop of $\hatLprojlow$ at $L^{\bfs}\cap\Lsharplow$ is $n-2$, so we have $\IFz=n-2$.

When $\varsigma>1$, the conclusion on the focusing indices follows from the proof of Corollary~\ref{coro:after-main-thm}. We now discuss the admissible condition. The projection $\hat{L}^{\bfs} \to \bfs$ is a diffeomorphism to $\bfs$ near $L^{\bfs} \cap \Lsharplow$. So $\hat{\mu}_I'$ is empty and the admissible condition holds automatically.

We now turn to the case $\varsigma<1$.
When $\varsigma \neq \frac{1}{k}$ for $k \in \N$, the projection $\hat{L}^{\bfs} \to \bfs$ is also a diffeomorphism to $\bfs$ near $L^{\bfs} \cap \Lsharplow$, hence it is admissible for the same reason as the case $\varsigma>1$ above.
However, the focusing indices are different from those in the case $\varsigma>1$.
The maximal rank drop happens when the geodesic flow finishes the first round of the circle (i.e. at time $\pi\varsigma$) and the degeneracy is the same as \eqref{eq:dy-combination-vanish} above. So we have $\IFz=\IF=n-2$.

The case $\varsigma = 1/k$ with $k\in\N$ and $k \geq 2$ follows in the same way since the shape of $L^{\bfs}$ near $L^{\bfs} \cap \Lsharplow$ is the same as the case $\varsigma=1$ above, except that now those geodesics are going through circles in $\mathbb{S}_{\varsigma}^{n-1}$ for $k$ times. 
The major difference with the case $\varsigma=1$ is that there is already a rank drop of $n-2$ when the geodesic flow finishes the first round of the circle, where the rank drop is $n-2$.
In this case, we have $\IFintz = n-2$ and $\IFz = n-2$.
\end{proof}
 
\subsection{Critical points at fixed s}
\label{subsec:critical_points_at_fixed_s}
In this subsection, we discuss the critical points of phase functions parametrizing Legendre conic pairs at fixed $s$.
For our phase function $\Phi = 1+\sigma+\sigma s \psi$, since an overall factor $e^{i\frac{\lambda}{x}(1+\sigma)}$ does not affect the pointwise bound on the kernel, we will consider 
\begin{equation} \label{eq:s-psi-def}
    s\psi  = sN_2 + s (y'_I - Y'_I) \cdot \hat{\mu'}_I + s(y'_{n-1}-Y'_{n-1}),
\end{equation}
where $N_2,Y'_I,Y'_{n-1}$ are as in \eqref{eq:hatLbf-low-components} or \eqref{eq:hatLbf-high-components-region1}.
By \eqref{eq:v-critical-fixed-s}, the derivative of the phase function $\psi$ in $\hat{\mu'}_I$ is $y'_I - Y'_I$. 

Let $\mathrm{Ran}(s)$ be the range of $Y'_I$ for fixed $s$.
Now for each fixed $\sigma,y,y_I',y_{II}'$, we define the set $\srange(\sigma,y,y_I',y_{II}')$ by 
\begin{equation} \label{eq:srange-def}
    s \in \srange(\sigma,y,y_I',y_{II}') \text{ if and only if } y_I' \in \mathrm{Ran}(s).
\end{equation}
We denote it by $\srange$ below for convenience since the estimates we prove are pointwise for fixed $(\sigma,y,y_I',y_{II}')$.

Next we record certain basic properties of the phase function and other quantities involved in terms of $\srange$. Since $\mathrm{Ran}(s)$ is open and changes continuously (say, in terms of the Hausdorff distance) in terms of $s$, if $s_0\in\srange$, then so does $s$ sufficiently close to it and consequently $\srange$ is open.

Now we turn to estimates.
First, for $s \notin \srange$, since the region of $\hat{\mu'}_I$ on which $Y'_I$ is defined via \eqref{eq:hatLbf-low-components}, equivalently the region on which we use $\hat{\mu'}_I$ as coordinates, contains a $\delta_0$-neighborhood of $\supp_{(\sigma,y,y',s)} \, a$ for certain $\delta_0>0$, where $\supp_{(\sigma,y,y',s)} \, a$ is the support of $a$ in $\hat{\mu'}_I$ with $(\sigma,y,y',s)$ fixed.
Since the derivative of $Y'_I$ is uniformly non-degenerate after removing the $s$-factor as in Proposition~\ref{prop:Y'I-derivative;low-high-region1-region-def}
 we know that the range of $Y'_I$ contains a $C\delta_0s$-neighborhood of the range of $Y'_I$ restricted to $\supp_{(\sigma,y,y',s)} \, a$ for some $C>0$ depending on the lower bound of the derivative of $Y'_I$ in Proposition~\ref{prop:Y'I-derivative;low-high-region1-region-def}. This means
\begin{equation}  \label{eq:lower-bound-Ec}
  |y'_I - Y'_I| \gtrsim s
\end{equation}
on $\supp_{(\sigma,y,y',s)} \, a$.
On the other hand, for $s \in \srange$, 
we know there is a unique $\hat{M}_I'$ such that 
\begin{align} \label{eq:def-hat-M'-I}
y'_I = Y'_I(\sigma,y,y_{II}',s,\hat{M'}_I(\sigma,y,y',s)).
\end{align}
The existence follows from the definition of $\srange$ while the uniqueness is due to \eqref{eq:derivative-Y'I-low}, which shows that for fixed $s$, the map sending $\hat{\mu'}_I$ to $Y'_{I}$ is a (local) diffeomorphism.

\begin{lemma} \label{lemma:M1-s-derivative-bound}
Let $\hat{M'}_{I}(\sigma,y,y'_{I},y'_{II},s)$ be defined as in \eqref{eq:def-hat-M'-I}, then it has at most conic singularity in $s$ in the sense that:
\begin{align} \label{eq:M1-s-derivative-bound-1}
  |\partial_s^m \hat{M'}_{I}(\sigma,y,y'_{II},s)| \lesssim  s^{-m}, \; m \in \N.
\end{align}

\end{lemma}

\begin{proof}

Differentiate both sides of \eqref{eq:def-hat-M'-I} in $s$, or apply the implicit function theorem, we have
\begin{align} \label{eq:M'I-s-derivative}
\partial_s\hat{M'}_I(\sigma,y,y'_{II},s) = - (\partial_{\hat{\mu}_1}Y'_I)^{-1}|_{\hat{\mu'}_I=\hat{M'}_{I}(\sigma,y,y'_{II},s)}\partial_s Y'_I.
\end{align}
The factor $\partial_s Y'_I$ is smooth and we only need to prove, for all $m \in \N_+$,
\begin{align} \label{eq:est-10}
\| \partial_s^{m-1}(\partial_{\hat{\mu}_1}Y'_I)^{-1}|_{\hat{\mu}_1=\hat{M'}_I(\sigma,y,y'_{II},s)}\|_{\R^k \to \R^k} \lesssim s^{-m}.
\end{align}
This is because, by \eqref{eq:derivative-Y'I-low}, we can write 
\begin{align}
(\partial_{\hat{\mu}_1}Y'_I)^{-1} = s^{-1}A,
\end{align}
with $A$ being a matrix with smooth components. Then the desired conclusion follows.

\end{proof}

Finally, we discuss the convexity of the phase function restricted to $\hat{\mu'}_I = \hat{M'}_I$.
Consider the case that $s\psi$ in \eqref{eq:s-psi-def} is constructed using $N_2,Y'_I,Y'_{n-1}$ in \eqref{eq:hatLbf-low-components} first.
As discussed before \eqref{eq:Ymetric-diagonalize}, we can select our coordinates so that $\mu'/|\mu'| = (0,...,0,1)$ at the point near which we are parametrizing $\hat{L}^{\bfs}$ (and $\hat{L}^{\bfs,\calchigh}$ in the high-energy setting). So we can assume that (recall that $\Ymetric$ is the metric on the cross section $Y$ at $\{x=0\}$)
\begin{equation} \label{eq: hat-mu'-small-delta-1}
  |\hat{\mu'}_I|,|\hat{\mu'}_{II}| \leq \delta_1, 
  \; \text{ and } \; |\Ymetric^{ij} - \delta^{ij}| \leq \delta_1, 
\end{equation}
on this region of $\hat{L}^{\bfs}$ (resp. $\hat{L}^{\bfs,\calchigh}$ in the high-energy setting) for a fixed small $\delta_1>0$ to be determined.

The first order $s$-derivative of $s\psi|_{\hat{\mu'}_I = \hat{M'}_I}$ is given by:
\begin{equation}
    \partial_s(s\psi|_{\hat{\mu'}_I = \hat{M'}_I}) = 
    \big(\partial_s(s\psi) \big)|_{\hat{\mu'}_I = \hat{M'}_I}
    + (\partial_s \hat{M'}_I) \cdot \partial_{\hat{\mu'}_I}(s\psi)|_{\hat{\mu'}_I = \hat{M'}_I}.
\end{equation}
Since we have $\partial_{\hat{\mu'}_I}(s\psi)|_{\hat{\mu'}_I = \hat{M'}_I} = 0$, hence 
\begin{equation} \label{eq:5.4-2}
    \partial_s(s\psi|_{\hat{\mu'}_I = \hat{M'}_I}) = 
    \big(\partial_s(s\psi) \big)|_{\hat{\mu'}_I = \hat{M'}_I} = y'_{n-1}-Y'_{n-1}(\sigma,y,y'_{II},\hat{M'}_I(s),s),
\end{equation}
where we abbreviated variables of $\hat{M'}_I$ other than $s$.

Similar to \eqref{eq:5.4-1}\eqref{eq:Y'-n-1-expansion-concrete-low}, we have
\begin{equation}  \label{eq:Y'-n-1-expansion-concrete-low-2}
  Y'_{n-1}(\sigma,y,y'_{II},\hat{\mu'}_I,s) = Y'_{n-1,0}(\sigma,y,y'_{II}) - s(\sum_{j \in I} \Ymetric^{(n-1)j}\hat{M'}_I + \sum_{j \in II} \Ymetric^{(n-1)j}\hat{M'}_{II} +1)  + O(s^2).
\end{equation}
Using the bound of $\partial_s\hat{M'}_I$ in \eqref{eq:M1-s-derivative-bound-1}, differentiating \eqref{eq:Y'-n-1-expansion-concrete-low-2} in $s$, we have
\begin{equation}
\partial_{ss}^2(s\psi|_{\hat{\mu'}_I = \hat{M'}_I})
= -1 + O(\delta_1) + O(s).
\end{equation}
After choosing $\delta_1,s_0$ to be small, we have $|\partial_{ss} (s\psi|_{\hat{\mu'}_{I}= \hat{M'}_I})  | \gtrsim 1$.
When $s\psi$ in \eqref{eq:s-psi-def} is instead defined using $N_2,Y'_I,Y'_{n-1}$ in \eqref{eq:hatLbf-high-components-region1}, argument above continues to apply since the expansion of $Y'_I$ above continues to apply by \eqref{eq: rescaled-perturbed-HG} and the same discussion afterwards concerning the flow of $\msf{H}_G$. The $\varrho = x'/s$-dependence of the error terms won't affect the final estimate, since that will at most introduce an $O(s^2 \times x'/s^2)=O(x')$-error after differentiating in $s$, which can be absorbed into the $O(s)$ term since $x'/s \lesssim 1$ in the region we are considering.  
Summarizing the discussion above, we have the following convexity of $s\psi|_{\hat{\mu'}_{I}= \hat{M'}_I}$.

\begin{lemma} \label{lemma:phase-restricted-convex-low-high-region1}
Let $s\psi$ be as in \eqref{eq:s-psi-def} arising in the parametrization of either $\hat{L}^{\bfs}$ or of $\hat{L}^{\bfs,\calchigh}$ in the region with $x'/s \lesssim 1$, then there is a $\delta_1$ as in \eqref{eq: hat-mu'-small-delta-1} and $s_0>0$, such that for $s \in [0,s_0]$:
  \begin{equation} \label{eq:phase-restricted-convex}
    |\partial_{ss} (s\psi|_{\hat{\mu'}_{I}= \hat{M'}_I})  | \gtrsim 1.
  \end{equation}
\end{lemma}

This uniform convexity has the following consequence
\begin{corollary} \label{coro:restricted-phase-three-regions}
   For each $(\sigma,y,y'_{II})$, potentially after further shrinking $s_0$, either of the following dichotomy happens.
\begin{enumerate}
\item When $y'_{n-1} - Y'_{n-1,0}(\sigma,y,y'_{II}) \geq 0$, for all $s \in [0,s_0] \cap \srange$, we have 
\begin{equation} \label{eq:restricted-phase-partial-s-2}
|\partial_s (s\psi|_{\hat{\mu'}_{I}=\hat{M'}_I})| \gtrsim s.
\end{equation} 
  \item When $y'_{n-1} - Y'_{n-1,0}(\sigma,y,y'_{II})<0$, for large enough constant $C>0$, if we divide $[0,s_0] \cap \srange$ into three regions:
\begin{itemize}
  \item Region 1: $s \geq C|y_{n-1}'-Y'_{n-1,0}(\sigma,y,y'_{II})|$,
  \item Region 2: $(2C)^{-1}|y_{n-1}'-Y'_{n-1,0}(\sigma,y,y'_{II})| \leq s \leq (2C)|y_{n-1}'-Y'_{n-1,0}(\sigma,y,y'_{II})|$,
  \item Region 3: $s \leq C^{-1}|y_{n-1}'-Y'_{n-1,0}(\sigma,y,y'_{II})|$,
\end{itemize}
then on Region 1 and Region 3, \eqref{eq:restricted-phase-partial-s-2} holds. On Region 2, \eqref{eq:phase-restricted-convex} continues to hold.
\end{enumerate}
 \end{corollary}

\begin{proof}
By \eqref{eq:5.4-2} and \eqref{eq:Y'-n-1-expansion-concrete-low-2}, for $s \in \srange$, we have
\begin{align} \label{eq:s-psi-restricted-s-derivative}
\begin{split}
  \partial_s\big(s\psi|_{\hat{\mu'}_{I} = \hat{M'}_I }\big) 
 & = \big(y_{n-1}'-Y_{n-1,0}'(\sigma,y,y'_{II}) \big) \\&+  s\Big(\sum_{j \in I} \Ymetric^{(n-1)j}\hat{M'}_I + \sum_{j \in II} \Ymetric^{(n-1)j}\hat{M'}_{II} +1\Big)  + O(s^2).
\end{split}
\end{align}
When $(y_{n-1}'-Y_{n-1,0}'(\sigma,y,y'_{II})) \geq 0$, we know \eqref{eq:restricted-phase-partial-s-2} holds after potentially choosing $s_0$ to be small.

Now we turn to the case when $y'_{n-1} - Y'_{n-1,0}(\sigma,y,y'_{II})<0$.
We choose $C$ large and define three regions as above. 
Let $F(s)$ be the sum of the last two terms in \eqref{eq:s-psi-restricted-s-derivative}.
In Region 1, $F(s)$ dominates $(y_{n-1}'-Y_{n-1,0}'(\sigma,y,y'_{II}))$ and we have \eqref{eq:restricted-phase-partial-s-2}.
On the other hand, in Region 3, $(y_{n-1}'-Y_{n-1,0}'(\sigma,y,y'_{II}) )$ dominates all other terms and we have
\begin{equation}  
 |\partial_s(s\psi|_{\hat{\mu'}_{I} = \hat{M'}_I })|
 \gtrsim |y_{n-1}'-Y_{n-1,0}'(\sigma,y,y'_{II})| \gtrsim s.
\end{equation}
\end{proof}
\subsection{The value of the phase function: the high-energy regime}
\label{subsec:value-phase-high}
We discuss the value of the phase function in the high-energy regime in this subsection.

Though we renormalized the value of $\tau$ in \eqref{eq:semiclassical-form-bf} when we approach boundaries to form a smooth section of ${}^{\bundlehigh}T_{\smf}^*X_{\rmb,\calchigh}^2$,  in fact we can choose not to do so except that this is not a smooth function down to the boundary any more. Let 
\begin{equation} \label{eq:def-rho-b}
    \rho_{\rmb} = (x^{-2}+(x')^{-2})^{-1/2}
\end{equation}
be the total boundary defining function of the boundary in $X_{\rmb}^2$, which can be identified with $\smf = X_{\rmb,\calchigh}^2 \cap \{h=0\}$.
Then $\rho_{\rmb}$ lifts to the defining function of
${}^{\bundlehigh}T_{\smf \cap ( \LB \cup \BFS \cup \RB )  }^*X_{\rmb,\calchigh}^2$ in ${}^{\bundlehigh}T_{\smf}^*X_{\rmb,\calchigh}^2$, which we still denote by $\rho_{\rmb}$.

Then over the interior of $X_{\rmb}^2$, using $(z,z',\zeta,\zeta')$ as its coordinates, $\tau$ restricted on $L^{\bfs,\calchigh}$ is the length of the lifted geodesic connecting $(z,z')$ with $\zeta,\zeta'$ as momentum at $z,z'$ respectively. 
Recall the discussion in \cite[Proposition~9.2]{HTW1}, if we consider
\begin{equation}
    \rho_{\rmb} d(z,z'),
\end{equation}
then it extends to a continuous function on $X_{\rmb}^2$ in a neighborhood of the lifted diagonal and is smooth except for a conic singularity at the diagonal. 

Now we remedy this issue of singularity at the diagonal and the restriction being close to the diagonal by switching to the flow time along the geodesic flow.
This is the signed distance near the diagonal and is more natural to be viewed as a function on $\hat{L}^{\bfs,\calchigh}$, which is the content of the rest of this subsection.

For how this resolves the singularity at the diagonal, it is illuminating to recall the concept of the conormal bundle of the diagonal.
Recall that ${}^{\Phi}T^*X_{\rmb}^2$ is the bundle defined in \eqref{def:pulled-back-bundle} and let $\NDb$ be the conormal bundle of the lifted diagonal in ${}^{\Phi}T^*X_{\rmb}^2$ defined by
\begin{equation}\label{eq:conormal-b-diagonal}
    \NDb = \{ \sigma = 1, y = y', \nu =-\nu', \mu=-\mu'\} \subset {}^{\Phi}T^*X_{\rmb}^2
\end{equation}
near $\bfs$, where un-primed and primed coordinates are as in \eqref{eq:sc-1-form} but lifted from ${}^{\sct}T^*X$ from the left and right factors respectively. And over the interior the characterization is
\begin{equation}
\{ z=z', \zeta = -\zeta' \} \subset {}^{\Phi}T^*X_{\rmb}^2
\end{equation}
if we use $(z,\zeta)$ as coordinates of the cotangent bundle.

In addition, consider the following map ${}^{\bundlehigh}T_{\smf}^*X_{\rmb,\calchigh}^2 \to {}^{\Phi}T^*X_{\rmb}^2$:
\begin{align} \label{eq:bundlehigh-Phi-projection}
\projhighlow: \quad   (z,z',\zeta,\zeta',\tau) \to (z,z',\zeta,\zeta')
\end{align}
over $X^{\circ} \times X^{\circ}$ with coordinates as in \eqref{eq:semiclassical-form-interior} and \eqref{eq:1-form-Phi-bundle-interior} respectively and extends to the boundary part ${}^{\bundlehigh}T_{\smf \cap \BFS}^*X_{\rmb,\calchigh}^2 \to {}^{\Phi}T_{\bfs}^*X_{\rmb}^2$ smoothly.
In addition, when we restrict $\projhighlow$ to $L^{\bfs,\calchigh}$, this is a diffeomorphism onto its image since $\tau$ is uniquely determined by other variables in a smooth (in terms of the smooth structure of ${}^{\bundlehigh}T_{\smf \cap \BFS}^*X_{\rmb,\calchigh}^2$) manner.
Using $\projhighlow$, $\NDb$ can be identified with its preimage under this projection intersecting with $L^{\bfs,\calchigh}$, which is, in terms of \eqref{eq:semiclassical-form-interior}, the following part of $L^{\bfs,\calchigh}$:
\begin{equation}
\NDhigh :=  \mathrm{clos} \{ h=0, z = z', \zeta = -\zeta', \tau =0 \}  \subset {}^{\bundlehigh}T_{\smf}^*X_{\rmb,\calchigh}^2.
\end{equation}
In addition, $\projhighlow$ also induces a unique map $\hat{\Pi}_{\smf}$ so that the diagram
\[ \begin{tikzcd}
{[{}^{\bundlehigh}T_{\smf}^*X; J^{\calchigh} ]} \arrow{r}{\beta_{\mathrm{LCP},\calchigh}} \arrow[swap]{d}{ \hat{\Pi}_{\smf} } & {}^{\bundlehigh}T_{\smf}^*X  \arrow{d}{\projhighlow } \\%
{[{}^{\Phi}T^*X_{\rmb}^2;J^{\calclow}]} \arrow{r}&  {}^{\Phi}T^*X_{\rmb}^2
\end{tikzcd} \]
commutes, where $J^{\calchigh}$ and $J^{\calclow}$ are defined in \eqref{eq:def-J-high} and \eqref{eq:def-J-low} and horizontal arrows are blow-down maps.
In terms of local coordinates over the region $x,|\mu'| \lesssim |\mu|$, $\hat{\Pi}_{\smf}$ is given by
\begin{equation} \label{eq:hat-Pi-sPhi-Phi}
  (\frac{x}{|\mu|},\sigma,y,y',\nu,\nu',|\mu|,\frac{|\mu'|}{|\mu|},\hat{\mu},\hat{\mu'},\overline{\tau})
  \to   (\frac{x}{|\mu|},\sigma,y,y',\nu,\nu',|\mu|, \frac{|\mu'|}{|\mu|},\hat{\mu},\hat{\mu'}),
\end{equation}
where $\overline{\tau}$ is as in \eqref{eq:semiclassical-form-bf}.

We make two observations here:
\begin{enumerate}
    \item The singularity of the distance function at the diagonal is resolved if we lift that function to $\hat{L}^{\bfs,\calchigh}$ and stop requiring the positive sign.   \label{item:dist-observation-1} 
    \item Restricting to a small neighborhood of the diagonal is unnecessary as well, after lifting to $\hat{L}^{\bfs,\calchigh}$.              \label{item:dist-observation-2}
\end{enumerate}
We explain \eqref{item:dist-observation-1} first. Notice that the singularity of the distance function is caused by the singularity of $\frac{z-z'}{|z-z'|}$
at the diagonal and sign switching of the gradient when one goes across the diagonal along a geodesic. Lifting to $L^{\bfs,\calchigh}$ replaces the diagonal with the conormal bundle of the diagonal (together with the other $\tau$-component, but restricted to be $0$), which exactly resolves this singularity.
More precisely, on $L^{\bfs,\calchigh}$, viewing it as the flow out of $\NDhigh$ as in  \eqref{eq:L-high-interior}, this $\tau$ is precisely the signed distance function, at least near $\NDhigh$, and is smooth on $L^{\bfs,\calchigh}$.
Up to this stage, using $L^{\bfs,\calchigh}$ instead of $\hat{L}^{\bfs,\calchigh}$ is sufficient. But there are singularities of the manifold $L^{\bfs,\calchigh}$ itself in the other extreme: when we approach $G_1^{\calchigh} \cap L^{\bfs,\calchigh}$, which corresponds to the case when $z,z'$ are approaching the `starting and ending points' of a geodesic.
But this is again resolved when we resolve it at $G_1^{\calchigh} \cap L^{\bfs,\calchigh}$ and further lift this function to $\hat{L}^{\bfs,\calchigh}$.

Now we discuss the behaviour of this flow time $\tau$ as we approach $\BFS \cap \smf$ on $L^{\bfs,\calchigh}$.
Comparing \eqref{eq:semiclassical-form-interior} with \eqref{eq:semiclassical-form-bf}, we know that the relation between $\overline{\tau}$ and $\tau$ is
\begin{equation}
   \tau - \frac{\xi}{x} - \frac{\xi'}{x'}   =  \overline{\tau},
\end{equation}
which means that $\overline{\tau}$ tends to the sojourn time (or, the renormalized length) as both left and right factors tend to two ends of the geodesic. 
This is because $\xi,\, \xi' \to \pm 1$ in that setting with the sign depending on which is chosen to go to the endpoint in the forward direction. See \cite[Section~15]{Hassell-Wunsch-semiclassical-resolvent} for more discussion on this perspective.

But in the current setting, we can stop incorporating $\tau$ as a part of the bundle, but instead view it as a function (or in fact, tempered distribution, as it has growth approaching the boundary) on 
\begin{equation}
\RXstarb := \hat{\Pi}_{\smf}(\hat{L}^{\bfs,\calchigh})  \subset [{}^{\Phi}T^*X_{\rmb}^2;J^{\calclow}],
\end{equation}
which roughly speaking is just obtained from $L^{\bfs,\calchigh}$ by forgetting the $\tau$-component.

The only thing left to check for this flow time $\tau$ to be well-defined as a function on $\RXstarb$ is that the flow time $\tau$ between the left and right components of a point in $\RXstarb$ is unique, but this follows from our non-trapping condition in Definition~\ref{def:non-trapping}, which implies that there is no (lifted) geodesic loop in ${}^{\sct}T^*X$.
\footnote{Notice that, as shown in \eqref{eq:conic-bichar}, those boundary bicharacteristic lines can form loops when projected to $\partial X$, but can't form loops if we take the frequency components into consideration as well. }
In addition, this also removes the restriction being close to the diagonal in \eqref{item:dist-observation-2}.  
We denote this function by $\mk{d}_X$ and it satisfies
\begin{equation} \label{eq:mkd-X}
  \mk{d}_X \in  \rho_{\bfs}^{-1}C^\infty(\RXstarb).
\end{equation}
When $\RXstarb$ projects to $X_{\rmb}^2$ diffeomorphically on some $U \subset \RXstarb$, then this $\mk{d}_X$ can locally be viewed as a function (with conic singularity when approaching the diagonal and $\rho_{\rmb}^{-1}$ growth near the boundary) on $X_{\rmb}^2$ as well and in fact this is what we will exploit.

Now we state the main conclusion of this subsection, which relates the value of our phase function to the generalization above of the distance function.

\begin{proposition} \label{prop:phase-equal-distance-high}
Suppose $\Phi$ parametrizes $L^{\bfs,\calchigh}$ in the sense of \eqref{eq:def-parametrization-high-energy-conic-codim3} in a region $\tilde{U} \subset \hat{L}^{\bfs,\calchigh}$ and is contained in the region $\theta = x'/x \lesssim 1$, with $C_{\Phi}$ as in \eqref{eq:def-C-Phi-high-energy-conic} and $U = \hat{\Pi}_{\smf} (\tilde{U}) \subset \RXstarb$; then we have
\begin{equation} \label{eq:phase-equal-distance-main-high-rb}
 (\paramconic^{-1})^*(\frac{\Phi}{\theta x}\Big|_{C_{\Phi}})  = (\hat{\Pi}_{\smf}|_{\hat{L}^{\bfs,\calchigh}})^*(\mk{d}_{X}),
\end{equation}
on $\tilde{U}$, where $\paramconic$ is the parametrization map sending points in $C_{\Phi}$ to points in $L^{\bfs,\calchigh}$ according to \eqref{eq:def-parametrization-high-energy-conic-codim3}.
Similarly, if it were over a region on which $\sigma = x/x' \lesssim 1$, we have
\begin{equation} \label{eq:phase-equal-distance-main-high-lb}
 (\paramconic^{-1})^*(\frac{\Phi}{\sigma x'}\Big|_{C_{\Phi}})  = (\hat{\Pi}_{\smf}|_{\hat{L}^{\bfs,\calchigh}})^*(\mk{d}_{X}).
\end{equation}
\end{proposition}
 
\begin{proof}
  The conclusion follows directly by observing that the coefficient of $d(\frac{1}{h})$ in \eqref{eq:def-parametrization-high-energy-conic-codim3}, which is what we defined to be $\mk{d}_X$, is precisely $\Phi/\theta x$.
  The remaining step is just pulling back all functions to $\tilde{U} \subset \hat{L}^{\bfs,\calchigh}$ so that we can actually write an identity.
\end{proof}

\begin{corollary} \label{coro:phase-value-lowerbound-high}
    One may choose the microlocal partition in Proposition~\ref{prop:microlocal-partition-high} so that 
    \begin{equation} \label{eq:Phi-lower-bound-off-diagonal}
     |\Phi| \geq \epsilon,        
    \end{equation}
 for a constant $\epsilon>0$ whenever $j \neq j'$.
\end{corollary}

\begin{proof}
      This follows from observing that in either case, the right hand side of \eqref{eq:phase-equal-distance-main-high-rb} or \eqref{eq:phase-equal-distance-main-high-lb} is comparable to the distance (in terms of $\Ymetric$) of the projection of our variables on $Y$, which can be chosen to have a lower bound, say $2\epsilon>0$ on $L^{\bfs,\calchigh}_{j,j'}$ defined in \eqref{eq:Lbf-high-jj'} when $j \neq j'$.
      Consequently, $\Phi$ parametrizing a part in $L^{\bfs,\calchigh}_{j,j'}$ satisfies $|\Phi|_{C_{\Phi}}| \geq 2\epsilon$.
      For any fixed $\epsilon'>0$, by Proposition~\ref{prop:1st-D-phase-lowerbound-bf-low} below and a non-stationary phase argument, we know that the contribution to the oscillatory integral of the part that is at least $\epsilon'$-away from $C_{\Phi}$ is Schwartz. So we can choose $\epsilon'$ small enough so that the phase function satisfies \eqref{eq:Phi-lower-bound-off-diagonal} on the region where it is used in the parametrizations.
\end{proof}

Finally we discuss the situation when we can `descend' $x\mk{d}_X$ to be a function on $X_{\rmb}^2$ and extend it up to and even across the part at distance $\pi$ in terms of their $y$-components.
Consider the case that $U$ is a neighborhood of  $q_0 \in \RXstarb$ with $\sigma = x/x' \lesssim 1$ that is in the image of $\partial \hat{L}^{\bfs,\calchigh}$ and $(L^{\bfs,\calchigh},\Lsharphigh)$ is admissible in the sense of Definition~\ref{definition:Lbf-boundary-admissible}.
Then $x \sim \rho_{\rmb}$ and $x\mk{d}_X$ extends to the projection of $L^{\bfs,\calchigh} \cap \Lsharphigh$, which corresponds to those points with $(y,y')$ components connected by geodesics of length $\pi$ on $Y$, taking value $1+\sigma$ there.\footnote{This corresponds to the choice of the boundary defining functions $x,x'$ in Section~\ref{subsec:Legendrian-geometry-high} so that $\Lsharphigh_\pm$ is given by $\overline{\nu} = \pm (1+\sigma),\mu,\mu'=0$ in terms of coordinates in \eqref{eq:tautological-1-form-bf}. }
Let $(\sigma_0,0,y_0,y_0')$ be the projection of $q_0$ on $X_{\rmb}^2$.
We can further extend $x\mk{d}_X$ so that it takes value $1+\sigma$ to points on $X_{\rmb}^2$ with $(y,y')$ near $(y_0,y_0')$ but at distance larger than $\pi$ on $Y$.
In this way, though $x\mk{d}_X$ is not $C^\infty$ anymore, it is still continuous near $(\sigma_0,0,y_0,y_0')$. 

\subsection{The value of phase functions: the low-energy regime}
\label{subsec:value-phase-low}
Now we turn to the function value of the phase function appearing in Legendre distributions in the low-energy regime.
We denote the (truncated) exact cone over $Y$ by
\begin{equation} \label{eq:def-exact-cone}
\Econe = [0,x_0)_{x} \times Y,
\end{equation}
and equip it with the metric $g_0 = \frac{dx^2}{x^4} + \frac{\Ymetric}{x^2}$, where $\Ymetric$ is the tensor obtained by extending $\Ymetric|_{x=0}$ `by constant in $x$' in terms of the notation in \eqref{conic_metric_2}. 
We also use $X_0^{\circ}=(0,x_0)_{x} \times Y$ to denote its interior.

For the exact cone $X_0$ in \eqref{eq:def-exact-cone}, we can still define its $\rmb$-double space by
\begin{equation}
  (X_0)_{\rmb}^2 = [X_0 \times X_0 ; \partial X_0 \times \partial X_0],
\end{equation}
and denote the blow-down map $ (X_0)_{\rmb}^2  \to X_0 \times X_0$ by $\beta_{\rmb,0}$.
Then we define ${}^{\sct}T^*X_0$ in the same way as ${}^{\sct}T^*X$ using coordinates as in \eqref{eq:sc-1-form} and define
\begin{equation}
  {}^{\Phi}T^*(X_0)_{\rmb}^2 = \beta_{\rmb,0}^*({}^{\sct}T^*X_0 \times {}^{\sct}T^*X_0).
\end{equation}
Coordinates on it are defined in the same way as in the ${}^{\Phi}T^*X_{\rmb}^2$ using the contact form \eqref{eq:tautological-1-form-bf}, which we recall here:
\begin{equation}   \label{eq:tautological-1-form-bf-X0}
\alpha = \overline{\nu} \frac{dx}{x^2} + \nu_1 \frac{d\sigma}{x} + \mu \cdot \frac{dy}{x} + \mu'\cdot \frac{dy'}{x'}.
\end{equation}
 
Then we define
\begin{equation} \label{eq:def-J0-low}
J_0^{\calclow} = \{x=0,\nu_1=0,\mu=0,\mu'=0\},
\end{equation}
which is the analogue of $J^{\calclow}$ in \eqref{eq:def-J-low}.

We also define all those objects in Section~\ref{subsec:value-phase-high} in the same way, except that we only consider geodesics staying within $x \leq x_0 - \delta_0$ for some $\delta_0 \ll x_0$.\footnote{Here we further shrink by $\delta_0$ to avoid the issue of hitting $x=x_0$ when taking the closure.}
In particular, we define the analogue of $\RXstarb$, which we denote by $\mathscr{R}_{X_0,\rmb}^*$, to be
\begin{align}
\begin{split}
\mathscr{R}_{X_0,\rmb}^* : =  & \mathrm{clos} \{ (q,q') \in {}^{\sct}T^*X_0^{\circ} \times {}^{\sct}T^*X_0^{\circ}  : 
 \; 
 \\ & q,q' \text{ lie in the same bicharacteristic lines staying in } \{  x <  x_0 - \delta_0 \} \} 
 \subset [{}^{\Phi}T^*(X_0)_{\rmb}^2;J_0^{\calclow}] ,
 \end{split}
\end{align}
where the closure is, as indicated by the inclusion above, taken in $[{}^{\Phi}T^*(X_0)_{\rmb}^2;J_0^{\calclow}]$.
In the same way as in Section~\ref{subsec:value-phase-high}, we can define a signed distance function on $\mathscr{R}_{X_0,\rmb}^*$ and we denote it by $\mk{d}_{X_0}$.
Let $\rho_{\rmb,0} = (x^{-2}+(x')^{-2})^{-1/2}$, which is the same expression as in \eqref{eq:def-rho-b} but just defined on $(X_0)_{\rmb}^2$ now; then we have
\begin{equation} \label{eq:mkd-X0}
  \mk{d}_{X_0} \in \rho_{\rmb,0}^{-1}C^\infty(\mathscr{R}_{X_0,\rmb}^*).
\end{equation}
In addition, the part of $\mathscr{R}_{X_0,\rmb}^*$ lying over $\bfs$ is exactly $\hat{L}^{\bfs}$:
\begin{equation}
  \mathscr{R}_{X_0,\rmb}^* \cap [{}^{\Phi}T_{\bfs}^*(X_0)_{\rmb}^2;J_0^{\calclow}] = \hat{L}^{\bfs}.
\end{equation}

 


Now we state the result for the phase function on its critical set and the proof is the same as that of Proposition~\ref{prop:phase-equal-distance-high}.

\begin{proposition} \label{prop:phase-equal-distance-X0}
Suppose $\Phi$ parametrizes $L^{\bfs}$ in the sense of Definition~\ref{def:Legendre-parametrization-low} in a region $\tilde{U} \subset L^{\bfs}$ and contained in the region $\theta = x'/x \lesssim 1$, with $C_{\Phi}$ as in \eqref{eq:C-Phi-defn-1} and $U =\projhighlow (\tilde{U}) \subset \mathscr{R}_{X_0,\rmb}^*$, then we have
\begin{equation} \label{eq:phase-equal-distance-main-1}  (\paramconic^{-1})^*(\Phi)/\theta x  =\mk{d}_{X_0},
\end{equation}
on $\tilde{U}$, where $\paramconic$ is the parametrization map sending points in $C_{\Phi}$ to points in $L^{\bfs}$ according to \eqref{def:Legendre-parametrization-low}.
Here $ (\paramconic^{-1})^*(\Phi)$ is a function on $\hat{L}^{\bfs}$ but it can be viewed as a function on $\mathscr{R}_{X_0,\rmb}^*$ by extending `by constant in $\rho_{\bfs,0}$'.

Similarly, if it is over a region on which $\sigma = x/x' \lesssim 1$, we have
\begin{equation} \label{eq:phase-equal-distance-main}
 (\paramconic^{-1})^*(\Phi)/\sigma x'  = \mk{d}_{X_0}.
\end{equation}
\end{proposition}

\begin{remark}
If we were dealing with this $X_0$-case only, then it is almost sufficient to use the proof of \cite[Proposition~2.6]{Hassell-Zhang2016Strichartz}, which shows, 
\begin{equation}
\Phi^2 = 1+\sigma^2 - 2\sigma \cos(s_l-s_r)
\end{equation}
at points where $d_v\Phi=0$, in terms of the parametrization in \eqref{eq:Lbf-definition-gamma^2} and \eqref{LPhiparam}.
We know $s_l - s_r$ is the (signed) length of the geodesic connecting $y,y'$ on $Y$, and \eqref{eq:phase-equal-distance-main-1} follows by the law of cosines. We rewrote it in this way to make it parallel to the high-energy case in which we take the entire $X_{\rmb}^2$ into consideration.
\end{remark}

Then the analogue of Corollary~\ref{coro:phase-value-lowerbound-high} in this setting, which can be proved in the same way except for using Proposition~\ref{prop:phase-equal-distance-X0} instead, is the following result.
\begin{corollary} \label{coro:phase-value-lowerbound-low}
    One may select the microlocal partition in Proposition~\ref{prop:microlocal-partition-low} so that 
    \begin{equation} \label{eq:Phi-lower-bound-off-diagonal-low}
     |\Phi| \geq \epsilon,        
    \end{equation}
 for a constant $\epsilon>0$ whenever $j \neq j'$.
\end{corollary}

\subsection{Lower bound of the first order derivatives of the phase function}

In this subsection, we give a lower bound of the first order derivatives of phase functions parametrizing an intersecting pair of Legendre submanifolds with conic points near their critical set  $C_{\Phi}$ in \eqref{eq:C-Phi-conic-pair-defn}.
Since they vanish on $C_{\Phi}$, the natural lower bound is the distance to $C_\Phi$.
Here the key observation is that $s$ and $v$ should be treated differently: $(s,v)$ should be viewed as `polar coordinates' with $s$ being the radial direction. 
After choosing the local parametrization of $(L^{\bfs},L^{\#})$ as in Section~\ref{subsec:conic-pair-geometry-and-phase-function}, we will use the Euclidean metric
\begin{equation} \label{eq:metric-LCP-parametrization}
 \mk{g} = d\sigma^2 + dy^2 + (dy')^2 + ds^2 + dv^2,
\end{equation}
on $\bfs \times \R_s \times \R^{k}_v$ after fixing coordinates to measure the distance in this part.

Now we give the lower bound of derivatives of the phase function in terms of the distance to $C_{\Phi}$.

\begin{proposition} \label{prop:1st-D-phase-lowerbound-bf-low}
Let $\Phi = 1+\sigma + s\sigma \psi$ be as in \eqref{eq:phase-conic-bf} and suppose it parametrizes  $(L^{\bfs},L^{\calchigh})$ non-degenerately near the image of $q_ 0 \in C_{\Phi} \cap \{s=0\}$ as in \eqref{eq:hat-Lbf-parametrization-bf-2}.
There is a neighborhood of $q_0$ on which
\begin{align} \label{eq:lower-bound-s-v-derivative}
|\partial_s(s\psi)|+|\partial_v\psi| \gtrsim d_{\mk{g}}((\sigma,y,y',v,s),C_{\Phi}),
\end{align}
where $d_{\mk{g}}$ is the distance in terms of Euclidean distance in this coordinate system.
Suppose $\Phi = 1+\sigma + s\sigma \psi$ is instead given by \eqref{eq:conic-pair-phase-construction-high-region1} and parametrizes $(L^{\bfs,\calchigh},\Lsharphigh)$ non-degenerately near the image (in $\hat{L}^{\bfs,\calchigh}$) of $q_0 \in C_{\Phi} \cap \{s=0\}$ and suppose that $(\hat{L}^{\bfs,\calchigh},\Lsharphigh)$ is admissible in the sense of Definition~\ref{definition:Lbf-boundary-admissible}, then there is a neighborhood of $q_0$ on which
\begin{align} \label{eq:lower-bound-s-v-derivative-high}
|\partial_s(s\psi)|+|\partial_v\psi| \gtrsim d_{\mk{g}}((\sigma,y,y',v,s,x'/s),C_{\Phi}).
\end{align}
If $\Phi=1+\sigma+\sigma\psi$ instead parametrizes $L^{\bfs}$ or $L^{\bfs,\calchigh}$ away from $L^{\bfs}\cap \Lsharplow$ or $L^{\bfs,\calchigh} \cap \Lsharphigh$, we have
\begin{align} \label{eq:lower-bound-v-derivative}
|\partial_v\psi(p)| \gtrsim d_{\mk{g}}(p,C_{\Phi}).
\end{align}
\end{proposition}


\begin{proof}
The proof of \eqref{eq:lower-bound-v-derivative} is the same (in fact simpler) than the case when $\Phi$ parametrizes Legendrian conic pairs, for which we give details below.
We consider the low-energy case in which $\Phi = 1+\sigma + s\sigma \psi$ is from \eqref{eq:phase-conic-bf} first.
Notice that \eqref{eq:non-degenerate-parametrization-conic-low} implies that
\begin{equation}  
 d_{(\sigma,y,y',v,s)}\psi, \; d_{(\sigma,y,y',v,s)}(\frac{\partial \psi}{\partial v^j}), \; j=1,2,...,k  \text{ \; are linearly independent at \; } q_0.
\end{equation}
Noticing that $d_{(\sigma,y,y',v,s)}(\partial_s(s\psi))=d_{(\sigma,y,y',v,s)}(\psi)$ at $s=0$, the non-degenerate condition above implies that the map
\begin{equation}
  (\sigma,y,y',v,s) \to (\partial_s(s\psi), \partial_v\psi)
\end{equation}
has non-degenerate derivative at $q_0$, hence the derivative remains non-degenerate near $q_0$, and this implies \eqref{eq:lower-bound-s-v-derivative}.

The high-energy case can be proved in the same way except that all functions have extra $\varrho$-dependence.

\end{proof}

Next we give an analogous estimate that lower-bounds derivatives of the phase in the setting introduced in Section~\ref{subsec:phase-Hessian}. 
Let $\ULCPhigh$ be the neighborhood of $q_0 \in \partial \hat{L}^{\bfs,\calchigh}$ as in Proposition~\ref{prop:Y'I-derivative;low-high-region1-region-def}, then $\hat{L}^{\bfs,\calchigh}$ is locally diffeomorphic to $X_{\rmb}^2$ over it before reaching the lift of $L^{\bfs,\calchigh} \cap \Lsharphigh$. Since $\tau$ (or $\overline{\tau}$) is a function of other variables, passing from $\hat{L}^{\bfs,\calchigh}$ to $\RXstarb$ does not change this diffeomorphic property, and $\RXstarb$ projects to $X_{\rmb}^2$ diffeomorphically over this region as well.
Then let $\mk{d}$ be the distance function defined before \eqref{eq:mkd-X}, as discussed there, we can view the renormalized distance $x\mk{d}$ as a function on $X_{\rmb}^2$ over this region. 
In addition, after further decomposing in $v$ and applying a linear transformation in $y$ (which in turn induces a transformation in frequencies), we may assume that we are in a region satisfying \eqref{eq: hat-mu'-small-delta-1}.

\begin{proposition} \label{prop:1st-dPhi-est}
Suppose $\Phi$ is as in \eqref{eq:conic-pair-phase-construction-high-region1} and parametrizes $(L^{\bfs,\calchigh},\Lsharphigh)$ non-degenerately near the image (in $\hat{L}^{\bfs,\calchigh}$) of $q_0 \in C_{\Phi} \cap \{s=0\}$ in a neighborhood as above and suppose that $(\hat{L}^{\bfs,\calchigh},\Lsharphigh)$ is admissible in the sense of Definition~\ref{definition:Lbf-boundary-admissible}, then there is a neighborhood of $q_0$ on which
\begin{equation} \label{eq:5.5-1} 
 \sigma^{1/2} s |x\mk{d}_X - \Phi|^{1/2}  \lesssim \big(|\partial_v\Phi| + |s\partial_s\Phi| \big).
\end{equation}
The same conclusion applies to the low-energy setting except that $\Phi$ is assumed to parametrize $(L^{\bfs},\Lsharplow)$ (hence no $\varrho$ dependence now) and $\mk{d}_X$ is replaced by $\mk{d}_{X_0}$.

\end{proposition}
Here $x\mk{d}_X$ is viewed as a function on $X_{\rmb}^2$ (as discussed at the end of Section~\ref{subsec:value-phase-high}) and lifted to $X_{\rmb}^2 \times \R^k_v \times [0,s_0)_s$.

\begin{proof}

Notice that \eqref{eq:5.5-1} is equivalent to 
\begin{equation}
 |x\mk{d}_X - \Phi|^{1/2}  
\lesssim \sigma^{1/2} (|y_I'-Y_I'|+|y'_{n-1}-Y'_{n-1}|),
\end{equation}
or
\begin{equation} \label{eq:5.5-3}
|x\mk{d}_X(\sigma,y,y',x') - \Phi| 
\lesssim \sigma (|y_I'-Y_I'|^2+|y'_{n-1}-Y'_{n-1}|^2).
\end{equation}

The left hand side vanishes when the right hand side vanishes by Proposition~\ref{prop:phase-equal-distance-high}.
In addition, by Proposition~\ref{prop:LCP-phase-derivatives}, the condition that $\partial_v\Phi = 0$ and $\partial_s\Phi = 0$ is equivalent to $y'_I - Y'_I = 0$ and $y'_{n-1}-Y'_{n-1} = 0$.
So the left hand side of \eqref{eq:5.5-3} has vanishing derivative in $v,s$ as well when the right hand side vanishes, hence it is controlled by the right hand side. 
Here we have the extra $\sigma$-factor because, as discussed, $\Phi$ can be written as $1+\sigma+s\sigma\psi$, or more precisely because of the particular form of the $N$-component in  
\eqref{eq:hatLbf-low-components} and \eqref{eq:hatLbf-high-components-region1}.
In addition, derivatives of $x\mk{d}_X$ on right variables (i.e., variables with primes) are $O(\sigma)$ as well, since they are $O(1/x')$ without the $x$ factor.


\end{proof}

For the region with $|\mu'|/x' \lesssim 1$, we have the following estimate for derivatives.
\begin{proposition} \label{prop:1st-dPhi-est-region2}
Suppose $\Phi$ is as in \eqref{eq:conic-pair-phase-construction-high-region2} and parametrizes $(L^{\bfs,\calchigh},\Lsharphigh)$ non-degenerately near the image (in $\hat{L}^{\bfs,\calchigh}$) of $q_0 \in C_{\Phi} \cap \{s=0\}$ in a neighborhood as above and suppose that $(\hat{L}^{\bfs,\calchigh},\Lsharphigh)$ is admissible in the sense of Definition~\ref{definition:Lbf-boundary-admissible}, then there is a neighborhood of $q_0$ on which
\begin{equation} \label{eq:5.5-2} 
 x' \sigma^{1/2} |x\mk{d}_X - \Phi|^{1/2}  \lesssim \big(|\partial_v\Phi| + |\partial_\varkappa \Phi| \big).
\end{equation}
\end{proposition}

\begin{proof}
    Using the second part of Proposition~\ref{prop:LCP-phase-derivatives} or more concretely \eqref{eq:5.1-2}, we know that \eqref{eq:5.5-2} is again equivalent to \eqref{eq:5.5-3} (now the previous $s$-factor is replaced by this $x'$-factor and again cancels out)
   except that $Y'_I,Y'_{n-1}$ are interpreted as those ones in \eqref{eq:hatLbf-high-components-region2}.
Then the rest of the proof remains the same.
\end{proof}

\section{\texorpdfstring{The pointwise bound of Legendre distributions\\ around the conic points of the intersecting Legendre pair}{The pointwise bound of Legendre distributions around the conic points of the intersecting Legendre pair}}
\label{sec:conic-points-pointwise-bound}

In this section, we prove the pointwise bound for Legendre distributions associated with the intersecting Legendre pair with conic points.
This estimate will demonstrate that, as stated in Section~\ref{subsec:Strategy-of-proof}, the geometric focusing at the end of $L^{\bfs}$ and $L^{\bfs,\calchigh}$ (i.e., when they meet $\Lsharplow$ and $\Lsharphigh$ respectively) won't affect dispersive estimates. 
In the low-energy setting, this corresponds to conjugate point pairs connected by a geodesic of length $\pi$ on the cross section $Y$.
In the high-energy setting, this not only includes those conjugate point pairs on $Y$, but also includes those conjugate point pairs on $\partial X \times \partial X$ connected by a geodesic going through the interior of $X$. In particular, this allows us to include dispersive estimates on Euclidean spaces (or with perturbations by scaling-critical potentials) as a special case of our results.

\subsection{Pointwise bound of Legendre distributions around the conic points of the intersecting Legendre pair: the low-energy part}
\label{subsec:conic-points-pointwise-bound-low}

In this subsection, we prove a pointwise bound for the Legendre distributions associated with intersecting pairs of Legendre submanifolds with conic points in the low-energy setting. 
It shows that when the intersecting pair $(L^{\bfs},\Lsharplow)$ is admissible, Legendre distributions associated with conic pairs obey the same  pointwise bound that does not depend on the degeneracy of the projection from $\hat{L}^{\bfs}$.
As mentioned in Section~\ref{subsec:Strategy-of-proof}, this confirms the claim that the geometric focusing on $Y$ exactly at time $\pi$ won't affect dispersive estimates.
This is achieved by exploiting the precise order of degeneracy of the parametrization Hessian in $v$ that we obtained in Section~\ref{sec:projection-and-phase}, which reflects the geometric structure of our intersecting pairs.

\begin{proposition} \label{prop:conic-pair-pointwise-bound-low}
Suppose that $\Phi(\sigma,y,y',v,s)$ is as in \eqref{eq:conic-pair-phase-construction-low}, with $v=\hat{\mu'}_I,s=\mu_{n-1}$, and parametrizes $(L^{\bfs},L^\#)$ over a region as in Proposition~\ref{prop:Y'I-derivative;low-high-region1-region-def} (i.e. the projection $\hat{L}^{\bfs} \to \bfs$ only degenerates at the lift of $L^{\bfs} \cap \Lsharplow$) and $(L^{\bfs},\Lsharplow)$ is admissible in the sense of Definition~\ref{definition:Lbf-boundary-admissible}, then
\begin{align} \label{eq:conic-pair-integral-bound-low}
\begin{split}
\Big|& \lambda^{n-1}  \int_0^\infty \int_{\R^{k}}  e^{i\lambda\Phi(\sigma,y,y',v,s)/x} \big(\frac{x'}{\lambda s}\big)^{(n-1)/2 - (k+1)/2}
  s^{n-2} \sigma^{\frac{n-1}{2}} a(\lambda,\sigma,y,y',\frac{x'}{\lambda},v,s) \, dv \, ds  \Big|
\\& \lesssim \lambda^{n-1} (x'/\lambda)^{(n-1)/2}\sigma^{(n-1)/2}
\end{split}
\end{align}
on the region $x/x' = \sigma \leq 2, x'/\lambda \leq 2$. The estimate near $\bfs \cap \rb$, with $x,x'$ exchanged and $\sigma$ replaced by $\theta=x'/x$ also holds.
\end{proposition}


\begin{proof}
By \eqref{eq:hatLbf-low-components}, 
we know that the phase function we are using, i.e. $\Phi(\sigma,y,y',v,s)$ in \eqref{eq:conic-pair-phase-construction-low}, can be written in the form:
\begin{align}
\Phi(\sigma,y,y',v,s) = 1+ \sigma + s \sigma \psi(\sigma,y,y', v,s).
\end{align}
So \eqref{eq:conic-pair-integral-bound-low} is equivalent to 
\begin{align} \label{eq:conic-osc-est-2}  
\Big|  \int_0^\infty \int_{\R^{k}}  e^{i\lambda s \sigma \psi(\sigma,y,y',v,s)/x} 
  s^{\frac{n}{2}-1+ \frac{k}{2}} a(\lambda,\sigma,y,y',\frac{x'}{\lambda},v,s) \, dv \, ds \Big| \lesssim \big(\frac{\lambda}{x'}\big)^{-(k+1)/2} ,
\end{align}
and the non-trivial part to justify is that this estimate holds uniformly down to $\frac{x'}{\lambda} \to 0$ over the region $x/x' \lesssim 1$.

We decompose the left hand side of \eqref{eq:conic-osc-est-2} into two parts according to $s \lesssim (\lambda/x')^{-1/2}$ and $s \gtrsim (\lambda/x')^{-1/2}$ and use different change of variables for them to overcome the degeneracy of the Hessian of the phase function on the left hand side of \eqref{eq:conic-osc-est-2} at its critical points. Concretely, denoting the left hand side of \eqref{eq:conic-osc-est-2} by $\mk{I}$, then we write
\begin{align} \label{eq:mkI-decomposition-1}
\mk{I} = \mk{I}_1 + \mk{I}_2,
\end{align}
where 
\begin{align}\label{eq:def-mkI-1,2}
\begin{split}
\mk{I}_1 = & \int_0^\infty \int_{\R^{k}} 
\chi_1(s (\lambda/x')^{1/2})
e^{i\lambda s \sigma \psi(\sigma,y,y',v,s)/x} 
  s^{\frac{n}{2}-1+ \frac{k}{2}} a(\lambda,\sigma,y,y',\frac{x'}{\lambda},v,s) \, dv \, ds, \\
\mk{I}_2 = & \int_0^\infty \int_{\R^{k}}  
(1-\chi_1(s (\lambda/x')^{1/2}))
e^{i\lambda s \sigma \psi(\sigma,y,y',v,s)/x} 
  s^{\frac{n}{2}-1+ \frac{k}{2}} a(\lambda,\sigma,y,y',\frac{x'}{\lambda},v,s) \, dv \, ds,
\end{split}
\end{align}
where $\chi_1(\cdot) \in C_c^\infty(\R)$ is supported in $[-2,2]$ and is identically one on $[-1,1]$.

Recall the form of amplitude in Definition~\ref{def:Legendrian-dis-conic-intersecting-low}, our amplitude $a(\lambda,\frac{x'}{\lambda s},\sigma,y,y',s,v)$ is supported in the region $s \gtrsim (\lambda/x')^{-1}$, so $a(\lambda,\frac{\lambda}{x's},\sigma,y,y',s,v)$ can be extended by zero to $s<0$ and the oscillatory integral can be written as an integral over $\R^{k+1}$.

Estimating $\mk{I}_1$ is straightforward: we have
\begin{align*}
\mk{I}_1 \lesssim \int_0^{2(\lambda/x')^{-1/2}}   s^{\frac{n}{2}-1+ \frac{k}{2}} ds
\lesssim  \Big( (\lambda/x')^{-1/2} \Big)^{\frac{n+k}{2}} \lesssim (\lambda/x')^{ -\frac{k+1}{2} }.
\end{align*}

For $\mk{I}_2$, we have $s\gtrsim (\lambda/x')^{-1/2}$ on the support of its amplitude. 
We consider two cases by a further decomposition in $s$ first. 
Let $\chi_{\srange}(s)$ be the characteristic function of $\srange$ defined in \eqref{eq:srange-def}, then we write
\begin{equation}  \label{eq:mkI2-decomposition-srange}
  \mk{I}_2 = \mk{I}^{\srange}_{2} + \mk{I}_{2}^{\srange^c},
\end{equation}
where 
\begin{align}
\begin{split}
\mk{I}_{2}^{\srange} = & \int_0^\infty \int_{\R^{k}}  
(1-\chi_1(s (\lambda/x')^{1/2})) \chi_{\srange}(s)
e^{i\lambda s \sigma \psi(\sigma,y,y',v,s)/x} 
  s^{\frac{n}{2}-1+ \frac{k}{2}} a(\lambda,\sigma,y,y',\frac{x'}{\lambda},v,s) \, dv \, ds,
\\ \mk{I}_{2}^{\srange^c} = & \int_0^\infty \int_{\R^{k}}  
(1-\chi_1(s (\lambda/x')^{1/2})) (1-\chi_{\srange}(s))
e^{i\lambda s \sigma \psi(\sigma,y,y',v,s)/x} 
  s^{\frac{n}{2}-1+ \frac{k}{2}} a(\lambda,\sigma,y,y',\frac{x'}{\lambda},v,s) \, dv \, ds.
  \end{split}
\end{align}

For $\mk{I}_{2}^{\srange^c}$, we have \eqref{eq:lower-bound-Ec}, which gives $|\partial_v\psi| \gtrsim s$ on the support of the amplitude. Then we integrate by parts in $v$ using the differential operator
\begin{equation} \label{eq:L1-def}
  L_1 = \frac{1}{|\partial_v \psi|^2} \sum_j \partial_{v_j}\psi \partial_{v_j},
\end{equation}
which has the (formal) adjoint
\begin{equation} \label{eq:L1-adjoint-def}
  L_1^* = - \sum_j \partial_{v_j} \frac{1}{|\partial_v \psi|^2} \partial_{v_j}\psi.
\end{equation}
More concretely, we write
\begin{equation}
    e^{i(\lambda/x')s\psi} = \big(\frac{\lambda}{x'}s\big)^{-1}L_1 e^{i(\lambda/x')s\psi},
\end{equation}
and then integrate by parts. 
Then we can classify terms by which factor the $\partial_{v_j}$ hits and all terms are bounded (modulo a constant) by $(\frac{\lambda}{x'})^{-1}|\partial_v\psi|^{-2}$ or $(\frac{\lambda}{x'})^{-1}|\partial_v\psi|^{-3}\partial_{v_iv_j}\psi$.
Noticing that $\partial_{v_jv_j}\psi$ is $O(s)$ as well by \eqref{eq:Hessian-psi-nondegenerate-low-highregion1}, 
each time we obtain a factor $O((\frac{\lambda}{x'})^{-1}s^{-2})$.
Then, if $\frac{k+2}{2} \in \N$, after $\frac{k+2}{2}$ times integration by parts, we have (notice that no smoothness in $s$ is involved for this part)
\begin{equation}
 |\mk{I}_{2}^{\srange^c}| \lesssim  (\frac{\lambda}{x'})^{-(k+2)/2}\int_{(\frac{\lambda}{x'})^{-1/2} \lesssim s \lesssim 1} s^{-2}ds \lesssim
  (\frac{\lambda}{x'})^{-(k+1)/2},
\end{equation}
where we used that the power of $s$ in the original integral is $\frac{n-2+k}{2} \geq k$ by Proposition~\ref{prop:minimal-parametrization-low}. When $\frac{k+3}{2} \in \N$ instead, we integrate by parts $\frac{k+3}{2}$ times to obtain 
\begin{equation} \label{eq:bound-mkI2-Ec-low}
  |\mk{I}_{2}^{\srange^c}| \lesssim  (\frac{\lambda}{x'})^{-(k+3)/2}\int_{(\frac{\lambda}{x'})^{-1/2} \lesssim s \lesssim 1} s^{-3}ds \lesssim
  (\frac{\lambda}{x'})^{-(k+1)/2}.
\end{equation}

Now we consider $\mk{I}_{2}^{\srange}$. For $s \in \srange$, we know there is a unique $v = \hat{M'}_I$ defined by \eqref{eq:def-hat-M'-I} such that the phase is critical with respect to $v$ (which is $\hat{\mu'}_I$) and the Hessian is non-degenerate after multiplying $s^{-1}$.
Applying the stationary phase lemma in $v$ with $s^2\frac{\lambda}{x'}$ as the large parameter and Corollary~\ref{coro:Hessian-lowerbound-low;high-region1} (with $v=\hat{\mu'}_I$), we have
\begin{align} \label{eq:6.1-1}
\begin{split}
\int_{\R^{k}} &  e^{i\frac{\lambda}{x'}s^2 (s^{-1}\psi) } s^{\frac{n}{2}-1+ \frac{k}{2}} a(\lambda,\sigma,y,y',\frac{x'}{\lambda},s,v) \, dv
\\ = & C (|\det(s^{-1}\partial^2_{vv}\psi)|)^{-1/2} e^{i (s\frac{\lambda}{x'}\psi)}  
\Big( (s^2\frac{\lambda}{x'})^{-k/2}  s^{\frac{n-2+k}{2}} a(\sigma,y,y',s,v) 
  \\ & + (\lambda/x')^{- (k+1)/2} s^{\frac{n-2-k}{2} - 1 }) La
  \Big) \Big|_{v= \hat{M}'_I}  +O( (\frac{\lambda}{x'})^{-(k+2)/2}s^{-2}),
\end{split}
\end{align}
where $L$ is a second order differential operator with smooth coefficients.
 
Notice that though the first order derivative of the phase in $v$
\begin{equation} \label{eq:6.1-2}
  \partial_v(s^{-1}\psi) = s^{-1}(y_I'-Y_{I,0}'(\sigma,y,y'_{II})) + O(s) ,
\end{equation}
is not uniformly bounded as $s \to 0$, but as the proof in \cite[Theorem~7.7.5]{hormanderbookvolI} shows, the reduction from this general phase to the exact quadratic one only needs a uniform bound of derivatives of the phase of the third or more orders. 
After reducing to the expansion of the exact quadratic phase, this linear term $s^{-1} (y_I'-Y_{I,0}'(\sigma,y,y'_{II}) ) \cdot v$ in the phase only causes a translation in the critical point and an overall shift of the phase, without changing the estimate of the remainder.

The $s$-integral of the $O( (\frac{\lambda}{x'})^{-(k+2)/2}s^{-2})$-remainder term over the region $s \gtrsim (\frac{\lambda}{x'})^{-1/2}$ is  bounded by
\begin{equation} \label{eq:est-mkI-23}
  (\frac{\lambda}{x'})^{-(k+2)/2}\int_{(\frac{\lambda}{x'})^{-1/2} \lesssim s \lesssim 1} s^{-2}ds \lesssim
  (\frac{\lambda}{x'})^{-(k+1)/2}.
\end{equation}

So it remains to consider the integral of the first two terms in the bracket on the right hand side of \eqref{eq:6.1-1}:
\begin{equation} \label{eq:6.2-mkI-21}
\mk{I}^{\srange}_{2, 1} = (\frac{\lambda}{x'})^{-k/2} \int \chi_{\srange}(s)(1-\chi_1(s (\lambda/x')^{1/2})) (\det(s^{-1}\partial^2_{vv}\psi))^{-1/2} e^{i (\frac{\lambda}{x'}s\psi)}  
\Big(   s^{\frac{n-2-
k}{2}} a(\sigma,y,y',s,v) 
  \Big) \Big|_{v= \hat{M}'_I} ds,
\end{equation}
and
\begin{equation} \label{eq:6.2-mkI-22}
\mk{I}^{\srange}_{2,2} = (\frac{\lambda}{x'})^{-(k+1)/2} \int   \chi_{\srange}(s)(1-\chi_1(s (\lambda/x')^{1/2})) (\det(s^{-1}\partial^2_{vv}\psi))^{-1/2} e^{i(\frac{\lambda}{x'}s\psi)}  
\big( s^{\frac{n-2-k}{2} - 1 } La \big) \Big|_{v= \hat{M}'_I} ds,
\end{equation}

We treat $\mk{I}^{\srange}_{2,1}$ in detail and the estimate for $\mk{I}^{\srange}_{2,2}$ can be derived in the same way.\footnote{Though the power of $s$ is `one order worse' in $\mk{I}^{\srange}_{2,2}$ compared with $\mk{I}^{\srange}_{2,1}$, on the other hand we have an extra $(\lambda/x')^{-1/2}$ in this term and we are in the region $s \geq (\lambda/x')^{-1/2}$, so the procedure estimating $\mk{I}^{\srange}_{2,1}$ applies to $\mk{I}^{\srange}_{2,2}$ as well. }
First of all, although $\hat{M'}_I$ is only defined for $s \in \srange$, this does not affect the oscillatory integral since we can just extend the integrand to be zero outside $\srange$, which will not make the regularity worse since we already have the factor $\chi_{\srange}(s)$.
In addition, this $\chi_{\srange}(s)$-factor does not worsen the regularity of the integrand because for $s$ near $\partial \srange$, $\hat{M}'_I$ is already outside $\supp \, a(\sigma,y,y',\bullet,s)$ in $\hat{\mu'}_I$.
Consequently, $\chi'_{\srange}(s)$ (which is the Dirac distribution on $\partial \srange$) won't contribute to  the derivative of the integrand in $s$.

Let $C>0$ and three regions be as in Corollary~\ref{coro:restricted-phase-three-regions}. We consider a partition of unity
\begin{equation} \label{eq:def-varrho-i}
\begin{split}
  1 =  \varrho_1\Big(\frac{s}{|y_{n-1}'-Y_{n-1,0}(\sigma,y,y'_{II})|}\Big)&+\varrho_2\Big(\frac{s}{|y_{n-1}'-Y_{n-1,0}(\sigma,y,y'_{II})|}\Big)\\&+\varrho_3\Big(\frac{s}{|y_{n-1}'-Y_{n-1,0}(\sigma,y,y'_{II})|}\Big),\end{split}
\end{equation}
where $\varrho_1$ is identically 1 on $[0,\frac{1}{2C}]$ and supported in $[\frac{1}{2C},\frac{1}{C}]$, $\varrho_2$ is identically 1 on $[\frac{1}{C},C]$ and supported in $[\frac{1}{2C},2C]$, $\varrho_3$ is identically one on $[2C,\infty)$ and supported in $[C,\infty)$.
Then we decompose $\mk{I}^{\srange}_{2,1}$ in \eqref{eq:6.2-mkI-21} into three parts:
\begin{equation} \label{eq:mkI-21-decomposition}
  \mk{I}^{\srange}_{2,1} = \mk{I}^{\srange}_{2,1,1}+ \mk{I}^{\srange}_{2,1,2} + \mk{I}^{\srange}_{2,1,3},
\end{equation}
where
\begin{align}
  \begin{split}  
  \mk{I}^{\srange}_{2,1,i} = & (\frac{\lambda}{x'})^{-k/2} \int_0^\infty  \chi_{\srange}(s)(1-\chi_1(s (\lambda/x')^{1/2})) \varrho_i\big(s/|y_{n-1}'-Y'_{n-1,0}(\sigma,y,y'_{II})|\big) \\&\big(\det(s^{-1}\partial^2_{vv}\psi)\big)^{-1/2} e^{i (\frac{\lambda}{x'}s\psi)} 
 \Big(   s^{\frac{n-2-k}{2}} a(\sigma,y,y',s,v) 
  \Big) \Big|_{v= \hat{M}'_I} ds.
  \end{split}
\end{align}

Consider $\mk{I}^{\srange}_{2,1,1}$ first. On the support of $\varrho_1(s/|y_{n-1}'-Y'_{n-1,0}(\sigma,y,y'_{II})|)$, by Corollary~\ref{coro:restricted-phase-three-regions}, the contribution of the term $s(y_{n-1}'-Y'_{n-1,0}(\sigma,y,y'_{II}) )$ to $\partial_s\Big(s\psi|_{\hat{\mu'}_{I} = \hat{M'}_I }\Big)$ dominates all other terms and we have
\begin{equation} \label{eq:ineq-6.1-2}
  \Big|\partial_s\Big(s\psi|_{\hat{\mu'}_{I} = \hat{M'}_I }\Big)\Big| \gtrsim s.
\end{equation}
So we can integrate by parts in $s$ and gain a $(\lambda/x')^{-1}$ factor.
Using Lemma~\ref{lemma:M1-s-derivative-bound} to bound derivatives of $\hat{M}'_I$, whenever the derivative hits other factors of the amplitude, it loses at most an $s^{-1},|y_{n-1}'-Y_{n-1,0}'(\sigma,y,y'_{II})|^{-1}$ or $(\lambda/x')^{1/2}$-factor, all of which are bounded by $s^{-1}$ (modulo a constant) on the support of the amplitude. For example, whenever $|y_{n-1}'-Y_{n-1,0}'(\sigma,y,y'_{II})|^{-1}$ arises, we are on the support of $\varrho'_1$; hence it is comparable to $s^{-1}$, and similarly for $(\lambda/x')^{1/2}$.
Together with \eqref{eq:ineq-6.1-2}, we know 
\begin{equation} \label{eq:ineq-6.1-5}
  |\mk{I}^{\srange}_{2,1,1}| \lesssim (\frac{\lambda}{x'})^{-k/2} (\frac{\lambda}{x'})^{-1} \int_{(\lambda/x')^{-1/2}}^1 s^{-2}ds \sim  (\frac{\lambda}{x'})^{-(k+1)/2}.
\end{equation}

Next we consider $\mk{I}^{\srange}_{2,1,3}$.
On the support of $\varrho_3(s/|y_{n-1}'-Y'_{n-1,0}(\sigma,y,y'_{II})|)$, we have $|y_{n-1}'-Y'_{n-1,0}(\sigma,y,y'_{II})| \geq Cs$.
Using Corollary~\ref{coro:restricted-phase-three-regions} again, we have
\begin{equation} \label{eq:ineq-6.1-3}
  \Big|\partial_s\Big(s\psi|_{\hat{\mu'}_{I} = \hat{M'}_I }\Big)\Big| \gtrsim s.
\end{equation}
Whenever the derivative hits the amplitude when we integrate by parts, we again at most lose $s^{-1}$. (Even though $s^{-1}$ does not bound $|y_{n-1}'-Y'_{n-1,0}(\sigma,y,y'_{II})|^{-1}$ on this entire region anymore, but this is still true on the support of $\varrho_3'$, which is the only place where we lose $|y_{n-1}'-Y'_{n-1,0}(\sigma,y,y'_{II})|^{-1}$.)
So similar to \eqref{eq:ineq-6.1-5}, we have 
\begin{equation}
  |\mk{I}^{\srange}_{2,1,3}| \lesssim (\frac{\lambda}{x'})^{-k/2} (\frac{\lambda}{x'})^{-1} \int_{(\lambda/x')^{-1/2}}^1 s^{-2}ds \sim  (\frac{\lambda}{x'})^{-(k+1)/2}.
\end{equation}

Now we consider the term $\mk{I}^{\srange}_{2,1,2}$. By \eqref{eq:phase-restricted-convex}, we have $|\partial_{ss}(s\psi|_{\hat{\mu'}_I = \hat{M'}_I})| \gtrsim 1$ on the support of the integrand and we can apply Van der Corput's lemma (in the form of \cite[Section~VIII.1.2, Corollary]{Stein-harmonic-textbook}) to estimate the contribution of this term: 
\begin{align} \label{eq:05}
\begin{split}
  (\frac{\lambda}{x'})^{-k/2}  &\int_0^\infty \varrho_2\big(s/|y_{n-1}'-Y'_{n-1,0}(\sigma,y,y'_{II})|\big)  (1-\chi_1(s (\lambda/x')^{1/2})) 
\\ & \Big(\det(s^{-1}\partial^2_{vv}\psi)\Big)^{-1/2} e^{i (s\frac{\lambda}{x'}\psi)}  s^{\frac{n-2-k}{2}} a(\sigma,y,y',s,v)|_{v = \hat{M'}_I(\sigma,y,y'_{II},s)}  ds
\\ \lesssim 
(\frac{\lambda}{x'})^{-(k+1)/2} &\int_0^C \Big|\frac{d}{ds}\Big(\varrho_2\big(s/|y_{n-1}'-Y'_{n-1,0}(\sigma,y,y'_{II})|\big) (1-\chi_1(s (\lambda/x')^{1/2})) \\& s^{\frac{n-2-k}{2}}  \Big((\det(s^{-1}\partial^2_{vv}\psi))^{-1/2}  a(\sigma,y,y',s,v)\Big)|_{v = \hat{M'}_I(\sigma,y,y'_{II},s)} \Big)\Big| ds.
\end{split}
\end{align}

Setting
\begin{align*}
& g_1 = \varrho_2\big(s/|y_{n-1}'-Y'_{n-1,0}(\sigma,y,y'_{II})|\big) , \quad 
g_2 = (1-\chi_1(s (\lambda/x')^{1/2})), 
\\ & g_3 = \Big( s^{\frac{n-2-k}{2}} \big(\det(s^{-1}\partial^2_{vv}\psi)\big)^{-1/2}  a(\sigma,y,y',s,v)\Big)|_{v = \hat{M'}_I(\sigma,y,y'_{II},s)},
\end{align*}
then the integral in \eqref{eq:05} is $\|\frac{d}{ds}(g_1g_2g_3)\|_{L^1}$ and we have
\begin{align*}
\Big\| \frac{d}{ds}(g_1g_2g_3)\Big\|_{L^1} \lesssim \|g_1'g_2g_3\|_{L^1} + \|g_1g_2'g_3\|_{L^1} + \|g_1g_2g_3'\|_{L^1}.
\end{align*}

For the first term on the right hand side, though the upper-bound of $|g_1'g_2g_3|$ is $O(|y_{n-1}'-Y'_{n-1,0}(\sigma,y,y'_{II})|^{-1})$, it is supported on an interval of length $O(|y_{n-1}'-Y'_{n-1,0}(\sigma,y,y'_{II})|)$, hence this term is $O(1)$.
Similarly, the amplitude in the second term is $O( (\lambda/x')^{1/2})$ and is supported on an interval of length $O((\lambda/x')^{-1/2})$, hence this term is $O(1)$ as well.
For the third term, $g_3'$ has two types of contributions, those from $\partial_sa, \, \partial_s\psi$ are clearly $O(1)$. Now we consider those contributions from
\begin{align*}
\partial_s\hat{M'}_I \cdot \partial_v a, \quad \partial_s\hat{M'}_I \cdot \partial_v((\det(s^{-1}\partial^2_{vv}\psi))^{-1/2}).
\end{align*}
Recall that as mentioned before \eqref{eq:6.1-1}, the factor $(\det(s^{-1}\partial^2_{vv}\psi))^{-1/2}$ is smooth by Corollary~\ref{coro:Hessian-lowerbound-low;high-region1}.
Using Lemma~\ref{lemma:M1-s-derivative-bound}, we know 
\begin{align*}
    |\partial_s\hat{M'}_I| \lesssim s^{-1}.
\end{align*} 
Combining with the support condition of the factor of $\varrho_2$, we have 
\begin{equation}
  \|g_1g_2g_3'\|_{L^1} \lesssim 
  \int_{(2C)^{-1}|y_{n-1}'-Y'_{n-1,0}(\sigma,y,y'_{II})|}^{2C|y_{n-1}'-Y'_{n-1,0}(\sigma,y,y'_{II})|} s^{-1} ds \lesssim 1,
\end{equation}
and we have completed the proof of \eqref{eq:conic-pair-integral-bound-low}.
\end{proof}

\subsection{Pointwise bound of Legendre distributions around the conic points of the intersecting Legendre pair: the high-energy part}
\label{subsec:conic-points-pointwise-bound-high}

Now we discuss the high-energy analogue of Proposition~\ref{prop:conic-pair-pointwise-bound-low}.
Consider the case in the region $\varrho = x'/s \lesssim 1$ first.

\begin{proposition} \label{prop:conic-pair-pointwise-bound-high-region1}
Suppose that $\Phi(\sigma,x',y,y',v,s,\varrho)$ as in \eqref{eq:conic-pair-phase-construction-high-region1} parametrizes $(L^{\bfs,\calchigh},\Lsharphigh)$ over a region as in Proposition~\ref{prop:Y'I-derivative;low-high-region1-region-def} (i.e. the projection $\hat{L}^{\bfs,\calchigh} \to \smf$ only degenerates at the lift of $L^{\bfs,\calchigh} \cap \Lsharphigh$), and satisfies the conditions in \eqref{eq:defn-condition-YI-Yn-high-region1}, then
\begin{align} \label{eq:conic-pair-integral-bound-high}
\begin{split}
\Big|& \lambda^{n-1}  \int_0^\infty \int_{\R^{k}}  e^{i\lambda\Phi(\sigma,y,y',v,s,\varrho)/x} \big(\frac{x'}{\lambda s}\big)^{(n-1)/2 - (k+1)/2} \sigma^{(n-1)/2}
  s^{n-2} a(\lambda,\sigma,y,y',\frac{x'}{\lambda},v,s) \, dv \, ds  \Big|
\\& \lesssim \lambda^{n-1} (x'/\lambda)^{(n-1)/2}\sigma^{(n-1)/2}
\end{split}
\end{align}
on the region $x/x' = \sigma \leq 2$. The estimate near $\bfs \cap \rb$, with $x,x'$ exchanged and $\sigma$ replaced by $\theta=x'/x$ also holds.
\end{proposition}

The proof of Proposition~\ref{prop:conic-pair-pointwise-bound-low} still applies in this setting except that we now use the part concerning the high-energy regime (only the part $\varrho \lesssim 1$) of the lemmas referred to in Sections~\ref{subsec:phase-Hessian}-\ref{subsec:critical_points_at_fixed_s} and we sketch it here for completeness. We still denote the oscillatory integral on the left hand side of \eqref{eq:conic-pair-integral-bound-high} (without the $h^{-(n-1)}$-factor) by $\mk{I}$.
The first step of decomposing $\mk{I}$ as in \eqref{eq:mkI-decomposition-1} and the way in which $\mk{I}_1$ can be bounded remain completely the same.
For $\mk{I}_2$, we still follow the strategy of considering the integral at fixed $s$ first.
Then we still decompose according to whether $y'_I$ is in the range of $Y'_I$ or not and denote the set of $s$ such that $y'_I$ is in this range by $\srangehigh$.
Then we decompose 
\begin{equation} \label{eq:mkI2-decomposition-srange-high}
  \mk{I}_2 = \mk{I}^{\srangehigh}_2 + \mk{I}^{(\srangehigh)^c}_2
\end{equation}
as in \eqref{eq:mkI2-decomposition-srange}.
We still use $\delta_0$ to denote the distance between the boundary of the support of $a(\sigma,y,y'_{II},\bullet,s,x'/s)$ and the boundary of the domain of $Y'_I(\sigma,y,y'_{II},\bullet,s,x'/s)$.
When $s \notin \srangehigh$, we know that $Y'_I$ is (up to a constant) at least $\delta_0 s$-away from $y'_I$ when restricted to the support of $a(\sigma,y,y'_{II},\bullet,s,x'/s)$. Then we integrate by parts using $L_1$ in \eqref{eq:L1-def} and bound $\mk{I}^{(\srangehigh)^c}_2$ as in \eqref{eq:bound-mkI2-Ec-low}.

Then we consider $\mk{I}^{\srangehigh}_2$, for which we have an oscillatory integral with a unique critical point $\hat{M'}_I$ now.
We apply the stationary phase method as in \eqref{eq:6.1-1}, where the estimate on the Hessian uses the part concerning the high-energy regime in Corollary~\ref{coro:Hessian-lowerbound-low;high-region1}.
Then we have the same expansion as in \eqref{eq:6.1-1} and denote the $s$-integral of those three terms (with extra $\varrho$-dependence now) there by 
\begin{equation}
  \mk{I}^{\srangehigh}_{2,1} , \, \mk{I}^{\srangehigh}_{2,2} , \, \mk{I}^{\srangehigh}_{2,3}
\end{equation}
respectively. The estimate for $\mk{I}^{\srangehigh}_{2,3}$ is still the same as in \eqref{eq:est-mkI-23}. 
The ways to bound $\mk{I}^{\srangehigh}_{2,1}$ and $\mk{I}^{\srangehigh}_{2,2}$ are the same and we consider $\mk{I}^{\srangehigh}_{2,1}$ to be parallel to the low-energy case.

For the phase restricted to $\hat{\mu'}_I = \hat{M'}_I$, as pointed out in Lemma~\ref{lemma:phase-restricted-convex-low-high-region1} and Corollary~\ref{coro:restricted-phase-three-regions}, those estimates, in particular \eqref{eq:phase-restricted-convex}\eqref{eq:restricted-phase-partial-s-2} continue to hold with the presence of $\varrho$-dependence in the current setting.
The way to introduce the partition of unity in \eqref{eq:def-varrho-i} and the corresponding estimates of $s$-derivatives of the phase on the support of each $\varrho_i$ are still valid. And the rest of the process goes through in the same way as in the low-energy setting, after \eqref{eq:mkI-21-decomposition}.

Now we turn to the $u_4$-term in \eqref{eq:conic-pair-u4-high}. 
To be more consistent with other parts of this section, instead of \eqref{eq:conic-pair-u4-high}, we consider its analogue that is supported away from $\smf \cap \BFS \cap \RB$, and the estimate for \eqref{eq:conic-pair-u4-high} can be derived in the same way.
For the same reason, we use $k+1$ in place of $k$ as the number of extra parameters.
This is because, as shown in Section~\ref{subsec:phase-Hessian}, our phase function $\Phi$ parametrizing this part of $\hat{L}^{\bfs,\calchigh}$ can be chosen as in \eqref{eq:conic-pair-phase-construction-high-region2}, which gives a direct correspondence between those $(v = \mu'_I/s,s)$ parameters in the region $x'/s \lesssim 1$ and our current parameters $(\hat{\mu'}_I=\mu'_I/x',\varkappa)$, which are just rescaled versions of each other.
The estimate for such a contribution is as follows.
\begin{proposition}   \label{prop:conic-pair-pointwise-bound-high-region2}
Let $\Phi$ be as in \eqref{eq:conic-pair-phase-construction-high-region2}. For $a$ that is compactly supported, we have
\begin{equation}  \label{eq:conci-pair-pointwisebound-high-region2}
\Big|(x')^{n-1}h^{-\frac{n-1}{2}-(k+1)/2} \int_{\R^{k+1}}   e^{i\frac{\Phi}{hx'\sigma}} a(h,x',\sigma,y,y',\hat{\mu'}_I,\varkappa) d\varkappa d\hat{\mu'}_I\Big| \lesssim h^{-\frac{n-1}{2}}(x')^{\frac{n-1}{2}} \sigma^{\frac{n-1}{2}}
\end{equation}
on the region $x/x' = \sigma \leq 2$. The estimate near $\bfs \cap \rb$, with $x,x'$ exchanged and $\sigma$ replaced by $\theta=x'/x$ also holds.
\end{proposition}

The proof follows a strategy similar to that in the region $x'/s \lesssim 1$ (which in turn follows in the same way as the proof of Proposition~\ref{prop:conic-pair-pointwise-bound-low}) with $\varkappa$ playing the role of $s$ and $\mu'_I/x'$ playing the role of $\mu'_I/s$ there, except that estimates about the phase function are replaced by their analogues in this region. 

First notice that \eqref{eq:conci-pair-pointwisebound-high-region2} is equivalent to
\begin{equation} \label{eq:6.2-5}
\big|(x')^{\frac{n-1}{2}}h^{-(k+1)/2} \int_{\R^{k+1}}   e^{i\frac{\Phi}{hx'\sigma}} a(h,x',\sigma,y,y',\hat{\mu'}_I,\varkappa) d\hat{\mu'}_Id\varkappa \big| \lesssim 1.
\end{equation}

This automatically holds if $x'h^{-1} \lesssim 1$. Now we consider the case $x'h^{-1} \gg 1$.
Using the concrete form of $\Phi$ in \eqref{eq:conic-pair-phase-construction-high-region2}, which can be rewritten as $\phi = 1 + \sigma + \sigma x' \psi$, with 
\begin{equation} \label{eq:def-psi-high-region2}
\psi = N_2+ (y'_I-Y'_I) \cdot \hat{\mu'}_I + \varkappa (y'_{n-1} - Y'_{n-1}),
\end{equation}
\eqref{eq:6.2-5} is equivalent to 
\begin{equation} \label{eq:6.2-6}
\big| \int_{\R^{k+1}}   e^{i h^{-1}\psi} a(h,x',\sigma,y,y',v) d\varkappa dv \big| \lesssim h^{(k+1)/2}(x')^{-\frac{n-1}{2}}.
\end{equation}

By further decomposing it using a partition of unity, we may assume that the amplitude $a(\bullet)$ is supported in a region on which Proposition~\ref{prop:lower-bound-hessian-high-region2} is valid.
We denote the oscillatory integral on the left hand side by $\mk{I}$. Now we compute derivatives of $\psi$. We still exploit the fact that $\hat{L}^{\bfs,\calchigh}$ is a Legendre submanifold, which gives \eqref{eq:1-form-vanish-low}, and we have
\begin{equation} \label{eq:psi-derivative-region2}
  \partial_{\hat{\mu'}_I}\psi = y'_I -Y'_I, \; \partial_{\varkappa}\psi = y'_{n-1}-Y'_{n-1}.
\end{equation}

Then we consider two cases depending on whether $y'_I,y'_{n-1}$ are in the range of $Y'_I,Y'_{n-1}$,  equivalently whether $\psi$ has a critical point.
Similar to Section~\ref{subsec:critical_points_at_fixed_s}, fixing $(\sigma,x',y,y'_{II})$, the region of $(\hat{\mu'},\varkappa)$ on which $Y'_I,Y'_{n-1}$ contains a $\delta_0$-neighborhood of $\supp\, a(\sigma,x',y,y'_{II},\bullet)$. In addition, we can make this $\delta_0$ uniform in $(\sigma,x',y,y'_{II})$ since it ranges continuously over a compact region. If $(y'_I,y'_{n-1})$ is not in the range of $Y'_I,Y'_{n-1}$, then in the same way as proving \eqref{eq:lower-bound-Ec}, but now instead using Proposition~\ref{prop:lower-bound-hessian-high-region2}, we have 
\begin{equation}
 \max \{ |y_I' - Y'_{I}(\sigma,x',y,y'_{II},\hat{\mu'}_I,\varkappa)|, \, |y'_{n-1}-Y'_{n-1}(\sigma,x',y,y'_{II},\hat{\mu'}_I,\varkappa)|  \}\gtrsim x'
\end{equation}
for all $\hat{\mu'}_I \in \supp \, a(\sigma,x',y,y'_{II},\bullet)$. 

Then we can integrate by parts using $L_1$ as in \eqref{eq:L1-def} and each time gain a factor $h (x')^{-1}$. Here we have three types of terms in $L_1^*a$, where $L_1^*$ is as in \eqref{eq:L1-adjoint-def}:
\begin{equation}
\frac{1}{|\partial_{\hat{\mu'}_I} \psi|^2} (\partial_{\hat{\mu'}_j}\psi) (\partial_{\hat{\mu'}_j}a), \quad
\frac{1}{|\partial_{\hat{\mu'}_I} \psi|^2} (\partial^2_{\hat{\mu'}_j\hat{\mu'}_j}\psi) a, \quad 
\frac{ (\partial_{\hat{\mu'}_i})(\psi \partial_{\hat{\mu'}_j} \psi) (\partial_{\hat{\mu'}_i\hat{\mu'}_j}\psi) }{|\partial_{\hat{\mu'}_I} \psi|^3}  a,
\end{equation}
or analogous terms with $\partial_{\hat{\mu'}_i}$ replaced by $\partial_{\varkappa}$.
For the first and third types of terms, we know they are bounded by $|\partial_{\hat{\mu'}_I} \psi|^{-1} = O((x')^{-1})$ up to a constant. 
For the second type of terms, just notice that $\partial^2_{\hat{\mu'}_j\hat{\mu'}_j}\psi = -\partial_{\hat{\mu'}_j}Y'_j$, which is $O(x')$ by Proposition~\ref{prop:lower-bound-hessian-high-region2}.
In sum, such terms are $O((x')^{-1})$ and we gain a factor $h(x')^{-1}$ after integration by parts.

Now we consider the case where $(y'_I,y'_{n-1})$ is in the range of $Y'_I,Y'_{n-1}$.
In this case, by \eqref{eq:psi-derivative-region2}, we have a unique critical point (in terms of $(\hat{\mu'}_I,\varkappa)$) of the phase, where the uniqueness follows from the fact that we are in a neighborhood defined in Proposition~\ref{prop:Y'I-derivative;low-high-region1-region-def}.
Now we apply the stationary phase method in $(\hat{\mu'}_I,\varkappa)$ with $x'h^{-1}$ as the large parameter with $\varkappa$ fixed. The phase function is
\begin{equation}
  (x')^{-1}\big( N_2+ (y'_I-Y'_I) \cdot \hat{\mu'}_I + \varkappa (y'_{n-1} - Y'_{n-1}) \big).
\end{equation}
As already mentioned in the proof of Proposition~\ref{prop:conic-pair-pointwise-bound-low} after \eqref{eq:6.1-2}, the stationary phase lemma still applies even when there is a linear term whose coefficients might grow in terms of our parameters. Hence applying the stationary phase method gives \eqref{eq:6.2-6}.

\section{The dispersive estimate for the Schr\"odinger equation}
\label{sec:dispersive-Schrodinger}

In this section, we prove Theorem~\ref{thm:dispersive-Schrodinger-1} and Theorem~\ref{thm:dispersive-Schrodinger-2}, i.e.,  dispersive estimates for Schr\"odinger equations.
In fact, as stated in the introduction, we will give even more refined microlocal characterization of the decay rate of contributions from different parts of $L^{\bfs}$ and $L^{\bfs,\calchigh}$.
We will first prove the conclusions on asymptotically conic manifolds, i.e. for $P$ in \eqref{eq:P-def}, and indicate changes needed for the exact cone case, i.e. $P$ in \eqref{eq:P-def-exact-cone}, at the end of this section.

\subsection{The estimate for the Schr\"odinger propagator}
\label{subsec:est-Schrodinger-kernel}


We prove Theorem~\ref{thm:Schrodinger-propagator-est-1} in this subsection. Let $\Jhigh,\Jlow$ be defined as before Proposition~\ref{prop: microlocal-partition-combined}. We set $J = \Jlow \cup \Jhigh$ for convenience below.

\begin{proof}[Proof of Theorem~\ref{thm:Schrodinger-propagator-est-1}]
We consider the claim about the asymptotically conic case, i.e., $P$ in \eqref{eq:P-def} first.
To prove Theorem~\ref{thm:Schrodinger-propagator-est-1}, we notice that using the microlocal partition $Q^{\calc}_j$ in Proposition~\ref{prop: microlocal-partition-combined}, we have
\begin{equation} \label{eq:decomposition-spectral-measure}
    \specm = (\sum_{j \in \overline{ \mk{J}}_{\calc} } Q^{\calc}_j) \specm (\sum_{j \in \overline{ \mk{J}}_{\calc} } Q^{\calc}_j)
    = \sum_{j,j' \in \overline{ \mk{J}}_{\calc} } Q^{\calc}_j \specm Q^{\calc}_{j'},
\end{equation}
which in turn gives
\begin{align} \label{eq:decomposition-propagator}
    \begin{split}
        e^{itP} = \int_0^\infty e^{it\lambda^2} \specm
        = \sum_{j,j' \in \overline{ \mk{J}}_{\calc} } \int_0^\infty e^{it\lambda^2} Q^{\calc}_j \specm Q^{\calc}_{j'}.
    \end{split}
\end{align}
So the part of Theorem~\ref{thm:Schrodinger-propagator-est-1} concerning $P$ in \eqref{eq:P-def} follows from the triangle inequality and summing the microlocalized estimates for the propagator below.
We postpone the discussion about modification needed for the exact cone $(X_0,g_0)$ after the proof of Proposition~\ref{prop:microlocal-dispersive-1}, which is our main technical estimate.
For the part replacing $\IF$ (resp. $\IFz$) by $\IFint$ (resp. $\IFintz$) when the conic intersecting pair is admissible, see Corollary~\ref{coro:Schrodinger-propagator-est-3} in the next subsection.
\end{proof}

\begin{proposition} \label{prop:microlocal-dispersive-1}
Let $P$ be as in \eqref{eq:P-def} and let $\{Q_j^{\calc}\}$ be a microlocal partition of unity as in Proposition~\ref{prop: microlocal-partition-combined}.
If either $j$ or $j'$ is $1$ or $\zf$, or $j = j'$, then
\begin{equation} \label{eq:7.1-0}
\Big|\int_0^\infty e^{i t \lambda^2} Q_j^{\calc} \specm Q_{j'}^{\calc}  \Big| \lesssim |t|^{-\frac{n}{2}}.
\end{equation}
For $j,j' \in J$, $j \neq j'$ and $k$ as in Corollary~\ref{coro:microlocalized-spectral-measure-osc-int-form-low} and Corollary~\ref{coro:microlocalized-spectral-measure-osc-int-form-high}, we have

\begin{equation} \label{eq:7.1-1}
\Big|\int_0^\infty e^{i t \lambda^2} Q_j^{\calc} \specm Q_{j'}^{\calc} \Big| \lesssim |t|^{-\frac{n}{2}}
\big(1+(|t|^{-1}\la z \ra \la z' \ra)^{k/2}\big).
\end{equation}
Notice that by Proposition~\ref{prop:IF-relation-parameter-number}, we have $k \leq \IF$.
\end{proposition}


\begin{proof}
Consider the low-energy case first for definiteness and the argument is uniform up to the range $\lambda \to \infty$.
Also, we only consider the region where $\sigma \lesssim 1$ since the case $\theta = x'/x \lesssim 1$ can be proved in the same way after switching primed and un-primed variables.
Let $J,\zf,1$ be as in Section~\ref{subsec:microlocal-partition-low}.
If at least one of $j$ or $j'$ is either $\zf$ or $1$, then by Proposition~\ref{prop:localized-specm-one-side-residual-low}, we need to consider two types of terms: terms like \eqref{QjEQ1-low-rb} and terms like \eqref{QjEQ1-c-low}.
For terms like \eqref{QjEQ1-low-rb}, all properties needed in the proof of \cite[Equation~(6-3)]{Hassell-Zhang2016Strichartz} are satisfied and we can bound the left hand side of \eqref{eq:7.1-1} in the same way.
For terms of the form \eqref{QjEQ1-c-low}, it can be bounded in the same way as in the proof of \cite[Proposition~6.1]{Hassell-Zhang2016Strichartz} since it satisfies \eqref{eq:phg-conormal-2}.

For the rest of the proof, we consider $Q_j^{\calc} \specm Q_{j'}^{\calc}$ with both $j,j' \in J$ and the corresponding $L^{\bfs}_{j,j'}$ in \eqref{eq:Lbf-low-jj'} is away from the conic intersecting pair. The case $j=j'$ follows from \cite[Proposition~6.1]{Hassell-Zhang2016Strichartz} and we consider $j \neq j'$ below.
By Corollary~\ref{coro:microlocalized-spectral-measure-osc-int-form-low}, we know $Q_j^{\calc} \specm Q_{j'}^{\calc}$ is a finite sum of pieces of the form:
\begin{equation} \label{eq:spectral-measure-local}
\int_{\R^k} \lambda^{n-1}  e^{\pm i \lambda \Phi/x} a(\lambda,z,z';v)dv,
\end{equation}
where $\pm \Phi$ parametrizes $L^{\bfs}$ locally and by Corollary~\ref{coro:phase-value-lowerbound-low}, we may assume $\Phi>\epsilon$ for a constant $\epsilon>0$ and the case $\Phi<-\epsilon$ is the same.
Here we consider the case that $a$ satisfies one of \eqref{eq:spectral-measure-symbol-bound-low-1} \eqref{eq:spectral-measure-symbol-bound-low-2}, and \eqref{eq:spectral-measure-symbol-bound-low-3}.
We consider the case that $a$ satisfies \eqref{eq:spectral-measure-symbol-bound-low-1} below, and the other two cases follow in the same way.

Correspondingly, the left hand side of \eqref{eq:7.1-1} is a finite sum of oscillatory integrals of the form
\begin{equation} \label{eq:7.1-5}
 \int_0^\infty \int_{\R^k} e^{it\lambda^2} \lambda^{n-1}  e^{\pm i \lambda \Phi(z,z',v) } a(\lambda,z,z';v)dv d\lambda.
\end{equation}
We consider the following change of variables
\begin{align} \label{eq:rescaled-lambda-phi}
\overline{\lambda} = t^{1/2}\lambda,  \quad \overline{r} = t^{-1/2}\Phi(z, z',v)/x,
\end{align}
under which quantities like $\sigma=x/x',\lambda/x,\lambda/x'$ remain unchanged.
In terms of those rescaled variables, \eqref{eq:7.1-5} becomes $t^{-n/2}I_\pm(a,\Phi)$, where
\begin{align*}
I_\pm(a,\Phi) = \int_0^\infty \int_{\R^k}  \overline{\lambda}^{n-1} e^{i\overline{\lambda}^2 }  e^{ \pm i \overline{\lambda} \overline{r} } a(t^{-1/2}\overline{\lambda},z,z';v)dv d\overline{\lambda}.
\end{align*}
Then we first decompose it according to the value of $\overline{\lambda}$. Let $\varphi \in C_c^\infty([\frac{1}{2},2])$ be such that 
\begin{align*}
\sum_{m \in \mathbb{Z} } \varphi(\frac{\overline{\lambda}}{2^m}) = 1,
\end{align*}
and we may assume that $\varphi$ is identically $1$ near $1$. Then we set 
\begin{align} \label{eq:def-varphi-0}
    \varphi_0(\overline{\lambda}) = \sum_{m \leq -1} \varphi(\frac{\overline{\lambda}}{2^m}).
\end{align} 
We have
\begin{align*}
I_\pm(a,\Phi) = I_{\pm,1}(a,\Phi)+I_{\pm,2}(a,\Phi),
\end{align*}
where
\begin{align*}
I_{\pm,1}(a,\Phi) = \int_0^\infty \int_{\R^k}  \overline{\lambda}^{n-1} e^{i\overline{\lambda}^2 }  e^{\pm i \overline{\lambda} \overline{r} } a(t^{-1/2}\overline{\lambda},z,z';v) \varphi_0(\overline{\lambda}) dv d\overline{\lambda},
\end{align*}
and 
\begin{align*}
I_{\pm,2}(a,\Phi) = \sum_{m \geq 0} \int_0^\infty \int_{\R^k} \overline{\lambda}^{n-1} e^{i\overline{\lambda}^2 }  e^{\pm i \overline{\lambda} \overline{r} } a(t^{-1/2}\overline{\lambda},z,z';v) \varphi(\frac{\overline{\lambda}}{2^m}) dv d\overline{\lambda}.
\end{align*}

Consider $I_{\pm,1}$ first. This part is uniformly bounded since $\overline{\lambda} \lesssim 1$ and the length of the interval is $O(1)$ as well.
The term $I_{+,2}$
\footnote{Here the part with $+$ sign in front of $\Phi$ is easier because we are considering the case $t>0$ and $\Phi>0$. If we consider the part $t<0$, then critical points of the phase will lie in this part and $I_{-,2}$ will play the role of the current $I_{+,2}$.} 
can be estimated by a non-stationary phase argument as in \cite[Proposition~6.1]{Hassell-Zhang2016Strichartz} for what is called $II^+$ there.

Now we turn to $I_{-,2}(a,\Phi)$. We further decompose according to the value of $\overline{r}\overline{\lambda}$. Concretely, we write
\begin{align*}
I_{-,2}(a,\Phi) = I_{-,2}^1(a,\Phi) + I_{-,2}^2(a,\Phi),
\end{align*}
where 
\begin{align*}
I_{-,2}^1(a,\Phi) = \sum_{m \geq 0} \int_0^\infty \int_{\R^k}  \overline{\lambda}^{n-1} e^{i\overline{\lambda}^2 }  e^{- i \overline{\lambda} \overline{r} } a(t^{-1/2}\overline{\lambda},z,z';v) \varphi(\overline{\lambda}/2^m) \varphi_0(4\overline{r}\overline{\lambda}) dv d\overline{\lambda},
\end{align*}
and
\begin{align*}
I_{-,2}^2(a,\Phi) = &  \sum_{m \geq 0} \sum_{m' \geq 0} \int_0^\infty \int_{\R^k}  \overline{\lambda}^{n-1} e^{i\overline{\lambda}^2 }  e^{- i \overline{\lambda} \overline{r} } a(t^{-1/2}\overline{\lambda},z,z';v) \varphi(\overline{\lambda}/2^m) \varphi(4\overline{r}\overline{\lambda}/2^{m'}) dv d\overline{\lambda}
\\ = & \sum_{m' \geq 0} \int  \overline{\lambda}^{n-1} e^{i\overline{\lambda}^2 }  e^{- i \overline{\lambda} \overline{r} } a(t^{-1/2}\overline{\lambda},z,z';v) (1-\varphi_0(\overline{\lambda})) \varphi(4\overline{r}\overline{\lambda}/2^{m'}) dv d\overline{\lambda} .
\end{align*}
Then $I_{-,2}^1(a,\Phi)$ can still be estimated by a non-stationary phase argument since on the support of the integrand we have $\overline{\lambda} \leq \frac{1}{4}$ by the $\varphi_0(4\overline{r}\overline{\lambda})$-factor while we have $\overline{\lambda} \geq \frac{1}{2}$ by the factor $(1-\varphi_0(\overline{\lambda}))$. In particular, we have $\partial_{\overline{\lambda}}(\overline{\lambda}^2-\overline{r}\overline{\lambda}) \geq \frac{1}{2} \overline{\lambda}$ and after integration by parts $N$-times this part can be estimated by 
\begin{align} \label{eq:01} 
\sum_{m \geq 0}\int_{\overline{\lambda} \sim 2^m} \overline{\lambda}^{n-1-2N} d\overline{\lambda} = O(1).
\end{align}

Finally, we turn to the most interesting part $I_{-,2}^2$, in which we actually exploit the structure of the spectral measure. As already observed, the critical point of the phase is when $2\overline{\lambda}-\overline{r}=0$, so we further decompose this part according to the size of this quantity. 
More precisely, $I_{-,2}^2$ can be decomposed as $\sum_{m=0}^\infty I_{C,m}$\footnote{Here C stands for critical.}, where
\begin{align} \label{eq:IC-0}
I_{C,0} = \sum_{m' \geq 0} \int  \overline{\lambda}^{n-1} e^{i\overline{\lambda}^2 }  e^{- i \overline{\lambda} \overline{r} } a(t^{-1/2}\overline{\lambda},z,z';v) (1-\varphi_0(\overline{\lambda}))
\varphi_0(|2\overline{\lambda}-\overline{r}|) \varphi(4\overline{r}\overline{\lambda}/2^{m'}) dv d\overline{\lambda}\\
=\int  \overline{\lambda}^{n-1} e^{i\overline{\lambda}^2 }  e^{- i \overline{\lambda} \overline{r} } a(t^{-1/2}\overline{\lambda},z,z';v) (1-\varphi_0(\overline{\lambda}))
\varphi_0(|2\overline{\lambda}-\overline{r}|) \big(1-\varphi_0(4\overline{r}\overline{\lambda})\big) dv d\overline{\lambda}
\end{align}
and for $m \geq 1$:
\begin{align} \label{eq:IC-m}
\begin{split}
I_{C,m} =  & \sum_{m' \geq 0} \int  \overline{\lambda}^{n-1} e^{i\overline{\lambda}^2 }  e^{- i \overline{\lambda} \overline{r} } a(t^{-1/2}\overline{\lambda},z,z';v) (1-\varphi_0(\overline{\lambda}))
\varphi(\frac{|2\overline{\lambda}-\overline{r}|}{2^m}) \varphi(4\overline{r}\overline{\lambda}/2^{m'}) dv d\overline{\lambda}
\\ = & \int  \overline{\lambda}^{n-1} e^{i\overline{\lambda}^2 }  e^{- i \overline{\lambda} \overline{r} } a(t^{-1/2}\overline{\lambda},z,z';v) (1-\varphi_0(\overline{\lambda}))
\varphi(\frac{|2\overline{\lambda}-\overline{r}|}{2^m}) (1-\varphi_0(4\overline{r}\overline{\lambda})) dv d\overline{\lambda} .
\end{split}
\end{align}
Since $\varphi$ is supported away from $0$, $\varphi(|\cdot|/2^m)$ is smooth.

Consider $I_{C,0}$ first. 
We can choose $\varphi_0$ so that $1-\varphi_0(\overline{\lambda})$ is supported in $\overline{\lambda} \geq 3/2$, which gives $|2\overline{\lambda}-\overline{r}| \leq 2$ (hence $\overline{r} \geq 1$) on the support of $(1-\varphi_0(\overline{\lambda}))
\varphi_0(|2\overline{\lambda}-\overline{r}|)$. 
By considering either $\overline{\lambda} \lesssim 1$ or $\overline{\lambda} \gg 1$, we know $\overline{\lambda} \sim \overline{r}$ in this region. 
In addition, since $\overline{r} = t^{-1/2} \Phi(\sigma,y,y',v)/x$ by definition and $\Phi \geq \epsilon > 0$ by our assumption, we know
\begin{equation}
   \overline{\lambda} \sim \overline{r} \sim t^{-1/2}x^{-1} : = \overline{d},
\end{equation}
where we used the condition $j \neq j'$. And we set 
\begin{equation}
    \overline{d'} = t^{-1/2}(x')^{-1} \sim  \sigma \overline{d}.
\end{equation}
Since $1/x$ is indeed comparable to the distance between $z$ and $z'$ when $j\neq j'$, this $\overline{d}$ will play the role of rescaled distance.
Combining \eqref{eq:spectral-measure-symbol-bound-low-1}, the $\overline{\lambda}$-integral above for fixed $v$, is bounded by 
\begin{align} \label{eq:7.1.4}
\begin{split}
& \int_{|2\overline{\lambda}-\overline{r}| \leq 1, \overline{\lambda} \sim \overline{d}} \overline{\lambda}^{n-1}
(1+\overline{\lambda}\overline{d})^{ -\frac{n-1-k_\ell}{2} } \sigma^{k_\ell/2} d\overline{\lambda}
\\ =  & \int_{|2\overline{\lambda}-\overline{r}| \leq 1, \overline{\lambda} \sim \overline{d}} \overline{\lambda}^{k_\ell} 
\big(\overline{\lambda}^{n-1-k_\ell}
(1+\overline{\lambda}\overline{d})^{-\frac{n-1-k_\ell}{2}} \big) \sigma^{k_\ell/2} d\overline{\lambda}
\\ \lesssim &   \int_{|2\overline{\lambda}-\overline{r}| \leq 1, \overline{\lambda} \sim \overline{d}} \overline{\lambda}^{k_\ell} \sigma^{k_\ell/2} d\overline{\lambda}  
\sim \max(1,(\overline{d'}\overline{d})^{k_\ell/2}), 
\end{split}
\end{align}
where we used the fact that the length of the $\overline{\lambda}$-interval we are integrating over is $O(1)$. 
Then we integrate over $v$ that are in the support of $a$, which is compact and obtain
\begin{equation}
    |I_{C,0}| \lesssim \max\Big(1,(\overline{d'}\overline{d})^{k_\ell/2}\Big). 
\end{equation}

For $I_{C,m}$ with $m \geq 1$, we know the derivative of the phase in $\overline{\lambda}$ is $(2\overline{\lambda}- \overline{r}) \sim 2^m$ and derivatives of all other factors in $\overline{\lambda}$ are uniformly bounded \footnote{For example, consider the $\varphi(4\overline{r}\overline{\lambda}/2^{m'})$ factor, which has derivative $\frac{4\overline{r}}{2^{m'}}\varphi'(4\overline{r}\overline{\lambda}/2^{m'})$, which is supported in the region $4\overline{r}\overline{\lambda}/2^{m'} \sim 1$. Since we are restricted to the region $\overline{\lambda} \gtrsim 1$ now and this is bounded.}.
In addition, the argument showing $\overline{r} \sim \overline{d}$ still applies in this region.
If $\overline{r} \lesssim 2^m$, then we know $\overline{\lambda} \lesssim 2^m$ as well on the support of the integrand.
Then we can integrate by parts in $\overline{\lambda}$ for $N$ times to obtain
\begin{equation} \label{eq:bound-IC-m}
    |I_{C,m}| \lesssim 2^{-mN} \int_{|2\overline{\lambda}-\overline{r}| \sim 2^m} \overline{\lambda}^{n-1}(1+\overline{\lambda} \overline{d})^{-\frac{n-1-k_\ell}{2}} \sigma^{k_\ell/2}  d\overline{\lambda}
    \lesssim 2^{-m}(\overline{d'}\overline{d})^{k_\ell/2},
\end{equation}
where we used $\overline{d} \sim \overline{r} \gtrsim \overline{\lambda}^{-1} \gtrsim 2^{-m}$ due to the factor $(1-\varphi_0(4\overline{r}\overline{\lambda}))$. Then we can sum over $m$. On the other hand, if $\overline{r} \gg 2^m$, then we know $\overline{\lambda} \gg 2^m$ and  $\overline{d} \sim \overline{r} \sim \overline{\lambda}$, which still gives \eqref{eq:bound-IC-m}.

In the proof below, we consider the contribution from the conic intersecting pair $(L^{\bfs},\Lsharplow)$ (or $(L^{\bfs,\calchigh},\Lsharphigh)$ in the high-energy regime).  
Let $\Phi$ and extra parameters $(v,s)$ be as in \eqref{eq:conic-pair-phase-construction-low} (or \eqref{eq:conic-pair-phase-construction-high-region1}\eqref{eq:conic-pair-phase-construction-high-region2} depending on the regime), if we treat $s$ and $v$ on an equal footing and apply the argument above with the pointwise bound in Corollary~\ref{coro:microlocalized-spectral-measure-osc-int-form-low}
(or Corollary~\ref{coro:microlocalized-spectral-measure-osc-int-form-high}) directly, then we are effectively running the same argument except that now we have $k+1$ extra parameters and eventually this leads to a $\sigma^{1/2} (1+\lambda|z|)^{1/2} \sim (\frac{\lambda}{x'})^{1/2}$ loss compared with \eqref{eq:7.1-1}.
This is remedied by exploiting the convexity of the phase function in $s$ and we explain this in detail below.

Recall the local expression in \eqref{QiEQj-s-lo-2},
our amplitude is smooth in $\frac{x'}{\lambda s}$. We apply the Taylor expansion to this amplitude in this variable as 
\begin{equation} \label{eq:a-exp-1}
    a_0(\lambda,0,\sigma,y,y',v,s) + \frac{x'}{\lambda s} a_1(\lambda,\frac{x'}{\lambda s},\sigma,y,y',v,s), 
\end{equation}
and then the contribution from the second term can be bounded in the same way as contributions away from conic Legendre pairs above, since after multiplying $(\frac{\lambda}{x'})^{1/2}$, which is the potential loss that we are dealing with, the factor in front of $a_1$ can be written as 
\begin{equation}
(\frac{x'}{\lambda s})^{1/2}s^{-1/2},
\end{equation}
while in the amplitude we have an $s^{\frac{n+k}{2}-1}$ factor that can absorb $s^{-1/2}$.

Then we apply the stationary phase lemma to the $s$-integral in the contribution from $a_0$.
More concretely, using the expressions in \eqref{eq:hatLbf-low-components} and \eqref{eq:conic-pair-phase-construction-low}, we can write the phase as 
\begin{equation}
    \frac{\lambda}{x}(1+\sigma) + \frac{\lambda}{x'}s\psi.
\end{equation}
Then $\psi$ has uniformly lower-bounded Hessian in $s$ at its critical point in $s$ by Proposition~\ref{prop:LCP-phase-derivatives} and
\eqref{eq:Y'-n-1-expansion-concrete-low} (or \eqref{eq:Y'-n-1-expansion-concrete-high-region1}\eqref{eq:Y'-n-1-expansion-region2-concrete} in other regions).
In addition, such a critical point always exists for $y_{n-1}'$ in a small neighborhood of the $y_{n-1}'$-component of the point around which we are writing this parametrization. This is because the derivative of $Y'_{n-1}$ in $s$ is bounded away from $0$ by the same proposition as above, so such $y_{n-1}'$ is in the range of $Y'_{n-1}$ and we have a critical point in $s$ by \eqref{eq:v-critical-fixed-s}.
Then we apply the stationary phase lemma in $s$ to gain a factor of $(\lambda/x')^{-1/2}$ and then we can apply the same argument in the first part.

In the high-energy regime, the amplitude in \eqref{eq:coro4.2-2} 
is only smooth in $x'/s$ or $x/s$ rather than $\frac{x}{\lambda s}$ or $\frac{x'}{\lambda s}$. So \eqref{eq:a-exp-1} becomes
\begin{equation} \label{eq:a-exp-high}
    a_0(\lambda,0,\sigma,y,y',v,s) + \frac{x'}{s} a_1(\lambda,x'/s,\sigma,y,y',v,s).
\end{equation}
The proof for the $a_0$-term remains the same while a further minor change is needed for the $a_1$-term.
Then the coefficient in front of the second term in \eqref{eq:a-exp-1} is only $x'/s$. So after multiplying $(\frac{\lambda}{x'})^{1/2}$, which is the potential loss that we are dealing with, the factor in front of $a_1$ can be written as 
\begin{equation}
 (x')^{1/2}\lambda^{1/2} s^{-1}.
\end{equation}
Then the $s^{-1}$-factor can be absorbed by $s^{\frac{n+k}{2}-1}$ and it leads to at most $s^{-1/2}$ type growth in $s$, which is integrable in $s$ after bounding the integral in all other variables using the argument above.
Then we can combine $ (x')^{1/2}\lambda^{1/2}$ with the extra $\sigma^{1/2}$ factor in \eqref{eq:spectral-measure-symbol-bound-high-2} to obtain
\begin{equation} \label{eq:temp-7.1-1}
    \sigma^{1/2} (x')^{1/2}\lambda^{1/2}.
\end{equation}
In the region $ \sigma^{1/2} (x')^{1/2}\lambda^{1/2} \lesssim |t|^{-1/2}$, the proof is completed by the same argument as in the first part.
In the region on which this quantity in \eqref{eq:temp-7.1-1} is $\gg t^{-1/2}$, we know $\lambda \gg x^{-1}t^{-1}$, equivalently in the rescaled variables above we have $\overline{\lambda} \gg \overline{x}^{-1}$. Then this part is dealt with via the non-stationary phase argument (in $\overline{\lambda}$) before \eqref{eq:01} and this factor $\lambda^{1/2}$ can be eliminated after integration by parts.

In summary, we have proved \eqref{eq:7.1-1}. The first part on the right hand side (i.e., the `1' in the bracket) comes from the contribution of small $\overline{r},\overline{\lambda}$, which gives $O(1)$ contribution as above after removing the $|t|^{-n/2}$-factor.
\end{proof}

Now we point out changes in the proof above needed for the second part of Theorem~\ref{thm:Schrodinger-propagator-est-1} concerning $P$ in \eqref{eq:P-def-exact-cone}.
\begin{proof}[Proof of Theorem~\ref{thm:Schrodinger-propagator-est-1} resumed.]
In the setting of an exact cone $(X_0,g_0)$, we use an analogue of Proposition~\ref{prop:microlocal-dispersive-1} in the same way as above, except that now the spectral measure is replaced by the one associated with $P$ in \eqref{eq:P-def-exact-cone} characterized as in the last part of Theorem~\ref{thm:spectral-measure-complete}.

The microlocal partition of unity is defined as follows.
We take the low-energy microlocal partition of unity in Proposition~\ref{prop:microlocal-partition-low} at any fixed $\lambda_0 \lesssim 1$ and then extend it  constantly in $\lambda$ to make it a family of partitions. 
Then the rest of the proof goes through in the same way and the maximum of the number of extra parameters in minimal parametrizations in this setting is $\IFz$ by the part of Proposition~\ref{prop:IF-relation-parameter-number} concerning $(X_0,g_0)$.
\end{proof}

Now we turn to prove the second estimate for the Schr\"odinger propagator, i.e. \eqref{eq:Schrodinger-propagator-est-2-1}. 
The statement concerning $P$ on $(X_0,g_0)$ defined in \eqref{eq:P-def-exact-cone} follows from replacing ingredients in the proof below concerning $P$ on $(X,g)$ by the corresponding ingredients for $(X_0,g_0)$ as in the last part of the proof of Theorem~\ref{thm:Schrodinger-propagator-est-1} above.
So we only give details about the part concerning $P$ in \eqref{eq:P-def} in the asymptotically conic case, 
which in turn follows from summing the following microlocalized estimate.

\begin{proposition} \label{prop:microlocal-dispersive-2}
    Let $j,j' \in J$, $j \neq j'$ and $\ell, k_\ell$ be as in Proposition~\ref{prop:microlocal-partition-low}, Proposition~\ref{prop:microlocal-partition-high}. For $k \geq k_\ell$, we have 
\begin{equation} \label{eq:7.1-2}
\Big|\int_0^\infty e^{i t \lambda^2} \lambda^{-k} Q_j^{\calc} \specm  Q_{j'}^{\calc} d\lambda \Big| \lesssim |t|^{-\frac{n-k}{2}}.
\end{equation}
For $j,j' \in J$ with $j=j'$ or one of $j,j'$ is $1$ or $\zf$, then \eqref{eq:7.1-2} holds for any $k$ such that  $ 0 \leq k \leq n-1$. 
\end{proposition}

\begin{proof}
     The proof is almost the same as Proposition~\ref{prop:microlocal-dispersive-1} and we address differences here. We still consider the case away from the conic intersecting pair first.
     The case either $j=j' \in J$ or one of $j,j'$ is $1$ or $\zf$ still follows from the same proof as in \cite[Proposition~6.1]{Hassell-Zhang2016Strichartz}, except that now the power of $\lambda$ is $n-1-k$ and consequently the power of $|t|$ obtained from the same rescaling with $\overline{\lambda}=|t|^{1/2}\lambda$ is $|t|^{-\frac{n-k}{2}}$ and all the steps in the proof remain the same.

Now we consider the case with $j,j' \in J$ and $j\neq j'$. In the local expression \eqref{eq:7.1-5} of the microlocalized propagator, we instead have\begin{equation} \label{eq:7.1-5'}
 \int e^{it\lambda^2} \lambda^{n-1-k}  e^{\pm i \lambda \Phi(z, z',v) } a(\lambda,z,z';v)dv d\lambda.
\end{equation}
Then we still introduce rescaled variables as in \eqref{eq:rescaled-lambda-phi}, which gives an overall $t^{-\frac{n-k}{2}}$ factor. 
Then the decomposition in $\overline{\lambda}$ and $\overline{r}\overline{\lambda}$ remains the same and we use the same letter to denote corresponding part of the oscillatory integral except that the polynomial factor is $\overline{\lambda}^{n-1-k}$ now.
Then the argument used to estimate $I_{\pm,1}$, $I_{+,2}$ remains the same.
The difference lies in the way to estimate the $I_{-,2}$-part.
More precisely, we again decompose it according to the value of $|2\overline{\lambda}-\overline{r}|$ as 
\begin{equation*}
\sum_{m=0}^\infty I_{C,m},   
\end{equation*}
where $I_{C,0}$ includes the part with $|2\overline{\lambda} - \overline{r}| \lesssim 1$ and $I_{C,m}, m \geq 1$ includes the part with $|2\overline{\lambda} - \overline{r}| \sim 2^m$. 
For $I_{C,0}$, using \eqref{eq:spectral-measure-symbol-bound-low-1} again, this oscillatory integral can be bounded by
\begin{align} \label{eq:7.1.4'}
\begin{split}
& \int_{|2\overline{\lambda}-\overline{r}| \leq 1, \overline{\lambda} \sim \overline{d}} \overline{\lambda}^{n-1-k}
(1+\overline{\lambda}\overline{d})^{-\frac{n-1-k}{2}} d\overline{\lambda} \lesssim 1.
\end{split}
\end{align}
For $I_{C,m}$, $m \geq 1$, the derivative of the phase with respect to $\overline{\lambda}$ is comparable to $2^m$. Then we integrate by parts in $\overline{\lambda}$ first to gain a $2^{-m}$ factor each and using the discussion before \eqref{eq:bound-IC-m} we have either $\overline{\lambda} \sim 2^m$ or $\overline{\lambda} \sim \overline{d}$, which gives
\begin{equation} \label{eq:bound-IC-m'}
    |I_{C,m}| \lesssim 2^{-mN} \int_{|2\overline{\lambda}-\overline{r}| \sim 2^m} \overline{\lambda}^{n-1-k}(1+\overline{\lambda} \overline{d})^{-\frac{n-1-k}{2}}  d\overline{\lambda}
    \lesssim 2^{-m},
\end{equation}
which completes the proof after summing over $m$.

For the part associated with the conic intersecting pair, the issue to remedy is still that we have $k+1$ extra parameters while we want to produce a bound as if there are only $k$ extra parameters. To this end, we first apply the last part of the proof of Proposition~\ref{prop:microlocal-dispersive-1}, which gives $(1+\overline{\lambda} \overline{d})^{-1/2}$ improvement, and then the rest of the argument above goes through.
\end{proof}

\subsection{Dispersive estimates for the Schr\"odinger equation: contribution from the Legendrian conic pair}
 \label{subsec:dispersive-conic-pair-contribution}
 
In this subsection, we prove the second part of Theorem~\ref{thm:Schrodinger-propagator-est-1} and Theorem~\ref{thm:Schrodinger-propagator-est-2} with the condition that $X$ is admissible in the sense of Definition~\ref{definition:Lbf-boundary-admissible}. 
These two parts will imply the corresponding dispersive estimates in the second part of Theorem~\ref{thm:dispersive-Schrodinger-1}, i.e., \eqref{eq:est-dispersive-1-2} and the second part of Theorem~\ref{thm:dispersive-Schrodinger-2}, i.e. \eqref{eq:est-dispersive-2-2}.
The corresponding part concerning $P$ defined for an exact cone $(X_0,g_0)$ follows from the same discussion in the previous subsection, i.e. the resumed part of the proof of Theorem~\ref{thm:Schrodinger-propagator-est-1}, with the only difference being that now we use the statement of 
Proposition~\ref{prop:IF-relation-parameter-number} concerning both $\IFintz$ for those local expressions of the microlocalized spectral measure that are not associated with the conic intersecting pair $(L^{\bfs},\Lsharplow)$, and use the statement concerning $\IFz$ for those pieces of the  microlocalized spectral measure associated with $(L^{\bfs},\Lsharplow)$.

So we only give details of the proof of \eqref{eq:est-dispersive-1-2} in the asymptotically conic case below. We prove this by giving microlocal dispersive estimates controlling the contribution to the propagator from the intersecting Legendre pair with conic points. 
The geometric setup is the same as in Section~\ref{subsec:phase-Hessian}: let $\ULCPlow$ (resp. $\ULCPhigh$) be a neighborhood of a point in  $\beta_{\mathrm{LCP},\calclow}^{-1}(L^{\bfs} \cap \Lsharplow)$ (resp. $\beta_{\mathrm{LCP},\calchigh}^{-1}(L^{\bfs,\calchigh} \cap \Lsharphigh$)) in $\hat{L}^{\bfs}$ (resp. $\hat{L}^{\bfs,\calchigh}$) on which the projection to $\bfs$ (resp. $X_{\rmb}^2$) only degenerates at $\beta_{\mathrm{LCP},\calclow}^{-1}(L^{\bfs} \cap \Lsharplow)$ (resp. $\beta_{\mathrm{LCP},\calchigh}^{-1}(L^{\bfs,\calchigh} \cap \Lsharphigh$)) with the maximal rank drop denoted by $k_\ell$. 
In fact, one can choose $\ULCPhigh$ first and then take $$\ULCPlow = \ULCPhigh \cap \big\{x/s = 0\big\}.$$

Let $L^{\bfs}_{j,j'}$ (resp. $L^{\bfs,\calchigh}_{j,j'}$) be as in \eqref{eq:Lbf-low-jj'} (resp. \eqref{eq:Lbf-high-jj'}), potentially after shrinking it (by shrinking $W_j,W_{j'}$), we may assume that it is contained in $\ULCPlow$ (resp. $\ULCPhigh$). In particular, $j$ and $j'$ are not $1$ or $\zf$ and $j \neq j'$.

\begin{proposition} \label{prop:microlocal-dispersive-LCP}
Suppose that $(\hat{L}^{\bfs,\calchigh},\Lsharphigh)$ is admissible in the sense of Definition~\ref{definition:Lbf-boundary-admissible},
for $P$ in \eqref{eq:P-def} and  $j,j',k_\ell$ above, then for $0 \leq k_0 \leq k_\ell$, we have 
\begin{equation} \label{eq:LCP-microlocalized-dispersive}
\Big|\int_0^\infty e^{i t \lambda^2} Q_j^{\calc} \specm   Q_{j'}^{\calc} \lambda^{-k_0} d\lambda \Big| \lesssim |t|^{-\frac{n-k_0}{2}}.
\end{equation}
Suppose that $(L^{\bfs},\Lsharplow)$ is admissible. Then for $P$ in \eqref{eq:P-def-exact-cone}, the same estimate holds, except that the microlocalizers on both sides of $\specm$ are replaced by $Q_j^{\flat},Q_{j'}^{\flat}$ for a fixed $\lambda_0 \lesssim 1$ that is extended by constant to all $\lambda$. 
\end{proposition}

We postpone the proof of Proposition~\ref{prop:microlocal-dispersive-LCP} for a moment and discuss its consequence now.
We sum the estimates in \eqref{eq:LCP-microlocalized-dispersive} for those $j,j'$ with $L^{\bfs,\calchigh}_{j,j'}$ near $L^{\bfs,\calchigh} \cap \Lsharphigh$
with estimates in Proposition~\ref{prop:microlocal-dispersive-1} and Proposition~\ref{prop:microlocal-dispersive-2} for those $j,j'$ with $L^{\bfs,\calchigh}_{j,j'}$ away from $L^{\bfs,\calchigh} \cap \Lsharphigh$.
Now the largest rank drop of the projection from those $L^{\bfs,\calchigh}_{j,j'}$ away from $L^{\bfs,\calchigh} \cap \Lsharphigh$ is at most $\IFint$. Contributions associated with $L^{\bfs,\calchigh}_{j,j'}$ near $L^{\bfs,\calchigh} \cap \Lsharphigh$ are estimated by Proposition~\ref{prop:microlocal-dispersive-LCP} and are not affected by the rank drop of the projection. In sum, we have the following estimate on the propagator.
\begin{corollary}
\label{coro:Schrodinger-propagator-est-3}
For $P$ defined by \eqref{eq:P-def} and $\IFint$ defined by \eqref{eq:IFint-def}, we have
\begin{equation}
|e^{itP}(z,z')| \lesssim |t|^{- \frac{n}{2} }\Big(1+\big(\frac{\la z \ra \la z' \ra}{|t|}\big)^{\IFint/2}\Big).
\end{equation}
For $P$ defined by \eqref{eq:P-def-exact-cone} on $(X_0,g_0)$, suppose that $(L^{\bfs},\Lsharplow)$ is admissible, then the same estimate holds with $\IFint$ replaced by $\IFintz$ defined in \eqref{eq:IFintz-def} on the right hand side.
\end{corollary}
After applying this to the initial data (i.e. integrate in $z'$), this implies the second part of Theorem~\ref{thm:dispersive-Schrodinger-1}, i.e. \eqref{eq:est-dispersive-1-2}. The similar estimate for the propagator proving the second part of Theorem~\ref{thm:dispersive-Schrodinger-2}, i.e. \eqref{eq:est-dispersive-2-2} is the following.
\begin{corollary}
\label{coro:Schrodinger-propagator-est-4}
For $P$ defined by \eqref{eq:P-def} and $\IFint$ defined by \eqref{eq:IFint-def}, we have
\begin{equation}
|e^{itP}(z,z')P^{-\frac{\IFint}{2}}| \lesssim |t|^{- \frac{n-\IFint}{2}}.
\end{equation}
For $P$ defined by \eqref{eq:P-def-exact-cone} on $(X_0,g_0)$, suppose that $(L^{\bfs},\Lsharplow)$ is admissible, then the same estimate holds with $\IFint$ replaced by $\IFintz$ defined in \eqref{eq:IFintz-def} on the right hand side.
\end{corollary}

Since the proof of Proposition~\ref{prop:microlocal-dispersive-LCP} is rather involved, we sketch the main ideas before the proof.
The overall strategy is to decompose the oscillatory integral according to the derivatives of the phase.
We first decompose according to the derivative in $\lambda$.
For the part with this derivative being large, we apply a non-stationary phase argument in $\lambda$.
For the part that this derivative is small, we further decompose according to the size of $x\mk{d}_X - \Phi$, which vanishes at critical (in $v,s$) points of $\Phi$ and quantifies the distance to the critical set. 
When $|x\mk{d}_X - \Phi|$ is small, then for each fixed $\lambda$, we apply the pointwise bound in Section~\ref{sec:conic-points-pointwise-bound} for Legendre distributions associated with Legendrian conic pairs, but with rescaled $\lambda,x,x'$:
\begin{equation} \label{eq:defn-rescaled-lambda-x-x'}
\overline{\lambda} =  t^{1/2}\lambda, \quad
\overline{x} = t^{1/2}x, \quad \overline{x'} = t^{1/2}x'. 
\end{equation}
When $|x\mk{d}_X - \Phi|$ is large, we integrate by parts in $v,s$ directly.
Finally, we estimate the $\lambda$-integral.

\begin{proof}[Proof of Proposition~\ref{prop:microlocal-dispersive-LCP}]
 Without loss of generality, we only consider the case $t>0$ and the case $t<0$ follows by symmetry. 
Also, we only need to consider the case $k_0=0$ and the same argument applies to other $k_0$, with the power of $\overline{\lambda}$ and $|t|$ shifted correspondingly when we introduce rescaled variables. More concretely, the estimate below for the part with $\overline{\lambda}$ being small remains the same, and this extra negative power of $\lambda$ introduces a negative power of $\overline{\lambda}$, which is in favour of all estimates in the remaining steps.

Consider the low-energy case first and we will indicate minor changes needed in the high-energy setting.
By the definition of Legendre distributions in Section~\ref{sec:Legendrian-distributions}, we know that (modulo a Schwartz term) $Q_j^{\calc} \specm   Q_{j'}^{\calc}$ has two parts: the first part is the part that corresponds to the part of $L^{\bfs}_{j,j'}$ that is away from $L^{\bfs} \cap \Lsharplow$; the second part is the part that is associated with the Legendrian conic pair $(L^{\bfs},\Lsharplow)$, which is given by oscillatory integrals like \eqref{QiEQj-s-lo}.
By the assumption on $L^{\bfs}_{j,j'}$, we know that the first part is a projectable Legendrian and the estimate in Proposition~\ref{prop:microlocal-dispersive-1} (or Proposition~\ref{prop:microlocal-dispersive-2}) with $k_\ell = 0$ applies.

Now we consider the contribution of $Q_j^{\calc} \specm   Q_{j'}^{\calc}$ from the piece like \eqref{QiEQj-s-lo} to the propagator, which is:
\begin{align} \label{eq:localized-propagator-LCP-1}
\begin{split}
  \int_0^{\infty} e^{it\lambda^2}\lambda^{n-1}  \int_0^\infty &\int_{\R^{k}}  e^{ \pm i\lambda\Phi(\sigma,y,y',v,s)/x} \big(\frac{x'}{\lambda s}\big)^{(n-1)/2 - (k+1)/2}
\\ & \sigma^{\frac{n-1}{2}}  s^{n-2} a(\lambda,\sigma,y,y',\frac{x'}{\lambda},v,s) \, dv \, ds  d\lambda.
\end{split}
\end{align}
Here we factor out the $-$ sign just for convenience, keeping $\Phi$ to be a positive quantity, which follows from the same discussion after \eqref{eq:spectral-measure-local}. In fact, 
since by definition of parametrizations in Section~\ref{subsec:Legendrian-geometry-low}, the value of $\pm \Phi$ is (say, in the region $\sigma \lesssim 1$) close to $\pm (1+\sigma)$ in the region that we are considering, where the sign depends on whether we are close to $L^{\bfs} \cap \Lsharplow_+$ or $L^{\bfs} \cap \Lsharplow_-$. So we can assume that $\Phi$ is close to $1+\sigma$.
In fact, by the definition of $L^{\bfs}$ using \eqref{eq:Lbf-definition-gamma^2}, we know that $\Phi$ parametrizes a region of $L^{\bfs}$ contained in the forward (in the left variable) flow-out of the diagonal locally if and only if $-\Phi$ parametrizes the region of $L^{\bfs}$ obtained by switching the left and right base variables and switching the sign of frequencies.
After substituting in \eqref{eq:defn-rescaled-lambda-x-x'}, \eqref{eq:localized-propagator-LCP-1} becomes
\begin{align} 
\begin{split}
t^{-n/2}I,
\end{split}
\end{align}
where
\begin{align} \label{eq:dispersive-Schrodinger-I-def}
\begin{split}
I =  \int_0^{\infty} \overline{\lambda}^{n-1}  \int_0^\infty &\int_{\R^{k}}  e^{i(\overline{\lambda}^2 \pm \overline{\lambda}\frac{\Phi(\sigma,y,y',v,s)}{\overline{x}})} \big(\frac{\overline{x'}}{\overline{\lambda} s}\big)^{(n-1)/2 - (k+1)/2}
  s^{n-2} \\ &  
  \sigma^{\frac{n-1}{2}} a(t^{-1/2}\overline{\lambda},\sigma,y,y',\frac{\overline{x'}}{\overline{\lambda}},v,s) \, dv \, ds  d\overline{\lambda}.
\end{split}
\end{align}
To prove \eqref{eq:localized-propagator-LCP-1}, we only need to show that $|I| \lesssim 1$.
Next we consider $I$ and decompose it according to $\overline{\lambda}$ as
\begin{equation} \label{eq:7.2-1} 
    I = I^{0} + I^{1},
\end{equation}
where 
\begin{align}
\begin{split}
I^0 =  \int_0^{\infty} \overline{\lambda}^{n-1} \phi_0(\overline{\lambda}) &\int_0^\infty \int_{\R^{k}}  e^{i(\overline{\lambda}^2 \pm \overline{\lambda}\frac{\Phi(\sigma,y,y',v,s)}{\overline{x}})} \big(\frac{\overline{x'}}{\overline{\lambda} s}\big)^{(n-1)/2 - (k+1)/2}
  s^{n-2} \\ & \sigma^{\frac{n-1}{2}}a(t^{-1/2}\overline{\lambda},\sigma,y,y',\frac{\overline{x'}}{\overline{\lambda}},v,s) \, dv \, ds  d\overline{\lambda},
  \end{split}
\end{align}
and 
\begin{align}
\begin{split}
I^1 = \int_0^{\infty} \overline{\lambda}^{n-1} (1-\phi_0(\overline{\lambda})) &\int_0^\infty \int_{\R^{k}}  e^{i(\overline{\lambda}^2 \pm \overline{\lambda}\frac{\Phi(\sigma,y,y',v,s)}{\overline{x}})} \big(\frac{\overline{x'}}{\overline{\lambda} s}\big)^{(n-1)/2 - (k+1)/2}
  s^{n-2} \\ & \sigma^{\frac{n-1}{2}}a(t^{-1/2}\overline{\lambda},\sigma,y,y',\frac{\overline{x'}}{\overline{\lambda}},v,s) \, dv \, ds  d\overline{\lambda},
  \end{split}
\end{align}
where $\phi_0$ is identically one on $(-\infty,1/2]$ and supported in $[0,1]$.
By the support condition in Definition~\ref{def:Legendrian-dis-conic-intersecting-low} and Definition~\ref{def:Legendrian-dis-conic-high}, $\frac{\overline{x'}}{\overline{\lambda}s}=\frac{x'}{\lambda s}$ is bounded on the support of the amplitude. So it is straightforward to see $|I^0| \lesssim 1$ and it remains to show $|I^1|\lesssim 1$.
For $I^1$ with $+$ sign in front of $\lambda\Phi/\overline{x}$, we integrate by parts in $\overline{\lambda}$. 
The derivative of the phase in $\overline{\lambda}$ is 
\begin{equation}
2\overline{\lambda} + \frac{\Phi}{\overline{x}} \geq 2\overline{\lambda}.
\end{equation}
So each time we improve by at least a factor $\overline{\lambda}^{-1}$. Here we used the symbolic property \eqref{eq:est-a-osc-main} of $a$ in $\lambda$ so that differentiating in $\overline{\lambda}$ won't introduce an extra $t^{-1/2}$ factor.
Then after $N > \frac{n+k}{2} +1$ times of integration by parts, we have
\begin{equation}
    |I^1| \lesssim \int_{1/2}^\infty \overline{\lambda}^{ \frac{n+k}{2} -N} d\overline{\lambda} \lesssim 1.
\end{equation}

Now we consider the case when we have a minus sign in the phase.
We first consider a dyadic decomposition in $\overline{\lambda}$ and a decomposition in $\overline{\lambda} \Phi/\overline{x}$.
Consider a partition of unity on $\R$:
\begin{align*}
1 = \phi_0(\cdot) + \sum_{j=1}^\infty \phi_j(\cdot),
\end{align*}
such that $\phi_0$ is as before: supported on $(-\infty,1]$ and is identically $1$ on $(-\infty,\frac{1}{2}]$; in addition here $\phi_j$ for $j \geq 1$ is supported on $[2^{j-2},2^{j+1}]$ and is identically $1$ on $[2^{j-1},2^j]$.
Using $1-\phi_0(\overline{\lambda}) = \sum_{j=1}^\infty \phi_j(\overline{\lambda})$, we can further decompose $I^{1}$ as
\begin{align}
I^1 = \sum I^1_{1} + I^1_2,
\end{align}
where
\begin{equation} \label{eq:dispersive-Schrodinger-I1-def}
\begin{split}
I^1_1 = &  \int_0^{\infty} \overline{\lambda}^{n-1}  \int_0^\infty \int_{\R^{k}} 
e^{i(\overline{\lambda}^2-\overline{\lambda}\frac{\Phi(\sigma,y,y',v,s)}{\overline{x}})} \Big(\frac{\overline{x'}}{\overline{\lambda} s}\Big)^{(n-1)/2 - (k+1)/2}
  s^{n-2} 
\\ & \sigma^{\frac{n-1}{2}}a(t^{-1/2}\overline{\lambda},\sigma,y,y',\frac{\overline{x'}}{\overline{\lambda}},v,s) 
 (1-\phi_0(\overline{\lambda}))
 \phi_0(4\overline{\lambda}\Phi/\overline{x})    \, dv \, ds  d\overline{\lambda},
 \end{split}
\end{equation}
and
\begin{align} \label{eq:dispersive-Schrodinger-I2-def}
I^1_2 = & \int_0^{\infty} \overline{\lambda}^{n-1}  \int_0^\infty \int_{\R^{k}}  e^{i(\overline{\lambda}^2-\overline{\lambda}\frac{\Phi(\sigma,y,y',v,s)}{\overline{x}})} \big(\frac{\overline{x'}}{\overline{\lambda} s}\big)^{(n-1)/2 - (k+1)/2}
  s^{n-2} 
  \\ & \sigma^{\frac{n-1}{2}}a(t^{-1/2}\overline{\lambda},\sigma,y,y',\frac{\overline{x'}}{\overline{\lambda}},v,s) 
 (1 - \phi_0(\overline{\lambda})) ( 1-\phi_0(4\overline{\lambda}\Phi/\overline{x}) )   \, dv \, ds  d\overline{\lambda}.
\end{align}

For $I^1_1$, we observe that on the support of $\phi_j(\overline{\lambda}) \phi_0(4\overline{\lambda}\Phi/\overline{x})$, we have:
\begin{align*}
\overline{\lambda} \in [2^{j-2},2^{j+1}] \geq \frac{1}{2},
\end{align*}
which in turn implies $\Phi/\overline{x} \leq \frac{1}{2}$ due to the $\phi_0(4\overline{\lambda}\Phi/\overline{x})$-factor. So we have

\begin{align} \label{eq:inq-15}
\partial_{\overline{\lambda}} \Big(\overline{\lambda}^2-\overline{\lambda}\frac{\Phi(\sigma,y,y',v,s)}{\overline{x}} \Big)
\geq \frac{1}{2} \overline{\lambda}.
\end{align}
In addition, derivatives of all other factors are $O(\overline{\lambda}^{-1})$:
\begin{itemize}
    \item The $\overline{\lambda}$-derivative of $a$ is $O(\overline{\lambda}^{-1})$ because of its symbolic (or conormal) property in $\overline{\lambda}$.
    \item The derivative of $1-\phi_0(\overline{\lambda})$ is supported on the region $\overline{\lambda} \lesssim 1$, hence is $O(\overline{\lambda}^{-1})$.
    \item The derivative of $\phi_0(\overline{\lambda}\Phi/\overline{x})$ is supported on the region $\overline{\lambda}\Phi/\overline{x} \sim 1$, and takes the form $\Phi/\overline{x}$ times a uniformly bounded function; hence the entire expression is $O(\overline{\lambda}^{-1})$.
\end{itemize}
So after integration by parts in $\overline{\lambda}$ for $N>\frac{n}{2}$ times we have
\begin{align*}
|I^1_1| \leq C_N    \int_{1/2}^\infty \overline{\lambda}^{n-1-2N} d\overline{\lambda} \lesssim 1.
\end{align*}

Now we consider $I^1_2$. If $\overline{x'} \gtrsim 1$, then
\begin{equation}
\frac{\Phi}{\overline{x}} = \frac{1+\sigma+s\sigma\psi}{\overline{x}} =  \frac{1+\sigma}{\overline{x}} + \frac{s\psi}{\overline{x'}}
\end{equation}
is smooth (with uniformly bounded derivatives) in $s,v$.
So we introduce a dyadic decomposition in terms of $2\overline{\lambda}-\frac{\Phi}{\overline{x}}$:
\begin{align*}
1 = \varphi_0\Big(2\overline{\lambda}-\frac{\Phi}{\overline{x}}\Big) + \sum_{j=1}^\infty \varphi_j\Big(2\overline{\lambda}-\frac{\Phi}{\overline{x}}\Big),
\end{align*}
where the $\varphi_0$-term is supported on the region $|2\overline{\lambda}-\frac{\Phi}{\overline{x}}| \lesssim 1$ and $\varphi_j$-term is supported on the region where 
\begin{equation} \label{eq:varphi-j-support-2-new}
10\times 2^j \leq |2\overline{\lambda}-\frac{\Phi}{\overline{x}}| \leq 20 \times 2^j.
\end{equation}
In addition, since now the size of the range of $\frac{\Phi}{\overline{x}}$ is $O(1)$, the range of $\overline{\lambda}$ on the support of $\varphi_j(2\overline{\lambda}-\frac{\Phi}{\overline{x}})$ is of size $O(2^j)$.

Using this dyadic decomposition, we decompose $I^1_2$ as
\begin{equation}
I^1_2 = \sum_{j=0}^\infty I^1_{2,j},
\end{equation}
where
\begin{align*}
I^1_{2,j} = & \int_0^{\infty} \overline{\lambda}^{n-1}  \int_0^\infty \int_{\R^{k}}  e^{i(\overline{\lambda}^2-\overline{\lambda}\frac{\Phi(\sigma,y,y',v,s)}{\overline{x}})} \big(\frac{\overline{x'}}{\overline{\lambda} s}\big)^{(n-1)/2 - (k+1)/2}
  s^{n-2} 
  \\& \sigma^{\frac{n-1}{2}}a(t^{-1/2}\overline{\lambda},\sigma,y,y',\frac{\overline{x'}}{\overline{\lambda}},v,s) \varphi_j(2\overline{\lambda}-\frac{\Phi}{\overline{x}})  (1 - \phi_0(\overline{\lambda})) ( 1-\phi_0(4\overline{\lambda}\Phi/\overline{x}) ) \, dv \, ds  d\overline{\lambda}.
\end{align*}

For $I^1_{2,0}$, we apply Proposition~\ref{prop:conic-pair-pointwise-bound-low} except with $x,x',\lambda$ replaced by rescaled ones in \eqref{eq:defn-rescaled-lambda-x-x'} and obtain
\begin{equation}
    |I^1_{2,0}|
    \lesssim \int_{|2\overline{\lambda} - \frac{\Phi}{\overline{x}}| \lesssim 1, \, \overline{x'}/\overline{\lambda} \lesssim 1}
    \overline{\lambda}^{n-1} \sigma^{\frac{n-1}{2}} (\overline{x'}/\overline{\lambda})^{\frac{n-1}{2}} d\overline{\lambda}.
\end{equation}
Then since $\overline{x'} \gtrsim 1$ in our current case, in combination with $|2\overline{\lambda} - \frac{\Phi}{\overline{x}}| \lesssim 1$ we know $\overline{\lambda} \lesssim 1$ and we obtain $|I^1_{2,0}| \lesssim 1$.

For $I^1_{2,j}$, we know 
\begin{equation}
 \partial_{\overline{\lambda}}\Big(  \overline{\lambda}^2-\overline{\lambda}\frac{\Phi(\sigma,y,y',v,s)}{\overline{x}}\Big) \gtrsim 2^j. 
\end{equation}
So we first integrate by parts in $\overline{\lambda}$ for $N$ times, and then apply Proposition~\ref{prop:conic-pair-pointwise-bound-low} to obtain
\begin{equation}
    |I^1_{2,j}| 
    \lesssim 2^{-jN}
    \int_{|2\overline{\lambda} - \frac{\Phi}{\overline{x}}| \lesssim 2^j, \, \overline{x'}/\overline{\lambda} \lesssim 1}
    \overline{\lambda}^{n-1} \sigma^{\frac{n-1}{2}} (\overline{x'}/\overline{\lambda})^{\frac{n-1}{2}} d\overline{\lambda}.
\end{equation}
If $\overline{x}^{-1} \lesssim 2^j$, then $\overline{\lambda} \lesssim 2^j$ in this region as well and we have (if we use $N \geq n+1$)
\begin{equation}
     |I^1_{2,j}| \lesssim 2^{-j(N-(n-1))} \times 2^j \leq 2^{-j}.
\end{equation}
If $\overline{x}^{-1} \gg 2^j$, then $\overline{\lambda} \sim \overline{x}^{-1} = \sigma^{-1} \overline{x'}^{-1}$ in this region, so the integrand is $O(1)$ and we have
\begin{equation}
    |I^1_{2,j}| \lesssim 2^{-jN} \times 2^j \lesssim 2^{-j(N-1)},
\end{equation}
which gives $\sum_j |I^1_{2,j}| \lesssim 1$ as well.

If $\overline{x'} \ll 1$, we use another decomposition to avoid the singular behaviour caused by differentiating the phase function.
Let $\mk{d}_{X_0}$ be the conic distance defined before \eqref{eq:mkd-X0}
and let 
\begin{equation} \label{eq:rescaled-mk-d-0}
    \overline{\mk{d}_{X_0}} = t^{-1/2}\mk{d}_{X_0}
\end{equation}
be the rescaled distance.
As sketched at the beginning of this subsection, now we introduce a dyadic decomposition in terms of $2\overline{\lambda}-|\overline{\mk{d}_{X_0}}|$:
\begin{align*}
1 = \varphi_0(2\overline{\lambda}-|\overline{\mk{d}_{X_0}}|) + \sum_{j=1}^\infty \varphi_j(2\overline{\lambda}-|\overline{\mk{d}_{X_0}}|),
\end{align*}
where the $\varphi_0$-term is supported on the region $|2\overline{\lambda}-|\overline{\mk{d}_{X_0}}|| \lesssim 1$ and $\varphi_j$-term is supported on the region where 
\begin{equation} \label{eq:varphi-j-support-2}
10\times 2^j \leq |2\overline{\lambda}-|\overline{\mk{d}_{X_0}}|| \leq 20 \times 2^j.
\end{equation}
Using this dyadic decomposition, we decompose $I^1_2$ as
\begin{equation}
I^1_2 = \sum_{j=0}^\infty I^1_{2,j},
\end{equation}
where
\begin{align*}
I^1_{2,j} = & \int_0^{\infty} \overline{\lambda}^{n-1}  \int_0^\infty \int_{\R^{k}}  e^{i(\overline{\lambda}^2-\overline{\lambda}\frac{\Phi(\sigma,y,y',v,s)}{\overline{x}})} \big(\frac{\overline{x'}}{\overline{\lambda} s}\big)^{(n-1)/2 - (k+1)/2}
  s^{n-2} 
  \\& \sigma^{\frac{n-1}{2}}a(t^{-1/2}\overline{\lambda},\sigma,y,y',\frac{\overline{x'}}{\overline{\lambda}},v,s) \varphi_j(2\overline{\lambda}-|\overline{\mk{d}_{X_0}}|)  (1 - \phi_0(\overline{\lambda})) ( 1-\phi_0(4\overline{\lambda}\Phi/\overline{x}) ) \, dv \, ds  d\overline{\lambda}.
\end{align*}

For each term $I^1_{2,j}$, we have $\overline{\lambda} \gtrsim 1$ on the support of the integrand due to the $(1 - \phi_0(\overline{\lambda}))$-factor.
For $I^1_{2,0}$, if $|\overline{\mk{d}_{X_0}}| \lesssim 1$, we have $\overline{\lambda} \lesssim 1$ on the support of the integrand due to the $\varphi_0(2\overline{\lambda}-|\overline{\mk{d}_{X_0}}|)$-factor.
If $|\overline{\mk{d}_{X_0}}| \gtrsim 1$, then $\overline{\lambda} \lesssim |\overline{\mk{d}_{X_0}}|$ on the support of the integrand again due to the $\varphi_0(2\overline{\lambda}-|\overline{\mk{d}_{X_0}}|)$-factor.
Then we can apply Proposition~\ref{prop:conic-pair-pointwise-bound-low} with $x,x',\lambda$ replaced by rescaled ones $\overline{x},\overline{x'},\overline{\lambda}$ introduced in \eqref{eq:defn-rescaled-lambda-x-x'} to obtain
\begin{equation}
  |I^1_{2,0}| \lesssim \int_{|2\overline{\lambda}-|\overline{\mk{d}_{X_0}}|| \lesssim 1} \overline{\lambda}^{\frac{n-1}{2}} 
  (\overline{x}')^{\frac{n-1}{2}} \sigma^{\frac{n-1}{2}} d\overline{\lambda}  \lesssim  \int_{|2\overline{\lambda}-|\overline{\mk{d}_{X_0}}|| \lesssim 1} \overline{\lambda}^{\frac{n-1}{2}} |\overline{\mk{d}_{X_0}}|^{-\frac{n-1}{2}} d\overline{\lambda}  \lesssim   1,
\end{equation}
where we made the same observation as in Corollary~\ref{coro:microlocalized-spectral-measure-osc-int-form-low}: products on the right hand side of \eqref{eq:conic-pair-integral-bound-low} can be bounded using:
\begin{align*}
\overline{x'}\sigma \lesssim |\overline{\mk{d}_{X_0}}|^{-1},
\end{align*}
which is equivalent to a more straightforward version: $x'\sigma \lesssim (d(z,z'))^{-1}$.
Here Proposition~\ref{prop:conic-pair-pointwise-bound-low} applies even though there is an extra factor $( 1-\phi_0(4\overline{\lambda}\Phi/\overline{x}))$.
This is because derivatives of this factor take the form (up to a constant factor and factors introduced by derivatives of $\Phi$):
\begin{equation}
    (\overline{\lambda}/\overline{x})^k \phi_0^{(j)}(4\overline{\lambda}\Phi/\overline{x}),
\end{equation}
which are supported in the region $\overline{\lambda}\Phi/\overline{x} \sim 1$. On the other hand, we have $|\Phi| = |1+\sigma+s\sigma\psi| \gtrsim 1$ and $\overline{\lambda} \gtrsim 1$ (due to the $1-\phi_0(\overline{\lambda})$-factor), hence this implies $\overline{\lambda}/\overline{x} \lesssim 1$ and all such derivatives are uniformly bounded and Proposition~\ref{prop:conic-pair-pointwise-bound-low} applies.

Now we consider $I^1_{2,j}$ with $j \geq 1$. 
For this term, the length of the $\overline{\lambda}$-interval we are integrating over becomes $O(2^j)$ and we need to exploit the  non-stationary nature of its phase on this part to control it.
For each $j \geq 1$, we further introduce a partition of unity in $\big||\overline{\mk{d}_{X_0}}|-\frac{\Phi(\sigma,y,y',v,s)}{\overline{x}}\big|$:
\begin{align*}
1 = \varrho_{j,0}\Big(|\overline{\mk{d}_{X_0}}|-\frac{\Phi(\sigma,y,y',v,s)}{\overline{x}}\Big) + \varrho_{j,1}\Big(|\overline{\mk{d}_{X_0}}|-\frac{\Phi(\sigma,y,y',v,s)}{\overline{x}}\Big),
\end{align*}
where the $\varrho_{j,0}$-term is supported on the region $\big||\overline{\mk{d}_{X_0}}|-\frac{\Phi(\sigma,y,y',v,s)}{\overline{x}}\big| \leq 2 \times 2^j$ and $\varrho_{j,1}$-term is supported on the region with
\begin{equation} \label{eq:varrho-j-support-2}
 \Big||\overline{\mk{d}_{X_0}}|-\frac{\Phi(\sigma,y,y',v,s)}{\overline{x}}\Big| \geq  2^j.
\end{equation}
They can be defined by choosing $\varrho_0 \in C_c^\infty([0,\infty))$ that is supported on $[0,2]$ and set
\begin{equation}
    \varrho_{j,0}(\bullet) = \varrho_0(2^{-j} \bullet), \quad   \varrho_{j,1}(\bullet) = 1-\varrho_0(2^{-j} \bullet).
\end{equation}
In this way, their $k$-th derivatives are bounded by
\begin{equation} \label{eq:varrhoj-derivative-bound}
    |\varrho_{j,0}^{(k)}|, \quad |\varrho_{j,1}^{(k)}| \lesssim_k 2^{-jk}.
\end{equation}
In addition, we notice that \eqref{eq:varrho-j-support-2} is equivalent to
\begin{align}
     \big|\overline{x}|\overline{\mk{d}_{X_0}}|-\Phi(\sigma,y,y',v,s)\big| \geq  2^j \overline{x}.
\end{align}
Since both terms on the left hand side are $O(1)$, we know that this term is present only when
\begin{align} \label{eq:7.2-4}
    \overline{x} \lesssim 2^{-j}.
\end{align}
Recall Proposition~\ref{prop:phase-equal-distance-X0}, we have $|\overline{\mk{d}_{X_0}}|-\frac{\Phi(\sigma,y,y',v,s)}{\overline{x}}= 0$ at critical points of the phase function, so this quantity is used to quantify the distance to the critical set of the phase.
Then we have the corresponding decomposition of $I^1_{2,j}$:
\begin{equation}
I^1_{2,j} = I^1_{2,j,0}+I^1_{2,j,1},
\end{equation}
where (for $\ell = 0,1$):
\begin{align*}
I^1_{2,j,\ell} = & \int_0^{\infty} \overline{\lambda}^{n-1}  \int_0^\infty \int_{\R^{k}}  e^{i(\overline{\lambda}^2-\overline{\lambda}\frac{\Phi(\sigma,y,y',v,s)}{\overline{x}})} \big(\frac{\overline{x'}}{\overline{\lambda} s}\big)^{(n-1)/2 - (k+1)/2} s^{n-2}  \sigma^{\frac{n-1}{2}}a(t^{-1/2}\overline{\lambda},\sigma,y,y',\frac{\overline{x'}}{\overline{\lambda}},v,s) 
   \\& \varphi_j(2\overline{\lambda}-|\overline{\mk{d}_{X_0}}|) (1 - \phi_0(\overline{\lambda})) ( 1-\phi_0(4\overline{\lambda}\Phi/\overline{x}) ) \varrho_{j,\ell}\big(|\overline{\mk{d}_{X_0}}|-\frac{\Phi(\sigma,y,y',v,s)}{\overline{x}}\big) \, dv \, ds  d\overline{\lambda}
\\ = & \int_0^{\infty} \overline{\lambda}^{\frac{n+k}{2}}  \int_0^\infty \int_{\R^{k}}  e^{i(\overline{\lambda}^2-\overline{\lambda}\frac{\Phi(\sigma,y,y',v,s)}{\overline{x}})} (\overline{x'})^{(n-1)/2 - (k+1)/2} s^{ \frac{n+k}{2}-1}  \sigma^{\frac{n-1}{2}}a(t^{-1/2}\overline{\lambda},\sigma,y,y',\frac{\overline{x'}}{\overline{\lambda}},v,s) 
   \\& \varphi_j(2\overline{\lambda}-|\overline{\mk{d}_{X_0}}|) (1 - \phi_0(\overline{\lambda})) ( 1-\phi_0(4\overline{\lambda}\Phi/\overline{x}) ) \varrho_{j,\ell}\big(|\overline{\mk{d}_{X_0}}|-\frac{\Phi(\sigma,y,y',v,s)}{\overline{x}}\big) \, dv \, ds  d\overline{\lambda}.
\end{align*}


For $I^1_{2,j,0}$, we have $|2\overline{\lambda}-|\overline{\mk{d}_{X_0}}|| \geq 10 \times 2^j$ and $||\overline{\mk{d}_{X_0}}|-\frac{\Phi(\sigma,y,y',v,s)}{\overline{x}}| \leq 2 \times 2^j$ on the support of its amplitude. 
This implies $|2\overline{\lambda} - \frac{\Phi(\sigma,y,y',v,s)}{\overline{x}}| \geq 8 \times 2^j$, which means we gain a factor $2^{-j}$ when we integrate by parts in $\overline{\lambda}$ from the derivative of the phase. 
Next we consider derivatives of other factors introduced after integration by parts.
\begin{enumerate}
  \item Whenever a $\overline{\lambda}$-derivative hits $a(t^{-1/2}\overline{\lambda},\sigma,y,y',\frac{\overline{x'}}{\overline{\lambda}},v,s)$, we use the symbolic estimate in \eqref{eq:spectral-measure-symbol-bound-low-1} to see that we gain a $\overline{\lambda}^{-1}$-factor.
  \item Whenever a $\overline{\lambda}$-derivative hits factors $(1 - \phi_0(\overline{\lambda}))$, $\overline{\lambda} \lesssim 1$ on the support of integrands of such terms. 
  \item Whenever a $\overline{\lambda}$-derivative hits $( 1-\phi_0(4\overline{\lambda}\Phi/\overline{x}) )$, we have $\overline{\lambda}\Phi/\overline{x} \lesssim 1$ on its support and we gain a $|\Phi/\overline{x}| \lesssim \overline{\lambda}^{-1}$-factor on the support of integrands of such terms. 
\end{enumerate}

So after $N \geq n+1$ times integration by parts, the result is bounded by
\begin{equation}
   \int_{|2\overline{\lambda}-|\overline{\mk{d}_{X_0}}|| \sim 2^j} \overline{\lambda}^{n-1-N} d\overline{\lambda} \lesssim 2^{-j},
\end{equation}
which is finite after summing in $j$.

For $I^1_{2,j,1}$, we integrate by parts in $v,s$ in this part.  We use \eqref{eq:varrho-j-support-2}, which holds on the support of the amplitude of this part, to quantify how non-stationary the phase is in this part.
Combining \eqref{eq:varrho-j-support-2} with Proposition~\ref{prop:1st-dPhi-est} we know
\begin{align} \label{eq:inq-16}
|s\partial_s\Phi| + |\partial_v\Phi| \gtrsim 2^{j/2} s  \overline{x}^{1/2}\sigma^{1/2}
\end{align}
in this region.
Then similar to the non-stationary phase argument using $L_1$ in \eqref{eq:L1-def}, we consider
\begin{equation} \label{eq:L2-def}
  L_2 = \frac{1}{|\partial_v \Phi|^2+s^2|\partial_s\Phi|^2} \Big(\sum_j \partial_{v_j}\Phi \partial_{v_j} + s^2\partial_s\Phi\partial_s \Big),
\end{equation}
with 
\begin{equation} \label{eq:L2-adjoint-def}
  L_2^* = - \sum_j \partial_{v_j} \Big(\frac{\partial_{v_j}\Phi}{|\partial_v \Phi|^2 + s^2|\partial_s\Phi|^2} \bullet \Big)
  - \partial_s \Big(\frac{s^2\partial_s\Phi}{|\partial_v \Phi|^2 + s^2|\partial_s\Phi|^2} \bullet \Big).
\end{equation}
Then each time we gain a factor $(\overline{\lambda}/\overline{x})^{-1}$ from this overall factor in the phase and we analyze other factors below.

For terms with all derivatives falling on $\varrho_{j,1}(|\overline{\mk{d}_{X_0}}|-\frac{\Phi(\sigma,y,y',v,s)}{\overline{x}})$,
factors involving derivatives are controlled by the coefficient of $L_2^*$ itself since they are of the form 
\begin{align} \label{eq:04}
\overline{x}^{-1}\frac{|\partial_{v_j}\Phi|^2}{|\partial_v \Phi|^2 + s^2|\partial_s\Phi|^2} = O(\overline{x}^{-1}), \quad 
\overline{x}^{-1}\frac{ s^2|\partial_s\Phi|^2}{|\partial_v \Phi|^2 + s^2|\partial_s\Phi|^2} 
 = O(\overline{x}^{-1}).
\end{align}
Then in combination with the $(\overline{\lambda}/\overline{x})^{-1}$-factor introduced when we integrate by parts, we eventually gain a factor of $2^{-j}\overline{\lambda}^{-1}$, where the $2^{-j}$ factor comes from \eqref{eq:varrhoj-derivative-bound}.
Then after $N \geq \max(k+1,2)$ times of integration by parts, we know that this term is bounded by
\begin{equation}
    2^{-jN} \int_{|2\overline{\lambda}-\overline{\mk{d}}_{X_0}| \sim 2^j, \overline{\lambda}\gtrsim 1} 
   \sigma^{\frac{n-1}{2}} \overline{x'}^{\frac{n-1-(k+1)}{2}} \overline{\lambda}^{\frac{n+k}{2}-N}d\overline{\lambda} 
\lesssim 2^{-jN} \int_{|2\overline{\lambda}-\overline{\mk{d}}_{X_0}| \sim 2^j, \overline{\lambda}\gtrsim 1} 
   \sigma^{\frac{n-1}{2}} \overline{x'}^{\frac{n-k-2}{2}} \overline{\lambda}^{\frac{n-k-2}{2}}d\overline{\lambda},
\end{equation}
If $\overline{\mk{d}}_{X_0} \lesssim 2^{j}$, then $\overline{\lambda} \lesssim 2^j$ and by \eqref{eq:7.2-4} we know $\sigma^{\frac{n-1}{2}} \overline{x'}^{\frac{n-k-2}{2}} \overline{\lambda}^{\frac{n-k-2}{2}} \lesssim 1$.
If $\overline{\mk{d}}_{X_0} \gg 2^{j}$, then $|2\overline{\lambda}-\overline{\mk{d}}_{X_0}| \sim 2^j$ shows $\overline{\lambda} \sim \overline{\mk{d}}_{X_0} \sim 1/\overline{x} = (\sigma \overline{x'})^{-1}$.
In both cases, the integrand in each term is $O(1)$ and the length of the interval is $O(2^j)$. So the total contribution of all such terms (i.e., after summing over $j$) is bounded by $\sum_j 2^{-j(N-1)} \lesssim 1$.

When the derivative falls on other factors,
by \eqref{eq:inq-16} and the same discussion as after \eqref{eq:L1-adjoint-def}, the factor introduced by derivatives on the denominator is $O(2^{-j/2}\overline{x}^{-1/2}\sigma^{-1/2}s^{-1})$. 
Combining the overall $(\overline{\lambda}/\overline{x})^{-1}$ factor, we gain a factor that is
\begin{equation} \label{eq:02}
O(2^{-j/2}\overline{\lambda}^{-1} \overline{x'}^{1/2} s^{-1}) 
\end{equation}
after integration by parts.
If $\frac{n+k}{2} \in \N$, then we integrate by parts for $N = \frac{n+k}{2} \geq k+1$ times and use $N_1,N_2$ to denote the times of derivative falling on $\varrho_{j,1}$ and all other factors respectively. 
In the current case we have $N_1+N_2 = N$ and $N_2 \geq 1$. The resulting oscillatory integral for fixed $\overline{\lambda}$ is bounded by: 
\begin{equation}
 2^{-j (N_1+\frac{N_2}{2}) }  \sigma^{\frac{n-1}{2}}
 \overline{x'}^{ \frac{n-1-(k+1)+N_2}{2} }
\overline{\lambda}^{\frac{n+k}{2}-N} 
\int_{s \gtrsim (\overline{\lambda}/\overline{x'})^{-1}} s^{ \frac{n+k}{2} - 1 -N_2 } ds
\lesssim  2^{-j (N_1+\frac{N_2}{2}) }  \sigma^{\frac{n-1}{2}} \overline{x'}^{ \frac{n-k-2+N_2}{2} }  \log(\overline{\lambda}/\overline{x'}),
\end{equation}
where we used the fact that we are considering the region $\overline{\lambda} \gtrsim 1$ and $\overline{x'}\ll 1$ now, which gives $\overline{\lambda}/\overline{x'}\geq 1$.
Consequently, we have
\begin{align} \label{eq:7.2-3}
\begin{split}
|I^1_{2,j,1}| & \lesssim \int_{|2\overline{\lambda} - |\overline{\mk{d}_{X_0}}||\sim 2^{j}, \, \overline{\lambda} \gtrsim 1} 
2^{-j (N_1+\frac{N_2}{2}) }  \sigma^{\frac{n-1}{2}} \overline{x'}^{ \frac{n-k-2+N_2}{2} }  \log(\overline{\lambda}/\overline{x'})d\overline{\lambda} .
\end{split}
\end{align}
If $|\overline{\mk{d}_{X_0}}| \lesssim 2^j$, then $\overline{\lambda} \lesssim 2^j$ and $\overline{x'}^{-1} \lesssim (\overline{x})^{-1} \sim |\overline{\mk{d}}_{X_0}| \lesssim 2^j$, in addition, we know $\sigma \overline{x'} \lesssim 2^{-j}$ by \eqref{eq:7.2-4} and we have
\begin{align} \label{eq:7.2-6}
\begin{split}
|I^1_{2,j,1}| & \lesssim \int_{|2\overline{\lambda} - |\overline{\mk{d}_{X_0}}||\sim 2^{j}, \, \overline{\lambda} \gtrsim 1} 
2^{-j (N_1+\frac{N_2}{2}) }  2^{-j\frac{n-k-2+N_2}{2}} \times j d\overline{\lambda} 
\lesssim j2^{-j(n-1)} \times 2^j,
\end{split}
\end{align}
which is $O(1)$ after summing in $j$ since $n \geq 3$.
If $|\overline{\mk{d}_{X_0}}| \gg 2^j$, then we know $\overline{\lambda} \sim |\overline{\mk{d}_{X_0}}| \sim 1/\overline{x} \gg 2^j$ in this region, and we have
\begin{align} 
\begin{split}
|I^1_{2,j,1}|
\lesssim 
\int_{|2\overline{\lambda} - |\overline{\mk{d}_{X_0}}||\sim 2^{j}, \, \overline{\lambda} \gtrsim 1}  
2^{-j (N_1+\frac{N_2}{2}) } \sigma^{\frac{n-1}{2}-\epsilon} \overline{x'}^{\frac{n-k-2+N_2}{2}-\epsilon} (\overline{x}^\epsilon|\log \overline{x}|) d\overline{\lambda} \lesssim 2^{-j(n-1-\epsilon)} \times 2^j,
\end{split}
\end{align}
where $\epsilon>0$ is a small fixed number and this contribution is again $O(1)$ after summing in $j$.

If $\frac{n+k}{2} \notin \N$, then $k \leq n-3$ (since $k \leq n-2$).
We integrate by parts for $N = \frac{n+k+1}{2} \geq 2$ times instead and the bound for fixed $\overline{\lambda}$ becomes
\begin{equation}
 2^{-j (N_1+\frac{N_2}{2}) } \sigma^{\frac{n-1}{2}}
  \overline{x'}^{ \frac{n-k-2+N_2}{2} } \overline{\lambda}^{\frac{n+k}{2}-N}  \int_{s \gtrsim (\overline{\lambda}/\overline{x'})^{-1}} s^{-3/2}ds  \sim 2^{-j (N_1+\frac{N_2}{2}) } \sigma^{\frac{n-1}{2}} \overline{x'}^{\frac{n-k-3+N_2}{2}}\overline{\lambda}^{\frac{n+k+1}{2}-N}.
\end{equation}
Then the rest of the argument about the $\overline{\lambda}$-integral goes through in the same way.

Now we discuss minor changes needed in other regions.
For the region with $\theta = x'/x  \lesssim 1$, the same argument applies with all primed and un-primed variables exchanged. 
In the high-energy case, i.e., when $\pm \Phi$ are instead parametrizing $(L^{\bfs,\calchigh},\Lsharphigh)$ near $L^{\bfs,\calchigh} \cap \Lsharphigh_\pm$.
We consider the part $\frac{x'}{s} \lesssim 1$ first. This corresponds to the terms contributed by the spectral measure that are like the oscillatory integrals on the left hand side of \eqref{eq:conic-pair-integral-bound-high} with phase function as in \eqref{eq:conic-pair-phase-construction-high-region1}.
In the step estimating $I^1_{2,0}$, we now use
Proposition~\ref{prop:conic-pair-pointwise-bound-high-region1} instead.
Also, in decompositions using $\overline{\mk{d}}_{X_0}$, we use $\overline{\mk{d}}_X$ defined in \eqref{eq:mkd-X} instead. Then in the step lower-bounding first order derivatives to estimate $I^1_{2,j,1}$, we still use Proposition~\ref{prop:1st-dPhi-est}, but now we use the part concerning the high-energy case. Then the remaining steps remain the same.

Finally, for the contribution to the propagator from the part with $\frac{s}{x'} \lesssim 1$, which corresponds to the contribution from the part of the spectral measure that consists of oscillatory integrals in the form of the left hand side of \eqref{eq:conci-pair-pointwisebound-high-region2} with phase functions as in \eqref{eq:conic-pair-phase-construction-high-region2}, 
the argument is the same as in the low-energy case and the region $\frac{x'}{s} \lesssim 1$ in the high-energy case except for the following minor differences.
Since the loss in $\overline{x'}$ is compensated directly when all factors $s$ are replaced by $\overline{x'}$ when we switch to this region, we only need to deal with the loss $h^{-\frac{k+1}{2}}$ now, which becomes $\overline{\lambda}^{\frac{k+1}{2}}$ after our scaling and the transformation $h=\lambda^{-1}$.
We use Proposition~\ref{prop:conic-pair-pointwise-bound-high-region2} to bound the analogue of $I^1_{2,0}$.
For $I^1_{2,j}$ with $j \geq 1$, we use Proposition~\ref{prop:1st-dPhi-est-region2} to lower bound derivatives of the phase for $I^1_{2,j,1}$. 
Then \eqref{eq:inq-16} is replaced by 
\begin{align} \label{eq:inq-17}
|\partial_\varkappa\Phi| + |\partial_v\Phi| \gtrsim 2^{j/2}  \overline{x'}\sigma^{1/2}.
\end{align}
Then the differential operator that we use in the non-stationary phase argument is replaced by 
\begin{equation} \label{eq:L2-def-high-region2}
  L_2 = \frac{1}{|\partial_v \Phi|^2+|\partial_\varkappa\Phi|^2} \Big(\sum_j \partial_{v_j}\Phi \partial_{v_j} + \partial_\varkappa\Phi\partial_\varkappa \Big),
\end{equation}
with 
\begin{equation} \label{eq:L2-adjoint-def-high-region2}
  L_2^* = - \sum_j \partial_{v_j} \Big(\frac{\partial_{v_j}\Phi}{|\partial_v \Phi|^2 + |\partial_\varkappa\Phi|^2} \bullet \Big)
  - \partial_\varkappa \Big(\frac{\partial_\varkappa\Phi}{|\partial_v \Phi|^2 + |\partial_\varkappa\Phi|^2} \bullet \Big).
\end{equation}
Then the rest of the proof applies, and in fact becomes simpler in the sense that now there is no special $s$-direction and we treat $\varkappa$ as one component of $v$-parameters. 
The bound for terms with all derivatives falling on $\varrho_{j,1}$ as in \eqref{eq:04} remains the same. 
The lower bound of derivatives of the phase function in \eqref{eq:inq-17}, together with the overall $\frac{1}{h\overline{x}}$-factor, shows that the analogue of \eqref{eq:02}, which is the improvement for each integration by parts, becomes
\begin{equation} \label{eq:03}
O(2^{-j/2}\overline{\lambda}^{-1} \overline{x'}^{-1/2}).
\end{equation}
Then the rest of the proof above applies.

\end{proof}

\begin{remark} \label{remark:microlocalized-est-Schrodinger}
It is worthwhile to summarize microlocal dispersive estimates that we have proved. 
In Proposition~\ref{prop:microlocal-dispersive-1}, each rank drop of the restriction of $\hatLprojhigh$ (resp. $\hatLprojlow$) to $L^{\bfs,\calchigh}_{j,j'}$ (resp. $\hat{L}^{\bfs}_{j,j'}$) causes $\big(\frac{\la z \ra \la z' \ra}{t}\big)^{1/2}$ order loss to the propagator from the contribution associated with $L^{\bfs,\calchigh}_{j,j'}$. If such degeneracy happens at the intersecting conic pair, then it does not cause a loss to the dispersive estimate if such an intersecting pair is admissible in the sense of Definition~\ref{definition:Lbf-boundary-admissible}. 
Similar numerology applies to dispersive estimates of the wave equation in the next section.
\end{remark}

We conclude this section with the proof of Corollary~\ref{coro:after-main-thm}.

\begin{proof}[Proof of Corollary~\ref{coro:after-main-thm}.]
    The result for $Y=\mathbb{S}^{n-1}$ and $X_0=\R^n\backslash\{0\}$ follows from Proposition~\ref{prop:example-admissible}.
    For the other part, i.e. for $Y$ with sectional curvature less than $1$, we notice that the exponential map is non-degenerate within distance $\pi$ by the Rauch comparison theorem. Using Proposition~\ref{prop:Lbf-RYstar-diffeomorphic}, we know that $d\hatLprojlow$ is always non-degenerate and we have $\IFz=0$. Then the conclusion follows by \eqref{eq:est-dispersive-1-1} with $\IF$ replaced by $\IFz$.
\end{proof}

\section{The dispersive estimate for the wave equation}
\label{sec:dispersive-wave}

We prove Theorem~\ref{thm:dispersive-wave}, i.e. dispersive estimates for wave equations, in this section.
We will first prove the conclusions on asymptotically conic manifolds, i.e. for $P$ in \eqref{eq:P-def}, and indicate changes needed for the exact cone case, i.e. $P$ in \eqref{eq:P-def-exact-cone}, at the end of this section.

\subsection{The estimate for the half-wave propagator}
We prove the first part of Theorem~\ref{thm:dispersive-wave} (i.e. estimate \eqref{eq:dispersive-wave-main-1})  in this subsection, which follows from the following global estimate for the half-wave propagator.

\begin{theorem}[Estimate for the half-wave propagator]
\label{thm:wave-propagator-est}
Let $\IF$ be as in \eqref{eq:IF-full-def}, and let $\varphi_K(\lambda)$ be supported in the region $\lambda \sim 2^K$, then
\begin{equation}
|\varphi_K(\sqrt{P})e^{it\sqrt{P}}(z,z')| \lesssim 
|t|^{- \frac{n-1-\IF}{2} }2^{K(\frac{n+1+\IF}{2})}.
\end{equation}
If $P$ is as in \eqref{eq:P-def-exact-cone} on $(X_0,g_0)$, then the same estimate holds with $\IF$ replaced by $\IFz$.
\end{theorem}
The adaptations needed for the exact cone case are still the same as those in the last part of the proof of Theorem~\ref{thm:Schrodinger-propagator-est-1}, and we will only discuss the asymptotically conic case in detail below.
To prove Theorem~\ref{thm:wave-propagator-est}, we use \eqref{eq:decomposition-spectral-measure} again, which  gives
\begin{align} \label{eq:decomposition-wave-propagator}
    \begin{split}
      \varphi_K(\sqrt{P})  e^{it\sqrt{P}} = \int_0^\infty \varphi_K(\lambda) e^{it\lambda} \specm
        = \sum_{j,j' \in \overline{ \mk{J}}_{\calc} } \int_0^\infty \varphi_K(\lambda) e^{it\lambda} Q^{\calc}_j \specm Q^{\calc}_{j'}.
    \end{split}
\end{align}

So Theorem~\ref{thm:wave-propagator-est} follows from the triangle inequality and summing the microlocalized estimate for the propagator below.
\begin{proposition} \label{prop:microlocal-dispersive-wave}
Let $\{Q_j^{\calc}\}$ be a microlocal partition of unity as in Proposition~\ref{prop: microlocal-partition-combined}.
If either $j$ or $j'$ is $1$ or $\zf$, or $j = j'$, then
\begin{equation} \label{eq:8.1-0}
|\int_0^\infty \varphi_K(\lambda) e^{i t \lambda} Q_j^{\calc} \specm Q_{j}^{\calc}  | \lesssim |t|^{-\frac{n-1}{2}} 2^{K\frac{n+1}{2}}.
\end{equation}
For $j,j' \in J$, $j \neq j'$ and $k$ as in Corollary~\ref{coro:microlocalized-spectral-measure-osc-int-form-low} and Corollary~\ref{coro:microlocalized-spectral-measure-osc-int-form-high}, we have
\begin{equation} \label{eq:8.1-1}
|\int_0^\infty \varphi_K(\lambda) e^{i t \lambda} Q_j^{\calc} \specm Q_{j'}^{\calc} | \lesssim |t|^{-\frac{n-1-k}{2}}2^{K(\frac{n+1+k}{2})}.
\end{equation}
In particular, by Proposition~\ref{prop:IF-relation-parameter-number}, we have $k \leq \IF$.
\end{proposition}

\begin{proof}
We consider the low-energy case first for definiteness and the argument is uniform up to the range $\lambda \to \infty$.
Also, we only consider the region where $\sigma \lesssim 1$ since the case $\theta = x'/x \lesssim 1$ can be proved in the same way after switching primed and un-primed variables.
Let $J,\zf,1$ be as in Section~\ref{subsec:microlocal-partition-low}.
If at least one of $j$ or $j'$ is either $\zf$ or $1$, then by Proposition~\ref{prop:localized-specm-one-side-residual-low}, we need to consider two types of terms: terms like \eqref{QjEQ1-low-rb} and terms like \eqref{QjEQ1-c-low}.
For terms like \eqref{QjEQ1-low-rb}, all properties needed in the proof of \cite[Equation~(6-3)]{Hassell-Zhang2016Strichartz} are satisfied and we can bound the left-hand side of \eqref{eq:8.1-1} in the same way.
For terms of the form \eqref{QjEQ1-c-low}, it can be bounded in the same way as the argument treating the case $j,j' \in J$ below, since $d(z,z') \sim 1/x$ in this region and it satisfies \eqref{eq:phg-conormal-2}.

Now we consider $Q_j^{\calc} \specm Q_{j'}^{\calc}$ with both $j,j' \in J$. For the case $j=j'$, we use \cite[Proposition~1.5]{Hassell-Zhang2016Strichartz} to write it as a finite sum of terms taking the form
\begin{equation}
    \lambda^{n-1} e^{\pm i \lambda d(z,z')} a_\pm(\lambda,z,z') + b(z,z'),
\end{equation}
where
\begin{align}
    \begin{split}
|\partial_\lambda^\alpha a_\pm(\lambda,z,z')| & \lesssim \lambda^{-\alpha}(1+\lambda d(z,z'))^{-\frac{n-1}{2}}, \\
  |\partial_\lambda^\alpha b(\lambda,z,z')| & \lesssim \lambda^{-\alpha}(1+\lambda d(z,z'))^{-N}, \; \forall N \in \mathbb{N}.      
    \end{split}
\end{align}
Then the estimate of the left hand side of \eqref{eq:8.1-1} follows in the same way as the argument for $j \neq j'$ below with the number of extra parameters $k = 0$.

For $j \neq j'$, consider the case with $L^{\bfs}_{j,j'}$ (or $L^{\bfs,\calchigh}_{j,j'}$) being away from the conic intersecting pair first. By Corollary~\ref{coro:microlocalized-spectral-measure-osc-int-form-low}, we know $Q_j^{\calc} \specm Q_{j'}^{\calc}$ is a finite sum of pieces of the form:
\begin{equation} 
\int \lambda^{n-1}  e^{\pm i \lambda \Phi/x} a(\lambda,z,z';v)dv,
\end{equation}
where $\pm \Phi$ parametrizes $L^{\bfs}$ locally and by Corollary~\ref{coro:phase-value-lowerbound-low}, we may assume $\Phi>\epsilon$ for a constant $\epsilon>0$ and the case $\Phi<-\epsilon$ is the same.
In addition, $a$ satisfies one of \eqref{eq:spectral-measure-symbol-bound-low-1} \eqref{eq:spectral-measure-symbol-bound-low-2}, and \eqref{eq:spectral-measure-symbol-bound-low-3}.
We consider the case that $a$ satisfies \eqref{eq:spectral-measure-symbol-bound-low-1} below and the other two cases follow in the same way.
In this case, the left hand side of \eqref{eq:8.1-1} is a finite sum of oscillatory integrals of the form
\begin{equation} \label{eq:8.1-5}
 \int_0^\infty \varphi_K(\lambda) e^{it\lambda} \lambda^{n-1}  e^{\pm i \lambda \Phi/x } a(\lambda,z,z';v)dv d\lambda,
\end{equation}
where $a$ satisfies \eqref{eq:spectral-measure-symbol-bound-low-1} and the phase is comparable to $d(z,z')$.

 When either $t \gg d(z,z')$ or $t \ll d(z,z') \sim \Phi/x$, we have
$|t-\Phi/x| \geq c|t|$, one can integrate by parts $N$-times in $\lambda$ to see that the oscillatory integral is bounded by 
\begin{align} \label{eq:wave-nonstationary-bound}
 \int_{2^{K-1}}^{2^{K+1}}  |t-\Phi/x|^{-N} \lambda^{n-1-N} (1+\lambda d(z,z') )^{- \frac{n-1-k}{2} } d\lambda
 \lesssim  |t|^{-N}2^{K(n-N)}(1+2^Kd(z,z'))^{-(n-1-k)/2} .
\end{align}
When $\frac{n-1}{2} \in \N$, we take $N = \frac{n-1}{2}$ and then the bound above is $\lesssim |t|^{-\frac{n-1}{2}}2^{K\frac{(n+1)}{2}}$.
When $\frac{n-2}{2} \in \N$, we then interpolate the bound for $N = \frac{n-2}{2}$ and $N = \frac{n}{2}$ to obtain the same bound.

Now we consider the case $t \sim d(z,z')$. Then the oscillatory integral is bounded by  
\begin{align*}
(2^{K+1}-2^{K-1}) \times 2^{K(n-1)}(1+2^K|t|)^{-\frac{n-1-k}{2}} \lesssim 2^{K (\frac{n+1+k}{2}) } (2^{-K}+|t|)^{-\frac{n-1-k}{2}},
\end{align*}
which completes the proof of this part.

For the contributions associated with the conic intersecting pair, 
the issue to remedy is again that we have $k+1$ extra parameters while we want to produce a bound as if there are only $k$ extra parameters. Then we first apply the last part of the proof of Proposition~\ref{prop:microlocal-dispersive-1}, which gives $(1+\lambda d(z,z'))^{-1/2}$ improvement, and then the rest of the argument above goes through.
\end{proof}

\subsection{Dispersive estimates for the wave equation: contribution from the Legendrian conic pair}
\label{subsec:dispersive-conic-pair-contribution-wave}

In this subsection we prove the second part of Theorem~\ref{thm:dispersive-wave}, i.e. estimate \eqref{eq:dispersive-wave-main-2} and the corresponding statement for $(X_0,g_0)$ with $\IFint$ replaced by $\IFintz$.

\begin{theorem}[Estimate for the half-wave propagator]
\label{thm:wave-propagator-est-admissible}
Suppose that $(L^{\bfs,\calchigh},\Lsharphigh)$ is admissible in the sense of Definition~\ref{definition:Lbf-boundary-admissible}, for $P$ in \eqref{eq:P-def}, $\IFint$ in \eqref{eq:IFint-def}, and $\varphi_K(\lambda)$ supported in the region $\lambda \sim 2^K$, then
\begin{equation}
\big|\varphi_K(\sqrt{P})e^{it\sqrt{P}}(z,z')\big| \lesssim |t|^{- \frac{n-1-\IFint}{2} }2^{K(\frac{n+1+\IFint}{2})}.
\end{equation}
For $P$ defined by \eqref{eq:P-def-exact-cone} on $(X_0,g_0)$, suppose that $(L^{\bfs},\Lsharplow)$ is admissible, then the same estimate holds with $\IFint$ replaced by $\IFintz$ on the right hand side.
\end{theorem}

\begin{proof}  
For definiteness, let $P$ be as in \eqref{eq:P-def} first.
 The proof is the same as that of Theorem~\ref{thm:wave-propagator-est}, except for the treatment of the microlocalized pieces of the spectral measure associated with the part near $L^{\bfs,\calchigh} \cap \Lsharphigh$ or $L^{\bfs} \cap \Lsharplow$ in the high- and low-energy regimes respectively. 
Consider the case $\sigma = x/x' \lesssim 1$ and the case $x'/x \lesssim 1$ can be proved in the same way.
In this case, we know that $z,z'$ are connected by a unique geodesic. For such contributions, we need to show that they satisfy \eqref{eq:8.1-1} with $k=0$.
They are of the form
\begin{equation}
    \int_0^\infty \varphi_K(\lambda) e^{i t \lambda} Q_j^{\calc} \specm Q_{j'}^{\calc},
\end{equation}
where $Q_j^{\calc} \specm Q_{j'}^{\calc}$ takes the form of the left hand side of one of \eqref{eq:conic-pair-integral-bound-low}, \eqref{eq:conic-pair-integral-bound-high} and \eqref{eq:conci-pair-pointwisebound-high-region2}.
We divide into two cases according to the derivative of the phase in $\lambda$.
When either $t \gg |\mk{d}(z,z')| \sim |\Phi/x|$ or $t \ll |\mk{d}(z,z')|  \sim |\Phi/x|$, we can apply the argument in the derivation of \eqref{eq:wave-nonstationary-bound} and obtain the result.
If instead $t \sim |\mk{d}(z,z')|$, then we apply Proposition~\ref{prop:conic-pair-pointwise-bound-low}, Proposition~\ref{prop:conic-pair-pointwise-bound-high-region1}, Proposition~\ref{prop:conic-pair-pointwise-bound-high-region2} to estimate the microlocalized spectral measure at fixed $\lambda$. Observing that 
\begin{align}
\sigma^{\frac{n-1}{2}} (x')^{\frac{n-1}{2}} \sim |\mk{d}(z,z')|^{-1} \sim t^{-1},    
\end{align}
we have
\begin{equation}
    |\int_0^\infty \varphi_K(\lambda) e^{i t \lambda} Q_j^{\calc} \specm Q_{j'}^{\calc}|
    \lesssim 2^{K} \times 2^{K \frac{n-1}{2}} \times |t|^{-\frac{n-1}{2}},
\end{equation}
which completes the proof.

Now for $P$ in \eqref{eq:P-def-exact-cone} on $(X_0,g_0)$, the modifications needed in the proof are the same as those indicated at the end of the proof of Theorem~\ref{thm:dispersive-Schrodinger-1}. Most importantly, the maximal number of extra parameters needed for the part of the microlocalized spectral measure associated with pieces away from $L^{\bfs} \cap \Lsharplow$ is $\IFintz$, while the part near this conic intersecting pair does not depend on the focusing index under our admissible assumption.
\end{proof}

\appendix

\section{Riemannian geometry interpretation of the focusing intensity}
\label{sec:riemannian_geometry_interpretation_of_the_focusing_intensity}
In this Appendix, we give an interpretation of the index of focusing in the context of Riemannian geometry. 
Much of the material here is standard, but it is usually not presented exactly in the form that we need.
We will only discuss the degeneracy of $\hatLprojlow: \; L^{\bfs} \to \bfs$ since the characterization of $\hatLprojhigh$ down to $x,x' \to 0$ using Riemannian geometry will not be simplified compared with its original definition or give more intuition compared with the $L^{\bfs}$-part. 
We will first show how $L^{\bfs}$ is related to $\RYstar$ in \eqref{eq:scr-RY-*-definition} below, and then discuss the connection between the degeneracy of its projection to $Y \times Y$ and the multiplicity of conjugate points.

The metric $\Ymetric$ on our boundary $\partial X =Y$ induces a function on $T^*Y$ by
\begin{equation}
\Ymetric^*(y,\mu) = \Ymetric^{-1}(y)(\mu,\mu),
\end{equation}
where $\Ymetric^{-1}$ is the inverse of the metric $\Ymetric$.
Then we denote the unit cosphere bundle, which is the level set $\{ \Ymetric^* = 1\} \subset T^*Y$, by $S^*Y$.
Then it is well-known that the geodesic flow on $TY$, under the identification of $TY$ and $T^*Y$ using the metric $h$, is the Hamilton flow (in terms of the natural symplectic structure on $T^*Y$) of $\frac{1}{2}\Ymetric^*$.

Recall $\RYstar$ defined in \eqref{eq:scr-RY-*-definition}, we will show that, away from the diagonal of $S^*Y \times S^*Y$, the drop of the rank of the exponential map (i.e., the multiplicity of conjugate points) equals the drop of the rank of $d\Pi_{\calclow}$, where
\begin{equation} \label{eq:def-Pi-Y}
  \Pi_{\calclow}: \; \RYstar \to Y \times Y
\end{equation}
is the projection.
We then define $\RY \subset SY \times SY$, which is the analogue of $\RYstar$ that is more suitable for applying standard geometric tools, to be
\begin{align} \label{eq:scr-R-definition}
\RY = \{ (y,v,y',v'): \exists \, \ell \in \R: (\gamma_{y,v}(\ell),\dot{\gamma}_{y,v}(\ell)) = (y',v') \},
\end{align}
where $\gamma_{y,v}$ is the unique (unit-speed) geodesic determined by $(y,v)$. Under the identification of $SY \times SY$ with $S^*Y \times S^*Y$ using the metric $\Ymetric$, $\RY$ is identified with $\RYstar$ in \eqref{eq:scr-RY-*-definition}.

Next we define $\Rprojlow$ to be the projection from $\RY$ to $Y \times Y$:
\begin{align} \label{eq:Rprojlow-def}
\Rprojlow: \quad (y,v,y',v') \to (\mathsf{y}=y,\mathsf{y}'=y').
\end{align}
Here $\mathsf{y},\mathsf{y}'$ are just introduced as maps $\RY \to Y$ to avoid abusing notations like $dy,dy'$. Both $d\mathsf{y}$ and $d\mathsf{y}'$ always have full rank, but the entire $d\Rprojlow=(d\msf{y},d\msf{y}')$ might degenerate, and this depends on how many $d\mathsf{y}'_j$ are linearly dependent on $d\mathsf{y}$. Equivalently, one may think of $\mathsf{y}$ as a family of parameters and characterize the degeneracy of $d\mathsf{y}'$ for fixed $\mathsf{y}$.
Geometrically, by the definition of $\RY$, the tangent space $T_{\msf{p}}\RY$ at $\msf{p}=(y_0,v_0,y_0',v_0')$ consists of two `radial'  directions\footnote{Corresponding to the left and right variables moving in the same and opposite directions along the geodesic.} summed with all variations of the geodesic $\gamma_0$ connecting $(y_0,v_0)$ and $(y_0',v_0')$.

Let $\msf{p} = (y_0,v_0,y_0',v_0') \in \RY$. We will investigate properties of $d\Rprojlow$ on $T_\msf{p}\RY$. 
In fact, this also gives a quantitative control of the degeneracy of $d\Rprojlow$.
To this end, we consider $\RY$ with the fixed left variable on the base manifold. That is, consider $\Ryz \subset SY$ defined by:
\begin{align} \label{eq:Ryz-definition}
\Ryz = \{ (y',v') : \exists \ell \in \R,\, v \in S_{y_0}Y \text{ such that } (\gamma_{y_0,v}(\ell),\dot{\gamma}_{y_0,v}(\ell)) = (y',v') \},
\end{align}
where $\gamma_{y_0,v}$ is the unique (unit-speed) geodesic determined by $(y_0,v) \in S_{y_0}Y$. 

Let 
\begin{equation} \label{eq:def-proj-calclow-right}   
  \tilde{\Pi}_{\calclow,2}: \Ryz \to Y
\end{equation}
be the projection and we will prove:
\begin{align} \label{eq:rkdpi=n-plus-rkdpi2}
\mathrm{rank}\, d\Rprojlow = n - 1  + \mathrm{rank}\, d\tilde{\Pi}_{\calclow,2}.
\end{align}
In addition, in local coordinates, $d\tilde{\Pi}_{\calclow,2}$ is a submatrix of $d\Rprojlow$ and it is the only part in which potential degeneracy of $d\Rprojlow$ happens.

This follows from the following characterization of $\RY$ as a flow out.  
We use $\mathrm{EXP}(\ell,\cdot,\cdot)$ to denote the Riemannian\footnote{As we have reserved the notation $\exp$ for the symplectic version.} exponential map:
\begin{equation}
\mathrm{EXP}(\ell,y,v) = (\gamma_{y,v}(\ell),\dot{\gamma}_{y,v}(\ell)).
\end{equation}
By definition of $\RY$, we have (not necessarily unique) $\ell_0 \in \R$ such that 
\begin{align*}
(\gamma_{y_0,v_0}(\ell_0),\dot{\gamma}_{y_0,v_0}(\ell_0)) = (y_0',v_0').
\end{align*}
Let $U_{(y_0,v_0)}$ be a neighborhood of $(y_0,v_0) \in SY$. Then a typical neighborhood of $\msf{p}$ in $\RY$ is 
\begin{align}  \label{eq:typical-neighbor-RY}
\mathcal{U}_{\msf{p}} = \{ (y,v, \mathrm{EXP}(\ell,y,v) ): (y,v) \in U_{(y_0,v_0)} , \; \ell \in (\ell_0-\delta,\ell_0+\delta) \}
\end{align}
for $\delta>0$ and one can take $(y,\ell,v)$ as coordinates on $\RY$ locally and view $(y',v')=\mathrm{EXP}(\ell,y,v)$ as functions of $(y,\ell,v)$.
In this coordinate system, $d\Rprojlow$ at $\mathsf{p}=(y_0,v_0,y_0',v_0') = (y_0,v_0,\mathrm{EXP}(\ell_0,y_0,v_0))$ takes the form
\begin{align} \label{eq:Rprojlow-derivative}
\begin{pmatrix}  
  \mathrm{Id} & 0 & 0 \\
 \frac{\partial y'}{\partial y} & v_0'  &  \frac{\partial y'}{\partial v}
\end{pmatrix}.
\end{align}
Here $\Id$ is $(n-1) \times (n-1)$. The rank of the lower right corner is exactly $\mathrm{rank}\, d\tilde{\Pi}_{\calclow,2}$ by definition since $(\ell,v)$ parametrizes $\Ryz$  
and we obtain \eqref{eq:rkdpi=n-plus-rkdpi2}.

On the other hand, $d\tilde{\Pi}_{\calclow,2}$ has a classical geometric interpretation. Consider the $(n-2)$-plane of $T_{y_0}Y$ that is orthogonal to $v_0$ and denote a basis of it by $e_1,...,e_{n-2}$\footnote{Rigorously speaking, they are further lifted to $T_{(y_0,v_0)}SY$.}. 
Under the push-forward of the geodesic flow, they become Jacobi fields $J_i$ determined by (with $R$ denoting the Riemann curvature tensor):
\begin{align} \label{eq:Jacobi-Ji}
\ddot{J}_i + R(J_i,\dot{\gamma}_{y_0,v_0})\dot{\gamma}_{y_0,v_0} = 0,
\end{align}
with initial conditions
\begin{align} \label{eq:Jacobi-initial-condition}
J_i(0) = 0,  \; \dot{J}_i(0) = e_i.
\end{align}
Then we know $d\mathrm{EXP}(\ell_0,y_0,\cdot)(e_i) = J_i(\ell_0)$. 
Since $e_1,...,e_{n-2}$ span the tangent plane to the fiber part of $T_{(y_0,v_0)}SY$, combining with the definition \eqref{eq:Ryz-definition},  $d\mathrm{EXP}(\ell_0,y_0,v_0)(e_i)$ spans the tangent plane of $\mathscr{R}_{y_0,\ell_0}$\footnote{Both of them are exactly parametrizing all those geodesic variations in normal directions, fixing the starting point $y_0$.}. After composing with $d\tilde{\Pi}_{\calclow,2}$, they become $J_i(\ell_0)$ and they form the $\frac{\partial y'}{\partial v}$-block in \eqref{eq:Rprojlow-derivative}.
In addition, we know 
\begin{equation} \label{eq:Jis0-v0'-orthogonal}
  \la J_i(\ell_0), v_0' \ra_{\Ymetric(y_0')} = 0
\end{equation}
since $\la J_i(0),v_0 \ra_{\Ymetric(y_0)} = 0$ and this inner product is preserved under the push-forward by the exponential map due to Gauss's lemma on the exponential map.

To summarize, the degeneracy of $d\tilde{\Pi}_{\calclow,2}$ (equivalently, the degeneracy of $d\Rprojlow$) corresponds to the degeneracy of the exponential map and the vanishing of normal Jacobi fields, which in turn corresponds to the multiplicity of conjugate points (see e.g. \cite[Section~5.3]{doCarmo}).

Similarly, denote the fiber over $(\sigma_0,y_0)$ of the projection to $(\sigma,y)$ components of $\hat{L}^{\bfs}$ by $\hat{L}^{\bfs}(\sigma_0,y_0)$, then in the same way as we compare $d\tilde{\Pi}_{\calclow,2}$ with $d\Rprojlow$
using \eqref{eq:Rprojlow-derivative}, we can view the projection $\hatLprojlow: \hat{L}^{\bfs} \to \bfs$ as always having full rank in $(\sigma,y)$-components and the degeneracy of $\hatLprojlow$ is equivalent to the degeneracy of $\hat{L}^{\bfs}(\sigma_0,y_0) \to Y$.
So in order to investigate the relationship between $\RYstar$ and $\hat{L}^{\bfs}$ and show that the degeneracy of $\Rprojlow$ corresponds to the degeneracy of $\hatLprojlow$, we only need to prove the same property between $\mathscr{R}^*_{y_0}$ and $\hat{L}^{\bfs}(\sigma_0,y_0)$, which is shown below.

\begin{proposition} \label{prop:Lbf-RYstar-diffeomorphic}
Fix $y_0 \in Y$ and $\sigma_0 \in (0,2)$, 
and let $\mathscr{R}^*_{y_0}$ be the image of $\mathscr{R}_{y_0}$ under the identification between $SY$ and $S^*Y$, then $\mathscr{R}^*_{y_0}$ is diffeomorphic to $\hat{L}^{\bfs}(\sigma_0,y_0)$ via a fiber preserving diffeomorphism $\RSlow$. \footnote{With RS standing for Riemannian-Symplectic, reflecting the feature of this correspondence.}
As a consequence, according to the discussion above, 
$\hatLprojlow$ degenerates when its projection to $Y \times Y'$ is a pair of conjugate points and the rank drop of $\hatLprojlow$ equals the multiplicity of this pair of conjugate points.
\end{proposition}

\begin{proof}

Consider a point $(y',\hat{\mu'}) = \exp(\ell H_{\frac{1}{2}\Ymetric^*,Y})(y_0,\hat{\mu}) \in \mathscr{R}^*_{y_0}$.
For definiteness of choice of coordinates, we will consider the case when $\ell$ is close to $\pi$ and the corresponding part on $\hat{L}^{\bfs}$ is close to the lift of $L^{\bfs} \cap \Lsharplow$, since the case away from this region is simpler and follows from the same argument.

By the definition of $\mathscr{R}^*_{y_0}$, $(\hat{\mu},\ell)$ gives a coordinate system on $\mathscr{R}^*_{y_0}$. We define the map $\RSlow:   \mathscr{R}^*_{y_0} \to \hat{L}^{\bfs}(\sigma_0,y_0)$ by
\begin{equation}
 \RSlow: (\hat{\mu},\ell) \to (\sigma_0, y_0, y', \overline{\nu},\frac{ \overline{\nu}-(1+\sigma)\nu_1 }{|\mu'|}, \frac{\mu}{|\mu'|}, \hat{\mu'},|\mu'|).
\end{equation}
That is, we send $(y',\hat{\mu'}) \in \mathscr{R}^*_{y_0}$ to the point in $\hat{L}^{\bfs}$ that corresponds to the same geodesic as the one starting at $(y_0,\hat{\mu})$ and ending at $(y',\hat{\mu'})$.
We will show that all other components are uniquely determined by $(y',\hat{\mu'})$, or in fact by $(\hat{\mu},\ell)$.

Firstly, recall the definition of $L^{\bfs}$ using \eqref{eq:Lbf-definition-gamma^2} and we will show that $(\sigma_0,\ell)$ determines $(s_l,s_r)$ in \eqref{eq:Lbf-definition-gamma^2}. We have $\ell = s_r - s_l$ and
\begin{align*}
\sigma_0 = \frac{x}{x'} = \frac{\sin s_l}{\sin s_r} = \frac{\sin (s_r-\ell)}{\sin s_r}.
\end{align*}
This gives
\begin{equation} \label{eq:ell-determined-by-sr}
 \ell = s_r - \arcsin (\sigma_0 \sin s_r).
\end{equation}
Notice that $s_r$ is close to $\pi$, so $\arcsin (\sigma_0 \sin s_r) \sim \sigma_0 (\pi-s_r)$. So (for fixed $\sigma_0$) $\ell$ and $s_r$ are smooth functions of each other.
Then $|\mu'| = \sin s_r$ is determined as well.
Next $\mu/|\mu'|$ is determined by $\mu/|\mu'| = \sigma_0\frac{\mu}{|\mu|}$.

Argument above also shows that $(\sigma_0,\ell)$ determines $\nu = \cos s_l$, $\nu' = \cos s_r$ as well.
Consequently, for $\overline{\nu}$, we have
\begin{equation}
\overline{\nu} = \nu + \sigma \nu' = \cos s_l + \sigma \cos s_r.
\end{equation}
For $\frac{ \overline{\nu}-(1+\sigma)\nu_1 }{|\mu'|}$, we have
\begin{equation} \label{eq:resolved-nu}
  \frac{ \overline{\nu}-(1+\sigma)\nu_1 }{|\mu'|} = \frac{ \nu - \nu' }{\sin s_r} = 
\frac{(1- (\sigma\sin s_r)^2)^{1/2} - (1- (\sin s_r)^2 )^{1/2}}{\sin s_r},
\end{equation}
where the potential singularity when $s_r \to 0$ is removable since $(1-\bullet^2)^{1/2} = 1+\bullet^2h_1(\bullet)$ for $\bullet$ small, with $h_1$ being a smooth function.

The last thing to show is that the inverse of $\RSlow$ is smooth. 
By \eqref{eq:ell-determined-by-sr}, $\ell$ is a smooth function of $s_r  = \arcsin |\mu'|$, hence a smooth function of $|\mu'|$.
Then $\hat{\mu}$ is determined via
\begin{align*}
(y_0,\hat{\mu})= \exp(-\ell H_{\frac{1}{2}\Ymetric^*})(y',\hat{\mu'}),
\end{align*}
completing the proof.
\end{proof}

\begin{remark} \label{remark:no-need-part}
In addition, \eqref{eq:resolved-nu} (and its analogue in the region $x'/x \lesssim 1$) shows that when we consider the parametrization of $\hat{L}^{\bfs}$ (and similarly for $\hat{L}^{\bfs,\calchigh}$), we never need to consider the region that $\overline{\nu}-(1+\sigma)\nu_1$ dominates $|\mu|,|\mu'|$, since the former is either $O(|\mu|^2)$ or $O(|\mu'|^2)$ as either of $|\mu|$, $|\mu'|$ tends to $0$.
\end{remark}

\section{Relaxed admissible condition}
\label{appendix:relax-condition}

In this appendix, we prove Proposition~\ref{prop:conic-pair-pointwise-bound-low}, Proposition~\ref{prop:conic-pair-pointwise-bound-high-region1}, and Proposition~\ref{prop:conic-pair-pointwise-bound-high-region2} under a relaxed version of the admissible condition in Definition~\ref{definition:Lbf-boundary-admissible}.
Using the same argument in Section~\ref{sec:dispersive-Schrodinger} and Section~\ref{sec:dispersive-wave}, this proves our main results under this relaxed condition.


\begin{definition}[Relaxed admissible condition] 
\label{def:relaxed-admissible}
Let $v=\hat{\mu'}_I$. We say that $X$ (or $L^{\bfs,\calchigh}$) satisfies the relaxed admissible condition if $Y'_I,Y'_{n-1}$ in various regions satisfy the following conditions.

\begin{enumerate}
\item
Consider the low-energy case (i.e. $\varrho = x'/s = 0$) first. We require $Y'_I,Y'_{n-1}$ in \eqref{eq:hatLbf-low-components} to satisfy the following conditions.
\begin{enumerate}
    \item The lower bound of the determinant:
\begin{align} \label{eq:determinant-low-1}
|\det \partial_{(s,v)}(Y'_{n-1},Y'_I)|\ge c s^k. 
\end{align}
\item  There is an integer $N>k+1$ such that the following holds. For every $1\le q\le N$, every choice of $m_\ell \in \N,\beta_\ell \in \N^k$, $\ell = 1,2,...q$, with $\sum_{\ell} m_\ell+|\beta_\ell|\le N$, $Y'_{\alpha_\ell}\in\{Y'_{n-1},Y'_{I,1},...,Y'_{I,k}\}$, and $D_\ell\in\{s\partial_s,\partial_{v_1},...,\partial_{v_k}\}$, we have\footnote{One can in principle include $D_\ell$ into other derivatives. But one should view $D_\ell$ as the differentiation appearing in the coefficients of $L$ in \eqref{eq:L-def-appendix} while others are differentiations appearing in the integration by parts. So this numerology makes it clearer what is the number of integration by parts. }
\begin{align}
\prod_{\ell=1}^q |(s\partial_s)^{m_\ell}\partial_v^{\beta_\ell}D_\ell Y'_{\alpha_\ell}|\le C_N s^{q-k}|\det \partial_{(s,v)}(Y'_{n-1},Y'_I)|.
\end{align}
\end{enumerate}

\item Consider the region $\varrho=x'/s\lesssim1$ now.
We require $Y'_I,Y'_{n-1}$ in \eqref{eq:hatLbf-high-components-region1} to satisfy the following conditions.
\begin{enumerate}
\item The lower bound of the determinant:
\begin{align}
|\det \partial_{(s,v)}(Y'_{n-1},Y'_I)|\ge c s^k .
\end{align}
\item There is an integer $N>k+1$ such that the following holds. 
For every $1\le q\le N$, every choice of $m_\ell \in \N,\beta_\ell \in \N^k$, $\ell = 1,2,...q$ with $\sum_\ell m_\ell+|\beta_\ell|\le N$, $Y'_{\alpha_\ell}\in\{Y'_{n-1},Y'_{I,1},...,Y'_{I,k}\}$, and $D_\ell\in\{s\partial_s-\varrho\partial_\varrho,\partial_{v_1},...,\partial_{v_k}\}$, we have
\begin{align}
\prod_{\ell=1}^q |(s\partial_s-\varrho\partial_\varrho)^{m_\ell}\partial_v^{\beta_\ell}D_\ell Y'_{\alpha_\ell}|\le C_N s^{q-k}|\det \partial_{(s,v)}(Y'_{n-1},Y'_I)|.
\end{align}
\end{enumerate}

\item Consider the region with $\varkappa = s/x' \lesssim 1$.
We require $Y'_I,Y'_{n-1}$ in \eqref{eq:hatLbf-high-components-region2} to satisfy the following conditions. 

\begin{enumerate}
    \item The lower bound of the determinant:
\begin{align}
|\det \partial_{(\varkappa,\hat{\mu'}_I)}(Y'_{n-1},Y'_I)|\ge c(x')^{k+1}.
\end{align}

\item There is an integer $N>k+1$ such that the following holds. 
For every $1\le q\le N$, every choice of $\beta_\ell \in \N^{k+1}$, $\ell = 1,2,...,q$, with $\sum_\ell |\beta_\ell|\le N$, $Y'_{\alpha_\ell}\in\{Y'_{n-1},Y'_{I,1},...,Y'_{I,k}\}$, and $D_\ell\in\{\partial_\varkappa,\partial_{\hat{\mu'}_{I,1}},...,\partial_{\hat{\mu'}_{I,k}}\}$, we have
\begin{align}
\prod_{\ell=1}^q |\partial_{(\varkappa,\hat{\mu'}_I)}^{\beta_\ell}D_\ell Y'_{\alpha_\ell}|\le C_N(x')^{q-(k+1)}|\det \partial_{(\varkappa,\hat{\mu'}_I)}(Y'_{n-1},Y'_I)|.
\end{align}
\end{enumerate}

\end{enumerate}

In all cases above, the estimates near $\bfs\cap\rb$ are required with $x,x'$ exchanged and $\sigma$ replaced by $\theta=x'/x$.
\end{definition}

\begin{remark} \label{remark: admissible-convex-version}
Now we explain why we call the condition in Definition~\ref{def:relaxed-admissible} a relaxed version of the admissible condition in Definition~\ref{definition:Lbf-boundary-admissible}.
The determinant lower bound is satisfied by Proposition~\ref{prop:Y'I-derivative;low-high-region1-region-def} and Proposition~\ref{prop:lower-bound-hessian-high-region2}.
The bound of derivatives can be obtained by Taylor expanding $Y'_I,Y'_{n-1}$ in $s$ or $x'$ around $0$ depending on the region, and noticing that the leading order term vanishes under the admissible condition in Definition~\ref{definition:Lbf-boundary-admissible}.
In fact, for the lower bound of the determinant to hold, we only need the following convexity of the boundary of $\hat{L}^{\bfs,\calchigh}$.
\begin{itemize}
\item For $Y'_I,\, Y'_{n-1}$ as in \eqref{eq:hatLbf-high-components-region1}, there is a constant $c>0$ such that
\begin{align} \label{eq:defn-relaxed-condition-YI-Yn-high-region1}
\begin{pmatrix}
 \partial_sY'_{n-1} & \partial_{\hat{\mu'}_I}Y'_{n-1}
 \\ \partial_s Y'_I & \partial_{\hat{\mu'}_I}Y'_{I}
\end{pmatrix} \Big|_{s=0}
\leq
c \begin{pmatrix}
-1 & 0
\\ 0 & 0
\end{pmatrix}.
\end{align}

\item  For $Y'_I,\, Y'_{n-1}$ as in \eqref{eq:hatLbf-high-components-region2}, we require
\begin{align} \label{eq:defn-relaxed-condition-YI-Yn-high-region2}
\begin{pmatrix}
 \partial_\varkappa Y'_{n-1} & \partial_{\hat{\mu'}_I}Y'_{n-1}
 \\ \partial_\varkappa Y'_I & \partial_{\hat{\mu'}_I}Y'_{I}
\end{pmatrix} \Big|_{x'=0}
\leq 0.
\end{align}
\end{itemize}
This convexity or the determinant bound \eqref{eq:determinant-low-1} is used to estimate the volume of the part near the critical point of the phase, which is crucial for the decay of the oscillatory integral. As shown by the counter-example in \cite[Section~6]{Cowling-etal-damping} in a slightly different setting, the volume estimate for the region near the critical set and the corresponding decay estimate of the oscillatory would fail if this convexity is violated. See also the discussion in \cite{damping-osc-Gressman}.
\end{remark}

We now return to the proof of Proposition~\ref{prop:conic-pair-pointwise-bound-low}, Proposition~\ref{prop:conic-pair-pointwise-bound-high-region1}, and Proposition~\ref{prop:conic-pair-pointwise-bound-high-region2}.
To be definite, we give details in the low-energy setting, i.e., Proposition~\ref{prop:conic-pair-pointwise-bound-low} and sketch minor changes needed for Proposition~\ref{prop:conic-pair-pointwise-bound-high-region1} and Proposition~\ref{prop:conic-pair-pointwise-bound-high-region2} after the proof.

\begin{proof}[Proof of Proposition~\ref{prop:conic-pair-pointwise-bound-low} under the relaxed admissible condition]

Recall that we only need to prove \eqref{eq:conic-osc-est-2}:
\begin{align*}
\Big|  \int_0^\infty \int_{\R^{k}}  e^{i\lambda s \sigma \psi(\sigma,y,y',v,s)/x} 
  s^{\frac{n}{2}-1+ \frac{k}{2}} a(\lambda,\sigma,y,y',\frac{x'}{\lambda},v,s) \, dv \, ds \Big| \lesssim \big(\frac{\lambda}{x'}\big)^{-(k+1)/2} .
\end{align*} 
If $\lambda/x'\lesssim1$, then this estimate is straightforward. Hence we can assume $\lambda/x'\gg1$.

Let $\mk{I}$ denote the left hand side of \eqref{eq:conic-osc-est-2}. We decompose $\mk{I}$ to be 
\begin{align}
\mk{I}=\mk{I}_1+\mk{I}_2
\end{align}
as in \eqref{eq:mkI-decomposition-1} and \eqref{eq:def-mkI-1,2}.
The estimate of $\mk{I}_1$ is unchanged:
\begin{align}
|\mk{I}_1|\lesssim \int_0^{2(\lambda/x')^{-1/2}}s^{\frac n2-1+\frac k2}\,ds\lesssim (\lambda/x')^{-\frac{n+k}{4}}\lesssim(\lambda/x')^{-\frac{k+1}{2}},
\end{align}
where the last inequality uses $k\le n-2$. It remains to estimate $\mk{I}_2$, where $s\gtrsim(\lambda/x')^{-1/2}$.

First consider the central region. For $0<\delta\lesssim1$, consider
\begin{align}
|\partial_s(s\psi)|+|\partial_v\psi|\le\delta.
\end{align}
By \eqref{eq:v-critical-fixed-s}, we have $(\partial_s(s\psi),\partial_v\psi)=(y'_{n-1}-Y'_{n-1},y'_I-Y'_I)$, which gives
\begin{align}
|\det \partial_{(s,v)}(\partial_s(s\psi),\partial_v\psi)|=|\det \partial_{(s,v)}(Y'_{n-1},Y'_I)|.
\end{align}
By \eqref{eq:determinant-low-1}, we have
\begin{align}
\begin{split}
\int_{\substack{s\gtrsim(\lambda/x')^{-1/2}\\ |\partial_s(s\psi)|+|\partial_v\psi|\le\delta}}s^{\frac n2-1+\frac k2}\,dv\,ds &\lesssim \int_{|\partial_s(s\psi)|+|\partial_v\psi|\le\delta}s^{\frac n2-1+\frac k2}|\det \partial_{(s,v)}(Y'_{n-1},Y'_I)|^{-1}\,d(\partial_s(s\psi))\,d(\partial_v\psi)\\
&\lesssim \int_{|\partial_s(s\psi)|+|\partial_v\psi|\le\delta}s^{\frac n2-1-\frac k2}\,d(\partial_s(s\psi))\,d(\partial_v\psi)
\; \lesssim \delta^{k+1},    
\end{split}
\end{align}
where we used $\frac{n}{2}-1-\frac{k}{2} \geq 0$ by Proposition~\ref{prop:minimal-parametrization-low-conic}.
Taking $\delta=(\lambda/x')^{-1/2}$ gives the desired contribution from this critical region.

It remains to estimate the complement. Decompose it dyadically by
\begin{align}
|\partial_s(s\psi)|+|\partial_v\psi|\sim\delta,\qquad \delta=2^j(\lambda/x')^{-1/2},\qquad j\ge0.
\end{align}
On the $j$-th dyadic region, we integrate by parts using the operator
\begin{align} \label{eq:L-def-appendix}
L=(\frac{i\lambda}{x'})^{-1}\frac{\partial_s(s\psi)\partial_s+s^{-1}\partial_v\psi\cdot\partial_v}{|\partial_s(s\psi)|^2+|\partial_v\psi|^2},
\end{align}
which satisfies $L e^{i\lambda s\sigma\psi/x}=e^{i\lambda s\sigma\psi/x}$.
Rewrite $L$ as
\begin{align}
L=(\frac{i\lambda}{x'})^{-1}s^{-1}\frac{\partial_s(s\psi)s\partial_s+\partial_v\psi\cdot\partial_v}{|\partial_s(s\psi)|^2+|\partial_v\psi|^2}.
\end{align}
Thus each application of $L^*$ gives one factor $(\lambda/x')^{-1}$ and one explicit $s^{-1}$ loss, and all $s$-derivatives are counted as conormal derivatives $s\partial_s$. When $s\partial_s$ hits the explicit factor $s^{-1}$, it preserves the same $s^{-1}$ order.

On the dyadic region,
\begin{align}
|\partial_s(s\psi)|^2+|\partial_v\psi|^2\sim\delta^2.
\end{align}
After applying $(L^*)^N$, every term is bounded by a finite sum of terms of the form
\begin{align}
C_N(\lambda/x')^{-N}\delta^{-N-q}s^{-N}s^{\frac n2-1+\frac k2}\prod_{\ell=1}^q |(s\partial_s)^{m_\ell}\partial_v^{\beta_\ell}D_\ell Y'_{\alpha_\ell}|,
\end{align}
where $0\le q\le N,\; Y'_{\alpha_\ell}\in\{Y'_{n-1},Y'_{I,1},...,Y'_{I,k}\},\; D_\ell\in\{s\partial_s,\partial_{v_1},...,\partial_{v_k}\},\; 
\sum_\ell m_\ell+|\beta_\ell|\le N$.

Indeed, each application of $L^*$ contributes the base factor comparable to $(\lambda/x')^{-1}s^{-1}\delta^{-1}$ when all derivatives are falling on the amplitude. 
Whenever a derivative hits $\partial_s(s\psi)$ or $\partial_v\psi$, one uses \eqref{eq:v-critical-fixed-s} and obtains one of the displayed derivative factors. Whenever a derivative hits the denominator or a dyadic cutoff depending on $|\partial_s(s\psi)|+|\partial_v\psi|$, one obtains the same type of factor and one additional power $\delta^{-1}$. Derivatives falling on the amplitude are harmless, and derivatives falling on previously produced factors only increase $m_\ell$ or $\beta_\ell$.

Now integrate the displayed bound over the dyadic region and apply the area formula to the map $(s,v)\mapsto(\partial_s(s\psi),\partial_v\psi)$. The Jacobian contributes $|\det \partial_{(s,v)}(Y'_{n-1},Y'_I)|^{-1}$.
For $q\ge1$, the condition in Definition~\ref{def:relaxed-admissible} gives
\begin{align}
\frac{\prod_{\ell=1}^q |(s\partial_s)^{m_\ell}\partial_v^{\beta_\ell}D_\ell Y'_{\alpha_\ell}|}{|\det \partial_{(s,v)}(Y'_{n-1},Y'_I)|}\lesssim s^{q-k}.
\end{align}
For $q=0$, the same conclusion is interpreted as $|\det \partial_{(s,v)}(Y'_{n-1},Y'_I)|^{-1}\lesssim s^{-k}$, which follows from the determinant lower bound in the definition. Therefore the weighted integrand in the variables $(\partial_s(s\psi),\partial_v\psi)$ is bounded by
\begin{align}
s^{\frac n2-1+\frac k2}s^{-N}s^{q-k}=s^{\frac{n-k-2}{2}+q-N}.
\end{align}
Since $s\gtrsim(\lambda/x')^{-1/2}$ and $q\le N$, we have
\begin{align}
s^{\frac{n-k-2}{2}+q-N}\lesssim(\lambda/x')^{\frac{N-q}{2}}.
\end{align}
The image of the $j$-th dyadic region under $(s,v)\mapsto(\partial_s(s\psi),\partial_v\psi)$ has volume $O(\delta^{k+1})$ with $\delta=2^j(\lambda/x')^{-1/2}$. 
Hence the contribution of such a term is bounded by
\begin{align}
C_N(\lambda/x')^{-N}\delta^{-N-q}\delta^{k+1}(\lambda/x')^{\frac{N-q}{2}}=C_N(\lambda/x')^{-\frac{k+1}{2}}2^{-j(N+q-k-1)}.
\end{align}
Since $N>k+1$ and $q\ge0$, the sum over $j$ converges and is bounded by:
\begin{align}
\sum_{j\ge0}C_N(\lambda/x')^{-\frac{k+1}{2}}2^{-j(N+q-k-1)}\lesssim(\lambda/x')^{-\frac{k+1}{2}}.
\end{align}
Combining the estimate of $\mk{I}_1$, the central region estimate, and the dyadic annulus estimate gives
\begin{align}
\Big|\int_0^\infty\int_{\R^k}e^{i\lambda s\sigma\psi/x}s^{\frac n2-1+\frac k2}a(\lambda,\sigma,y,y',\frac{x'}{\lambda},v,s)\,dv\,ds\Big|\lesssim(\lambda/x')^{-\frac{k+1}{2}}.
\end{align}
This is \eqref{eq:conic-osc-est-2}, hence \eqref{eq:conic-pair-integral-bound-low} follows. The region near $\bfs\cap\rb$ is obtained by interchanging $x,x'$ and replacing $\sigma$ by $\theta=x'/x$. This proves Proposition~\ref{prop:conic-pair-pointwise-bound-low} under the condition in Definition~\ref{def:relaxed-admissible}.
\end{proof}

The proof of Proposition~\ref{prop:conic-pair-pointwise-bound-high-region1} and Proposition~\ref{prop:conic-pair-pointwise-bound-high-region2} follows from the same proof, except that we instead use the corresponding conditions in Definition~\ref{def:relaxed-admissible} and we have the following minor modifications. 
For Proposition~\ref{prop:conic-pair-pointwise-bound-high-region1}, due to the $\varrho$-dependence, $s\frac{d\cdot}{ds}$ acting on $Y'_{n-1},Y'_I$ becomes $s\partial_s-\varrho\partial_{\varrho}$, and all steps remain the same. 
For Proposition~\ref{prop:conic-pair-pointwise-bound-high-region2}, there is no special direction $s$ anymore and we treat all directions isotropically. The proof follows in the same way using the lower bound of the determinant and the upper bound of derivatives in the last part of Definition~\ref{def:relaxed-admissible}.
As pointed out at the beginning of this appendix, using Proposition~\ref{prop:conic-pair-pointwise-bound-low} under this relaxed condition, we can apply the argument in Section~\ref{sec:dispersive-Schrodinger} and Section~\ref{sec:dispersive-wave} to produce our main results under this relaxed condition.

\section{Index of notations}
\label{sec:index-notation}

Background spaces and blow-down maps:
\begin{itemize}
\item $X$: an asymptotically conic manifold. $X_0$: an exact cone. See the discussion in Section~\ref{subsec:set-up}.
    \item  $X_{\rmb}^2$: the b-double space, defined in  \eqref{eq:b-double-space}. 
    \item $\beta_{\mathrm{b}}$: the blow-down map from $X_{\rmb}^2$ to $X^2$, see \eqref{eq:beta-b}.
   \item  $X_{\rmb,\flat}^2$: the low-energy space. See \eqref{eq:def-low-space}.
   \item $X_{\rmb, \calchigh}^2$: the high-energy space, see \eqref{eq:def-high-energy-space}.
   \item $X_{\calc}^2$: the combined space, see \eqref{eq: definition, combined double space}.
\end{itemize}

Cotangent bundles:
\begin{itemize}
\item ${}^{\sct}T^*X$: the scattering cotangent bundle, see the discussion before \eqref{eq:sc-frame}.
\item ${}^{\Phi}T^* X_{\mathrm{b}}^2$: the lifted scattering cotangent bundle, see \eqref{def:pulled-back-bundle}. 
\item ${}^{\Phi}N^*Z_{\lb}$, ${}^{\Phi}N^*Z_{\rb}$: the bundle introduced via boundary fibrations, see \eqref{eq:NZ-lb-def}.
    \item ${}^{\calc}\overline{T}^*X$: the compactified combined cotangent bundle, see \eqref{eq:compactified-combined-cotangent-bundle}.
\end{itemize}

Variables:
\begin{itemize}
    \item $(x,y)$: coordinates on $X$ or $X_0$ near infinity. $\{x=0\} = \partial X$ at infinity and $y$ is a coordinate system on $\partial X = Y$.
    \item $(\tau,\mu)$: scattering frequencies, see \eqref{eq:sc-1-form}.
    \item $\sigma,\theta$: coordinates on $X_{\rmb}^2$ and $X_{\calc}^2$ parametrizing the front-face introduced by the blow up, $\sigma = x/x'$ and $\theta = x'/x$ in the interior.
    \item $(\overline{\nu},\nu_1,\mu,\mu')$: frequencies in the low-energy regime near $\LB \cap \BFS$. \eqref{eq:tautological-1-form-bf}.
    \item $(z,z',\zeta,\zeta',\tau)$: position and frequency variables in the interior, see \eqref{eq:semiclassical-form-interior}. 
    \item $(\overline{\nu},\nu_1,\mu,\mu',\overline{\tau})$: frequencies used in the high-energy regime near $\LB \cap \BFS$, see \eqref{eq:semiclassical-form-bf-lb}.
    \item $Y'_I,Y'_{n-1}$: see \eqref{eq:hatLbf-low-components}\eqref{eq:hatLbf-high-components-region1}\eqref{eq:hatLbf-high-components-region2}.
    \item $\overline{\lambda}, \overline{x}, \overline{x'}$: rescaled variables, see \eqref{eq:defn-rescaled-lambda-x-x'}.
\end{itemize}

Boundary surfaces:
\begin{itemize}
    \item $\zf, \lb_0, \bfs_0, \rb_0$: boundary surfaces at zero energy, see \eqref{eq:defining-functions-low-1}.
    \item $\LB, \BFS, \RB, \smf$: see \eqref{eq:defining-functions-low-3}.
    \item $\rho_\bullet$ with $\bullet$ being one of the boundary surfaces above: defining function of $\bullet$, see \eqref{eq:defining-functions-low-1}\eqref{eq:defining-functions-low-3}.
\end{itemize}

Indices:
\begin{itemize}
    \item $\IF$: the maximal rank drop including the  intersecting conic pair part, see \eqref{eq:IF-full-def}.
    \item $\IFint$: the maximal rank drop excluding the  intersecting conic pair part, \eqref{eq:IFint-def}.
    \item $\IFz,\IFintz$: the analogue of $\IF$ and $\IFint$ defined for exact cones, see \eqref{eq:IFz-def}\eqref{eq:IFintz-def}.
    \item $k_\ell$: the rank drop of the projection from $\mk{G}_\ell$, see the discussion after     \eqref{eq:G-ell-01}.
\end{itemize}

Index sets
\begin{itemize}
    \item $\Jlow$:  see \eqref{eq:sum-chi-j-1}. \quad
    $\Jhigh$: see \eqref{eq:def-Jhigh}.
\end{itemize}

Metrics and operators:
\begin{itemize}
    \item $g_0,\Ymetric$: Metrics on the exact cone and its cross section, see \eqref{eq:def-exact-conic-metric}.
    \item $g,\XYmetric$: Metrics on the asymptotically conic manifold and its cross section, see \eqref{conic_metric_2}.
    \item $P$: the operator on $X$ or $X_0$, see \eqref{eq:P-def} and \eqref{eq:P-def-exact-cone}. 
    \item $\mk{g}$: see \eqref{eq:metric-LCP-parametrization}.
\end{itemize}
Distance functions:
\begin{itemize}
\item $\mk{d}_{X}$: defined before \eqref{eq:mkd-X}.
\item $\mk{d}_{X_0}$: defined before \eqref{eq:mkd-X0}.
\end{itemize}

Density bundles:
\begin{itemize}
    \item $\Omega_{\flat}^{1/2}$: the low-energy density bundle, see \eqref{eq:low-energy-density-1}\eqref{eq:low-energy-density-2}\eqref{eq:low-energy-density-3}.
    \item $\Omega_{\calchigh}^{1/2}$: the high-energy density bundle, see \eqref{eq:high-energy-half-density}.
\end{itemize}

Pseudodifferential algebras
\begin{itemize}
    \item $\Psi_{\flat}^{m}(X;\Omega_{\flat}^{1/2})$: the low-energy pseudodifferential algebra, see \eqref{eq:PsiDO-low-def}.
    \item $\Psi_{\sct,h}^{m,0,0}(X;{}^{\Phi}\Omega^{1/2})$: the high-energy (semiclassical) pseudodifferential algebra, see \eqref{eq:PsiDO-high-def-1}\eqref{eq:PsiDO-high-def-2}.
    \item $\Psi_{\calc}^{m,l,k}$: the combined pseudodifferential algebra, see Definition~\ref{def:combined-PsiDO}.
\end{itemize}

Wavefront sets:
\begin{itemize}
    \item $\WF'_{\flat}(A)$: the low-energy wavefront set, see Definition~\ref{def:low-energy-WF}.
    \item $\WF'_{\sct,h}(A)$: the high-energy wavefront set, see Definition~\ref{def:high-WF}.
    \item $\WF'_{\calc}(A)$: the combined wavefront set, see \eqref{eq:combined-WF-def}.
\end{itemize}

Legendre distributions:
\begin{itemize}
\item $I^{m,r_\lb,r_\rb}(X_{\rmb}^2,\lambda L;{}^{\Phi}\Omega^{1/2})$: the class of Legendre distributions at fixed energy level $\lambda$, see  Definition~\ref{def:Legendre-fixed-energy}.
\item $I_{\calclow}^{m,r_{\LB},r_{\RB};\mathcal{B}}(X_{\rmb,\flat}^2,L ; \Omega_{\flat}^{1/2})$: the class of Legendre distributions associated with the low-energy part, see Definition~\ref{def:Legendre-distribution-low-energy}.
\item $I_{\calclow}^{m, p; r_{\LB}, r_{\RB}; \mathcal{B}}(X_{\rmb,\flat}^2, (L^{\bfs}, L^{\#}); \Omega_{\flat}^{1/2})$: the class of Legendre distributions associated with a conic intersecting pair, see Definition~\ref{def:Legendrian-dis-conic-intersecting-low}.
\item $I^{m_{\calchigh},m ,r_{\LB},r_{\RB}}_{\calchigh}(X_{\rmb,\calchigh}^2,G; \Omega_{\calchigh}^{1/2})$: the class of Legendre distributions in the high-energy regime, see Definition~\ref{def:Legendre-distribution-high-energy}.

\item $I_{\calchigh}^{m_{\calchigh},p;r_{\LB},r_{\RB}}(X_{\rmb,\calchigh}^2,(L^{\bfs,\calchigh},G_1^{\calchigh});\Omega_{\calchigh}^{1/2})$: the class of Legendre distributions associated with a conic intersecting pair in the high-energy case, see Definition~\ref{def:Legendre-dis-conic-high-RB-LB}.
\end{itemize}

Legendre submanifolds:
\begin{itemize}
    \item $L^{\bfs}$: the Legendre submanifold of ${}^{\Phi}T_{\bfs}^*X_{\rmb}^2$ used for the low-energy part for asymptotically conic manifolds and all energy levels for exact cones, see \eqref{eq:Lbf-definition-gamma^2}.
    \item $L_{j,j'}^{\bfs}$: microlocalized $L^{\bfs}$, see \eqref{eq:Lbf-low-jj'}.
    \item $\hat{L}^{\bfs}$: the resolved $L^{\bfs}$, see \eqref{eq:hat-Lbf-def}. 
    \item $\Lsharplow$, $\Lsharplow_\pm$: the Legendre submanifold capturing purely outgoing or incoming oscillations, see 
    \eqref{eq:Lsharplow-def}\eqref{eq:Lsharplow-pm-def}.
    \item $(L^{\bfs}, L^{\#})$: the conic intersecting pair at low-energy, see the discussion after \eqref{eq:hat-Lbf-def}.
    
    \item $L^{\bfs,\calchigh}$: the Legendre submanifold used for the high-energy regime, capturing the propagation in the interior of $X$, see \eqref{eq:L-high-interior}.
    \item $\Lsharphigh$: the Legendre submanifold of ${}^{\Phi}T_{\smf}^*X_{\rmb,\calchigh}^2$ capturing purely outgoing or incoming oscillations,  see \eqref{eq:def-Lsharphigh}.
    \item $\hat{L}^{\bfs,\calchigh}$: the blown up version of $L^{\bfs,\calchigh}$, see \eqref{eq:hatLbf-high-def}.
    \item $(L^{\bfs,\calchigh},\Lsharphigh)$: the conic intersecting pair in the high-energy regime, see the discussion after \eqref{eq:hatLbf-high-def}.
\end{itemize}

Other submanifolds or neighbourhoods:
\begin{itemize}
\item $\Sigma$: the characteristic variety, see \eqref{eq:def-Sigma}.
    \item $\hat{L}^{\bfs}(\sigma_0,y_0)$: The part of $\hat{L}^{\bfs}$ with fixed $(\sigma,y)$ components, see the discussion before Proposition~\ref{prop:Lbf-RYstar-diffeomorphic}.
    \item $\RYstar$: see \eqref{eq:scr-RY-*-definition}.
    \item $\RY$: see \eqref{eq:scr-R-definition}.
    \item $\Ryz$: the part of $\RY$ with fixed left position, see \eqref{eq:Ryz-definition}.
    \item $\mk{G}_{\ell}$: see discussion around \eqref{eq:G-ell-01}.
\end{itemize}

Maps:
\begin{itemize}
  \item $\hatLprojhigh$: the projection from $L^{\bfs,\calchigh}$ to $\smf$: see \eqref{eq:def-Lprojhigh}.
    \item $\hatLprojlow$: the projection from $\hat{L}^{\bfs}$ to $\bfs$, see \eqref{eq:def-Lprojlow}.
    \item $\Pi_{\calclow}$: the projection from $\RYstar$ to $Y \times Y$, see \eqref{eq:def-Pi-Y}.
    \item $\Rprojlow$: the projection from $\RY$ to $Y \times Y$, see \eqref{eq:Rprojlow-def}.
    \item $\tilde{\Pi}_{\calclow,2}$: the projection from $\Ryz$ to $Y$, see \eqref{eq:def-proj-calclow-right}. 
\end{itemize}

\bibliographystyle{plain}
\bibliography{dispersive-conjugate}

\end{document}